\documentclass[fontsize=11pt,paper=A4,
	fleqn,
	bibliography=totoc 
	]{scrbook}

\usepackage[ngerman]{babel} 
\usepackage[T1]{fontenc}
\usepackage[latin1]{inputenc} 

\usepackage{chngcntr} 

\usepackage{amsmath,  
            amssymb,  
            amsthm,   
            amsfonts, 
            }
\usepackage{mathtools}
\usepackage{braket}   
\usepackage{extpfeil} 

\usepackage{multirow}
\usepackage{bigdelim}

\allowdisplaybreaks[3]

\newcommand{\IN}{\ensuremath{\mathbb{N}}}
\newcommand{\IZ}{\ensuremath{\mathbb{Z}}}
\newcommand{\IQ}{\ensuremath{\mathbb{Q}}}
\newcommand{\IR}{\ensuremath{\mathbb{R}}}
\newcommand{\IC}{\ensuremath{\mathbb{C}}}

\newcommand{\abs}[1]{{\left\lvert#1\right\rvert}}
\newcommand{\ceil}[1]{\left\lceil #1 \right\rceil}
\newcommand{\floor}[1]{\left\lfloor #1 \right\rfloor}

\newcommand{\isomorphic}{\cong}

\DeclareMathOperator{\tr}{tr}
\DeclareMathOperator{\sgn}{sgn}
\DeclareMathOperator{\id}{id}
\DeclareMathOperator{\diag}{diag}
\DeclareMathOperator{\ord}{ord}

\DeclareMathOperator{\CharFld}{char}
\DeclareMathOperator{\QuotFld}{Quot}

\DeclareMathOperator{\im}{im}
\DeclareMathOperator{\rad}{rad}
\DeclareMathOperator{\Hom}{Hom}
\DeclareMathOperator{\End}{End}
\DeclareMathOperator{\Irr}{Irr}

\DeclareSymbolFont{extraup}{U}{zavm}{m}{n}
\DeclareMathSymbol{\varheart}{\mathalpha}{extraup}{86}
\DeclareMathSymbol{\vardiamond}{\mathalpha}{extraup}{87}

\newcounter{theoremnumber}
\numberwithin{theoremnumber}{section}
\newtheoremstyle{dotless} 
			{\bigskipamount}    
			{0.0em}             
			{\nopagebreak}      
			{}                  
			{\bfseries}         
			{:}                 
			{\newline}          
			{}                  
\newtheoremstyle{dotless2} 
			{\bigskipamount}    
			{0.0em}             
			{}                  
			{}                  
			{\bfseries}         
			{:}                 
			{0.5em}             
			{}                  
\swapnumbers

\theoremstyle{dotless2}
\newtheorem{remark}[theoremnumber]{}
\theoremstyle{dotless}
\newtheorem{theorem}[theoremnumber]{Satz}
\newtheorem{theoremdef}[theoremnumber]{Satz und Definition}
\newtheorem{lemma}[theoremnumber]{Lemma}
\newtheorem{lemmadef}[theoremnumber]{Lemma und Definition}
\newtheorem{corollary}[theoremnumber]{Korollar}

\newtheorem{definition}[theoremnumber]{Definition}
\newtheorem{example}[theoremnumber]{Beispiel}
\newtheorem{conjecture}[theoremnumber]{Vermutung}

\newtheorem*{convention}{Vereinbarung}

\numberwithin{table}{section}
\numberwithin{figure}{section}

\usepackage{tikz}
\usetikzlibrary{decorations.markings}
\usetikzlibrary{matrix,arrows}
\usetikzlibrary{calc}

\tikzset
{
	halfarrow/.style=
	{
		postaction={decorate},
		decoration={markings,mark=at position 0.8 with {\arrow{to}}}
	},
	arrow/.style=
	{
		->,very thick,font=\scriptsize
	},
	desc/.style=
	{
		fill=white,inner sep=2pt,font=\scriptsize
	},
	Vertex/.style =
	{
		inner sep = 1pt,
		outer sep = 2pt,
		minimum size=15pt,
		circle,
		draw=black!70,
		thick,
		font=\scriptsize
	},
	EdgeT/.style =
	{
		ultra thick,
		draw=black,
		font=\scriptsize
	},
	EdgeTdir/.style =
	{
		->,
		draw=black,
		font=\scriptsize
	},
	EdgeI/.style =
	{
		->,
		draw=black!70,
		font=\scriptsize
	},
	EdgeI2/.style =
	{
		->,
		draw=black!50,
	}
}

\usepackage{listings}
\newtheorem{algorithm}{Algorithmus}

\usepackage[numbers]{natbib}
\usepackage[pdfpagelabels=true,plainpages=false]{hyperref}
\hypersetup{
colorlinks=true,
linkcolor=red,
urlcolor=blue,
citecolor=blue,
linktocpage=true, 
pdfinfo=
	{  
		Title={Gyojas W-Graph-Algebra und zelluläre Struktur von Iwahori-Hecke-Algebren},
		Author={Johannes Hahn},
		Keywords={coxeter groups, hecke algebra, cellular algebra, asymptotic algebra, W-graphs, W-graph algebra},
		Subject={representation theory}
	},
pdfpagelayout={OneColumn},
pdfstartview= 
}

\usepackage{imakeidx}
\makeindex[name=terms,title=Stichwortverzeichnis,intoc,columns=2,
	program=makeindex,options=-s terms.ist]
\makeindex[name=symbols,title=Symbolverzeichnis,intoc,columns=2,
	program=makeindex,options=-s symbols.ist]
\indexsetup{level=\chapter*,toclevel=chapter,headers={\indexname}{\indexname}}
\let\imkiindex\index
\renewcommand\index[1]{\imkiindex[#1]}

\usepackage[normalem]{ulem}
\newcommand*{\udot}{\dotuline}

\usepackage{enumitem} 

\title{Zelluläre Struktur von Iwahori-Hecke-Algebren und Gyojas $W$-Graph-Algebra}
\author{Johannes Hahn \\ Friedrich-Schiller-Universität Jena}

\begin{document}

\frontmatter
\begin{titlepage}

\newcommand{\HRule}{\rule{\linewidth}{0.5mm}} 

\center 
 

\HRule \\[0.4cm]
\renewcommand{\baselinestretch}{1.75}\normalsize
{ \huge \bfseries Gyojas $W$-Graph-Algebra und zelluläre Struktur von Iwahori-Hecke-Algebren}\\[0.4cm] 
\renewcommand{\baselinestretch}{1.00}\normalsize
\HRule \\[1.5cm]
 

\textsc{\Large Dissertation}\\[0.5cm] 
\textsc{\large zur Erlangung des akademischen Grades \linebreak doctor rerum naturalium (Dr.\,rer.\,nat.)}\\[0.5cm] 

\vfill
vorgelegt dem Rat der Fakultät für Mathematik und Informatik der Friedrich-Schiller-Universität Jena von\linebreak
\textsc{\small Dipl.\,Math. Johannes Hahn,\linebreak geboren am 20.01.1987 in Rostock}

\pagebreak

\vspace*{\fill}
\begin{minipage}{\textwidth}
\bfseries Gutachter:
\begin{enumerate}[label=\arabic*.]
	\item Prof.\,Dr.\,Burkhard Külshammer, Universität Jena
	\item PD.\, Dr.\,Jürgen Müller, Universität Jena
	\item Prof.\,Dr.\,Meinolf Geck, Universität Stuttgart
\end{enumerate}
Tag der öffentlichen Verteidigung: Di., 27.02.2014
\end{minipage}
\end{titlepage}
\cleardoublepage
\chapter*{Danksagung}

Ich bedanke mich ganz herzlich bei Prof.\ Dr. Burkhard Külshammer und PD\ Dr. Jürgen Müller für ihre fachliche Unterstützung während der Arbeit an dieser Dissertation, Prof.\ Dr. Meinolf Geck für das Aufdecken eines kritischen Fehlers in der ersten Version des Manuskripts sowie der Deutschen Forschungsgemeinschaft für die Finanzierung des Projekts "`Computing with Hecke algebras"' im Rahmen dessen diese Dissertation entstanden ist.

Zu großem Dank für die Hilfe bei der Korrektur beider Versionen des Manuskripts verpflichtet fühle ich mich auch Martin Brandenburg, Rene Marczinzik, Dominic Michaelis sowie Sabrina Gemsa.

\cleardoublepage
\phantomsection
\addcontentsline{toc}{chapter}{\contentsname}
\tableofcontents

\mainmatter

\addchap{Einleitung} 

Seit mehr als einem Jahrhundert ist die Klassifikation der einfachen Lie"=Algebren über den komplexen Zahlen durch ihre Wurzelsysteme bekannt. Ein Wurzelsystem ist im Wesentlichen eine Menge von Geraden im $\IR^n$, deren gegenseitige Lage durch gewisse Axiome stark eingeschränkt wird und so die Klassifikation erlaubt. Man kann jedem Wurzelsystem $\Phi$ auf natürliche Weise eine Gruppe zuordnen: Die orthogonalen Komplemente (bzgl. einer symmetrischen, nichtentarteten Bilinearform) zu den durch $\Phi$ bestimmten Geraden sind Hyperebenen und die Spiegelungen an diesen Hyperebenen erzeugen eine Untergruppe der orthogonalen Gruppe. Diese Konstruktion liefert zum Beispiel die sogenannten Weyl"=Gruppen der Lie"=Algebren, in denen eine Vielzahl von geometrischen, gruppentheoretischen, darstellungstheoretischen und kombinatorischen Informationen über die Lie"=Algebra und die zugehörigen Lie"=Gruppen kodiert ist. Aus diesem Grunde waren Weyl"=Gruppen auch schon immer ein zentraler Bestandteil der Theorie der Lie"=Gruppen und in späteren Jahrzehnten auch der Theorie der reduktiven, algebraischen Gruppen.

\bigskip
Ein klassisches Resultat zeigt, dass Spiegelungsgruppen mit den Coxeter"=Gruppen identisch sind, das heißt Gruppen, die sich durch eine spezielle Präsentation mit Erzeugern und Relationen schreiben lassen: Zu jeder Spiegelungsgruppe $W$ gibt es eine (bis auf Konjugation eindeutige) Menge von Spiegelungen $S\subseteq W$ derart, dass $W$ die Präsentation
\[W=\braket{s\in S \mid s^2=1 \ \text{und}\ \underbrace{sts\ldots}_{m_{st}} = \underbrace{tst\ldots}_{m_{st}} \ \text{für alle}\ s,t\in S}\]
hat. Dabei bezeichnet $m_{st}$ die Ordnung von $st$ in $W$. Ist umgekehrt für gewisse Zahlen $m_{st}\in\IN$ eine Gruppe durch eine derartige Präsentation gegeben, dann ist $W$ auch als Spiegelungsgruppe realisierbar.

Aufgrund von engen Beziehungen zur Knotentheorie werden die hierbei auftretenden Relationen vom Typ $\smash{\underbrace{sts\ldots}_{m_{st}} = \underbrace{tst\ldots}_{m_{st}}}$ als Zopfrelationen bezeichnet.

\bigskip
Iwahori"=Hecke"=Algebren oder kurz Hecke"=Algebren sind enge Verwandte der Spiegelungsgruppen. Ist etwa $(W,S)$ eine Coxeter"=Gruppe, dann ist durch diese Daten eine $\IZ[v^{\pm1}]$"~Al\-ge\-bra $H=H(W,S)$ bestimmt, die Hecke"=Algebra von $(W,S)$, welche ebenfalls wichtige Informationen über die Lie- oder reduktive, algebraische Gruppe $G$ kodiert, beispielsweise kohomologische Informationen über Fahnenvarietäten $G/B$, wobei $B$ eine Borel"=Untergruppe von $G$ ist. 

Präziser ist die Hecke"=Algebra $H(W,S)$ definiert als die $\IZ[v^{\pm1}]$-Algebra mit den Erzeugern $(T_s)_{s\in S}$ und den quadratischen Relationen $T_s^2 = 1+(v-v^{-1})T_s$ für alle $s\in S$ und den Zopfrelationen $\underbrace{T_s T_t T_s\ldots}_{m_{st}} = \underbrace{T_t T_s T_t\ldots}_{m_{st}}$ für alle $s,t\in S$. Man erkennt sofort, dass die Präsentation der Gruppe $W$ (genauer der Gruppenalgebra $\IZ[W]$) sich darin wiederfindet, wenn man auf $v=1$ spezialisiert. Es stellt sich heraus, dass, genau wie $\IZ[W]$, auch die Hecke"=Algebra eine Basis hat, die mit den Elementen der Gruppe indiziert ist, üblicherweise mit $T_w$ für $w\in W$ bezeichnet wird und durch die Spezialisierung $v=1$ in die Standardbasis von $\IZ[W]$ übergeht. Die Hecke"=Algebra wird deshalb als "`Deformation"' des Gruppenrings $\IZ[W]$ bezeichnet. In diesem Sinne sind alle Informationen, die in $W$ und $\IZ[W]$ vorhanden sind, auch in $H$ vorhanden und durch den zusätzlichen Parameter $v$ sogar detaillierter aufgeschlüsselt. Daher sind Hecke"=Algebren seit Jahrzehnten ein ebenso wichtiger Bestandteil der Theorie wie die Coxeter-Gruppen selbst.

\bigskip
Die Hecke"=Algebren der symmetrischen Gruppen tragen eine Vielzahl interessanter Strukturen und spielen daher in vielen anderen Kontexten eine wichtige Rolle. Unter anderem war durch die Knuth"=Robinson"=Schensted"=Korrespondenz seit langem eine Eigenschaft bekannt, die durch Graham und Lehrer 1996 unter dem Namen "`zelluläre Algebra"' axiomatisiert und systematisch untersucht wurde (siehe \cite{GrahamLehrer}). Im Jahr 2007 bewies Meinolf Geck, dass nicht nur Hecke"=Algebren der symmetrischen Gruppen, sondern alle Hecke"=Algebren endlicher Coxeter"=Gruppen zelluläre Algebren sind (unter gewissen, harmlosen Zusatzvoraussetzungen, siehe \cite{geck2007hecke} und \cite{geck2009leading} für eine vereinfachte Konstruktion).

Die Zellbasen, die Geck konstruierte, sind von der Wahl bestimmter Matrixdarstellungen als Eingabedaten abhängig, sogenannter balancierter Darstellungen. Dies sind in einem gewissen Sinne Darstellungen von minimaler Komplexität. Sie haben unter anderem die Eigenschaft, dass sie Darstellungen von Lusztigs asympotischer Algebra $J$ induzieren.

\bigskip
Im Jahr 2008 wurden durch Geck und Jürgen Müller die Zerlegungszahlen für die exzeptionellen Spiegelungsgruppen bestimmt, um die Gültigkeit von James' Vermutung in diesen Fällen zu bestätigen (siehe \cite{GeckMueller}). Dabei benutzten sie spezielle Matrixdarstellungen der Hecke"=Algebren, die sogenannten $W$"~Graph"=Darstellungen, die eine verblüffende Ähnlichkeit mit denjenigen Matrixdarstellungen aufwiesen, die sich aus den Eigenschaften zellulärer Algebren ergeben. Es wurde daher vermutet (\citep[2.7.13]{geckjacon} und \citep[4.5]{GeckMueller}), dass in der Tat jede $W$"~Graph"=Darstellung gleich einer Zelldarstellung in Gecks Konstruktion sein müsse für geeignet gewählte Eingabedaten.

\bigskip
$W$"~Graphen sind besonders sparsame, kombinatorisch kodierte Matrixdarstellungen von $H$. Sie wurden zuerst von Kazhdan und Lusztig in ihrem wegweisenden Artikel \cite{KL} definiert. Es ist nicht unmittelbar klar, aber mit Hilfe der Lusztig-Vermutungen (die für viele wesentliche Fälle bereits nachgewiesen wurden) beweisbar, dass es zu jedem Isomorphietyp irreduzibler Darstellungen von $H$ auch mindestens einen $W$"~Gra\-phen gibt, der eine Matrixdarstellung dieses Isomorphietyps kodiert. Dies wurde zuerst von Gyoja in \cite{Gyoja} bewiesen. Gyoja definierte dafür eine Algebra, die in dieser Arbeit unter der Bezeichnung $W$"~Graph"=Algebra verallgemeinert betrachtet wird.

In Gyojas Arbeit taucht auch bereits die Vermutung auf, dass der Quotient der $W$"~Graph"=Algebra nach ihrem Jacobson"=Radikal wieder die Hecke"=Algebra von $W$ zurückliefern sollte. Ich beweise unter anderem, dass die Gültigkeit dieser Vermutung auch die Gültigkeit der Geck"=Müller"=Vermutung über $W$"~Graphen nach sich zieht.

\bigskip
Das Ziel dieser Arbeit ist die genauere Untersuchung und -- in Spezialfällen -- der Beweis dieser Vermutungen. Die Arbeit gliedert sich dabei in vier Kapitel.

Das erste Kapitel dient vor allem der Einführung der notwendigen Begrifflichkeiten aus der Theorie der Coxeter-Gruppen, der Iwahori"=Hecke"=Algebren und der Kazhdan-Lusztig-Theorie.

Das zweite Kapitel definiert das Konzept der balancierten Darstellung. Es wird eine Konstruktion der asymptotischen Algebra $J$ angegeben, die die Konstruktion durch Geck und Jacon in \cite{geckjacon} in einer naheliegenden Weise verallgemeinert. Das führt unter anderem zu der Einsicht, dass die asymptotische Algebra von der Wahl der speziellen Basis $(T_w)_{w\in W}$ unabhängig ist.

Das dritte Kapitel definiert zelluläre Algebren und ihre Zellmoduln. Das Vorgehen orientiert sich dabei sowohl am Originalartikel von Graham und Lehrer (\cite{GrahamLehrer}) als auch an den Arbeiten von Steffen König und Changchang Xi (\cite{KoenigXi_cellalg_1}, \cite{KoenigXi_cellalg_3}, \cite{KoenigXi_cellalg_no_cells}).

Im vierten Kapitel werden dann die neuen Ergebnisse über Gyojas $W$"~Graph"=Algebra präsentiert. Es werden zunächst $W$"~Graphen und verschiedene Versionen von $W$"~Graph"=Algebren definiert, die Gyojas ursprüngliche Definition verallgemeinern. Dann wird eine Verbindung zu Lusztigs asympotischer Algebra aus Kapitel zwei hergestellt, indem gezeigt wird, dass der Lusztig-Isomorphismus $\phi: H\to J$ durch die kanonische Einbettung der Hecke"=Algebra in die $W$"~Graph"=Algebren faktorisiert. Im darauffolgenden Abschnitt wird eine bislang anscheinend unbekannte, alternative Präsentation von Gyojas $W$"~Graph"=Algebra durch Erzeuger und Relationen im Einparameterfall hergeleitet, die sie als Quotient einer Pfadalgebra realisiert. Der letzte Abschnitt des Kapitels hat dann das Ziel, die $W$"~Graph"=Zerlegungsvermutung zu motivieren und sie mit Hilfe der zuvor erhaltenen Präsentation für die Spezialfälle $I_2(m)$, $A_3$, $A_4$ und $B_3$ zu beweisen.

\bigskip
Im Anhang dieser Dissertation werden Algorithmen für einige der im Hauptteil angeschnittenen Probleme und mögliche Optimierungen gegenüber naiven Ansätzen besprochen. Der erste Abschnitt enthält unter anderen einen Algorithmus zur Berechnung von Bruhat"=Intervallen und einen Algorithmus zur Berechnung von Kazhdan-Lusztig-Polynomen. Der zweite Abschnitt enthält einen Algorithmus zum Auffinden einer äquivalenten, balancierten Darstellung zu jeder gegebenen Matrixdarstellung der Hecke"=Algebra. Der dritte Abschnitt gibt dann einen Algorithmus zum Bestimmen der Zelldarstellungen aus Gecks Konstruktion an und liefert so eine Möglichkeit, die Geck"=Müller"=Vermutung in konkreten Fällen computergestützt zu überprüfen.

\chapter{Coxeter-Gruppen und Hecke-Algebren}
\section{Coxeter-Gruppen}

\begin{definition}[Coxeter-Gruppen, siehe {\cite{bourbaki_elements_lie_5}}]
\index{terms}{Coxeter!-Gruppe}\index{terms}{Coxeter!-Matrix}\index{terms}{Dynkin-Diagramm}\index{terms}{Coxeter!-Gruppe!Rang}\index{terms}{Coxeter-Dynkin-Graph|see{Dynkin-Diagramm}}
\index{symbols}{WS@$(W,S)$}\index{symbols}{mst@$m_{st}$}
Sei $S$ eine endliche Menge. Eine Matrix $M=(m_{st})\in(\IN\cup\set{\infty})^{S\times S}$ heißt \udot{Coxeter"=Matrix}, wenn $M$ symmetrisch ist, $m_{ss}=1$ und $m_{st}\geq 2$ für alle $s\neq t$ aus $S$ gilt.

Es sei $W$ eine Gruppe. Wir nennen das Datum $(W,S)$ eine \udot{Coxeter"=Gruppe}, wenn es eine Coxeter-Matrix gibt derart, dass
\[W=\braket{S \mid (st)^{m_{st}}=1 \text{ für alle }s,t\in S}.\]
Dabei wird $(st)^\infty=1$ als die Abwesenheit einer Relation bzw. als triviale Relation verstanden.

Wir nennen $\abs{S}$ \udot{Rang} der Coxeter"=Gruppe.

Der \udot{Coxeter"=Dynkin"=Graph} oder das \udot{Dynkin"=Diagramm} von $(W,S)$ ist der Graph mit Eckenmenge $S$, in dem $s,t\in S$ genau dann durch eine Kante verbunden sind, wenn $m_{st}>2$ ist. Wir entscheiden uns dafür, die ungerichtete Variante dieses Graphen zu benutzen. Wir vereinbaren außerdem, dass die Kante $s-t$ des Graphen mit $m_{st}$ beschriftet sein soll, falls $m_{st}>3$ ist, sodass der Graph die Coxeter"=Matrix und damit die Gruppe eindeutig bestimmt.
\end{definition}

\begin{theorem}[Spiegelungsdarstellung]
\index{terms}{Spiegelungsdarstellung!einer Coxeter-Gruppe}
Mit den Bezeichnungen aus der Definition definiere $V:=\IR^S$, wobei die Basisvektoren mit $e_s$ für $s\in S$ bezeichnet seien. Dann gilt:
\begin{enumerate}
	\item Definiere die Bilinearform $\beta:V\times V\to\IR$ durch $\beta(e_s,e_t) := -\cos\left(\frac{\pi}{m_{st}}\right)$ (wobei $\frac{\pi}{\infty}:=0$ sei). Dann ist durch
	\[\rho(s)v := v - 2\beta(e_s,v)e_s\]
	eine Darstellung $\rho: W\to\End_{\IR}(V)$ definiert, die sogenannte \udot{Spie\-ge\-lungs\-dar\-stel\-lung} von $W$.
	\item $\rho$ ist injektiv.
	\item Die zur Konstruktion verwendete Coxeter-Matrix ist durch $(W,S)$ eindeutig bestimmt. Genauer ist $m_{st}$ gleich der Ordnung von $st$ in $W$ für alle $s,t\in S$.
\end{enumerate}
\end{theorem}
\begin{proof}
Siehe \citep[5.3+5.4]{humphreys1992coxeter}.
\end{proof}

\subsection{Kombinatorik I: Wortkombinatorik}

\begin{definition}[Längenfunktion, reduzierte Wörter]
\index{terms}{Coxeter!-Gruppe!Längenfunktion}\index{terms}{reduzierter Audruck}
\index{symbols}{lw@$l(w)$}
Jedes $w\in W$ lässt sich nach Definition als Produkt $w=s_1 s_2 \ldots s_n$ für geeignete $s_i\in S$ schreiben. Das minimale $n\in\IN$, für das so eine Darstellung existiert, heißt \udot{Länge} von $w$ und wird als $l(w)$ notiert. Ein Ausdruck $w=s_1 s_2 \ldots s_n$ mit $n=l(w)$ heißt \udot{reduzierter Ausdruck} für $w$.
\end{definition}

\begin{theorem}
\index{terms}{Austauscheigenschaft}\index{terms}{Austauscheigenschaft!starke}\index{terms}{Löschungseigenschaft}
Sei $(W,S)$ eine Coxeter"=Gruppe. Es gelten dann folgende Eigenschaften:
\begin{enumerate}
	\item Austauscheigenschaft:

	Sei $w=s_1\ldots s_n\in W$ und $s\in S$ derart, dass $l(sw)<l(w)$. Dann gibt es einen Index $i$ so, dass
	\[sw = s_1 \ldots \widehat{s_i} \ldots s_n.\]
	\item Starke Austauscheigenschaft:
	
	Sei $w=s_1\ldots s_n\in W$ und $t\in\Set{wsw^{-1} | w\in W, s\in S}$ derart, dass $l(tw)<l(w)$. Dann gibt es einen Index $i$ so, dass
	\[tw = s_1 \ldots \widehat{s_i} \ldots s_n.\]
	\item Löschungseigenschaft:

	Sei $w=s_1\ldots s_n\in W$ mit $l(w)<n$. Dann gibt es Indizes $i<j$ so, dass
	\[w = s_1 \ldots \widehat{s_i} \ldots \widehat{s_j} \ldots s_n.\]
\end{enumerate}
Umgekehrt ist jede dieser Eigenschaften für eine Gruppe $W'$, die von einer Menge von Involutionen $S'\subseteq W'$ erzeugt wird, äquivalent dazu, dass $(W',S')$ eine Coxeter"=Gruppe ist.
\end{theorem}
\begin{proof}
Siehe \citep[1.9+5.8]{humphreys1992coxeter}.
\end{proof}

\begin{theorem}[Tits, siehe \cite{tits1969probleme}]
\index{terms}{Satz von!Tits}
Sei $(W,S)$ eine Coxeter"=Gruppe. Auf der Menge $S^\ast$ der formalen Wörter mit Buchstaben aus $S$ definiere folgende Operationen:
\begin{enumerate}
	\item[I] Lösche ein Teilwort der Form $ss$.
	\item[II] Ersetze ein Teilwort der Form $sts\ldots$ durch $tst\ldots$, wobei beide Wörter Länge ${m_{st}<\infty}$ haben.
\end{enumerate}
Mit diesen Bezeichnungen gilt:
\begin{enumerate}
	\item Für jedes $w=s_1 s_2 \ldots s_n\in W$ gibt es dann eine Folge von solchen Operationen, die das Wort $s_1 \ldots s_n\in S^\ast$ in einen reduzierten Ausdruck von $w$ überführen.
	\item Je zwei reduzierte Ausdrücke von $w\in W$ unterscheiden sich nur durch eine Folge von Typ-II-Operationen.
\end{enumerate}
\end{theorem}
\begin{proof}
Siehe \citep[3.4.2]{davis2008geometry}.
\end{proof}

\begin{corollary}[Satz von Matsumoto, siehe \cite{matsumoto1964generateurs}]
\index{terms}{Satz von!Matsumoto}
Sei $(W,S)$ eine Coxeter"=Gruppe und $M$ ein beliebiges Monoid sowie $f: S\to M$ eine Funktion derart, dass für alle $s,t\in S$ mit $m_{st}<\infty$ stets
\[\underbrace{f(s)f(t)f(s)\ldots}_{m_{st}} = \underbrace{f(t)f(s)f(t)\ldots}_{m_{st}}\]
gilt. Dann gibt es genau eine Fortsetzung $F:W\to M$ mit $F(s_1 \ldots s_n)=f(s_1)\ldots f(s_n)$ für alle reduzierten Ausdrücke $s_1\ldots s_n$.
\end{corollary}
\begin{proof}
Siehe z.\,B. \citep[1.2.2]{geckpfeiffer}.
\end{proof}

\begin{remark}
\index{terms}{Zopfrelationen}
Dies ist in der Tat äquivalent zur zweiten Aussage des Satzes von Tits. Die Implikation vom Satz von Tits zum Satz von Matsumoto ist offenkundig, die Rückrichtung folgt, indem man das \udot{Zopf"=Monoid}, also das Monoid mit der Monoidpräsentation
\[B^+(W,S) := \braket{S \mid \underbrace{sts\ldots}_{m_{st}}=\underbrace{tst\ldots}_{m_{st}} \,\text{für alle $s,t\in S$ mit }m_{st}<\infty},\]
für $M$ einsetzt. Aus der Eindeutigkeit der Fortsetzung $F$ folgt dann, dass jedes $w\in W$ ein eindeutiges Element von $B^+(W,S)$ definiert. Das ist genau die Aussage, dass sich je zwei reduzierte Ausdrücke desselben Elements durch Typ"~II"=Operationen ineinander überführen lassen.

Der zitierte Beweis ist jedoch unabhängig vom Satz von Tits.
\end{remark}

\begin{lemma}[Längstes Element]
\index{terms}{Coxeter!-Gruppe!längstes Element}
\index{symbols}{w0@$w_0$}
Ist $(W,S)$ eine endliche Coxeter"=Gruppe, dann existiert genau ein Element $w_0\in W$ maximaler Länge. Es gilt
\begin{enumerate}
	\item $w_0^2=1$ sowie
	\item $l(ww_0)=l(w_0)-l(w)=l(w_0 w)$ für alle $w\in W$.
\end{enumerate}
\end{lemma}
\begin{proof}
Siehe z.\,B. \citep[5.6]{humphreys1992coxeter}
\end{proof}

\subsection{Klassifikation der endlichen Coxeter-Gruppen}

\begin{lemmadef}[Parabolische Untergruppen]
\index{terms}{Parabolische Untergruppe}
\index{symbols}{WJ@$W_J$}
Sei $(W,S)$ eine Coxeter"=Gruppe und $J\subseteq S$ beliebig. Definiere dann die Untergruppe $W_J:=\braket{J}\leq W$. Mit diesen Bezeichnungen ist $(W_J,J)$ eine Coxeter"=Gruppe.
\end{lemmadef}
\begin{proof}
Siehe z.\,B. \citep[5.5]{humphreys1992coxeter}.
\end{proof}

\begin{lemmadef}[Irreduzible Komponenten]
\index{terms}{Coxeter!-Gruppe!irreduzible}\index{terms}{Zusammenhangskomponente}
Eine Coxeter"=Gruppe $(W,S)$ heißt \udot{irreduzibel}, falls das Dynkin"=Diagramm von $(W,S)$ zusammenhängend ist.

Die Zerlegung $S=\coprod_{i=1\ldots n} S_i$ in Zusammenhangskomponenten induziert eine Zerlegung der Gruppe in irreduzible Faktoren $W=\prod_{i=1\ldots n} W_{S_i}$.
\end{lemmadef}
\begin{proof}
Siehe z.\,B. \citep[6.1]{humphreys1992coxeter}
\end{proof}

\begin{theorem}[Klassifikation der endlichen Coxeter-Gruppen]
\index{terms}{Dynkin-Diagramm}\index{terms}{Coxeter!-Gruppe!Klassifikation}\index{terms}{Coxeter!-Gruppe!irreduzible}\index{terms}{Coxeter!-Gruppe!exzeptionelle}
\index{symbols}{An@$A_n$}\index{symbols}{Bn@$B_n$}\index{symbols}{Dn@$D_n$}\index{symbols}{E6E7E8@$E_6$, $E_7$, $E_8$}\index{symbols}{F4@$F_4$}\index{symbols}{G2@$G_2$}\index{symbols}{I2m@$I_2(m)$}\index{symbols}{H3H4@$H_3$, $H_4$}
Die endlichen, irreduziblen Coxeter"=Gruppen sind genau diejenigen mit den Dynkin"=Diagrammen in Abbildung \ref{fig:fin_coxgrps}.
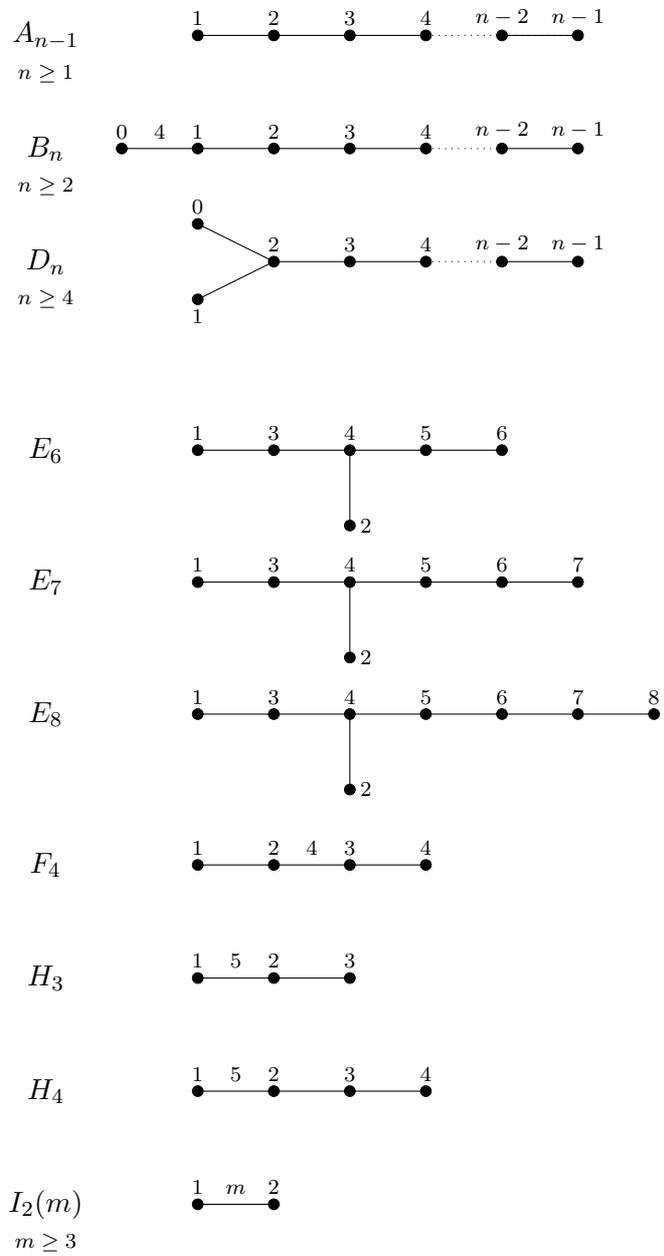
\begin{figure}[p]
	\centering
	\begin{tikzpicture}
\coordinate (A) at (-2,0);
\node at (A) {$A_{n-1}$};
\node[yshift=-0.5cm,font=\scriptsize] at (A) {$n\geq 1$};

\coordinate (A1)  at (0,0);
\coordinate (A2)  at (1,0);
\coordinate (A3)  at (2,0);
\coordinate (A4)  at (3,0);
\coordinate (An2) at (4,0);
\coordinate (An1) at (5,0);

\node[above,font=\scriptsize] at (A1)  {$1$};
\node[above,font=\scriptsize] at (A2)  {$2$};
\node[above,font=\scriptsize] at (A3)  {$3$};
\node[above,font=\scriptsize] at (A4)  {$4$};
\node[above,font=\scriptsize] at (An2) {$n-2$};
\node[above,font=\scriptsize] at (An1) {$n-1$};

\draw[fill=black] (A1)  circle (2pt);
\draw[fill=black] (A2)  circle (2pt);
\draw[fill=black] (A3)  circle (2pt);
\draw[fill=black] (A4)  circle (2pt);
\draw[fill=black] (An2) circle (2pt);
\draw[fill=black] (An1) circle (2pt);

\path[-]
(A1) edge (A2)
(A2) edge (A3)
(A3) edge (A4)
(An2) edge (An1);
\path[dotted]
(A4) edge (An1);

\coordinate (B) at (-2,-1.5);
\node at (B) {$B_n$};
\node[yshift=-0.5cm,font=\scriptsize] at (B) {$n\geq 2$};

\coordinate (B0)  at (-1,-1.5);
\coordinate (B1)  at ( 0,-1.5);
\coordinate (B2)  at ( 1,-1.5);
\coordinate (B3)  at ( 2,-1.5);
\coordinate (B4)  at ( 3,-1.5);
\coordinate (Bn2) at ( 4,-1.5);
\coordinate (Bn1) at ( 5,-1.5);

\node[above,font=\scriptsize] at (B0)  {$0$};
\node[above,font=\scriptsize] at (B1)  {$1$};
\node[above,font=\scriptsize] at (B2)  {$2$};
\node[above,font=\scriptsize] at (B3)  {$3$};
\node[above,font=\scriptsize] at (B4)  {$4$};
\node[above,font=\scriptsize] at (Bn2) {$n-2$};
\node[above,font=\scriptsize] at (Bn1) {$n-1$};

\draw[fill=black] (B0)  circle (2pt);
\draw[fill=black] (B1)  circle (2pt);
\draw[fill=black] (B2)  circle (2pt);
\draw[fill=black] (B3)  circle (2pt);
\draw[fill=black] (B4)  circle (2pt);
\draw[fill=black] (Bn2) circle (2pt);
\draw[fill=black] (Bn1) circle (2pt);

\path[-]
(B0) edge node[above,font=\scriptsize]{$4$} (B1)
(B1) edge (B2)
(B2) edge (B3)
(B3) edge (B4)
(Bn2) edge (Bn1);
\path[dotted]
(B4) edge (Bn2);

\coordinate (D) at (-2,-3);
\node at (D) {$D_n$};
\node[yshift=-0.5cm,font=\scriptsize] at (D) {$n\geq 4$};

\coordinate (D0)  at (0,-2.5);
\coordinate (D1)  at (0,-3.5);
\coordinate (D2)  at (1,-3);
\coordinate (D3)  at (2,-3);
\coordinate (D4)  at (3,-3);
\coordinate (Dn2) at (4,-3);
\coordinate (Dn1) at (5,-3);

\node[above,font=\scriptsize] at (D0)  {$0$};
\node[below,font=\scriptsize] at (D1)  {$1$};
\node[above,font=\scriptsize] at (D2)  {$2$};
\node[above,font=\scriptsize] at (D3)  {$3$};
\node[above,font=\scriptsize] at (D4)  {$4$};
\node[above,font=\scriptsize] at (Dn2) {$n-2$};
\node[above,font=\scriptsize] at (Dn1) {$n-1$};

\draw[fill=black] (D0)  circle (2pt);
\draw[fill=black] (D1)  circle (2pt);
\draw[fill=black] (D2)  circle (2pt);
\draw[fill=black] (D3)  circle (2pt);
\draw[fill=black] (D4)  circle (2pt);
\draw[fill=black] (Dn2) circle (2pt);
\draw[fill=black] (Dn1) circle (2pt);

\draw (D0) -- (D2);
\draw (D1) -- (D2);
\draw (D2) -- (D3) -- (D4);
\draw[dotted] (D4) -- (Dn2);
\draw (Dn2) -- (Dn1);

\coordinate (E6) at (-2,-5.5);
\node at (E6) {$E_6$};

\coordinate (E61) at (0,-5.5);
\coordinate (E62) at (2,-6.5);
\coordinate (E63) at (1,-5.5);
\coordinate (E64) at (2,-5.5);
\coordinate (E65) at (3,-5.5);
\coordinate (E66) at (4,-5.5);

\node[above,font=\scriptsize] at (E61) {$1$};
\node[right,font=\scriptsize] at (E62) {$2$};
\node[above,font=\scriptsize] at (E63) {$3$};
\node[above,font=\scriptsize] at (E64) {$4$};
\node[above,font=\scriptsize] at (E65) {$5$};
\node[above,font=\scriptsize] at (E66) {$6$};

\draw[fill=black] (E61) circle (2pt);
\draw[fill=black] (E62) circle (2pt);
\draw[fill=black] (E63) circle (2pt);
\draw[fill=black] (E64) circle (2pt);
\draw[fill=black] (E65) circle (2pt);
\draw[fill=black] (E66) circle (2pt);

\draw (E61) -- (E63) -- (E64) -- (E65) -- (E66);
\draw (E64) -- (E62);

\coordinate (E7) at (-2,-7.25);
\node at (E7) {$E_7$};

\coordinate (E71) at (0,-7.25);
\coordinate (E72) at (2,-8.25);
\coordinate (E73) at (1,-7.25);
\coordinate (E74) at (2,-7.25);
\coordinate (E75) at (3,-7.25);
\coordinate (E76) at (4,-7.25);
\coordinate (E77) at (5,-7.25);

\node[above,font=\scriptsize] at (E71) {$1$};
\node[right,font=\scriptsize] at (E72) {$2$};
\node[above,font=\scriptsize] at (E73) {$3$};
\node[above,font=\scriptsize] at (E74) {$4$};
\node[above,font=\scriptsize] at (E75) {$5$};
\node[above,font=\scriptsize] at (E76) {$6$};
\node[above,font=\scriptsize] at (E77) {$7$};

\draw[fill=black] (E71) circle (2pt);
\draw[fill=black] (E72) circle (2pt);
\draw[fill=black] (E73) circle (2pt);
\draw[fill=black] (E74) circle (2pt);
\draw[fill=black] (E75) circle (2pt);
\draw[fill=black] (E76) circle (2pt);
\draw[fill=black] (E77) circle (2pt);

\draw (E71) -- (E73) -- (E74) -- (E75) -- (E76) -- (E77);
\draw (E74) -- (E72);

\coordinate (E8) at (-2,-9.0);
\node at (E8) {$E_8$};

\coordinate (E81) at (0,-9.0);
\coordinate (E82) at (2,-10.0);
\coordinate (E83) at (1,-9.0);
\coordinate (E84) at (2,-9.0);
\coordinate (E85) at (3,-9.0);
\coordinate (E86) at (4,-9.0);
\coordinate (E87) at (5,-9.0);
\coordinate (E88) at (6,-9.0);

\node[above,font=\scriptsize] at (E81) {$1$};
\node[right,font=\scriptsize] at (E82) {$2$};
\node[above,font=\scriptsize] at (E83) {$3$};
\node[above,font=\scriptsize] at (E84) {$4$};
\node[above,font=\scriptsize] at (E85) {$5$};
\node[above,font=\scriptsize] at (E86) {$6$};
\node[above,font=\scriptsize] at (E87) {$7$};
\node[above,font=\scriptsize] at (E88) {$8$};

\draw[fill=black] (E81) circle (2pt);
\draw[fill=black] (E82) circle (2pt);
\draw[fill=black] (E83) circle (2pt);
\draw[fill=black] (E84) circle (2pt);
\draw[fill=black] (E85) circle (2pt);
\draw[fill=black] (E86) circle (2pt);
\draw[fill=black] (E87) circle (2pt);
\draw[fill=black] (E88) circle (2pt);

\draw (E81) -- (E83) -- (E84) -- (E85) -- (E86) -- (E87) -- (E88);
\draw (E84) -- (E82);

\coordinate (F4) at (-2,-11);
\node at (F4) {$F_4$};

\coordinate (F41)  at ( 0,-11);
\coordinate (F42)  at ( 1,-11);
\coordinate (F43)  at ( 2,-11);
\coordinate (F44)  at ( 3,-11);

\node[above,font=\scriptsize] at (F41)  {$1$};
\node[above,font=\scriptsize] at (F42)  {$2$};
\node[above,font=\scriptsize] at (F43)  {$3$};
\node[above,font=\scriptsize] at (F44)  {$4$};

\draw[fill=black] (F41)  circle (2pt);
\draw[fill=black] (F42)  circle (2pt);
\draw[fill=black] (F43)  circle (2pt);
\draw[fill=black] (F44)  circle (2pt);

\path[-]
(F41) edge (F42)
(F42) edge node[above,font=\scriptsize]{$4$} (F43)
(F43) edge (F44);

\coordinate (H3) at (-2,-12.5);
\node at (H3) {$H_3$};

\coordinate (H31)  at ( 0,-12.5);
\coordinate (H32)  at ( 1,-12.5);
\coordinate (H33)  at ( 2,-12.5);

\node[above,font=\scriptsize] at (H31)  {$1$};
\node[above,font=\scriptsize] at (H32)  {$2$};
\node[above,font=\scriptsize] at (H33)  {$3$};

\draw[fill=black] (H31)  circle (2pt);
\draw[fill=black] (H32)  circle (2pt);
\draw[fill=black] (H33)  circle (2pt);

\path[-]
(H31) edge node[above,font=\scriptsize]{$5$} (H32)
(H32) edge (H33);

\coordinate (H4) at (-2,-14);
\node at (H4) {$H_4$};

\coordinate (H41)  at ( 0,-14);
\coordinate (H42)  at ( 1,-14);
\coordinate (H43)  at ( 2,-14);
\coordinate (H44)  at ( 3,-14);

\node[above,font=\scriptsize] at (H41)  {$1$};
\node[above,font=\scriptsize] at (H42)  {$2$};
\node[above,font=\scriptsize] at (H43)  {$3$};
\node[above,font=\scriptsize] at (H44)  {$4$};

\draw[fill=black] (H41)  circle (2pt);
\draw[fill=black] (H42)  circle (2pt);
\draw[fill=black] (H43)  circle (2pt);
\draw[fill=black] (H44)  circle (2pt);

\path[-]
(H41) edge node[above,font=\scriptsize]{$5$} (H42)
(H42) edge (H43)
(H43) edge (H44);

\coordinate (I2) at (-2,-15.5);
\node at (I2) {$I_2(m)$};
\node[yshift=-0.5cm,font=\scriptsize] at (I2) {$m\geq 3$};

\coordinate (I21)  at ( 0,-15.5);
\coordinate (I22)  at ( 1,-15.5);

\node[above,font=\scriptsize] at (I21)  {$1$};
\node[above,font=\scriptsize] at (I22)  {$2$};

\draw[fill=black] (I21)  circle (2pt);
\draw[fill=black] (I22)  circle (2pt);

\path[-]
(I21) edge node[above,font=\scriptsize]{$m$} (I22);

\end{tikzpicture}
	\caption{Die Dynkin-Diagramme endlicher Coxeter-Gruppen.}
	\label{fig:fin_coxgrps}
\end{figure}

Die Typen $E_6, E_7, E_8$, $F_4$, $G_2=I_2(6)$, $H_3, H_4$ und $I_2(m)$ für $m\geq 5$ heißen \udot{exzeptionelle} oder \udot{Ausnahmetypen}.
\end{theorem}
\begin{proof}
Siehe \citep[Ch.\,2]{humphreys1992coxeter}.
\end{proof}

\subsection{Kombinatorik II: Die Bruhat-Chevalley-Ordnung}

\begin{theoremdef}[Bruhat-Ordnung]
\index{terms}{Bruhat!-Ordnung}\index{terms}{Liftungseigenschaft}\index{terms}{Ketteneigenschaft}
\index{symbols}{$\leq$}
Sei $(W,S)$ eine Coxeter"=Gruppe und $w=s_1 \ldots s_n$ ein reduzierter Ausdruck für $w\in W$. Sei $x\in W$ ein weiteres Gruppenelement. Wir definieren, dass $x\leq w$ genau dann gelten soll, falls Indizes $1\leq i_1<\ldots<i_k\leq n$ existieren derart, dass $x=s_{i_1} \cdots s_{i_k}$ ist.

Dies ist eine partielle Ordnung auf $W$, die \udot{Bruhat"=Chevalley"=Ordnung} oder kurz \udot{Bruhat"=Ordnung} genannt wird. Es gilt:
\begin{enumerate}
	\item Das kleinste Element der Bruhat"=Ordnung ist $1$ und das größte ist -- sofern existent -- das Element maximaler Länge von $W$.
	\item Ketteneigenschaft:
	
	Zu je zwei $x,y\in W$ mit $x<y$ gibt es Elemente $x=x_0<x_1<\ldots<x_k=y$ so, dass $l(x_i)=l(x_{i-1})+1$ für $i=1,\ldots,k$ gilt. Insbesondere gilt $l(x)<l(y)$ und für alle $s\in S$ gilt weiter $sx<x \iff l(sx)<l(x)$.
	\item Liftungseigenschaft:
	
	Für $s\in S$ mit $l(sx)<l(x)$ gilt:
	\[w\leq x \iff \begin{cases} sw\leq sx & \text{falls}\ l(sw)<l(w) \\ w\leq sx & \text{falls}\ l(sw)>l(w)\end{cases}\]
\end{enumerate}
\end{theoremdef}
\begin{proof}
Beweise können z.\,B. in \citep[Ch.\,2]{bjorner2005combinatorics} gefunden werden.
\end{proof}

\begin{remark}
Der Satz von Tits zeigt, dass das Wort-Problem in Coxeter"=Gruppen algorithmisch lösbar ist, indem reduzierte Darstellungen von Elementen berechnet werden (Er liefert jedoch nicht die effizienteste Methode dafür!). Es kann also für beliebige $x\in W$ und $s\in S$ entschieden werden, welche der Alternativen $l(sx)<l(x)$ oder $l(sx)>l(x)$ gilt (und analog für Rechtsmultiplikation). Aus der Liftungseigenschaft ergibt sich, dass auch $x<y$ algorithmisch entschieden werden kann.

Der Algorithmus, um $x<y$ zu entscheiden, funktioniert wie folgt: Durch iteratives Anwenden der Liftungseigenschaft wird die Länge von $y$ reduziert und in jedem Schritt $x$ entsprechend der Liftungseigenschaft durch $x$ oder $sx$ ersetzt. Abgebrochen werden kann, sobald $l(x)\geq l(y)$ ist (d.\,h. spätestens dann, wenn $y$ vollständig zu $y=1$ gekürzt wurde). In diesem Fall ist $x\leq y$ genau dann gegeben, wenn $x=y$ ist. Algorithmus \ref{algo:bruhat_simple} gibt eine Pseudocode-Implementierung dieses Algorithmus an.

\bigbreak
Der Algorithmus liefert ebenfalls eine Möglichkeit, um die Definition der Bruhat"=Ordnung effektiv zu machen: Indem man Buch führt, welches $s\in S$ in jedem Schritt verwendet und ob $x$ durch $sx$ ersetzt wurde, findet man eine reduzierte Darstellung von $y$ und ein Teilwort dieser Darstellung, das gleich $x$ ist. In Algorithmus \ref{algo:bruhat_subwords} ist eine Pseudocode-Implementierung angegeben.
\end{remark}

\begin{remark}
\index{terms}{Coxeter!-Gruppe!reduzible}\index{terms}{Bruhat!-Ordnung}\index{terms}{Liftungseigenschaft}
Sei $(W,S)$ eine reduzible Coxeter"=Gruppe, etwa $S=S_1 \coprod S_2$ und $W=W_1\times W_2$. Dann gilt $(x_1, x_2) \leq (y_1, y_2)$ genau dann, wenn $x_1\leq y_1$ und $x_2\leq y_2$ gelten, d.\,h. die Bruhat"=Ordnung zerlegt sich ebenfalls als Produkt der Bruhat"=Ordnungen. Das kann man beispielsweise mit Hilfe der Liftungseigenschaft einfach einsehen.
\end{remark}

\begin{theorem}[Bruhat-Intervalle]\label{bruhat:intervals}
\index{terms}{Bruhat!-Ordnung}\index{terms}{Bruhat!-Intervall}
\index{symbols}{xy@$[x,y]$}
Sei $(W,S)$ eine Coxeter"=Gruppe, $u,w\in W$ so, dass ein $s\in S$ existiert mit $sw>w$ und $su>u$. Für Elemente $x,z\in W$ sei das Intervall zwischen $x$ und $z$ in der Bruhat"=Ordnung wie gewohnt mit $[x,z]:=\Set{y\in W | x\leq y \leq z}$ bezeichnet. Dann gilt:

Die Abbildung $\eta: [u,w]\times\Set{1,s} \to [u,sw]$ mit
\[\eta(x,1):=x \quad \text{und}\]
\[\eta(x,s):=\begin{cases} x & \text{falls}\ sx<x \\ sx & \text{falls}\ sx>x\end{cases}\]
ist surjektiv.
\end{theorem}
\begin{proof}
Siehe \citep[Prop.\,5.1]{Reading}.
\end{proof}

\begin{remark}
Dieser Satz liefert nun eine Verfeinerung des Algorithmus zur Entscheidung der Bruhat"=Ordnung, die zusätzlich fähig ist, Bruhat"=Intervalle zu berechnen. Algorithmus \ref{algo:bruhat_interval} gibt eine Pseudocode-Implementierung dieses Algorithmus an.
\end{remark}
\section{Iwahori-Hecke-Algebren}

\begin{remark}
Für diesen Abschnitt fixieren wir eine Coxeter"=Gruppe $(W,S)$. Aus der universellen Eigenschaft der Präsentation $\smash{W=\braket{s\in S \mid s^2=1, \underbrace{sts\ldots}_{m_{st}} = \underbrace{tst\ldots}_{m_{st}} \ \text{für alle}\ s,t\in S}}$ lässt sich das folgende Lemma ableiten:
\end{remark}

\begin{lemma}[$W^\text{ab}$ und Konjugationsklassen in $S$]
\index{terms}{Coxeter!-Gruppe!Konjugationsklassen}\index{terms}{Coxeter!-Gruppe!Abelianisierung}\index{terms}{Dynkin-Diagramm}
In jeder Coxeter"=Gruppe gilt:
\begin{enumerate}
	\item $W^\text{ab} := W/[W,W]$ ist isomorph zu $(\IZ/2)^z$, wobei $z$ die Anzahl der Zusammenhangskomponenten des Graphen ist, der entsteht, indem man im Dynkin"=Diagramm von $(W,S)$ alle Kanten $s \text{---} t$ mit $m_{st}\in 2\IN$ oder $m_{st}=\infty$ entfernt.
	\item Zwei Elemente $s,t\in S$ sind genau dann konjugiert in $W$, wenn es eine Folge $s=s_0, s_1, \ldots, s_k=t$ gibt so, dass $\ord(s_i s_{i+1})<\infty$ und ungerade ist, d.\,h. wenn sie in derselben Zusammenhangskomponente dieses Graphen liegen.
\end{enumerate}
\end{lemma}

\begin{remark}
Zur Definition der Iwahori"=Hecke"=Algebra und verwandter Algebren wird folgender vorbereitender Satz benötigt:
\end{remark}
\begin{theorem}[Generische Algebren]
\index{terms}{Zopfrelationen}
Sei $R$ ein kommutativer Ring und $(a_s,b_s)_{s\in S}$ eine Familie von Elementen von $R$ derart, dass $(a_s, b_s)=(a_t, b_t)$ gilt, wann immer $m_{st}<\infty$ und ungerade ist (d.\,h. wann immer $s$ und $t$ in $W$ konjugiert sind).

Es gibt eine eindeutig bestimmte, assoziative $R$"~Algebra $H:=H_R(W,S,(a_s,b_s)_{s\in S})$, die durch die folgenden, äquivalenten Präsentationen beschrieben wird:
\begin{enumerate}
	\item Erzeuger $(T_w)_{w\in W}$ und die Relationen
	\[\forall s\in S, w\in W: T_s T_w = \begin{cases} T_{sw} & \text{falls } l(sw) > l(w) \\ a_s T_{sw} + b_s T_w & \text{falls } l(sw) < l(w)\end{cases}.\]

	\item Erzeuger $(T_s)_{s\in S}$ und die Relationen
	\[\forall s\in S: T_s^2 = a_s\cdot 1_H + b_s T_s,\]
	\[\forall s,t\in S: \underbrace{T_s T_t T_s \ldots }_{m_{st}} = \underbrace{T_t T_s T_t \ldots}_{m_{st}}.\]
	Dann sind die Elemente $T_w$ durch $T_{s_1} \ldots T_{s_n}$ gegeben, wobei $w=s_1 \ldots s_n$ eine reduzierte Darstellung ist (vgl. Satz von Matsumoto\index{terms}{Satz von!Matsumoto}).
\end{enumerate}
Weiter ist $H$ als $R$"~Modul frei mit Basis $(T_w)_{w\in W}$.
\end{theorem}
\begin{proof}
Siehe \citep[7.1-7.3]{humphreys1992coxeter}.
\end{proof}

\begin{definition}[Gewichtsfunktionen auf Coxeter-Gruppen]
\index{terms}{Coxeter!-Gruppe!-mit-Gewicht}\index{terms}{Gewichtsfunktion}
\index{symbols}{WSL@$(W,S,L)$}\index{symbols}{L@$L$}
Sei $\Gamma$ eine total geordnete, abelsche Gruppe. Eine Funktion $L:S\to\Gamma$ heißt \udot{Ge\-wichts\-funk\-tion}, falls für alle konjugierten $s,t\in S$ stets $L(s)=L(t)$ gilt. Die Fortsetzung via Matsumotos Satz auf ganz $W$ werden wir in diesem Fall ebenfalls mit $L$ bezeichnen.
\end{definition}

\begin{example}
\index{terms}{Coxeter!-Gruppe!reduzible}\index{terms}{Coxeter!-Gruppe!Längenfunktion}
Die Längenfunktion $l:W\to\IZ$ ist eine Gewichtsfunktion. Für irreduzible Coxeter"=Gruppen, deren Dynkin"=Diagramm keine Kanten $s-t$ mit $m_{st}\in 2\IN$ oder $m_{st}=\infty$ hat, ist dies im Wesentlichen die einzige Funktion, weil alle Gewichtsfunktionen konstant auf $S$ sein müssen und damit ein Vielfaches von $l$ sind.
Für die anderen irreduziblen Coxeter"=Gruppen (im endlichen Fall sind das nur $B_n$, $F_4$ und $I_2(m)$ mit $2\mid m$) gibt es hingegen eine Vielzahl von Gewichtsfunktionen.

Die Gewichtsfunktionen reduzibler Coxeter"=Gruppen sind durch beliebige Kombinationen von Gewichtsfunktionen auf den irreduziblen Komponenten gegeben.
\end{example}

\begin{convention}[Ringe von Laurent-Polynomen]\label{def:laurent_polynomials}
\index{terms}{Laurent-Polynome}
\index{symbols}{RGamma@$R[\Gamma]$}
Wir werden die Bezeichnung $\Gamma$ für eine total geordnete, abelsche Gruppe ab jetzt festhalten. Zu dieser Gruppe werden wir sehr häufig den Gruppenring $R[\Gamma]$ ($R$ ein kommutativer Ring) benutzen. Um die additive Schreibweise in $\Gamma$ mit der multiplikativen Schreibweise im Gruppenring zu vereinen, bezeichnen wir das dem Gruppenelement $\gamma$ entsprechenden Element der Standardbasis von $R[\Gamma]$ mit $v^\gamma$. Es gilt also $v^{\gamma+\gamma'}=v^\gamma\cdot v^{\gamma'}$.

Falls $\Gamma\isomorphic\IZ^k$ ist, wird durch jede Wahl einer $\IZ$"~Basis ein Isomorphismus zum Ring der Laurent"=Polynome in $k$ Unbestimmten induziert: $R[\Gamma]\isomorphic R[X_1^{\pm 1}, \ldots, X_k^{\pm 1}]$.

\medbreak
Ist $R$ ein Integritätsbereich, so ist $R[\Gamma]$ nullteilerfrei für alle torsionsfreien, abelschen Gruppen $\Gamma$: Wählt man wie oben eine Totalordnung auf $\Gamma$, dann erhält man eine "`Gradfunktion"' $R[\Gamma]\to\Gamma\cup\Set{-\infty}$ und kann den Beweis, dass Polynomringe nullteilerfrei sind, wörtlich übernehmen. Insbesondere können wir den Quotientenkörper bilden. Ist $R=F$ ein Körper, so werden wir für diesen Körper die Bezeichnung $F(\Gamma)$ verwenden. Ist $\Gamma\isomorphic\IZ^k$ so wird durch Wahl einer Basis, ähnlich wie zuvor, ein Isomorphismus $F(\Gamma)\to F(X_1, \ldots X_k)$ induziert.

\medbreak
Wir heben hervor, dass wir mit $R[\Gamma]$ stets diesen Gruppenring meinen, in dem also insbesondere die $v^\gamma$ transzendent über $R$ sind, selbst dann, wenn $R$ selber schon ein Laurent-Polynomring war.
\end{convention}

\begin{definition}[Einparameterfall]
\index{terms}{Einparameterfall}\index{terms}{Multiparameterfall}
Falls in obiger Situation $L$ konstant auf $S$ ist, sprechen wir vom \udot{Einparameterfall}. Dies ist im Wesentlichen durch $\Gamma=\IZ$ und $L(s)=1$ realisiert, alle Betrachtungen in diesem Fall können darauf zurückgeführt werden. Dann ist insbesondere $\IZ[\Gamma]=\IZ[v^{\pm1}]$ der gewöhnliche Ring der Laurent"=Polynome.

Entsprechend sind alle Fälle, in denen $L$ nicht konstant ist, unter der Bezeichnung \udot{Multiparameterfall} zusammengefasst.
\end{definition}

\begin{definition}[Hecke-Algebren]
\index{terms}{Iwahori-Hecke-Algebra|see{Hecke-Algebra}}\index{terms}{Hecke-Algebra}
\index{symbols}{vs@$v_s$}\index{symbols}{HWSL@$H(W,S,L)$}
Sei $(W,S,L)$ eine Coxeter"=Gruppe und $L:W\to\Gamma$ eine Gewichtsfunktion. Wir führen die Abkürzung $v_s:=v^{L(s)}$ ein.

Die \udot{Iwahori"=Hecke"=Algebra} (oder kurz \udot{Hecke"=Algebra}) $H=H(W,S,L)$ ist die $\IZ[\Gamma]$"~Al\-ge\-bra mit Erzeugern $(T_s)_{s\in S}$ und den Relationen
\[T_s^2 = 1+(v_s-v_s^{-1})T_s \quad\text{sowie}\]
\[\underbrace{T_s T_t T_s \ldots }_{m_{st}} = \underbrace{T_t T_s T_t \ldots}_{m_{st}}\]
für alle $s,t\in S$.
\end{definition}

\begin{remark}
\index{terms}{Hecke-Algebra!$\IZ[q]$-Form}\index{symbols}{Tw@$\dot{T}_w$}
Dies ist die Definition, die \cite{geckjacon} verwendet. Es gibt abweichende Definitionen, z.\,B. in \cite{geckpfeiffer}. Dort wird $H$ als $\IZ[q]$"~Algebra analog mit Erzeugern $\dot{T_s}$, den gleichen Zopfrelationen aber den quadratischen Relationen $\dot{T_s}^2 = q 1+(q-1)\dot{T_s}$ definiert. Die so entstehende Algebra hat den Vorteil, auch auf $q=0$ spezialisiert werden zu können, ist aber keine symmetrische Algebra. Unsere Definition liefert hingegen eine symmetrische Algebra.
	
Wenn man erst die Koeffizienten in der zweiten Definition zu $\IZ[q^{\pm 1/2}]$ erweitert hat, erhält man jedoch dieselbe Algebra. Die Elemente $v^{L(w)} T_w$ erfüllen dann die Relationen der $\dot{T}_w$ in der zweiten Definition, wenn man $q:=v^2$ setzt. Umgekehrt erfüllen die Elemente $q^{-l(w)/2} \dot{T}_w$ die Relationen in unserer Definition.
\end{remark}
\begin{remark}	
Die diversen Spezialisierungen dieser Algebra (bzw. der anderen Variante) werden ebenfalls als Hecke"=Algebren bezeichnet. Wir werden nur diese generischen Varianten verwenden.
\end{remark}
\begin{remark}
\index{terms}{Coxeter!-Gruppe!reduzible}\index{terms}{Hecke-Algebra}
Sei $(W,S)$ eine reduzible Coxeter"=Gruppe, etwa $S=S_1 \coprod S_2$ und $W=W_1\times W_2$. Ist $L:W\to\Gamma$ eine Gewichtsfunktion, so ist $L_i:=L_{|S_i}$ eine Gewichtsfunktion für die Coxeter"=Gruppe $(W_i,S_i)$ und die Hecke"=Algebra $H:=H(W,S,L)$ zerlegt sich als Tensorprodukt $H = H_1 \otimes_{\IZ[\Gamma]} H_2$, wobei $H_i:=H(W_i, S_i, L_i)$ die korrespondierende parabolische Unteralgebra sei. Der Isomorphismus ist ganz explizit einfach durch die Einbettung der parabolischen Unteralgebren und Multiplikation $h_1\otimes h_2\mapsto h_1 h_2$ gegeben. Die Standardbasen werden dabei aufeinander abgebildet, d.\,h. $T_u\otimes T_v \mapsto T_{uv}$. Mit dieser Beobachtung lassen sich viele Konstruktionen auf den Fall irreduzibler Coxeter"=Gruppen zurückführen.
\end{remark}

\begin{convention}
\index{terms}{Coxeter!-Gruppe!-mit-Gewicht}
Wir werden ab jetzt stets annehmen, dass $L(s)\geq 0$ für alle $s\in S$ gilt. Dies ist keine wesentliche Einschränkung, denn definiert man $\varepsilon_s := +1$, falls $L(s)\geq 0$ ist, und $\varepsilon_s:=-1$, falls $L(s)<0$ ist, so ist durch $\abs{L}(s):=\varepsilon_s L(s)$ eine Gewichtsfunktion definiert und $T_s\mapsto \varepsilon_s T_s$ ein Isomorphismus $H(W,S,L)\to H(W,S,\abs{L})$.

Eine Einschränkung auf $L(s)>0$ ist hingegen wirklich eine Beschränkung der Allgemeinheit, jedoch kann man die Darstellungstheorie von solchen Hecke"=Algebren zurückführen auf die Darstellungstheorie von Hecke"=Algebren von (anderen!) Coxeter"=Gruppen mit strikt positiver Gewichtsfunktion. In der Tat erhält man in dieser Situation eine Zerlegung $H(W,S,L) = H(W_1, S_1, L_1)\rtimes\Omega$ für zwei Coxeter"=Gruppen $W_1,\Omega\leq W$ (von denen i.\,A. nur $\Omega$ eine parabolische Untergruppe ist), siehe z.\,B. \citep[2.4]{geckjacon}.
\end{convention}

\subsection{Struktureigenschaften}

\begin{convention}
Wir nehmen nun an, dass $(W,S)$ eine endliche Coxeter"=Gruppe, $L:W\to\Gamma$ eine Gewichtsfunktion und $H=H(W,S,L)$ ist.
\end{convention}

\begin{lemma}
\index{terms}{Symmetrische Algebra}
Durch die Spurfunktion $\tau:H\to\IZ[\Gamma], T_w \mapsto \delta_{1,w}$ wird $H$ zu einer symmetrischen $\IZ[\Gamma]$"~Algebra. Hierbei bilden $(T_w)_{w\in W}$ und $(T_{w^{-1}})_{w\in W}$ ein Paar dualer Basen.
\end{lemma}

\begin{lemma}[Ein Antiautomorphismus von $H$]
\index{symbols}{$\ast$}
Definiere $\ast: H\to H$ durch $\IZ[\Gamma]$-lineare Fortsetzung von
\[T_w^\ast := T_{w^{-1}}.\]
Mit dieser Bezeichnung ist $\ast$ ein involutiver Antiautomorphismus von $H$.
\end{lemma}

\begin{theorem}
\index{symbols}{ZW@$\IZ_W$}\index{symbols}{QW@$\IQ_W$}
Setze
\[\IZ_W:=\IZ[\textstyle{2\cos(\frac{2\pi}{m_{st}}}) \mid s,t\in S] \quad\text{und}\]
\[\IQ_W:=\QuotFld(\IZ_W).\]
Dann gilt:
\begin{enumerate}
	\item $\IQ_W$ ist ein Zerfällungskörper von $W$.
	\item Für $K:=\IQ_W(\Gamma)$ ist $KH$ zerfallend halbeinfach.
\end{enumerate}
\end{theorem}
\begin{proof}
In \citep[6.3.8]{geckpfeiffer} wird Teil a. bewiesen und in \citep[9.3.5]{geckpfeiffer}, dass $K$ ein Zerfällungskörper für $KH$ ist. Dass $KH$ halbeinfach ist, folgt aus Tits' Deformationssatz (siehe \citep[7.4.6]{geckpfeiffer})\index{terms}{Satz von!Tits' Deformationssatz}, denn $v^\gamma\mapsto 1$ definiert einen Homomorphismus $\IZ_W[\Gamma]\to\IQ_W$ bezüglich dessen $H$ zu $\IQ_W[W]$ spezialisiert. Aufgrund des Satzes von Maschke und Teilaussage a. ist $\IQ_W[W]$ zerfallend halbeinfach, also ist auch $KH$ halbeinfach.
\end{proof}

\begin{remark}
Aus dem Deformationssatz von Tits folgt außerdem, dass $\chi\mapsto\chi_1$ eine Bijektion $\Irr(KH)\to\Irr(W)$ ist. Dabei bezeichne $\chi_1$ das Bild von ${\chi: H\to\IC[\Gamma]}$ unter der Spezialisierung $v^\gamma\mapsto 1$.

Man kann daher eine gemeinsame Indexmenge $\Lambda$ für die Isomorphietypen der einfachen Moduln beider Algebren wählen. Wir werden dieser Tradition folgen und die Charaktere von $KH$ mit $\chi_\lambda$ bezeichnen.
\end{remark}

\subsection{Kazhdan-Lusztig-Theorie}

\begin{remark}
Die beiden folgenden Definitionen gehen zurück auf \cite{KL} im Einparameterfall und \cite{lusztig2003hecke} sowie \cite{geckjacon} im Multiparameterfall.
\end{remark}

\begin{theoremdef}[Kazhdan-Lusztig-Basis]\label{KL:def:KL_basis}
\index{terms}{Kazhdan-Lusztig!-Basis}\index{terms}{Spiegelung eines Laurent-Polynoms}
\index{symbols}{Cw@$C_w$}\index{symbols}{$\overline{\phantom{m}}$}
Definiere $\overline{\phantom{m}}:\IZ[\Gamma]\to\IZ[\Gamma]$ durch $v^\gamma\mapsto v^{-\gamma}$.
\begin{enumerate}
	\item Durch $\sum_{w\in W} a_w T_w \mapsto \sum_{w\in W} \overline{a_w} T_{w^{-1}}^{-1}$ ist ein $\overline{\phantom{m}}$-semilinearer Ringautomorphismus von $H$ gegeben, den wir ebenfalls mit $\overline{\phantom{m}}$ bezeichnen.
	\item Es gibt genau eine $\IZ[\Gamma]$"~Basis $(C_w)_{w\in W}$ von $H$ mit den Eigenschaften
	\begin{enumerate}
		\item $C_w \in T_w + \sum_{y\in W} \IZ[\Gamma_{>0}] T_y$ und
		\item $\overline{C_w} = C_w$.
	\end{enumerate}
	\item Ist $W$ endlich, so ist $C_w$ auch eindeutig bestimmt durch die Eigenschaften
	\begin{enumerate}
		\item $C_w \in T_w + \sum_{y\in W} \IZ[\Gamma_{>0}] T_y$ und
		\item $C_w T_{w_0} \in \sum_{y\in W} \IZ[\Gamma_{\leq 0}] T_y$.
	\end{enumerate}
	Dabei ist $w_0$ das längste Element von $W$.
\end{enumerate}
Die $C_w$ werden \udot{Kazhdan-Lusztig-Basis} von $H$ genannt. Sie haben zusätzlich die folgende Eigenschaft:
\begin{enumerate}[resume]
	\item $C_w^\ast = C_{w^{-1}}$.
\end{enumerate}
\end{theoremdef}

\begin{theoremdef}[Kazhdan-Lusztig-Polynome, \cite{KL}]\label{KL:def:KL_poly}
\index{terms}{Kazhdan-Lusztig!-Polynom}\index{terms}{Kazhdan-Lusztig!-$\mu$}
\index{symbols}{Pyw@$P_{y,w}^\ast$}\index{symbols}{Pyw@$P_{y,w}$}\index{symbols}{muyws@$\mu_{y,w}^s$}
Sei $P_{y,w}^\ast\in\IZ[\Gamma]$ und $\mu_{y,w}^s\in\IZ[\Gamma]$ durch
\[\smash{C_w = \sum_{y\in W} (-1)^{l(y)+l(w)} \overline{P_{y,w}^\ast} T_y}\]
bzw.
\[\smash{C_s C_w = C_{sw} + \sum_{y<w} (-1)^{l(y)+l(w)+1} \mu_{y,w}^s C_y}\]
für alle $w\in W$,$s\in S$ definiert. Dann gilt:
\begin{enumerate}
	\item $P_{w,w}^\ast = 1$,	$P_{y,w}^\ast\in\IZ[\Gamma_{<0}]$ für $y<w$ und $P_{y,w}^\ast = 0$ für $y \not\leq w$.
	\item Für feste $w\in W$ und $y<w$ gilt $P_{y,w}^\ast = P_{y^{-1},w^{-1}}^\ast$ sowie
	\begin{enumerate}
		\item Für alle $t\in S$ mit $tw<w$:
		\[P_{y,w}^\ast = \begin{cases} P_{ty,tw}^\ast & L(t)=0 \\
		v_t^{-1} P_{ty,w}^\ast & L(t)>0 \wedge ty>y \\
		\smash{v_t P_{ty,w}^\ast + P_{ty,tw}^\ast - \sum_{\substack{z\in W \\ y\leq z < tw \\ tz < z}} P_{y,z}^\ast \mu_{z,tw}^t} & L(t)>0 \wedge ty<y
		\end{cases}\]
		\item Für alle $t\in S$ mit $wt<w$:
		\[P_{y,w}^\ast = \begin{cases} P_{yt,wt}^\ast & L(t)=0 \\
		v_t^{-1} P_{yt,w}^\ast & L(t)>0 \wedge yt>y \\
		\smash{v_t P_{yt,w}^\ast + P_{yt,wt}^\ast - \sum_{\substack{z\in W \\ y\leq z < wt \\ zt < z}} P_{y,z}^\ast \mu_{z,tw^{-1}}^t} & L(t)>0 \wedge yt<y
		\end{cases}\]
	\end{enumerate}
	\item $P_{y,w} := v^{L(w)-L(y)} P_{y,w}^\ast \in\IZ[\Gamma_{\geq 0}]$. Falls $L(s)>0$ für alle $s\in S$ ist, dann ist $P_{y,w}\neq 0$ und hat konstanten Term $1$.
	\item Für alle $y,w\in W, s\in S$ gilt $\mu_{y,w}^s \neq 0 \implies L(s)>0$ und $sy<y<w<sw$.
	\item Außerdem gilt für alle $y,w\in W, s\in S$ mit $L(s)>0$ und $sy<y<w<sw$:
	\begin{enumerate}
		\item $\smash{\mu_{y,w}^s - v_s P_{y,w}^\ast + \sum_{\substack{z\in W \\ y<z<w \\ sz<z}} P_{y,z}^\ast \mu_{z,w}^s \in\IZ[\Gamma_{<0}]}$.
		\item $\overline{\mu_{y,w}^s} = \mu_{y,w}^s$.
		\item $v_s \mu_{y,w}^s \in\IZ[\Gamma_{>0}]$.
		\item Falls $L(s)=1$ für alle $s\in S$ ist, ist $\mu_{y,w}^s\in\IZ$ und gleich dem Koeffizienten von $v^{-1}$ in $P_{y,w}^\ast\in\IZ[v^{-1}]$.
	\end{enumerate}
\end{enumerate}
Die $P_{y,w}$ heißen \udot{Kazhdan-Lusztig-Polynome}. Die $\mu_{y,w}^s$ werden wir als \udot{Kazhdan-Lusztig-$\mu$} bezeichnen.
\end{theoremdef}

\begin{example}
\begin{itemize}
	\item Man kann zeigen, dass $C_1=T_1=1$ sowie $C_s = T_s - v_s$ für alle $s\in S$ gilt. Insbesondere ist $\Set{C_s | s\in S}$ ein Erzeugendensystem der $\IZ[\Gamma]$"~Algebra $H$.
	\item In einer endlichen Coxeter"=Gruppe gilt für das längste Element $w_0$ und beliebige $x\in W$ stets $P_{x,w_0}=1$.
	\item Man kann ebenfalls zeigen (siehe etwa \citep[7.12.a]{humphreys1992coxeter}), dass für Dieder"=Gruppen im Einparameterfall $P_{x,y}=\begin{cases} 1 & \text{falls}\,x\leq y \\ 0 &\text{sonst}\end{cases}$ gilt.
\end{itemize}
\end{example}

\begin{remark}
Der Satz liefert einen rekursiven Algorithmus, um $P_{y,w}^\ast$ und $\mu_{y,w}^s$ simultan zu berechnen. Indem wir auf der Menge $\Set{(y,w)\in W\times W | y\leq w}$ die partielle Ordnung
\[(y',w')\sqsubset(y,w) :\iff w'<w \text{ oder } (w=w' \text{ und } y'>y)\]
definieren, erhalten wir, dass alle Terme der rechten Seite in den Teilaussagen b.i. und b.ii. beziehungsweise alle Summanden in e.i. echt kleiner bezüglich dieser Ordnung sind. Das kleinste Element ist $(1,1)$. Da außerdem jedes Element $w\in W$ nur endlich viele Vorgänger (nämlich $\leq 2^{l(w)}$) in der Bruhat"=Ordnung hat, hat die partielle Ordnung keine unendlichen absteigenden Ketten, d.\,h. wir können diese Rekursion tatsächlich in einen Algorithmus umwandeln.

Man beachte, dass aufgrund der großen Anzahl von nicht-verschwindenden Kazhdan-Lusztig-Polynomen und der Art der Rekursion die Berechnung der $P_{x,y}$ mit diesem Algorithmus eine sehr anspruchsvolle Aufgabe sein kann, falls das Intervall $[x,y]$ sehr groß ist. Die Berechnung von allen Kazhdan-Lusztig-Polynomen für $E_8$ ist beispielsweise an der Grenze der Machbarkeit für heutige Computer aufgrund des enormen Zeit- und Speicherplatzbedarfs.

Die Wahl von $t\in S$ mit $tw<w$ ist uns in b. und c. freigestellt. Im Sinne einer zeit- und speicherplatzeffizienten Implementierung der Rekursion ist es wünschenswert, den jeweils dritten Fall in b. und c. so selten wie möglich tatsächlich zu benutzen und nur so wenig Werte von $P_{y,w}^\ast$ wie möglich dauerhaft zu speichern. Dies kann man dadurch erreichen, dass man zuerst nach $t\in S$ sucht, die $tw<w$, aber $ty>y$ beziehungsweise $wt<w$, aber $yt>y$ erfüllen. Paare $(y,w)$, für die $y<w$ und 
\[\forall t: (tw<w \implies L(t)>0 \wedge ty<y) \wedge (wt<w \implies L(t)>0 \wedge yt<y)\]
gilt, heißen \udot{kritische Paare}\index{terms}{Kritisches Paar}. Es ist mit den ersten beiden Fällen von b.i. und b.ii. problemlos möglich, aus der Kenntnis von $P_{y,w}^\ast$ für alle kritischen Paare die restlichen Werte zu berechnen. Man sollte daher die Polynome nur für kritische Paare abspeichern und auch nur für kritische Paare tatsächlich die Summation ausführen.

Im Einparameterfall ist keine gesonderte Berechnung der $\mu_{y,w}^s$ nötig, wie aus e.iv. hervorgeht. Man braucht also nur für den Multiparameterfall einen speziellen Algorithmus für die $\mu$-Werte.

Eine Pseudocode-Implementierung dieses Algorithmus ist in \ref{algo:KL_poly_mu} zu finden.
\end{remark}
\begin{remark}
\index{terms}{Coxeter!-Gruppe!reduzible}\index{terms}{Kazhdan-Lusztig!-Basis}
Es sei $(W,S)$ reduzibel, etwa $S=S_1 \coprod S_2$ und $W=W_1\times W_2$. Es seien weiter $L_i:=L_{|S_i}$ die eingeschränkten Gewichtsfunktionen. Dann ist, wie bereits festgestellt, $H:=H(W,S,L)$ das Tensorprodukt der parabolischen Unteralgebren $H_i:=H(W_i,S_i,L_i)$. Da sowohl $\overline{\phantom{m}}$ als auch die Standardbasen $T$ der beteiligten Hecke"=Algebren mit den parabolischen Einbettungen $H_i\hookrightarrow H$ verträglich sind, ist $(C_{w_1} \cdot C_{w_2})_{w_i\in W_i}$ eine Basis von $H$, die die Bedingungen der Definition erfüllt, also gilt $C_{w_1 w_2} = C_{w_1} C_{w_2}$ für alle $w_i\in W_i$.

In der gleichen Situation zerlegen sich auch die Polynome $P_{y,w}^\ast$ und $\mu_{y,w}^s$ entsprechend: Es gilt $P_{y_1 y_2, w_1 w_2}^\ast = P_{y_1, w_1}^\ast \cdot P_{y_2, w_2}^\ast$ für alle $y_i,w_i\in W_i$ sowie
\[\mu_{y_1 y_2, w_1 w_2}^s =\begin{cases} \mu_{y_1,w_1}^s & \text{falls }s\in S_1\text{ und }y_2=w_2 \\ \mu_{y_2,w_2}^s & \text{falls } s\in S_2\text{ und }y_1=w_1 \\ 0 & \text{sonst}\end{cases}\]
für alle $y_i,w_i\in W_i$ und alle $s\in S$ mit $s(y_1 y_2)<y_1 y_2 < w_1 w_2 < s (w_1 w_2)$.
\end{remark}

\begin{corollary}[Symmetrieeigenschaften]\label{KL:lemma:symmetry_KL_poly}
\index{terms}{Kazhdan-Lusztig!-Polynom}\index{terms}{Kazhdan-Lusztig!-$\mu$}
Sei ein Gruppenhomomorphismus $\Gamma\to\Set{\pm 1}$ gegeben. Wir bezeichnen ihn der Übersichtlichkeit halber mit $\gamma\mapsto (-1)^\gamma$ und definieren einen Automorphismus $\theta$ von $\IZ[\Gamma]$ via $\theta(v^\gamma):=(-1)^{\gamma}v^\gamma$. Mit diesen Bezeichnungen gilt:
\begin{enumerate}
	\item $\theta(P_{y,w}^\ast) = (-1)^{L(w)+L(y)} P_{y,w}^\ast$.
	\item $\theta(\mu_{y,w}^s) = (-1)^{L(w)+L(y)+L(s)} \mu_{y,w}^s$.
\end{enumerate}
Insbesondere folgt:
\begin{enumerate}[resume]
	\item $P_{y,w}\in\IZ[2\Gamma]$.
	\item $v^{L(w)-L(y)} \cdot v_s \mu_{y,w}^s\in \IZ[2\Gamma]$.
\end{enumerate}
\end{corollary}
\begin{proof}
Der Beweis von a. und b. erfolgt durch Induktion entlang der oben definierten Ordnung $\sqsubseteq$ mittels der Rekursionsformeln. Für $y=w$ und $y\not\leq w$ gilt die Behauptung sicherlich, da $P_{w,w}^\ast = 1$ bzw. $P_{y,w}^\ast=0$ sowie $\mu_{y,w}^s = 0$ ist.

Gelten die Behauptungen nun für alle $(y',w')\sqsubset(y,w)$, dann gelten sie auch für $(y,w)$ aufgrund der Rekursion: Wir wählen ein $t\in S$ mit $tw<w$. Falls $L(t)=0$ ist, gilt
\[\theta(P_{y,w}^\ast) = \theta(P_{ty,tw}^\ast) = (-1)^{L(tw)+L(ty)} P_{ty,tw}^\ast = (-1)^{L(w)+L(y)} P_{y,w}^\ast.\]
Falls $L(t)>0$ und $ty>y$ ist, gilt 
\[\theta(P_{y,w}^\ast) = \theta(v_t^{-1} P_{ty,w}^\ast) = (-1)^{-L(t)}v_t^{-1} (-1)^{L(w)+L(ty)} P_{ty,w}^\ast = (-1)^{L(w)+L(y)} P_{y,w}^\ast.\]
Ist schließlich $L(t)>0$ und $ty<y$, so gilt
\begin{align*}
	\theta(P_{y,w}^\ast) &= \theta(v_t P_{ty,w}^\ast)+\theta(P_{ty,tw}^\ast) - \sum_z \theta(P_{y,z}^\ast \mu_{z,tw}^t) \\
	&= (-1)^{L(t)} v_t (-1)^{L(w)+L(ty)} P_{ty,w}^\ast + (-1)^{L(tw)+L(ty)} P_{ty,tw}^\ast \\
	&\phantom{\text{= }} - \sum_z (-1)^{L(y)+L(z)} P_{y,z}^\ast (-1)^{L(z)+L(tw)+L(t)} \mu_{z,tw}^t \\
	&= (-1)^{L(w)+L(y)}\cdot \Big(v_t P_{ty,w}^\ast + P_{ty,tw}^\ast - \sum_z P_{y,z}^\ast \mu_{z,tw}^t \Big) \\
	&= (-1)^{L(w)+L(y)} P_{y,w}^\ast.
\end{align*}
Eine analoge Anwendung der Rekursionsformel für $\mu$ zeigt den Induktionsschritt. Insbesondere ergibt sich daraus nun
\[\theta(P_{y,w}) = \theta\big(v^{L(w)-L(y)} P_{y,w}^\ast\big) = (-1)^{L(w)-L(y)} v^{L(w)-L(y)} (-1)^{L(w)+L(y)} P_{y,w}^\ast = P_{y,w}\]
und analog ist auch $v^{L(w)-L(y)} v_s \mu_{xy}^s$ invariant unter $\theta$.

c. und d. ergeben sich nun, da a. und b. für alle $\Gamma\to\Set{\pm 1}$ gelten. Da es zu jedem $\gamma\in\Gamma\setminus2\Gamma$ einen Homomorphismus $\Gamma\to\Set{\pm 1}$ mit $\gamma\mapsto-1$ gibt, folgt aus der eben bewiesenen Invarianz von $P_{y,w}$ unter allen solchen $\theta$, dass der Koeffizient vor $v^\gamma$ in $P_{y,w}$ gleich Null sein muss.
\end{proof}

\begin{remark}
In Kontexten wie \cite{humphreys1992coxeter} und \cite{geckpfeiffer}, wo man die Hecke"=Algebra mit der $\dot{T}$"~Basis und den Relationen
\[\dot{T_s}^2 = q_s + (q_s-1)\dot{T_s}\]
definiert ($q_s:=v_s^2$), erhält man insbesondere, dass die Kazhdan-Lusztig-Polynome in $\IZ[q_s | s\in S]$ liegen. Das Korollar liefert die Möglichkeit, dies direkt einzusehen, ohne die $\IZ[q]$-Form der Hecke"=Algebra benutzen zu müssen. Nach meinem Kenntnisstand ist obige Beweismethode noch nicht (öffentlich) bemerkt worden.
\index{terms}{Hecke-Algebra!$\IZ[q]$-Form}
\end{remark}

\subsection{Zellen}

\begin{definition}[Zellen, siehe {\cite{stembridge2012finiteness}}]\label{KL:def:cells}
\index{symbols}{$\preceq$}
Sei $k$ ein kommutativer Ring, $H$ eine $k$"~Algebra sowie $M$ ein $H$"~Modul, der als $k$"~Modul frei ist mit Basis $C\subseteq M$. Dann definiere eine Quasiordnung auf $C$ wie folgt: Wir legen fest, dass $x\leftarrow y$ genau dann gelten soll, falls ein $h\in H$ existiert derart, dass in der Darstellung $hy = \sum_{z\in C} a_{zy} z$ mit $a_{zy}\in k$ der Koeffizient $a_{xy}\neq 0$ ist. Die Quasiordnung $\preceq$ sei dann definiert als der transitive Abschluss von $\leftarrow$.

Die Äquivalenzklassen der von $\preceq$ induzierten Äquivalenzrelation (d.\,h. $x\sim y$ genau dann, wenn $x\preceq y \wedge y\preceq x$) werden als \udot{Zellen} von $(M,C)$ bezeichnet.
\end{definition}

\begin{lemmadef}
\index{terms}{Zellmodul einer Zelle}
Die Quasiordnung kann in der Situation der Definition wie folgt charakterisiert werden: Für alle $y\in C$ ist
\[\Set{x\in C | x\preceq y} = \bigcap \Set{A\subseteq C | y\in A \text{ und }\textrm{span}_k A \text{ ist ein  $H$-Untermodul von $M$}}.\]

Insbesondere kann zu jeder Zelle auf kanonische Weise ein $H$"~Modul-mit-Basis assoziiert werden. Diese werden \udot{Zellmoduln von $(M,C)$} genannt.
\end{lemmadef}

\begin{remark}
Da ein $k$-Untermodul genau dann ein $H$-Untermodul ist, wenn er unter Multiplikation mit einem festen $k$"~Algebra"=Erzeugendensystem von $H$ abgeschlossen ist, reicht es in der Definition von $\leftarrow$ das Element $h$ aus einem solchen Erzeugendensystem zu wählen.

$\leftarrow$ definiert daher für jedes Erzeugendensystem $S\subseteq H$ einen gerichteten Graphen mit der Eckenmenge $C$. Die Zellen sind genau die starken Zusammenhangskomponenten dieses Graphen.
\index{terms}{stark zusammenhängend}\index{terms}{Zusammenhangskomponente!starke}
\end{remark}
\begin{remark}
Sind $H_1, H_2$ zwei $k$"~Algebren und $(M_i, C_i)$ jeweils ein $H_i$"~Modul-mit-Basis wie in der Definition, dann lässt sich der Graph des $H_1\otimes_k H_2$"~Moduls-mit-Basis $(M_1\otimes_k M_2, C_1 \times C_2)$ aus den Graphen von $(M_1, C_1)$ und $(M_2, C_2)$ konstruieren:
	
Ist $S_i\subseteq H_i$ je ein Erzeugendensystem der beiden Algebren, dann ist $S_1\otimes 1 \cup 1\otimes S_2$ ein Erzeugendensystem von $H_1\otimes H_2$ und der dazugehörige Graph ist das sogenannte kartesische Produkt der beiden Graphen, d.\,h. $(x_1, x_2)\leftarrow (y_1, y_2)$ gilt genau dann, wenn $x_1=x_2$ und $y_1\leftarrow y_2$ oder umgekehrt $x_1\leftarrow x_2$ und $y_1=y_2$ gilt.
	
Ist $1\in S_i$ und wählt man $\Set{s_1\otimes s_2 | s_i\in S_i}$ als Erzeugendensystem, erhält man das sogenannte Tensorprodukt der beiden Graphen, d.\,h. das Produkt im kategorientheoretischen Sinne: $(x_1, x_2)\leftarrow (y_1, y_2)$ gilt genau dann, wenn $x_1\leftarrow x_2$ und $y_1\leftarrow y_2$ gilt.
	
Insbesondere sind die Zellen des Produkts genau die kartesischen Produkte der Zellen.
	
Die Quasiordnung erfüllt entsprechend auch $(x_1, x_2)\preceq (y_1, y_2) \iff x_1\preceq x_2$ und $y_1\preceq y_2$ für alle $x_i,y_i\in M_i$.
\end{remark}

\begin{definition}[Kazhdan-Lusztig-Zellen, siehe {\cite{KL}}]\label{KL:def:KL_cells}
\index{terms}{Kazhdan-Lusztig!-Basis}\index{terms}{Kazhdan-Lusztig!-Zellen}\index{terms}{Kazhdan-Lusztig!-$W$-Graph}\index{terms}{Kazhdan-Lusztig!-$\mu$}
\index{symbols}{$\preceq$!$\preceq_\mathcal{L}$,$\preceq_\mathcal{R}$,$\preceq_\mathcal{LR}$}\index{symbols}{$\sim_\mathcal{L}$,$\sim_\mathcal{R}$,$\sim_\mathcal{LR}$}
Sei $(W,S)$ eine Coxeter"=Gruppe und $H$ die dazugehörige Hecke"=Algebra. Wir fixieren die Kazhdan-Lusztig-Basis $\Set{C_w | w\in W}$.

\medbreak
Die Zellen von $(H,\Set{C_w | w\in W})$ aufgefasst als Links-, Rechts- oder Bimodul über $H$ heißen \udot{Links-}, \udot{Rechts-} bzw. \udot{zweiseitige Kazhdan-Lusztig-Zellen}.

\medbreak
Die in der Definition angegebene Quasiordnung werden wir entsprechend mit $\preceq_\mathcal{L}$, $\preceq_\mathcal{R}$ beziehungsweise $\preceq_\mathcal{LR}$ sowie die zugehörige Äquivalenzrelation mit $\sim_\mathcal{L}$, $\sim_\mathcal{R}$ beziehungsweise $\sim_\mathcal{LR}$ bezeichnen.

Den gerichteten Graphen auf $W$, der durch das Erzeugendensystem $\Set{C_s | s\in S}$ definiert ist, nennen wir \udot{Kazhdan-Lusztig-$W$"~Graphen} (dies ist tatsächlich ein $W$"~Graph im Sinne von Definition \ref{def:W_graph}, wenn man die $\mu_{xy}^s$ als Kantengewichte wählt).
\end{definition}

\subsection{Lusztigs Vermutungen}

\begin{remark}
Die Kazhdan-Lusztig-Basis hat tiefliegende kombinatorische, geometrische und darstellungstheoretische Interpretationen, wenn $W$ eine Weyl-Gruppe einer Lie-Gruppe oder reduktiven algebraischen Gruppe ist. Die algebraischen Konsequenzen dieser Eigenschaften sind für allgemeine Coxeter"=Gruppen und allgemeine Gewichtsfunktionen zum Teil noch unbewiesen, werden jedoch in der Theorie oft als Voraussetzungen benutzt. Daher wollen wir die wichtigsten dieser Vermutungen noch einmal in diesem Abschnitt zusammenfassen.

Wir fixieren dafür eine Coxeter"=Gruppe $(W,S)$, eine Gewichtsfunktion $L:W\to\Gamma$ und die dazugehörige Hecke"=Algebra $H$.
\end{remark}

\begin{definition}[Strukturkonstanten der KL-Basis. Siehe \cite{KL}]\label{KL:def:h_xyz}
\index{terms}{Kazhdan-Lusztig!-Basis}
\index{symbols}{hxyz@$h_{x,y,z}$}

Die Strukturkonstanten der Kazhdan-Lusztig-Basis in dieser Hecke"=Algebra seien mit $h_{xyz}\in\IZ[\Gamma]$ bezeichnet, d.\,h. es gilt:
\[C_x C_y = \sum_{z\in W} h_{xyz} C_z\]
für alle $x,y,z\in W$.
\end{definition}

\begin{conjecture}[Positivitätsvermutung von Kazhdan und Lusztig. Siehe \cite{KL}]\label{KL:conj:positivity}
\index{terms}{Einparameterfall}
Im Einparameterfall gilt:
\begin{enumerate}
	\item Die Kazhdan-Lusztig-Polynome haben nichtnegative Koeffizienten, d.\,h. $P_{x,y}\in\IN[v]$ für alle $x,y\in W$.
	\item Für alle $x,y,z\in W$ haben die nichtverschwindenden Koeffizienten der Strukturkonstanten $h_{x,y,z}$ dasselbe Vorzeichen. Genauer soll $(-1)^{l(x)+l(y)+l(z)}h_{x,y,z}\in\IN[v]$ sein.
\end{enumerate}
\end{conjecture}

\begin{remark}
Ist $W$ eine Weyl"=Gruppe einer halbeinfachen, komplexen Lie"=Algebra, so ist diese Vermutung unabhängig durch Beilinson und Bernstein (\cite{beilinson1981localisation}) sowie etwa zeitgleich durch Brylinski und Kashiwara (\cite{brylinski1981kazhdan}) bewiesen worden. Diesen Beweisen liegen tiefe algebro-geometrische Interpretationen der Kazhdan-Lusztig-Polynome zugrunde.

Da man die Kazhdan-Lusztig-Polynome für Diedergruppen kennt und die Strukturkonstanten ebenfalls explizit ausschreiben kann, konnte man die Gültigkeit der Vermutung auch für $W=I_2(m)$ für alle $m\in\IN_{\geq 3}$ nachweisen (siehe \cite{ducloux2006positivity}).

Durch explizites Nachrechnen ist ein Computerbeweis der Vermutung für die Typen $H_3$ und $H_4$ erbracht worden (siehe \cite{ducloux2006positivity} und \cite{alvis1987left}).

\index{terms}{Coxeter!-Gruppe!reduzible}\index{terms}{Kazhdan-Lusztig!-Polynom}
Ist $(W,S)$ reduzibel, etwa $S=S_1 \coprod S_2$ und $W=W_1\times W_2$ und sind weiter $L_i:=L_{|S_i}$ die eingeschränkten Gewichtsfunktionen, dann sind, wie bereits festgestellt, die Kazhdan-Lusztig-Basen und Kazhdan-Lusztig-Polynome mit dieser Zerlegung verträglich. Das überträgt sich auf die Strukturkonstanten: Es gilt $h_{x_1x_2, y_1 y_2, z_1 z_2} = h_{x_1, y_1, z_1} \cdot h_{x_2, y_2, z_2}$ für alle $x_i,y_i,z_i\in W_i$.
	
Daher reicht es, die beiden Teile der Positivitätsvermutung für irreduzible Coxeter"=Gruppen nachzuprüfen. Da dies für die nichtkristallographischen Gruppen durch explizite Rechnungen und für Weyl"=Gruppen durch allgemeine Argumente geschehen ist, gilt die Vermutung somit für alle endlichen Coxeter"=Gruppen.
\end{remark}
\begin{remark}
Im Dezember 2012 veröffentlichen Ben Elias und Geordie Williamson einen rein algebraischen Beweisansatz der Positivitätsvermutung für beliebige Coxeter"=Gruppen, siehe \cite{elias2012hodge}.
\end{remark}
\begin{remark}
Im Multiparameterfall gilt die Positivitätsvermutung schon in einfachen Fällen nicht mehr. Es können negative Koeffizienten in den Polynomen auftreten. Als Ersatz für die Positivitätseigenschaften formulierte Lusztig in \cite{lusztig2003hecke} eine Reihe schwächerer Vermutungen, die alle wesentlichen algebraischen Konsequenzen der Positivitätsvermutung implizieren sollen. Um diese Vermutungen formulieren zu können, benötigen wir weitere Definitionen.
\end{remark}

\begin{definition}[Lusztigs Funktionen $\textbf{a}(z)$ und $\Delta(z)$, siehe \cite{lusztig2003hecke} und {\citep[2.3]{geckjacon}}]\label{KL:def:Lusztig_a}
\index{terms}{Kazhdan-Lusztig!-Polynom}
\index{symbols}{az@$\textbf{a}(z)$}\index{symbols}{Deltaz@$\Delta(z)$}
Definiere für alle $x,y,z\in W$ nun
\[\textbf{a}(z) := \min\Set{\gamma\in\Gamma | v^\gamma h_{xyz}\in\IZ[\Gamma_{\geq 0}] \,\text{für alle}\,x,y\in W}\]
und $c_{x,y,z}\in\IZ$ durch $v^{\textbf{a}(z)} h_{x,y,z} \equiv c_{x,y,z^{-1}} \mod \IZ[\Gamma_{>0}]$.

\medbreak
Definiere, falls $P_{1,z}^\ast\neq 0$ ist, außerdem
\[\Delta(z) := \max\Set{\gamma\in\Gamma | v^\gamma P_{1,z}^\ast \in \IZ[\Gamma_{\leq 0}]}\]
sowie $n_z\in\IZ\setminus\Set{0}$ durch $v^{\Delta(z)} P_{1,z}^\ast \equiv n_z \mod \IZ[\Gamma_{<0}]$.
\end{definition}

\begin{definition}[Duflo-Involutionen, siehe \cite{lusztig2003hecke}]\label{def:KL_Duflo_inv}
\index{symbols}{D@$\mathcal{D}$}
Mit obigen Bezeichnungen definiere
\[\mathcal{D}:=\Set{z\in W | P_{1,z}^\ast\neq 0 \,\text{und}\, \textbf{a}(z)=\Delta(z)}.\]
\end{definition}

\begin{conjecture}[Lusztig-Vermutungen, siehe \cite{lusztig2003hecke}]\label{KL:conj:Lusztig_P1_P15}
\index{terms}{Lusztig-Vermutungen}\index{terms}{Kazhdan-Lusztig!-Zellen}\index{terms}{Parabolische Untergruppe}
\index{symbols}{P15@\textbf{P1}-\textbf{P15}}
Für beliebige Gewichtsfunktionen $L:W\to\Gamma$ gilt:
\begin{enumerate}[label=\textbf{P\arabic*}]
	\item Für alle $z\in W$ mit $P_{1,z}^\ast\neq 0$ gilt $\textbf{a}(z)\leq\Delta(z)$.
	\item Für alle $d\in\mathcal{D}$ und $x,y\in W$ mit $c_{x,y,d}\neq 0$ gilt $x=y^{-1}$.
	\item Für alle $y\in W$ existiert genau ein $d\in\mathcal{D}$ mit $c_{y^{-1},y,d}\neq 0$.
	\item Für alle $x,y\in W$ mit $x\preceq_\mathcal{LR} y$ gilt $\textbf{a}(x)\geq\textbf{a}(y)$. Insbesondere gilt $\textbf{a}(x)=\textbf{a}(y)$, falls $x\sim_\mathcal{LR}y$.
	\item Für alle $y\in W$, $d\in\mathcal{D}$ mit $c_{y^{-1},y,d}\neq 0$ gilt $c_{y^{-1},y,d} n_d = (-1)^{l(d)}$.
	\item Für alle $d\in\mathcal{D}$ gilt $d^2=1$.
	\item Für alle $x,y,z\in W$ gilt $c_{x,y,z} = c_{y,x,z}$.
	\item Für alle $x,y,z\in W$ mit $c_{x,y,z}\neq 0$ gilt $x\sim_\mathcal{L} y^{-1}$, $y\sim_\mathcal{L} z^{-1}$ und $z\sim_\mathcal{L} x^{-1}$.
	\item Für alle $x,y\in W$ mit $x\preceq_\mathcal{L} y$ und $\textbf{a}(x)=\textbf{a}(y)$ gilt $x\sim_\mathcal{L} y$.
	\item Für alle $x,y\in W$ mit $x\preceq_\mathcal{R} y$ und $\textbf{a}(x)=\textbf{a}(y)$ gilt $x\sim_\mathcal{R} y$.
	\item Für alle $x,y\in W$ mit $x\preceq_\mathcal{LR} y$ und $\textbf{a}(x)=\textbf{a}(y)$ gilt $x\sim_\mathcal{LR} y$.
	\item Sei $I\subseteq S$ und $W_I\leq W$ die von $I$ erzeugte parabolische Untergruppe. Für alle $y\in W_I$ ist der Wert $\textbf{a}(y)$ unabhängig davon, ob man ihn bezüglich $(W,S,L)$ oder bezüglich $(W_I,I,L_{|I})$ berechnet.
	\item Jede Kazhdan-Lusztig-Linkszelle $\mathfrak{C}\subseteq W$ enthält genau ein Element $d\in\mathcal{D}$ und es gilt $c_{x^{-1},x,d}\neq 0$ für alle $x\in\mathfrak{C}$.
	\item Für alle $z\in W$ gilt $z\sim_\mathcal{LR} z^{-1}$.
	\item Für alle $w,w',x,y\in W$ mit $\textbf{a}(x)=\textbf{a}(y)$ gilt
	\[\sum_{z\in W} h_{x,w',z}\otimes h_{w,z,y} = \sum_{z\in W} h_{z,w',y}\otimes h_{w,x,z}\]
	als Gleichung in $\IZ[\Gamma] \otimes_\IZ \IZ[\Gamma]$.
\end{enumerate}
\end{conjecture}

\begin{remark}
Die Positivitätsvermutung impliziert \textbf{P1} bis \textbf{P15} im Einparameterfall (siehe \cite{lusztig2003hecke}).\index{terms}{Einparameterfall} Der allgemeine Beweis dieser Vermutungen, insbesondere im Fall $B_n$ mit beliebiger Gewichtsfunktion, steht jedoch noch aus.
\end{remark}
\begin{remark}
\index{terms}{Coxeter!-Gruppe!reduzible}
Auch diese Eigenschaften sind gut verträglich mit einer Zerlegung $W=W_1\times W_2$ einer reduziblen Coxeter"=Gruppe. Man kann zeigen, dass $\textbf{a}(z_1 z_2)=\textbf{a}(z_1)+\textbf{a}(z_2)$, $\Delta(z_1 z_2) = \Delta(z_1)+\Delta(z_2)$ sowie $n_{z_1 z_2} = n_{z_1} n_{z_2}$ für alle $z_i\in W_i$ gilt (wobei man $\Delta(z)=\infty$ setzt, falls $P_{1,z}^\ast=0$ sein sollte).
	
Es folgt dann, dass die Gültigkeit von \textbf{P1}-\textbf{P15} für beide Faktoren die Gültigkeit von \textbf{P1}-\textbf{P15} für das Produkt nach sich zieht. Ist das der Fall, so gilt auch $\mathcal{D}=\mathcal{D}_1\cdot\mathcal{D}_2$ für die Duflo-Involutionen. Erneut wäre es also wieder ausreichend, die Lusztig-Vermutungen für irreduzible Coxeter-Gruppen nachzuweisen.

\end{remark}
\begin{remark}
Anstatt die vollen Lusztig-Vermutungen anzunehmen, haben Geck und Jacon in \cite{geckjacon}, drei schwächere Vermutungen formuliert, die für ihre Zwecke ausreichen. Dazu konstruieren sie ohne weitere Annahmen eine Version von Lusztigs asymptotischer Algebra $J$, die sie $\widetilde{J}$ nennen und welche unter Annahme von \textbf{P1} und \textbf{P4} kanonisch zu $J$ isomorph ist.
\end{remark}

\chapter{Darstellungen I: Balanciertheit}
\setcounter{section}{-1}
\section{Setup}

\begin{remark}
Die Abschnitte 2.1 und 2.2 haben den Zweck, die Konstruktion der asymptotischen Algebra, die in \citep[1.4+1.5]{geckjacon} gegeben wird, auf eine naheliegende Weise zu verallgemeinern.
\end{remark}

\begin{convention}
\index{terms}{$\ast$-symmetrisch}\index{terms}{Schur-Element}\index{terms}{Symmetrische Algebra}\index{terms}{Bewertung}\index{terms}{Bewertung!-sring}
\index{symbols}{K@$K$}\index{symbols}{O@$\mathcal{O}$}\index{symbols}{F@$F$}\index{symbols}{Gamma@$\Gamma$}\index{symbols}{nu@$\nu$}\index{symbols}{m@$\mathfrak{m}$}
\index{symbols}{Lambda@$\Lambda$}
\index{symbols}{clambda@$c_\lambda$}
\index{symbols}{$\ast$}\index{symbols}{x@$x^\vee$}
Wir legen dazu folgende Voraussetzungen und Notationen für die Abschnitte \thechapter.1 und \thechapter.2 fest:
\begin{itemize}
	\item $K$ sei ein Körper, $\nu: K\twoheadrightarrow\Gamma\cup\Set{\infty}$ eine Bewertung, $\mathcal{O}=\Set{x\in K | \nu(x)\geq 0}$ der zugehörige Bewertungsring, $\mathfrak{m}=\Set{x\in K | \nu(x)>0}$ sein maximales Ideal und $F=\mathcal{O}/\mathfrak{m}$ sein Restklassenkörper.
	
	\item Wir werden die Bezeichnung $\nu$ auch für die Bewertung von Matrizen, Zeilen- und Spaltenvektoren mit Einträgen aus $K$ verwenden. Wir definieren also für alle ${A\in K^{n\times m}}$ die Fortsetzung
	\[\nu(A) := \min\Set{\nu(A_{ij}) | 1\leq i \leq n, 1\leq j\leq m}\in\Gamma\cup\Set{\infty}.\]
	
	\item $H$ sei eine endlichdimensionale $K$"~Algebra.
	\item $H$ sei weiter eine symmetrische $K$"~Algebra mit Spurform $\tau: H\to K$.
	\item $\Lambda$ sei eine Indexmenge für die einfachen $H$"~Moduln.
	\item $H$ sei zerfallend halbeinfach. Insbesondere hat jeder einfache Modul ein Schur"=Element (siehe \citep[Ch.\,7]{geckpfeiffer} für eine Definition von Schur"=Elementen und Anwendungen in der Darstellungstheorie symmetrischer Algebren). Wir bezeichnen das Schur"=Element der Moduln mit Isomorphietyp $\lambda\in\Lambda$ mit $c_\lambda$.
	\item $\ast:H\to H$ sei ein $K$"~linearer Antiautomorphismus mit $h^{\ast\ast}=h$.
	\item Es gebe eine \udot{$\ast$"~symmetrische} $K$"~Basis $B\subseteq H$, d.\,h. $B$ ist eine Basis mit $B^\ast=B$ so, dass $b^\ast$ für alle $b\in B$ mit dem dualen Basiselement $b^\vee$ übereinstimmt. Mit anderen Worten soll $\tau(b\cdot c^\ast)=\delta_{bc}$ für alle $b,c\in B$ gelten.
\end{itemize}
\end{convention}

\begin{remark}
Die Bedingung, dass es eine $\ast$"~symmetrische Basis gibt, impliziert ${\tau(h^\ast)=\tau(h)}$ für alle $h\in H$: Es gilt $\tau(bc^\ast) = \delta_{bc} = \delta_{cb} = \tau(cb^\ast) = \tau((bc^\ast)^\ast)$ für alle $b,c\in B$. Da diese Bedingung $K$"~bilinear ist, folgt $\tau(xy^\ast) = \tau((xy^\ast)^\ast)$ für alle $x,y\in H$ und somit insbesondere für $y=1$.

Über Körpern mit hinreichend vielen Quadratwurzeln gilt auch die Umkehrung, falls $\CharFld(K)\neq 2$. Dazu betrachten wir die Bilinearform $(x,y)\mapsto\tau(x^\ast y)$. Ist die Spurform $\ast$"~invariant, so ist diese Bilinearform symmetrisch und eine $\ast$"~symmetrische Basis ist nichts anderes als eine Orthonormalbasis für diese Bilinearform.
\end{remark}

\begin{example}[Laurent-Polynome und rationale Funktionen]
\index{terms}{Laurent-Polynome}
\index{symbols}{nu@$\nu$}
Für jede total geordnete, abelsche Gruppe $(\Gamma,+,\leq)$ und jeden Körper $F$ ist der Gruppenring $F[\Gamma]$ ein Integritätsbereich, wie wir schon in \ref{def:laurent_polynomials} bemerkt haben. Mit den dort vereinbarten Schreibweisen ist auf $F[\Gamma]$ eine kanonische Bewertung durch
\[\nu\Big(\sum_{\gamma\in\Gamma} r_\gamma v^\gamma\Big):=\min\Set{\gamma\in\Gamma | r_\gamma\neq 0}\]
definiert. Sie kann durch $\nu(\frac{f}{g})=\nu(f)-\nu(g)$ eindeutig auf $K:=F(\Gamma)$ fortgesetzt werden. Der Restklassenkörper von $\mathcal{O}$ kann via $\sum_{\gamma\geq 0} r_\gamma v^\gamma \mapsto r_0$ mit $F$ selbst identifiziert werden.
\end{example}

\begin{example}[Hecke-Algebren]\label{ex:Hecke1}
\index{terms}{Hecke-Algebra}
Das Beispiel von vorrangigem Interesse ist das der Hecke"=Algebren. Dafür sei $(W,S,L)$ eine endliche Coxeter"=Gruppe mit Gewichtsfunktion $L:W\to\Gamma$. Wie im vorherigen Beispiel benutzen wir die kanonische Bewertung auf $K:=\IQ_W(\Gamma)$ und identifizieren den Restklassenkörper $F$ mit $\IQ_W$ selbst.

Die Hecke"=Algebra $H=H(W,S,L)$ erfüllt nun die Voraussetzungen. Es gibt die Standardbasis $(T_w)_{w\in W}$, die kanonische Spurform $\tau(T_w) := \delta_{w,1}$ sowie den Antiautomorphismus $T_w^\ast := T_{w^{-1}}$ bezüglich dessen $(T_w)$ eine $\ast$"~symmetrische Basis ist.
\end{example}

\begin{lemma}\label{symm_alg:symmetric_base_change}
\index{terms}{Matrix!orthogonale}\index{terms}{$\ast$-symmetrisch}
\index{symbols}{On@$O_n(K)$}
Sind $(b_i)_{i=1\ldots n}$ und $(c_i)_{i=1\ldots n}$ zwei $\ast$"~symmetrische Basen von $H$, dann ist die Basiswechselmatrix eine orthogonale Matrix. Umgekehrt liefert jeder Basiswechsel mit $A\in O_n(K)$ wieder eine $\ast$"~symmetrische Basis.
\end{lemma}
\begin{proof}
Sei $A\in K^{n\times n}$ die Matrix mit $b_i = \sum_{k=1}^n A_{ik} c_k$. Wenden wir $\ast$ an, so erhalten wir
\[b_i^\vee = b_i^\ast = \sum_{k=1}^n A_{ik} c_k^\ast = \sum_{k=1}^n A_{ik} c_k^\vee\]
und somit:
\begin{align*}
	\delta_{ij} &= \tau(b_i b_j^\vee) \\
	&= \sum_{k,l=1\ldots n} A_{ik} A_{jl} \underbrace{\tau(c_k c_l^\vee)}_{=\delta_{kl}} \\
	&= \sum_{s=1\ldots n} A_{is}A_{js} \\
	&= (AA^\text{Tr})_{ij}
\end{align*}

Die Umkehrung ergibt sich aus der analogen Rechnung.
\end{proof}
\section{Balancierte Darstellungen}

\begin{definition}[$a$-Werte]
\index{terms}{Schur-Element}
\index{symbols}{alambda@$a_\lambda$}
Betrachte die Schur"=Elemente $c_\lambda$. Wir definieren $a_\lambda\in\frac{1}{2}\Gamma$ durch
\[a_\lambda := -\tfrac{1}{2}\nu(c_\lambda)\in\tfrac{1}{2}\Gamma.\]
\end{definition}

\begin{definition}[Partielle Schnitte, $f$-Werte]
\index{terms}{Partieller Schnitt}
\index{symbols}{flambda@$f_\lambda$}
Ist $a_\lambda\in\Gamma$ für alle $\lambda\in\Lambda$, dann setze $\Gamma_0:=\braket{a_\lambda \mid \lambda\in\Lambda} \leq \Gamma$.

Einen Homomorphismus $\Gamma_0\to K^\times, \gamma\mapsto v^\gamma$ mit $ \nu(v^\gamma)=\gamma$ nennen wir einen \udot{partiellen Schnitt}. Einen Homomorphismus $\Gamma\to K^\times$ mit dieser Eigenschaft nennen wir einen \udot{globalen Schnitt}.

Ist ein partieller Schnitt gewählt worden, dann definieren wir $f_\lambda:=v^{2a_\lambda} c_\lambda$. Man beachte, dass dies wegen $\nu(c_\lambda)=-2a_\lambda$ in $\mathcal{O}^\times$ liegt.
\end{definition}

\begin{remark}
Man beachte, dass $\frac{1}{2}\Gamma$ echt größer sein kann als $\Gamma$, weil $\Gamma$ nicht $2$-teilbar zu sein braucht. Im Beispiel der Hecke"=Algebren gilt beispielsweise oft $\Gamma\isomorphic\IZ$, sodass $\Gamma\neq\frac{1}{2}\Gamma$ tatsächlich der Fall ist. Es wird sich in diesem Beispiel jedoch herausstellen, dass $\nu(c_\lambda)\in 2\Gamma$ gilt, sodass $a_\lambda\in\Gamma$ ist.
\end{remark}
\begin{remark}
Partielle Schnitte existieren immer, weil $\Gamma_0$ ja frei abelsch ist. Globale Schnitte hingegen brauchen nicht zu existieren, weil beispielsweise $\Gamma$ teilbar sein könnte, während $K^\times$ dies nicht zu sein braucht.

Wenn $\Gamma$ insgesamt endlich erzeugt ist, kann man sogar einen globalen Schnitt wählen. Im Fall der Funktionenkörper $K=F(\Gamma)$ ist durch die Einbettung $\Gamma\to F[\Gamma]\to F(\Gamma)$ eine kanonische Wahl eines solchen Schnittes möglich. Weil das bei beliebigen Bewertungsringen nicht der Fall ist, benötigen wir obige Definition. Die meisten der nachfolgenden Überlegungen sind unabhängig von der Wahl des partiellen Schnittes.

$f_\lambda$ ist jedoch von der Wahl des Schnittes abhängig, aber wegen $f_\lambda = v^{2a_\lambda} c_\lambda$ ist die Nebenklasse von $f_\lambda$ in $\mathcal{O}^\times / (\mathcal{O}^\times)^2$ eindeutig bestimmt. Wählt man nämlich einen anderen Schnitt, dann ist der entsprechende $f$-Wert $\widetilde{f_\lambda} = (\frac{\widetilde{v}^{a_\lambda}}{v^{a_\lambda}})^2 f_\lambda$. Falls wir also $F\subseteq\IR$ betrachten, ist etwa das Vorzeichen von $f_\lambda$ eindeutig festgelegt.
\end{remark}
\begin{remark}
\index{terms}{Komplexität}
Ist $\rho:H\to K^{d\times d}$ eine Matrixdarstellung von $H$, so liefern die Bewertungen $\nu(\rho(b))$ der Matrizen der Basiselemente ein (recht grobes) Maß für die "`Komplexität"' dieser Darstellung. Falls $K=F(\Gamma)$ ein Körper von rationalen Funktionen ist und sogar $\rho(b)\in F[\Gamma]^{d\times d}$ für alle $b\in B$ gilt, ist $\nu(\rho(b))$ gleich dem kleinsten in einem Eintrag $\rho(b)_{ij}$ vorkommenden Exponenten. Minimale und maximale vorkommende Exponenten sind ein grobes Maß für die Speicherkapazität, die die Speicherung dieser Matrizen benötigen würde.

Das folgende Lemma liefert eine Abschätzung für dieses Komplexitätsmaß und eine Motivation für die Definition balancierter Darstellungen als Darstellungen von minimaler Komplexität.
\end{remark}

\begin{lemmadef}[Balancierte Darstellungen]
\index{terms}{Darstellung!balancierte}\index{terms}{Führende Koeffizienten}
\index{symbols}{cx@$c(x)$}
Sei $B$ eine $\ast$"~symmetrische Basis von $H$. Ist $\rho: H\to K^{d\times d}$ eine irreduzible Matrixdarstellung vom Isomorphietyp $\lambda\in\Lambda$, so gilt
\[\nu(\rho(b)) \leq -a_\lambda\]
für mindestens ein $b\in B$.

Wir nennen $\rho$ \udot{balanciert}, falls $a_\lambda\in\Gamma$ und für alle $\ast$"~symmetrischen Basen $B$ diese Schranke scharf ist, also
\[\forall b\in B: \nu(\rho(b)) \geq -a_\lambda\]
gilt.

Ist $\rho$ balanciert und $\gamma\mapsto v^\gamma$ ein partieller Schnitt, dann definieren wir die Matrizen der \udot{führenden Koeffizienten} $c(x)\in F^{d\times d}$ durch
\[\forall x\in B: c(x) := v^{a_\lambda} \rho(x) \mod \mathfrak{m}.\]
\end{lemmadef}
\begin{proof}
Wäre die Schranke verletzt, d.\,h. wäre $\nu(\rho(b))>-a_\lambda$ für alle $b\in B$, dann folgt aus den Schur-Relationen:
\[v^{2a_\lambda} c_\lambda = \sum_{x\in B} (v^{a_\lambda} \rho(x)_\mathfrak{st})(v^{a_\lambda} \rho(x^\vee)_\mathfrak{ts}) \in\mathfrak{m}\]
Andererseits ist $a_\lambda$ gerade so definiert worden, dass $v^{2a_\lambda} c_\lambda$ eben nicht mehr in $\mathfrak{m}$ liegt.
\end{proof}

\begin{remark}
Ist $\gamma\mapsto\widetilde{v}^\gamma$ ein weiterer partieller Schnitt, dann gilt $\nu(v^\gamma)=\gamma=\nu(\widetilde{v}^\gamma)$, also $\frac{\widetilde{v}^\gamma}{v^\gamma}\in\mathcal{O}^\times$. Die führenden Koeffizienten bezüglich dieser beiden Schnitte unterscheiden sich genau um die Konstante $\frac{\widetilde{v}^{a_\lambda}}{v^{a_\lambda}}\mod\mathfrak{m}$.
\end{remark}

\subsection{Einschub: Formal reelle Körper}

\begin{remark}
Wir erinnern an die Definition formal reeller Körper und ihre elementaren Eigenschaften (siehe etwa \citep[Ch.\,20]{lorenz2007algebra} für eine ausführlichere Behandlung).
\end{remark}

\begin{lemmadef}
\index{terms}{formal reell|(}
Für einen Körper $L$ sind äquivalent:
\begin{enumerate}
	\item Es gibt eine Totalordnung auf $L$, mit der $L$ zu einem angeordneten Körper wird.
	\item $-1$ ist keine Summe von Quadraten in $L$.
	\item Es gibt ein Element, das keine Summe von Quadraten ist, und $L$ ist nicht von Charakteristik $2$.
	\item $\sum_{i=1}^n x_i^2 = 0 \implies \forall i: x_i=0$.
\end{enumerate}
Gegebenenfalls heißt $L$ \udot{formal reell}.
\end{lemmadef}

\begin{lemma}
Ist in unserer Situation der Restklassenkörper $F$ formal reell, so ist auch $K$ formal reell.
\end{lemma}
\begin{proof}
Seien $x_i\in K^\times$ mit $-1 = \sum_{i=1}^k x_i^2$. Dann setze $\alpha := -\min\Set{\nu(x_i) | i=1,\ldots,k}$ und wähle ein $v^\alpha\in K$ mit Bewertung $\alpha$. Dann erhalten wir:
\[-v^{2\alpha} = \sum_{i=1}^k (\underbrace{v^\alpha x_i}_{\in\mathcal{O}})^2\]
Nach Definition ist $v^\alpha x_i\in\mathcal{O}^\times$ für mindestens ein $i$. Wenn wir also modulo $\mathfrak{m}$ reduzieren, verschwinden auf der rechten Seite nicht alle Quadrate. Weil $F$ formal reell ist, ist die rechte Seite daher ungleich Null. Das heißt, dass auch die linke Seite ungleich Null modulo $\mathfrak{m}$ ist. Daher muss $\alpha=0$ sein.

Das hieße aber, dass $x_i\in\mathcal{O}$ für alle $i$ und $x_i\in\mathcal{O}^\times$ für mindestens ein $i$ ist und somit $-1 \equiv \sum_{i=1}^k x_i^2 \mod\mathfrak{m}$ gilt im Widerspruch dazu, dass $F$ formal reell ist. Also ist, wie behauptet, auch $K$ formal reell.
\end{proof}

\begin{lemma}[Symmetrische, semidefinite Matrizen über formal reellen Körpern]\label{formally_real:symmetric_matrices}
\index{terms}{Matrix!positiv semidefinite}\index{terms}{Matrix!symmetrische}
Ist $L$ ein formal reeller Körper und sind $X_i\in L^{n\times m}$ Matrizen, dann gilt:
\begin{enumerate}
	\item $X:=\sum_{i=1}^k X_i^\text{Tr} X_i$ ist positiv semidefinit in dem Sinne, dass $v^\text{Tr} X v$ für alle $v\in L^m$ eine Summe von Quadraten ist.
	\item Weiter gilt:
	\[\forall v\in L^m: v^\text{Tr} X v = 0 \implies X_i v=0\]
	Und somit insbesondere $X=0 \implies \forall i: X_i=0$.
	\item Ist $L=K$ und $F$ formal reell, so gilt für die Bewertungen:
	\begin{enumerate}
		\item $\nu(X) = 2\min\Set{\nu(X_i) | 1\leq i\leq k}$.
		\item $\nu(X_{jj}) = 2\min\Set{\nu(X_i e_j) | 1\leq i\leq k}$ für alle $j=1,\ldots,m$.
	\end{enumerate}
\end{enumerate}
\end{lemma}
\begin{proof}
a. Sei $v\in L^m$. Dann gilt:
\begin{align*}
	v^\text{Tr} X v &= \sum_{i=1}^k (X_i v)^\text{Tr}(X_i v) \\
	&= \sum_{i=1}^k \sum_{j=1}^n (X_i v)_j^2 \\
	&\geq 0
\end{align*}

b. ergibt sich dann aus a. wie folgt:
\begin{align*}
	0 &= v^\text{Tr} X v \\
	\implies \forall i,j: 0 &= (X_i v)_j \\
	\implies \forall i: 0 &= X_i v
\end{align*}
Wenn nun also $X=0$ ist, gilt $v^\text{Tr} X v=0$ für alle $v$ und nach obiger Überlegung auch $X_i v=0$ für alle $i$ und alle $v$. Das heißt aber $X_i=0$.

Setze für c. nun $\alpha:=\min\Set{\nu(X_i) | 1\leq i\leq k}$. Wähle ein $v^\alpha\in K$ mit Bewertung $\alpha$. Dann ist $v^{-\alpha} X_i\in\mathcal{O}^{n\times m}$ für alle $i$ und $v^{-\alpha} X_i\not\equiv 0 \mod\mathfrak{m}^{n\times m}$ für mindestens ein $i$. Es ist also $v^{-2\alpha} X = \sum_i (v^{-\alpha} X_i)^\text{Tr} (v^{-\alpha} X_i)\in\mathcal{O}^{m\times m}$ und $v^{-2\alpha} X \not\equiv 0 \mod\mathfrak{m}^{m\times m}$ aufgrund von a. Das zeigt $\nu(X)=2\alpha$.

Es ist $X_{jj} = e_j^\text{Tr} X e_j = \sum_i (X_i e_j)^\text{Tr} (X_i e_j)$. Dieselbe Überlegung angewendet auf die $1\times m$-Matrizen $X_i e_j$ ergibt die zweite Behauptung in c.
\end{proof}

\begin{lemma}[Orthogonale Matrizen über $K$ und $\mathcal{O}$]\label{formally_real:orthogonal_matrices}
\index{terms}{Matrix!orthogonale}
Ist in unserer Situation $F$ formal reell, so ist
\[O_n(K) = O_n(\mathcal{O}).\]
\end{lemma}
\begin{proof}
Die eine Inklusion ist klar. Sei für die andere $A\in O_n(K)$ beliebig. Dann setze $\alpha:=-\nu(A)$ und wähle $v^\alpha\in K$ mit Bewertung $\alpha$. Dann ist $v^\alpha A\in\mathcal{O}^{n\times n}$, $v^\alpha A \not\equiv 0 \mod\mathfrak{m}^{n\times n}$ und $v^{2\alpha} I = (v^\alpha A)^\text{Tr}(v^\alpha A)$. Weil $F$ formal reell ist, ist wegen \ref{formally_real:symmetric_matrices} die Matrix auf der linken Seite ungleich Null modulo $\mathfrak{m}^{n\times n}$. Das heißt, dass $2\alpha=0$ sein muss, d.\,h. $A\in\mathcal{O}^{n\times n}$, wie behauptet.
\end{proof}

\begin{corollary}
\index{terms}{Darstellung!balancierte}\index{terms}{$\ast$-symmetrisch}
Ist $F$ formal reell, dann hängt die Balanciertheit von $\rho: H\to K^{d\times d}$ nicht von der Wahl der $\ast$"~symmetrischen Basis ab.
\end{corollary}
\begin{proof}
Sind $(b_i)_{i=1\ldots n}$ und $(c_i)_{i=1\ldots n}$ zwei $\ast$"~symmetrische Basen von $H$, dann ist die Basiswechselmatrix $A\in O_n(K)=O_n(\mathcal{O})$, wie in \ref{symm_alg:symmetric_base_change} gesehen.
	
Jetzt gilt für alle Matrizendarstellungen:
\[\rho(b_i) = \sum_{k=1}^n A_{ik} \rho(c_k)\]
Wenn also $\nu(\rho(c_k))\geq -a_\lambda$ für alle $k$ gilt, gilt das auch für $\rho(b_i)$, weil die Koeffizienten in $\mathcal{O}$ liegen.
\end{proof}

\subsection{Ein Kriterium für Balanciertheit über formal reellen Körpern}

\begin{remark}
Wir beweisen nun ein hinreichendes Kriterium für Balanciertheit.
\end{remark}

\begin{theorem}[Balanciertheit aus invarianten Bilinearformen]\label{balanced_reps:invariant_blf1}
\index{terms}{Invariante Bilinearform|(}\index{terms}{Führende Koeffizienten}\index{terms}{Darstellung!balancierte|(}
Sei $F$ formal reell. Sei $\rho: H\to K^{d\times d}$ eine Matrizendarstellung vom Isomorphietyp $\lambda$. Es gilt:
\begin{enumerate}
	\item Gibt es ein $\Omega\in\mathcal{O}^{d\times d}$ mit
	\[\forall x\in B: \Omega\rho(x) = \rho(x^\ast)^\text{Tr} \Omega\]
	und ist
	\[D:=\Omega+\mathfrak{m}^{d\times d}\]
	eine Diagonalmatrix $D=\diag(d_\mathfrak{s}) \in GL_d(F)$, dann ist $a_\lambda\in\Gamma$ und $\rho$ balanciert. Die Führenden-Koeffizienten-Matrizen erfüllen die zusätzliche Bedingung
	\[\forall x\in B: d_\mathfrak{t} \cdot c(x^\ast)_\mathfrak{ts} = d_\mathfrak{s} \cdot c(x)_\mathfrak{st}.\]
	\item Gibt es ein $\Omega\in GL_d(\mathcal{O})$ mit $\Omega^\text{Tr}=\Omega$ und
	\[\forall x\in B: \Omega\rho(x) = \rho(x^\ast)^\text{Tr} \Omega\]
	dann ist $a_\lambda\in\Gamma$ und $\rho$ balanciert.
\end{enumerate}
\end{theorem}
\begin{proof}
Wir setzen
\[a := -\min\Set{\nu(\rho(x)_\mathfrak{st}) | x\in B,\mathfrak{s},\mathfrak{t}=1,\ldots,d}\]

Wir wählen ein $v^a \in K^\times$ mit Bewertung $a$ und setzen:
\[\widetilde{c}(x) := v^a\cdot \rho(x) \mod\mathfrak{m}^{d\times d}\]
Es gilt jetzt also $v^a \rho(x)\in\mathcal{O}^{d\times d}$ für alle $x\in B$ und $v^a \rho(x)\not\equiv 0 \mod\mathfrak{m}^{d\times d}$ für mindestens ein $x\in B$.

\medbreak
Unser Ziel ist es nun, $\nu(c_\lambda)=-2a$, d.\,h. $a=a_\lambda$ zu zeigen. Wir setzen zunächst in die Voraussetzungen ein und erhalten:
\[\Omega \cdot \widetilde{c}(x) = \widetilde{c}(x^\ast)^\text{Tr} \cdot \Omega\]
Weil nun die Restklasse $\Omega+\mathfrak{m}^{d\times d}$ eine Diagonalmatrix ist, erhalten wir die Gleichungen
\[\forall \mathfrak{s},\mathfrak{t}: d_\mathfrak{s} \widetilde{c}(x)_\mathfrak{st} = \widetilde{c}(x^\ast)_\mathfrak{st}^\text{Tr} d_\mathfrak{t} = \widetilde{c}(x^\vee)_\mathfrak{ts} d_\mathfrak{t}\]
Das setzen wir nun in die Schur-Relationen ein:
\begin{alignat*}{2}
	 v^{2a} \cdot c_\lambda  &= \sum_{x\in B} \underbrace{v^a \rho(x)_\mathfrak{st}}_{\in\mathcal{O}} \cdot \underbrace{v^a \rho(x^\vee)_\mathfrak{ts}}_{\in\mathcal{O}} \in\mathcal{O} & \implies 2a+\nu(c_\lambda)\geq 0 \\
	 &= \sum_{x\in B} \widetilde{c}(x)_\mathfrak{st} \cdot \widetilde{c}(x^\vee)_\mathfrak{ts} \mod\mathfrak{m}\\
	 &=\sum_{x\in B} \widetilde{c}(x)_\mathfrak{st} \frac{d_\mathfrak{s}}{d_\mathfrak{t}} \widetilde{c}(x)_\mathfrak{st} \\
	 &= \frac{d_\mathfrak{s}}{d_\mathfrak{t}} \sum_{x\in B} \widetilde{c}(x)_\mathfrak{st}^2
\end{alignat*}
Weil $F$ nach Annahme formal reell ist, ist die Summe ungleich Null. Weil $d_\mathfrak{s}\neq 0$ für alle $\mathfrak{s}$ ist, ist also der gesamte Term ungleich Null. Das zeigt, dass $\nu(c_\lambda)+2a=0$ ist, wie gewünscht. Die Matrizen $\widetilde{c}(x)$ sind insbesondere die Führende"=Koeffizienten"=Matrizen von $\rho$. Das zeigt die Behauptung a.

\bigbreak
In der Situation von b. ist $\Omega\mod\mathfrak{m}$ symmetrisch und hat eine von Null verschiedene Determinante, induziert also eine nichtentartete, symmetrische Bilinearform auf $F^d$. Weil $\CharFld(F)=0$ ist, können wir eine Orthogonalbasis wählen, d.\,h. wir finden einen Basiswechsel $P\in GL_d(\mathcal{O})$ derart, dass 
\[\Omega':=P^\text{Tr} \Omega P \equiv \diag(d_\mathfrak{s}) \mod\mathfrak{m}^{d\times d}\]
ist (finde ein passendes $\overline{P}\in GL_d(F)$ und wähle irgendein Urbild $P\in GL_d(\mathcal{O})$). Die entsprechend geänderte Darstellung $\rho':=P^{-1} \rho P$ erfüllt dann die Voraussetzungen von a. Weil Balanciertheit unter $GL_d(\mathcal{O})$-Konjugation invariant ist, ist $\rho$ damit balanciert.
\end{proof}
\begin{definition}\label{balanced_reps:strictly_balanced}
\index{terms}{Darstellung!balancierte!strikt}
Eine Matrixdarstellung, die nicht nur balanciert ist, sondern auch die stärkere Bedingung aus a. erfüllt, wollen wir \udot{strikt balanciert} nennen.
\end{definition}

\begin{remark}
Es bleibt zu zeigen, dass solche invarianten Bilinearformen auch immer existieren:
\end{remark}
\begin{theorem}\label{balanced_reps:invariant_blf2}
\index{terms}{Invariante Bilinearform|)}\index{terms}{Matrix!positiv definite}
Sei $F$ formal reell. Sei $V_\lambda$ ein einfacher $H$"~Modul vom Isomorphietyp $\lambda\in\Lambda$. Dann gilt:
\begin{enumerate}
	\item Es gibt eine symmetrische Bilinearform $\braket{\cdot,\cdot}_\lambda: V_\lambda\times V_\lambda\to K$, die $H$"~invariant im Sinne von
	\[\forall h\in H: \braket{h^\ast v,w}_\lambda = \braket{v,h w}_\lambda\]
	und positiv definit ist in dem Sinne, dass $\braket{w,w}_\lambda$ stets eine Summe von Quadraten ist, die genau dann Null ist, wenn $w=0$ ist. Insbesondere ist $\braket{\cdot,\cdot}_\lambda$ nichtentartet.
	\item Für jedes $w\in V_\lambda\setminus\Set{0}$ gibt es ein $\gamma\in\Gamma$ mit $\nu(\braket{w,w}_\lambda)=2\gamma$.
	\item Es gilt $a_\lambda\in\Gamma$ und es gibt eine balancierte Matrixdarstellung vom Isomorphietyp~$\lambda$.
\end{enumerate}
\end{theorem}
\begin{proof}
Sei $(v_i)_{i=1\ldots d}$ eine $K$"~Basis von $V:=V_\lambda$. Wir bezeichnen die zugehörige Matrixdarstellung mit $\rho$.

Dann ist $\widehat{V}=\Hom_K(V,K)$ via $h\cdot f:v\mapsto f(h^\ast\cdot v)$ ein $H$"~Modul und bzgl. der dualen Basis $(v_i^\ast)_{i=1\ldots d}$ hat dieser die Matrizendarstellung $\widehat{\rho}:H\to K^{d\times d}, h\mapsto\rho(h^\ast)^\text{Tr}$.

Wir setzen dann
\[\Omega_1 := \sum_{x\in B} \rho(x)^\text{Tr} \rho(x) = \sum_{x\in B} \widehat{\rho}(x^\ast)\rho(x) = \sum_{x\in B} \widehat{\rho}(x^\vee)\cdot 1_{d\times d}\cdot \rho(x).\]
Für eine $K$"~lineare Abbildung $f: V\to\widehat{V}$ ist die Gaschütz-Ikeda-Projektion $\Phi(f):V\to\widehat{V}$ durch
\[\Phi(f)v:= \sum_{x\in B} x^\vee \cdot f(x\cdot v)\]
definiert (siehe \citep[7.1.9]{geckpfeiffer}). Die Matrix $\Omega_1$ ist somit die Darstellungsmatrix von $\Phi(f)$ für die durch $f(v_i):=v_i^\ast$ definierte lineare Abbildung. $\Phi(f)$ ist $H$"~linear (siehe \citep[7.1.10]{geckpfeiffer}), d.\,h. $\Phi(f)(h\cdot v) = h\cdot\Phi(f)(v)$ für alle $h\in H$. In Matrizenschreibweise heißt das
\[\forall h\in H: \Omega_1 \rho(h)=\widehat{\rho}(h) \Omega_1 = \rho(h^\ast)^\text{Tr} \Omega_1.\]

Die durch $\Omega_1$ auf $V$ definierte Bilinearform $\braket{\cdot,\cdot}_\lambda$ ist nach Konstruktion symmetrisch und aufgrund des eben Bewiesenen auch $H$"~invariant. Wir müssen jetzt noch zeigen, dass $\Omega_1$ nichtentartet ist. Das folgt daraus, dass $F$, und somit auch $K$, formal reell ist. Dann zeigt uns das Lemma \ref{formally_real:symmetric_matrices}, dass $\Omega_1$ ungleich Null ist. Weil aber $V$ und $\widehat{V}$ einfache Moduln sind und $\Omega_1$ die Darstellungsmatrix von $\Phi(f): V\to\widehat{V}$ ist, sagt uns das Lemma von Schur, dass $\Omega_1$ sogar invertierbar sein muss.

Nach Konstruktion ist auch klar, dass $w^\text{Tr} \Omega_1 w=\sum_{x\in B} (\rho(x) w)^\text{Tr}\cdot (\rho(x)w)$ eine Summe von Quadraten ist.

\bigbreak
b. Sei nun $0\neq w\in K^d$ beliebig. Dann ist $w^\text{Tr} \Omega_1 w$ eine Summe von Quadraten. Aus \ref{formally_real:symmetric_matrices} folgt $\nu(w^\text{Tr}\Omega_1 w)\in 2\Gamma$.

\bigbreak
c. Weil $\braket{\cdot,\cdot}_\lambda$ positiv definit ist, gibt es eine Orthogonalbasis $(w_i)_{i=1\ldots d}$ bzgl. $\braket{\cdot,\cdot}_\lambda$ von $V$. Setze jeweils $\gamma_i:=-\frac{1}{2}\nu(w_i)$. Das ist aufgrund von b. in $\Gamma$. Wähle nun Elemente $v^{\gamma_i}\in K$ mit Bewertung $\gamma_i$ und setze $u_i:=v^{\gamma_i} w_i$. Dann ist $(u_i)_{i=1\ldots d}$ immer noch eine Orthogonalbasis von $V_\lambda$, aber jetzt ist $\braket{u_i,u_i}_\lambda \in \mathcal{O}^\times$.

Wenn wir nun $\Omega$ als Darstellungsmatrix von $\braket{\cdot,\cdot}_\lambda$ und $\rho$ als Matrixdarstellung von $V$ bzgl. der Basis $(u_i)_{i=1\ldots d}$ wählen, sind die Bedingungen aus dem vorherigen Satz erfüllt.
\end{proof}

\begin{corollary}\label{balanced_reps:characters}
\index{terms}{formal reell|)}\index{terms}{Charakter}
Sei $F$ formal reell. Dann gilt für alle $\lambda\in\Lambda$:
\begin{enumerate}
	\item Für alle $h\in H$ ist $\chi_\lambda(h)=\chi_\lambda(h^\ast)$.
	\item Für alle $x\in B$ ist $\nu(\chi_\lambda(x)) \geq -a_\lambda$ mit Gleichheit für mindestens ein $x\in B$.
\end{enumerate}
\end{corollary}
\begin{proof}
Aufgrund des eben bewiesenen Satzes können wir eine balancierte Matrixdarstellung $\rho: H\to K^{d_\lambda\times d_\lambda}$ vom Isomorphietyp $\lambda$ und eine invariante Bilinearform mit Darstellungsmatrix $\Omega$ wählen. Haben wir das, dann erhalten wir
\[\chi_\lambda(x) = \tr(\rho(x)) = \tr(\Omega^{-1}\rho(x^\ast)^\text{Tr}\Omega) = \tr(\rho(x^\ast)^\text{Tr}) = \tr(\rho(x^\ast)) = \chi_\lambda(x^\ast).\]
Nun folgt zum Einen
\[v^{a_\lambda} \chi_\lambda(x) = \sum_\mathfrak{s} \underbrace{v^{a_\lambda} \rho(x)_\mathfrak{ss}}_{\in\mathcal{O}}\]
und zum anderen
\[\sum_{x\in B} (v^{a_\lambda} \chi_\lambda(x))^2 = \sum_{x\in B} v^{2a_\lambda} \chi_\lambda(x)\chi_\lambda(x^\ast) = v^{2a_\lambda} c_\lambda d_\lambda \equiv f_\lambda d_\lambda \mod\mathfrak{m}\]
Daher muss mindestens ein Summand auf der linken Seite in $\mathcal{O}^\times$ liegen.
\end{proof}

\subsection{Eigenschaften balancierter Darstellungen}

\begin{theorem}[Schur-Relationen für führende Koeffizienten]\label{schur_relations_leading_coeff}
\index{terms}{Darstellung!balancierte|)}\index{terms}{Schur-Relationen}
Sei $B$ eine $\ast$"~invariante Basis und $F$ formal reell wie zuvor.

Sind $\rho_\lambda: H\to K^{d_\lambda\times d_\lambda}$ balancierte Matrizendarstellungen für alle $\lambda\in\Lambda$, dann gilt für die zugeordneten führenden Koeffizienten:
\begin{enumerate}
	\item Für alle $\lambda,\mu\in\Lambda$ gilt:
	\[\sum_{x\in B} c^\lambda(x)_\mathfrak{st} c^\mu(x^\ast)_\mathfrak{uv} = \begin{cases} f_\lambda \delta_\mathfrak{sv} \delta_\mathfrak{tu} & \lambda=\mu \\ 0 & \text{sonst} \end{cases}\]
	\item Für alle $x,y\in B$ gilt:
	\[\sum_{\substack{\lambda\in\Lambda \\ 1\leq\mathfrak{s},\mathfrak{t}\leq d_\lambda}} f_\lambda^{-1} c^\lambda(x)_\mathfrak{st} c^\lambda(y^\ast)_\mathfrak{ts} = \begin{cases} 1 & x=y \\ 0 & \text{sonst} \end{cases}\]
	\item Falls $\lambda,\mathfrak{s},\mathfrak{t}$ gegeben sind, gibt es ein $x\in B$ mit $c^\lambda(x)_\mathfrak{st}\neq 0$. Ist umgekehrt $x\in B$ gegeben, so gibt es $\lambda,\mathfrak{s},\mathfrak{t}$ mit $c^\lambda(x)_\mathfrak{st}\neq 0$.
	
	Insbesondere ist
	\[-2a_\lambda = \min\Set{\nu(\rho_\lambda(x)_\mathfrak{st}) | x\in B, 1\leq\mathfrak{s},\mathfrak{t}\leq d_\lambda}\]
\end{enumerate}
\end{theorem}
\begin{proof}
a. folgt sofort, indem man die Schur-Relationen
\[\sum_{x\in B} \rho_\lambda(x^\vee)_\mathfrak{st} \rho_\mu(x)_\mathfrak{uv} = \begin{cases} c_\lambda \delta_\mathfrak{sv} \delta_\mathfrak{tu} & \lambda=\mu \\ 0 & \text{sonst}\end{cases}\]
mit $v^{2a_\lambda}$ multipliziert und modulo $\mathfrak{m}$ reduziert.

\bigbreak
b. folgt aus a., indem man die Gleichung als Matrizengleichung auffasst, wobei $B$ die Indexmenge für die Spalten und $\Set{(\lambda,\mathfrak{s},\mathfrak{t}) | \lambda\in\Lambda, 1\leq\mathfrak{s},\mathfrak{t}\leq d_\lambda}$ die Indexmenge für die Zeilen ist. Aussage a. behauptet dann eine Gleichung der Form $X \cdot Y^\text{Tr} = \diag(f_\lambda)$ und b. behauptet die entsprechende Gleichung $Y \cdot X^\text{Tr} \diag(f_\lambda)^{-1} = 1$.

\bigbreak
c. folgt aus a. und b.
\end{proof}
\section{Die asymptotische Algebra}

\begin{convention}
Wir nehmen in diesem Abschnitt an, dass $F$ formal reell ist. Wir wählen außerdem ein für alle Mal einen partiellen Schnitt $\Gamma\supseteq\langle a_\lambda \mid \lambda\in\Lambda\rangle\to K^\times, \gamma\mapsto v^\gamma$.
\end{convention}

\begin{definition}[$J$-Algebra nach Geck, vgl. {\citep[1.5]{geckjacon}}]\label{J_alg:def:J_algebra}
\index{terms}{Darstellung!balancierte}\index{terms}{Führende Koeffizienten}\index{terms}{Asymptotische Algebra|(}
\index{symbols}{tx@$t_x$}\index{symbols}{gammaxyz@$\gamma_{x,y,z}$}\index{symbols}{D@$\mathcal{D}$}\index{symbols}{nz@$n_z$}\index{symbols}{J@$J$|(}

Sei $B$ eine $\ast$"~symmetrische Basis von $H$. Definiere in dieser Situation
\[\gamma_{x,y,z} := \sum_{\lambda\in\Lambda} \sum_{\mathfrak{s},\mathfrak{t},\mathfrak{u}} f_\lambda^{-1} c^\lambda(x)_\mathfrak{st} c^\lambda(y)_\mathfrak{tu} c^\lambda(z)_\mathfrak{us}\]
\[n_x := \sum_{\lambda\in\Lambda} \sum_{\mathfrak{s}} f_\lambda^{-1} c^\lambda(x^\ast)_\mathfrak{ss}\]
\[\mathcal{D} := \set{x\in B | n_x \neq 0}\]
für alle $x,y,z\in B$.

Die \udot{asymptotische Algebra} (bzgl. $B$) ist nun definiert als die $F$"~Algebra $J$ mit der $F$"~Basis $(t_x)_{x\in B}$ und der Multiplikation
\[t_x t_y := \sum_{z} \gamma_{x,y,z} t_{z^\ast}.\]
\end{definition}

\begin{remark}
Die folgenden beiden Beweise verallgemeinern die Behandlung der asymptotischen Algebra in \cite{geckjacon} von der dort betrachteten konkreten Situation, dass $H$ die Hecke"=Algebra ist, auf den hier betrachteten allgemeinen Fall. Die Beweise sind allerdings fast wortgleich.
\end{remark}

\begin{theorem}\label{J_alg:welldef_gamma_n}
\index{terms}{Darstellung!balancierte}\index{terms}{Führende Koeffizienten}\index{terms}{Asymptotische Algebra}
In obiger Situation gilt für alle $x,y,z\in B$:

\begin{enumerate}
	\item Wohldefiniertheit: $\gamma_{x,y,z}$ und $n_z$ sind eindeutig durch $H$ bestimmt und hängen nicht von der Wahl der $\rho_\lambda$ ab. Genauer gilt:
	\[\gamma_{x,y,z} = \sum_{\lambda\in\Lambda} f_\lambda^{-1} v^{3a_\lambda} \chi_\lambda(xyz) \mod\mathfrak{m}\]
	\[n_x = \sum_{\lambda\in\Lambda} f_\lambda^{-1} v^{a_\lambda} \chi_\lambda(x^\ast) \mod\mathfrak{m}\]
\end{enumerate}
Es gelten weiterhin folgende Rechenregeln:
\begin{enumerate}[resume]
	\item $\gamma_{x,y,z} = \gamma_{y,z,x}$
	\item $\sum\limits_{z\in B} \gamma_{x^\ast,y,z} n_z = \delta_{xy}$
	\item $\gamma_{x,y,z} = \gamma_{y^\ast,x^\ast,z^\ast}$
	\item $n_x=n_{x^\ast}$
\end{enumerate}
\end{theorem}
\begin{proof}
Die Behauptungen in a. ergeben sich direkt aus den Definitionen:
\begin{align*}
	\gamma_{x,y,z} &= \sum_{\lambda\in\Lambda} \sum_{\mathfrak{s},\mathfrak{t},\mathfrak{u}} f_\lambda^{-1} c^\lambda(x)_\mathfrak{st} c^\lambda(y)_\mathfrak{tu} c^\lambda(z)_\mathfrak{us} \\
	&= \sum_{\lambda\in\Lambda} \sum_{\mathfrak{s},\mathfrak{t},\mathfrak{u}} f_\lambda^{-1} (v^{a_\lambda} \rho_\lambda(x)_\mathfrak{st})(v^{a_\lambda} \rho_\lambda(y)_\mathfrak{tu})(v^{a_\lambda} \rho_\lambda(z)_\mathfrak{us}) \mod\mathfrak{m} \\
	&= \sum_\lambda f_\lambda^{-1} v^{3a_\lambda} \sum_\mathfrak{s} \rho_\lambda(xyz)_\mathfrak{ss} \\
	&= \sum_\lambda f_\lambda^{-1} v^{3a_\lambda} \chi_\lambda(xyz) \\
	n_x &= \sum_{\lambda\in\Lambda} \sum_{\mathfrak{s}} f_\lambda^{-1} c^\lambda(x^\ast)_\mathfrak{ss} \\
	&= \sum_\lambda f_\lambda^{-1} \sum_\mathfrak{s} v^{a_\lambda} \rho_\lambda(x^\ast)_\mathfrak{ss} \mod\mathfrak{m} \\
	&= \sum_\lambda f_\lambda^{-1} v^{a_\lambda} \chi_\lambda(x^\ast)
\end{align*}

b. ist erfüllt, weil $\gamma_{x,y,z}$ bereits symmetrisch bzgl. zyklischer Shifts definiert wurde.

c. folgt aus den Schur-Relationen:
\begin{align*}
	\sum_{z\in B} \gamma_{x^\ast,y,z} n_z &= \sum_{z\in B} \left(\sum_{\lambda\in\Lambda} \sum_{\mathfrak{s},\mathfrak{t},\mathfrak{u}} f_\lambda^{-1} c^\lambda(x^\ast)_\mathfrak{st} c^\lambda(y)_\mathfrak{tu} c^\lambda(z)_\mathfrak{us}\right) \cdot \left(\sum_{\mu\in\Lambda} \sum_{\mathfrak{v}} f_\mu^{-1} c^\mu(z^\ast)_\mathfrak{vv}\right) \\
	&= \sum_{\lambda,\mu} \sum_{\mathfrak{s},\mathfrak{t},\mathfrak{u},\mathfrak{v}} f_\lambda^{-1} f_\mu^{-1} c^\lambda(x^\ast)_\mathfrak{st} c^\lambda(y)_\mathfrak{tu} \smash{\underbrace{\sum_{z\in B} c^\lambda(z)_\mathfrak{us} c^\mu(z^\ast)_\mathfrak{vv}}_{=\delta_{\lambda\mu}\delta_\mathfrak{uv}\delta_\mathfrak{sv}\cdot f_\lambda}} \\
	&= \sum_\lambda \sum_{\mathfrak{s},\mathfrak{t}} f_\lambda^{-1} f_\lambda^{-1} c(x^\ast)_\mathfrak{st} c(y)_\mathfrak{ts}\cdot f_\lambda \\
	&= \sum_\lambda \sum_{\mathfrak{s},\mathfrak{t}} f_\lambda^{-1} c(x^\ast)_\mathfrak{st} c(y)_\mathfrak{ts} \\
	&= \delta_{xy}
\end{align*}

\bigbreak
d.+e. folgen aus a., weil wir bereits festgestellt haben, dass $\chi_\lambda(h)=\chi_\lambda(h^\ast)$ für alle $h\in H$ gilt. Damit gilt:
\begin{align*}
	\gamma_{x,y,z} &= \sum_\lambda f_\lambda^{-1} v^{3a_\lambda} \chi_\lambda(xyz) \mod\mathfrak{m} \\
	&= \sum_\lambda f_\lambda^{-1} v^{3a_\lambda} \chi_\lambda(z^\ast y^\ast x^\ast) \\
	&= \sum_\lambda f_\lambda^{-1} v^{3a_\lambda} \chi_\lambda(y^\ast x^\ast z^\ast) \\
	&= \gamma_{y^\ast, x^\ast, z^\ast} \\
	n_x &= \sum_\lambda f_\lambda^{-1} v^{a_\lambda} \chi_\lambda(x^\ast) \mod\mathfrak{m} \\
	&= \sum_\lambda f_\lambda^{-1} v^{a_\lambda} \chi_\lambda(x) \\
	&= n_{x^\ast} \qedhere
\end{align*}
\end{proof}

\begin{theorem}
\index{terms}{Symmetrische Algebra}\index{terms}{$\ast$-symmetrisch}
In obiger Situation gilt:
\begin{enumerate}
	\item $J$ ist eine assoziative $F$"~Algebra mit Einselement
	\[1_J = \sum_{d\in D} n_d t_d\]
	\item Durch $\overline{\tau}(t_x):=n_x$ ist eine Spurform auf $J$ gegeben, die $J$ zu einer symmetrischen Algebra macht.
	\item Durch $t_x^\ast := t_{x^\ast}$ ist ein $F$"~linearer Antiautomorphismus auf $J$ gegeben und $(t_x)_{x\in B}$ ist eine $\ast$"~symmetrische Basis von $J$.
\end{enumerate}
\end{theorem}
\begin{proof}
a. folgt wieder aus den Schur-Relationen. Wir formen zunächst die Behauptung um:
\begin{align*}
	(t_x t_y) t_z &= t_x (t_y t_z) \\
	\iff \sum_u \gamma_{x,y,u} t_{u^\ast} t_z &= \sum_u \gamma_{y,z,u} t_x t_{u^\ast} \\
	\iff \sum_{u,v} \gamma_{x,y,u} \gamma_{u^\ast,z,v} t_{v^\ast} &= \sum_{u,v} \gamma_{x,u^\ast,v} \gamma_{y,z,u} t_{v^\ast} \\
	\iff \forall v\in B: \sum_u \gamma_{x,y,u} \gamma_{u^\ast,z,v} &= \sum_u \gamma_{x,u^\ast,v} \gamma_{y,z,u}
\end{align*}
und das rechnen wir jetzt nach:
\begin{align*}
	\sum_u \gamma_{x,y,u} \gamma_{u^\ast,z,v} &= \sum_{\lambda,\mu} \sum_{\substack{\mathfrak{s},\mathfrak{t},\mathfrak{u} \\ \mathfrak{a},\mathfrak{b},\mathfrak{c}}} f_\lambda^{-1} c^\lambda(x)_\mathfrak{st} c^\lambda(y)_\mathfrak{tu} \cdot f_\mu^{-1} c^\mu(z)_\mathfrak{bc} c^\mu(v)_\mathfrak{ca} \cdot \smash{\underbrace{\sum_u c^\lambda(u)_\mathfrak{us} c^\mu(u^\ast)_\mathfrak{ab}}_{=\delta_{\lambda\mu} \delta_\mathfrak{ub} \delta_\mathfrak{sa} \cdot f_\lambda}} \\
	&= \sum_\lambda f_\lambda^{-1} \sum_{\substack{\mathfrak{s},\mathfrak{t},\mathfrak{u} \\ \mathfrak{c}}}  c^\lambda(x)_\mathfrak{st} c^\lambda(y)_\mathfrak{tu} \cdot c^\lambda(z)_\mathfrak{uc} c^\lambda(v)_\mathfrak{cs} \\
	&= \sum_\lambda f_\lambda^{-1} v^{4a_\lambda} \chi_\lambda(xyzv) \mod \mathfrak{m} \\
	\sum_u \gamma_{x,u^\ast,v} \gamma_{y,z,u} &= \sum_{\lambda,\mu} \sum_{\substack{\mathfrak{s},\mathfrak{t},\mathfrak{u} \\ \mathfrak{a},\mathfrak{b},\mathfrak{c}}} f_\lambda^{-1} c^\lambda(x)_\mathfrak{st} c^\lambda(v)_\mathfrak{us} \cdot f_\mu^{-1} c^\mu(y)_\mathfrak{ab} c^\mu(z)_\mathfrak{bc} \cdot \smash{\underbrace{\sum_u c^\lambda(u^\ast)_\mathfrak{tu} c^\mu(u)_\mathfrak{ca}}_{=\delta_{\lambda\mu} \delta_\mathfrak{ta} \delta_{uc}\cdot f_\lambda}} \\
	&= \sum_\lambda f_\lambda^{-1} \sum_{\substack{\mathfrak{s},\mathfrak{t},\mathfrak{u} \\ \mathfrak{b}}} c^\lambda(x)_\mathfrak{st} c^\lambda(v)_\mathfrak{us} \cdot c^\lambda(y)_\mathfrak{tb} c^\lambda(z)_\mathfrak{bu} \\
	&= \sum_\lambda f_\lambda^{-1} \sum_{\substack{\mathfrak{s},\mathfrak{t},\mathfrak{u} \\ \mathfrak{b}}} c^\lambda(x)_\mathfrak{st} \cdot c^\lambda(y)_\mathfrak{tb} c^\lambda(z)_\mathfrak{bu} c^\lambda(v)_\mathfrak{us} \\
	&= \sum_\lambda f_\lambda^{-1} v^{4a_\lambda} \chi_\lambda(xyzv) \mod \mathfrak{m}
\end{align*}

b. und c. folgen aus obigem Lemma. c. folgt z.\,B. sofort aus $\gamma_{x,y,z} = \gamma_{y^\ast,x^\ast,z^\ast}$. Mit dessen Hilfe zeigen wir jetzt, dass $(t_x)$ eine $\ast$"~symmetrische Basis von $J$ ist:
\begin{align*}
	\overline{\tau}(t_{x^\ast} t_y) &= \sum_z \gamma_{x^\ast,y,z} \overline{\tau}(t_{z^\ast}) \\
	&= \sum_z \gamma_{x^\ast,y,z} n_{z^\ast} \\
	&= \sum_z \gamma_{x^\ast,y,z} n_z \\
	&= \delta_{xy} \qedhere
\end{align*}
\end{proof}

\begin{remark}
Die folgenden beiden Lemmata kommen in der Behandlung der asymptotischen Algebra bei \cite{geckjacon} nicht vor, sind aber auch für sich genommen interessante Ergebnisse, weil sie zeigen, dass die Konstruktion nicht so unnatürlich ist, wie sie erscheint.
\end{remark}

\begin{lemma}[Basis-Unabhängigkeit]\label{J_alg:base_independent}
\index{terms}{$\ast$-symmetrisch}
Sei $(b_i)_{i=1\ldots n}$ die Basis, mit der wir die ganze Zeit gearbeitet haben. Sei $b_i' = \sum_{j=1}^n A_{ij} b_j$ eine weitere $\ast$"~symmetrische Basis von $H$. Definiere dann $c^\lambda(x)'$, $\gamma_{x,y,z}'$, $n_x'$ und $J'$ analog wie in \ref{J_alg:def:J_algebra}.

Bezeichne mit $\sigma,\sigma'\in\operatorname{Sym}(n)$ die beiden Bijektionen mit $b_i^\ast = b_{\sigma(i)}$ bzw. $(b_i')^\ast = b_{\sigma'(i)}'$.

Definiere dann $\alpha:J\to J'$ durch
\[\alpha(t_x) := \sum_i A_{ix} t_i'.\]
(Man beachte, dass nach \ref{formally_real:orthogonal_matrices} $A_{ij}\in\mathcal{O}$ ist und $J$ als $\mathcal{O}/\mathfrak{m}$"~Algebra definiert wurde.)
\begin{enumerate}
	\item $\alpha$ ist ein Isomorphismus von $F$"~Algebren.
	\item $\alpha$ überführt die Spurformen von $J$ und $J'$ ineinander, d.\,h. $\tau'\circ\alpha = \tau$.
	\item $\alpha$ kommutiert mit den Involutionen von $J$ und $J'$, d.\,h. $\alpha(x^\ast) = \alpha(x)^\ast$ für alle $x\in J$.
\end{enumerate}
\end{lemma}
\begin{proof}
Wir wissen aus \ref{formally_real:orthogonal_matrices}, dass $A$ in $O_n(\mathcal{O})$ liegt. Die Reduktion liegt also in $O_n(F)$. Insbesondere ist $\alpha$ ein $F$-Vektorraumisomorphismus mit Darstellungsmatrix $A\mod\mathfrak{m}$ und $\alpha^{-1}(t_i') = \sum_j A_{ij} t_j$.

Es gilt nach Voraussetzung:
\begin{align}
	b_i' &= \sum_{j=1}^n A_{ij} b_j
\label{eq:J_alg1}\tag{1} \\
	b_j &= \sum_{k=1}^n A_{kj} b_k'
\label{eq:J_alg2}\tag{2}
\end{align}
Analoges gilt dann auch für die Führenden-Koeffizienten-Matrizen.

\bigbreak
Wir werten zunächst $(b_i')^\ast$ auf zwei verschiedene Weisen aus:
\begin{align*}
	(b_i')^\ast &= b_{\sigma'(i)}' \\
	&\overset{\text{\eqref{eq:J_alg1}}}{=} \sum_j A_{\sigma'(i),j} b_j \\
	&= \sum_k A_{\sigma'(i),\sigma(k)} b_{\sigma(k)} \\
	(b_i')^\ast &\overset{\text{\eqref{eq:J_alg1}}}{=} (\sum_k A_{ik} b_k)^\ast \\
	&= \sum_k A_{ik} b_k^\ast \\
	&= \sum_k A_{ik} b_{\sigma(k)}
\end{align*}
Also ergibt sich für alle $i,k\in\Set{1,\ldots,n}$:
\begin{equation}
	A_{i,k} = A_{\sigma'(i),\sigma(k)}
	\label{eq:J_alg3}\tag{3}
\end{equation}

\bigbreak
Weil \eqref{eq:J_alg2} analog auch für die Führenden-Koeffizienten-Matrizen gilt, folgt:
\begin{align*}
	\gamma_{x,y,z} &= \sum_\lambda f_\lambda^{-1} \sum_{\mathfrak{s},\mathfrak{t},\mathfrak{u}} c^\lambda(x)_\mathfrak{st} c^\lambda(y)_\mathfrak{tu} c^\lambda(z)_\mathfrak{us} \\
	 &= \sum_\lambda f_\lambda^{-1} \sum_{\mathfrak{s},\mathfrak{t},\mathfrak{u}} \sum_{i,j,k} A_{ix} c^\lambda(i)_\mathfrak{st}' A_{jy} c^\lambda(j)_\mathfrak{tu}' A_{kz} c^\lambda(k)_\mathfrak{us}' \\
	 &= \sum_{i,j,k} \sum_\lambda \sum_{\mathfrak{s},\mathfrak{t},\mathfrak{u}} f_\lambda^{-1} A_{ix} A_{jy} A_{kz}  c^\lambda(i)_\mathfrak{st}' c^\lambda(j)_\mathfrak{tu}' c^\lambda(k)_\mathfrak{us}' \\
 	 &= \sum_{i,j,k} \gamma_{i,j,k}' A_{ix} A_{jy} A_{kz}
 	 \label{eq:J_alg4}\tag{4}
\end{align*}

Es ergibt sich:
\begin{align*}
	\alpha(t_x t_y) &= \alpha\Big(\sum_z \gamma_{x,y,z} t_{\sigma(z)}\Big) \\
	&= \sum_z \sum_a \gamma_{x,y,z} A_{a\sigma(z)} t_a' \\
	&= \sum_{z,a} \gamma_{x,y,z} A_{\sigma'(a)\sigma(z)} t_{\sigma'(a)}' \\
	&\overset{\text{\eqref{eq:J_alg3}}}{=} \sum_{z,a} \gamma_{x,y,z} A_{az} t_{\sigma'(a)}' \\
	&\overset{\text{\eqref{eq:J_alg4}}}{=} \sum_{z,a} \sum_{i,j,k} \gamma_{i,j,k}' A_{ix} A_{jy} A_{kz} A_{az} t_{\sigma'(a)}' \\
	&= \sum_{i,j,k} \sum_a \gamma_{i,j,k}' A_{ix} A_{jy} \smash{\underbrace{\Big(\sum_z A_{az} (A^\text{Tr})_{zk}\Big)}_{=\delta_{ka}}} t_{\sigma'(a)}' \\
	&= \sum_{i,j,k} \gamma_{i,j,k}' A_{ix} A_{jy} t_{\sigma'(k)}' \\
	&= \sum_{i,j} A_{ix} A_{jy} \Big(\sum_k \gamma_{i,j,k}' t_{\sigma'(k)}'\Big) \\
	&= \sum_{i,j} A_{ix} A_{jy} t_i' t_j' \\
	&= \alpha(t_x)\alpha(t_y)
\end{align*}
Das zeigt a.

Als Algebra-Isomorphismus bildet $\alpha$ insbesondere die Einselemente aufeinander ab, d.\,h. wir folgern
\begin{align*}
	\sum_i n_i' t_i' &= 1_{J'} \\
	&= \alpha(1_J) \\
	&= \alpha\Big(\sum_z n_z t_z\Big) \\
	&= \sum_{i,z} n_z A_{iz} t_i' \\
	\implies \forall i: n_i' &= \sum_z A_{iz} n_z
	\label{eq:J_alg5}\tag{5}
\end{align*}

Daraus wiederum können wir b. folgern:
\begin{align*}
	\tau'(\alpha(t_x)) &= \tau'\Big(\sum_i A_{ix} t_i'\Big) \\
	&= \sum_i A_{ix} n_i' \\
	&\overset{\text{\eqref{eq:J_alg5}}}{=} \sum_i \sum_j A_{ix} A_{ij} n_i \\
	&= \sum_j \smash{\underbrace{\Big(\sum_i (A^\text{Tr})_{ji} A_{ix}\Big)}_{=\delta_{jx}}} n_i \\
	&= n_x \\
	&= \tau(t_x)
\end{align*}
Damit haben wir b. gezeigt.

Es bleibt c. zu zeigen. Das folgt ebenfalls aus \eqref{eq:J_alg3}:
\begin{align*}
	\alpha(t_x^\ast) &= \alpha(t_{\sigma(x)}) \\
	&= \sum_i A_{i\sigma(x)} t_i' \\
	&= \sum_i A_{\sigma'(i)\sigma(x)} t_{\sigma'(i)}' \\
	&\overset{\text{\eqref{eq:J_alg3}}}{=} \sum_i A_{ix} (t_i')^\ast \\
	&= \Big(\sum_i A_{ix} t_i'\Big)^\ast \\
	&= \alpha(t_x)^\ast \qedhere
\end{align*}
\end{proof}

\begin{lemma}[Verträglichkeit mit Tensorprodukten]
Es seien $(H^1,\ast)$, $(H^2,\ast)$ zwei endlichdimensionale, zerfallend halbeinfache $K$"~Algebren mit involutiven Antiautomorphismen, $\tau^1$, $\tau^2$ je eine Spurform, die $H^1$ bzw. $H^2$ zu einer symmetrischen Algebra macht, $B^1=(b_i^1)_{i=1\ldots n}$, $B^2=(b_j^2)_{j=1\ldots m}$ je eine $\ast$"~symmetrische Basis sowie $J^1$, $J^2$ die dazugehörigen asymptotischen Algebren. Es gilt dann:
\begin{enumerate}
	\item $H:=H^1\otimes_K H^2$ ist eine zerfallend halbeinfache $K$"~Algebra, $\tau(x\otimes y):=\tau^1(x)\tau^2(y)$ ist eine Spurform, $(x\otimes y)^\ast := x^\ast \otimes y^\ast$ ein involutiver Antiautomorphismus und $B:=(b_i^1 \otimes b_j^2)_{(i,j)\in\Set{1,\ldots,n}\times\set{1,\ldots,m}}$ eine $\ast$"~symmetrische Basis von $H$.
	\item Sei $J$ die asymptotische Algebra von $H$. Durch $t_{x^1\otimes x^2} \leftrightarrow t_{x^1} \otimes t_{x^2}$ für alle $x^1\in B^1$ und alle $x^2\in B^2$ ist ein Isomorphismus von $F$"~Algebren $J \leftrightarrow J^1\otimes_F J^2$ gegeben. Die Involution und die Spurform auf $J$ sind von den Involutionen und Spurformen auf $J^1$ bzw. $J^2$ induziert, d.\,h. $(a^1\otimes a^2)^\ast = (a^1)^\ast \otimes (a^2)^\ast$ und $\overline{\tau}(a^1\otimes a^2)=\overline{\tau^1}(a^1)\overline{\tau^2}(a^2)$ für alle $a^1\in J^1, a^2\in J^2$.
\end{enumerate}
\end{lemma}
\begin{proof}
a. Es ist einfach einzusehen, dass $H$ eine $K$"~Algebra mit involutivem Antiautomorphismus, $\tau$ eine Spurform und $B^1\otimes B^2$ eine $\ast$"~symmetrische Basis von $H$ ist. Insbesondere ist die Spurform daher nichtausgeartet und $H$ somit eine symmetrische $K$"~Algebra.

Sind nun $V^1$ und $V^2$ einfache Moduln von $H^1$ bzw. $H^2$, dann ist $V:=V^1\otimes_K V^2$ ein $H$"~Modul und man überzeugt sich leicht davon, dass die Schurelemente $c_{V^1\otimes V^2} = c_{V^1} c_{V^2}$ erfüllen. Insbesondere sind die Tensorprodukte irreduzibler Moduln selbst absolut irreduzibel, da ihre Schurelemente ungleich Null sind. Aus einer Dimensionsbetrachtung folgt, dass $H$ zerfallend halbeinfach ist.

\medbreak
b. Sind $\Lambda^1$ und $\Lambda^2$ indizierende Mengen für die einfachen $H^1$- bzw. $H^2$"~Moduln, so können wir daher $\Lambda:=\Lambda^1\times\Lambda^2$ als indizierende Menge für die einfachen $H$"~Moduln wählen. Aus der Betrachtung der Schurelemente folgt so auch, dass $a_\lambda = a_{\lambda^1}+a_{\lambda^2}$ und $f_\lambda = f_{\lambda^1} f_{\lambda^2}$ für alle $\lambda=(\lambda^1,\lambda^2)$ gilt. Es gilt außerdem für die Charaktere $\chi_\lambda(h^1\otimes h^2) = \chi_{\lambda^1}(h^1)\chi_{\lambda^2}(h^2)$ für alle $h^1\in H, h^2\in H^2$.

Daher erhalten wir
\begin{align*}
	\gamma_{x,y,z} &= \sum_{\lambda\in\Lambda} f_\lambda^{-1} v^{3a_\lambda} \chi_\lambda(xyz) \mod\mathfrak{m} \\
	&= \sum_{\substack{\lambda^1\in\Lambda^1 \\ \lambda^2\in\Lambda^2}} f_{\lambda^1}^{-1} f_{\lambda^2}^{-1} v^{3a_{\lambda^1}} v^{3a_{\lambda^2}} \chi_{\lambda^1}(x^1 y^1 z^1) \chi_{\lambda^2}(x^2 y^2 z^2) \\
	&= \Big(\sum_{\lambda^1\in\Lambda^1} f_{\lambda^1}^{-1} v^{3a_{\lambda^1}} \chi_{\lambda^1}(x^1 y^1 z^1) \Big)\cdot\Big(\sum_{\lambda^2\in\Lambda^2} f_{\lambda^2}^{-1} v^{3a_{\lambda^2}} \chi_{\lambda^2}(x^2 y^2 z^2) \Big) \\
	&= \gamma_{x^1,y^1,z^1} \cdot \gamma_{x^2,y^2,z^2}
\end{align*}
für alle Basiselemente $x=x^1\otimes x^2, y=y^1\otimes y^2, z=z^1\otimes z^2\in B$. Daraus folgt die Behauptung, dass die angegebene Abbildung ein Isomorphismus von $F$"~Algebren ist. Insbesondere werden die Einselemente aufeinander abgebildet, woraus wir $n_x = n_{x^1}\cdot n_{x^2}$ erhalten. Das zeigt die Verträglichkeitsaussage für die Spurformen. Die Verträglichkeit von den Involutionen folgt, weil es sich für die Basiselemente bereits aus der Definition ergibt.
\end{proof}

\begin{conjecture}
Angesichts dieser letzten beiden Lemmata ist es naheliegend zu vermuten, dass es eine abstrakte, basisfreie Konstruktion von $J$ aus $H$ geben sollte.

Weitergehend könnte $J$ sogar eine universelle Eigenschaft erfüllen. Vielleicht kann der folgende Satz dafür als Vorbild dienen.
\end{conjecture}

\subsection{Darstellungen von \texorpdfstring{$J$}{J}}

\begin{theorem}[Führende Koeffizienten und Darstellungen von $J$, siehe {\citep[1.5.7]{geckjacon}}]
\index{terms}{Darstellung!balancierte|(}\index{terms}{Führende Koeffizienten}
\index{symbols}{rho@$\overline{\rho}$}
Ist $\rho_\lambda$ eine balancierte Darstellung von $H$, dann ist durch die Definition
\[\overline{\rho}_\lambda(t_x) := c_\lambda(x)\]
für alle $x\in B$ eine absolut irreduzible Matrixdarstellung $\overline{\rho}_\lambda: J\to F^{d_\lambda\times d_\lambda}$ gegeben. Jede irreduzible Darstellung von $J$ kann so erhalten werden. Insbesondere ist $J$ zerfallend halbeinfach und die Aussagen in Lemma \ref{schur_relations_leading_coeff} sind tatsächlich die echten Schur-Relationen und $f_\lambda$ die echten Schur-Elemente von $J$.
\end{theorem}
\begin{proof}
Das ist einfach nachzurechnen. Nach Definition ist $\overline{\rho}_\lambda$ bereits $F$"~linear. Wir zeigen, dass die Abbildung auch mit der Multiplikation verträglich ist:
\begin{align*}
	\overline{\rho}_\lambda(t_x t_y)_\mathfrak{ab} &= \sum_z \gamma_{x,y,z} \overline{\rho}_\lambda(t_{z^\ast})_\mathfrak{ab} \\
	&= \sum_z \gamma_{x,y,z} c^\lambda(z^\ast)_\mathfrak{ab} \\
	&= \sum_\mu \sum_{\mathfrak{s},\mathfrak{t},\mathfrak{u}} f_\mu^{-1} c^\mu(x)_\mathfrak{st}c^\mu(y)_\mathfrak{tu} \smash{\underbrace{\sum_z c^\mu(z)_\mathfrak{us} c^\lambda(z^\ast)_\mathfrak{ab}}_{=\delta_{\lambda\mu} \delta_\mathfrak{ub} \delta_\mathfrak{sa} \cdot f_\lambda}} \\
	&= \sum_{\mathfrak{t}} c^\lambda(x)_\mathfrak{at}c^\lambda(y)_\mathfrak{tb} \\
	&= (c^\lambda(x)\cdot c^\lambda(y))_\mathfrak{ab} \\
	&= (\overline{\rho}_\lambda(t_x)\cdot\overline{\rho}_\lambda(t_y))_\mathfrak{ab}
\end{align*}
Analog beweisen wir auch $\overline{\rho}_\lambda(1_J)=I$:
\begin{align*}
	\overline{\rho}_\lambda(1_J)_\mathfrak{ab} &= \sum_z n_z \overline{\rho}_\lambda(t_z)_\mathfrak{ab} \\
	&= \sum_z n_z c^\lambda(z)_\mathfrak{ab} \\
	&= \sum_\mu \sum_\mathfrak{s} f_\mu^{-1} \smash{\underbrace{\sum_z c^\mu(z^\ast)_\mathfrak{ss} c^\lambda(z)_\mathfrak{ab}}_{\delta_{\lambda\mu} \delta_\mathfrak{sa} \delta_\mathfrak{sb} \cdot f_\lambda}} \\
	&= \delta_\mathfrak{ab}
\end{align*}

Die Schur"=Relationen für die Führenden"=Koeffizienten"=Matrizen (Lemma \ref{schur_relations_leading_coeff}) zeigen nun, dass die so definierten $J$"~Moduln absolut irreduzibel und paarweise nichtisomorph sind. Weil außerdem $\sum_\lambda d_\lambda^2 = \dim_K H = \abs{B} = \dim_F J$ gilt, ist $J$ sogar halbeinfach und alle irreduziblen Darstellungen kommen in der so konstruierten Liste vor.
\end{proof}

\begin{remark}
$J$ ist von der Basiswahl unabhängig, wie wir oben gesehen haben. Wir haben nun gesehen, dass $J$ zerfallend halbeinfach ist und die absolut irreduziblen Moduln die Dimensionen $d_\lambda$ haben. Damit ist der abstrakte Isomorphietyp von $J$ als $F$"~Algebra von der Wahl des partiellen Schnittes $\gamma\mapsto v^\gamma$ unabhängig. Diese Wahl spielt jedoch tatsächlich eine kleine Rolle, denn die Schur-Elemente von $J$, also die $f_\lambda$, hängen von der Wahl dieses Schnittes ab. Anders formuliert: Die induzierte Spurform $\overline{\tau}$ hängt von der Wahl des partiellen Schnittes ab.
\end{remark}

\begin{lemma}[Elementare Eigenschaften der Darstellungen von $J$]
\index{terms}{Darstellung!balancierte}\index{terms}{Asymptotische Algebra|)}
Mit den Bezeichnungen von oben gilt:
\begin{enumerate}
	\item Für den Charakter $\chi_\lambda^J: J\to F$ der irreduziblen $J$-Darstellung mit Index $\lambda$ und den Charakter $\chi_\lambda^H: H\to K$ der irreduziblen $H$-Darstellung mit Index $\lambda$ gilt:
	\[\forall x\in B: \chi_\lambda^J(t_x) = v^{a_\lambda} \chi_\lambda^H(x) \mod\mathfrak{m}\]
	\item Für zwei balancierte, irreduzible Matrixdarstellungen $\rho_1, \rho_2: H\to K^{d\times d}$ sind äquivalent:
	\begin{enumerate}
		\item $\rho_1 \isomorphic \rho_2$.
		\item $\overline{\rho_1} \isomorphic \overline{\rho_2}$.
	\end{enumerate}
	\item Genauer gesagt können in der Situation von b. Isomorphismen geliftet und reduziert werden:
	\begin{enumerate}
		\item Ist $M'\in GL_d(K)$ mit $M'\rho_1 = \rho_2 M'$, dann gibt es ein $v^\alpha\in K^\times$ mit Bewertung $\alpha$ so, dass $M:=v^\alpha M'\in GL_d(\mathcal{O})$ ist und $\overline{M}\overline{\rho_1} = \overline{\rho_2} \overline{M}$.
		\item Ist umgekehrt $\overline{M}\in GL_d(F)$ mit  $\overline{M}\overline{\rho_1} = \overline{\rho_2} \overline{M}$ gegeben, dann gibt es ein Urbild $M\in GL_d(\mathcal{O})$ mit $M\rho_1 = \rho_2 M$.
	\end{enumerate}	
\end{enumerate}
\end{lemma}
\begin{proof}
a. folgt sofort daraus, dass man balancierte Darstellungen wählen kann. Ist nämlich $\rho_\lambda: H\to K^{d_\lambda\times d_\lambda}$ balanciert vom Isomorphietyp $\lambda\in\Lambda$, dann ist
\begin{align*}
	\chi_\lambda^J(x) &= \sum_\mathfrak{s} \overline{\rho}_\lambda(x)_\mathfrak{ss} \\
	&= \sum_\mathfrak{s} v^{a_\lambda} \rho_\lambda(x)_\mathfrak{ss} \mod\mathfrak{m} \\
	&= v^{a_\lambda} \chi_\lambda^H(x).
\end{align*}

\bigbreak
b. Weil $H$ und $J$ beide zerfallend halbeinfach sind, sind die Darstellungen äquivalent genau dann, wenn ihre Charaktere gleich sind. Zusammen mit den Schur-Relationen folgt somit 
\begin{align*}
	\rho_\lambda &\isomorphic \rho_\mu \\
	\iff\chi_\lambda^H &= \chi_\mu^H \\
	\iff \lambda &= \mu \\
	\iff \chi_\lambda^J &= \chi_\mu^J \\
	\iff \overline{\rho_\lambda} &\isomorphic \overline{\rho_\mu}.
\end{align*}

\bigbreak
c. Beide Richtungen sind einfach: Falls $M'\in GL_d(K)$ mit $M'\rho_1=\rho_2 M'$ ist, setze ${\alpha:=-\nu(M')}$. Dann ist $M:=v^\alpha M'\in\mathcal{O}^{d\times d}$ und $M\not\equiv 0 \mod\mathfrak{m}^{d\times d}$. Indem wir die Gleichung mit $v^{a_\lambda+\alpha}$ durchmultiplizieren (wobei $\lambda\in\Lambda$ der Isomorphietyp der beiden Darstellungen sei), erhalten wir:
\[\forall x\in B: M(v^{a_\lambda}\rho_1(x)) = (v^{a_\lambda}\rho_2(x))M.\]
Nach Annahme sind $\rho_1, \rho_2$ balanciert, sodass wir nun modulo $\mathfrak{m}$ reduzieren können und
\[\forall x\in B: \overline{M} \overline{\rho}_1(x) = \overline{\rho}_2(x) \overline{M}\]
erhalten. Weil $\overline{M}\neq 0$ ist und $\overline{\rho}_1, \overline{\rho}_2$ absolut irreduzible Darstellungen von $J$ sind, folgt hieraus $\overline{M}\in GL_d(F)$ und daher auch $M\in GL_d(\mathcal{O})$. Damit ist eine Richtung gezeigt.

\medbreak
Sei für die andere Richtung $\overline{M}$ gegeben. Nach Annahme ist $\overline{\rho}_1\isomorphic\overline{\rho}_2$, d.\,h. nach a., dass auch $\rho_1$ und $\rho_2$ isomorph sind. Wir erhalten also ein $X\in GL_d(K)$ mit $X\rho_1 = \rho_2 X$. Indem wir das obige Argument wiederholen, können wir o.\,B.\,d.\,A. $X\in GL_d(\mathcal{O})$ annehmen. Dann ist $\overline{X}$ ein Isomorphismus $\overline{\rho}_1\to\overline{\rho}_2$. Weil diese Darstellungen nun absolut irreduzibel sind, gibt es ein $\overline{c}\in F^\times$ mit $\overline{M} = \overline{c} \overline{X}$. Indem wir ein Urbild $c\in\mathcal{O}^\times$ wählen und $M:=cX$ setzen, erhalten wir eine Matrix $M\in\mathcal{O}^{d\times d}$, deren Projektion genau die gegebene Matrix $\overline{M}$ ist. Insbesondere ist $M$ invertierbar. Nach Konstruktion gilt auch $M\rho_1=\rho_2 M$.
\end{proof}

\begin{remark}
Der folgende Satz verallgemeinert \citep[1.5.11]{geckjacon} auf die naheliegende Weise.
\end{remark}
\begin{lemma}[Invariante Bilinearformen]\label{balanced_reps:invariant_blf3}
\index{terms}{Invariante Bilinearform}
Sei $R\subseteq F$ ein Hauptidealring mit $F=\QuotFld(R)$ und $f_\lambda^{\pm 1}, \gamma_{x,y,z}, n_w\in R$ für alle $\lambda\in\Lambda$ und alle $x,y,z,w\in B$. Dann gilt:
\begin{enumerate}
	\item Es gibt zu jedem Isomorphietyp $\lambda$ balancierte Darstellungen $\rho_\lambda$ dieses Typs mit $\overline{\rho_\lambda}(t_x)\in R^{d_\lambda\times d_\lambda}$ für alle $x\in B$.
	\item Es gibt eine nichtentartete, symmetrische, invariante Bilinearform $B_\lambda\in GL_d(R)$ für $\overline{\rho_\lambda}$, d.\,h.
	\[B_\lambda \overline{\rho_\lambda}(a) = \overline{\rho_\lambda}(a^\ast)^\text{Tr} B_\lambda\]
	für alle $a\in J$.
\end{enumerate}
\end{lemma}
\begin{proof}
a. $J_0:=\sum_{x\in B} Rt_x$ ist nach Annahme eine $R$-Unteralgebra von $J$ mit $J=FJ_0$. Da $R$ ein Hauptidealring ist, können wir wie gewohnt zu jeder irreduziblen Darstellung $V$ von $J$ einen $J_0$"~Modul $V_0\subseteq V$ mit $FV_0=V$ finden, der als $R$"~Modul frei ist. Durch Wahl einer $R$"~Basis erhalten wir so, wie gewünscht, eine Matrixdarstellung $\alpha: J\to F^{d_\lambda\times d_\lambda}$ vom selben Isomorphietyp mit $\alpha(t_x)\in R^{d_\lambda\times d_\lambda}$. Ist nun $\rho_\lambda: H\to K^{d_\lambda\times d_\lambda}$ eine balancierte Matrixdarstellung von $H$ mit Isomorphietyp $\lambda$, dann ist $\overline{\rho_\lambda}$ eine Darstellung von $J$ vom Isomorphietyp $\lambda$ und somit äquivalent zu $\alpha$, d.\,h. es gibt eine Matrix $P\in GL_{d_\lambda}(\mathcal{O})$ mit $\overline{P^{-1}\rho_\lambda P}=\alpha$. Das zeigt a.

\medbreak
b. Da $(t_x)_{x\in B}$ eine $\ast$"~symmetrische Basis von $J$ ist, ist durch $\widetilde{B}:=\sum_{x\in B} \overline{\rho_\lambda}(t_x)^\text{Tr} \overline{\rho_\lambda}(t_x)$ eine positiv semidefinite (im Sinne von \ref{formally_real:symmetric_matrices}) invariante Bilinearform gegeben. Sie ist auch positiv definit, denn ist $u\in F^{d_\lambda}$ ein Vektor mit $u^\text{Tr} \widetilde{B} u = 0$, dann folgt aus \ref{formally_real:symmetric_matrices}, dass $\overline{\rho_\lambda}(t_x)u=0$ für alle $x\in B$ ist. Damit ist insbesondere $u=\overline{\rho_\lambda}(1_J)u=0$.

Wenn wir nun dabei $\rho_\lambda$ wie in a. wählen, folgt $\widetilde{B}\in R^{d_\lambda\times d_\lambda}$. Nun setzen wir 
\[B_\lambda := \operatorname{ggT}(\widetilde{b}_{ij} \mid i,j=1,\ldots,d_\lambda)^{-1} \widetilde{B}.\]
Dies ist die gesuchte Matrix. Dazu müssen wir $\det(B_\lambda)\in R^\times$ beweisen. Betrachte dazu ein beliebiges Primideal $\mathfrak{p}\in\operatorname{Spec}(R)$ und den Quotientenkörper $k:=\QuotFld(R/\mathfrak{p})$. Da die Einträge von $B_\lambda$ teilerfremd sind, ist $B_\lambda\not\equiv 0 \mod\mathfrak{p}^{d_\lambda\times d_\lambda}$. Weil wir $f_\lambda^{\pm 1}\in R$ vorausgesetzt haben, ist $f_\lambda\in R^\times$. Wir setzen in die Schur-Relationen \ref{schur_relations_leading_coeff} ein und erhalten
\[\sum_{x\in B} \overline{\rho_\lambda}(t_x^\ast)_\mathfrak{st} \overline{\rho_\lambda}(t_x)_\mathfrak{uv} = \begin{cases} f_\lambda\mod\mathfrak{p} & \text{falls}\,\mathfrak{s}=\mathfrak{v}, \mathfrak{t}=\mathfrak{u} \\ 0 &\text{sonst}\end{cases}\]
für alle $\mathfrak{s},\mathfrak{t},\mathfrak{u},\mathfrak{v}\in\Set{1,\ldots,d}$. Da $f_\lambda\not\equiv 0 \mod\mathfrak{p}$ ist, folgt aus der Theorie der Schur-Elemente, dass $\overline{\rho_\lambda}: kJ_0 \to k^{d_\lambda\times d_\lambda}$ eine absolut irreduzible Darstellung ist (siehe \citep[Ch.\,7]{geckpfeiffer}).

Nun ist aber $B_\lambda\mod\mathfrak{p}$ die Darstellungsmatrix eines $J$"~linearen Operators zwischen den beiden absolut irreduziblen Matrixdarstellungen $a\mapsto\overline{\rho_\lambda}(a)$ und $a\mapsto\overline{\rho_\lambda}(a^\ast)^\text{Tr}$ von $kJ_0$. Da $B_\lambda\not\equiv 0 \mod\mathfrak{p}^{d_\lambda\times d_\lambda}$ ist, folgt aus dem Lemma von Schur also, dass dieser Operator invertierbar ist, d.\,h. $\det(B_\lambda)\not\equiv 0 \mod\mathfrak{p}$. Da $\mathfrak{p}$ beliebig war, folgt ${\det(B_\lambda)\in R^\times}$.
\end{proof}

\begin{remark}
Wie in \citep[1.5.11]{geckjacon} bemerkt wurde, ist für die Aussage in b. nicht unbedingt nötig, dass $R$ ein Hauptidealring ist. Dies tritt z.\,B. für Hecke"=Algebren von nichtkristallographischen Coxeter"=Gruppen auf. Dort ist $\IZ_W$ i.\,A. kein Hauptidealring mehr, durch die explizite Angabe von $W$"~Graphen kann man jedoch trotzdem jede irreduzible Darstellung von $H$ über $\IZ_W[\Gamma]$ und daher jede irreduzible Darstellung von $J$ über $\IZ_W$ realisieren. In den Fällen, wo das auftritt (dies ist bei irreduziblen Coxeter"=Gruppen nur für $I_2(m)$ der Fall), kann eine solche invariante Bilinearform explizit angegeben werden.
\end{remark}
\begin{remark}
Jede invariante Bilinearform $\Omega\in GL_d(K)$ für $\rho_\lambda$ liefert auch eine invariante Bilinearform für $\overline{\rho_\lambda}$, indem $\Omega$ zunächst durch ein geeignetes $v^\alpha\Omega$ ersetzt wird so, dass $\nu(\Omega)=0$ ist. Die Reduktion $\Omega\mod\mathfrak{m}$ ist dann invariant für $\overline{\rho_\lambda}$. Da invariante Bilinearformen wegen Schurs Lemma bis auf Skalare eindeutig bestimmt sind, folgt, dass die hinreichende Bedingung aus \ref{balanced_reps:invariant_blf1} sogar notwendig ist: Jede balancierte Darstellung hat eine invariante Bilinearform, deren Reduktion symmetrisch und nichtentartet (eben sogar definit) ist.
\index{terms}{Darstellung!balancierte|)}
\end{remark}

\subsection{Beispiele}

Nach den langen Vorbereitungen schauen wir uns einige Beispiele an:
\begin{example}[Halbeinfache Reduktionen]
Es ist in vielen Beispielen der Fall, dass $H$ als Skalarerweiterung $KH_0=K\otimes_R H_0$ gewonnen wurde für einen Unterring $R\subseteq K$ und eine $R$"~Algebra $H_0$, die als $R$"~Modul frei und endlich erzeugt ist, d.\,h. es gibt eine $K$"~Basis $B\subseteq H$ bzgl. der alle Strukturkonstanten in $R$ liegen.

Wir nehmen nun an, dass $R=\mathcal{O}$ ist und auch die restliche Struktur bereits über $R=\mathcal{O}$ vorhanden ist, d.\,h. $\tau$ schränkt sich zu einer Spurform $H_0\to\mathcal{O}$ ein, bzgl. der $H_0$ bereits symmetrisch ist, $\ast$ schränkt sich zu einer Involution von $H_0$ ein und $B$ ist $\ast$"~symmetrisch.

\medbreak
Weil dann $H_0$ eine $\mathcal{O}$"~Algebra ist, die endlich erzeugt als $\mathcal{O}$"~Modul ist, lässt sich jede endlichdimensionale Darstellung von $H$ bereits über $\mathcal{O}$ realisieren, d.\,h. zu jedem $\lambda\in\Lambda$ gibt es Matrixdarstellungen $\rho_\lambda: H_0\to\mathcal{O}^{d_\lambda\times d_\lambda}$, die (als $H$"~Moduln) den Isomorphietyp $\lambda$ haben.

Da damit die Schurelemente $c_\lambda$ allesamt in $\mathcal{O}$ liegen, ist automatisch $a_\lambda=-\frac{1}{2}\nu(c_\lambda) \leq 0$. Weil sich $1\in H_0$ als $\mathcal{O}$"~Linearkombination von $B$ schreiben lässt, muss in der Tat $a_\lambda=0$ sein, denn sonst wäre das Bild $\rho_\lambda(H_0)$ ja vollständig in einem echten Ideal von $\mathcal{O}^{d_\lambda\times d_\lambda}$ enthalten.

Daraus ergibt sich, dass die $c_\lambda$ Einheiten in $\mathcal{O}$ und die Darstellungen $\rho_\lambda$ bereits automatisch balancierte Darstellungen sind. Es ist $\braket{a_\lambda}=0 \leq \Gamma$, d.\,h. wir müssen keinen partiellen Schnitt wählen, denn er ist sowieso eindeutig bestimmt: $v^\gamma=1$. Es gilt damit:
\[\forall\lambda\in\Lambda\forall x\in B: c^\lambda(x) = \rho_\lambda(x) \mod\mathfrak{m}\]
und $f_\lambda=c_\lambda$.

\medbreak
Wenn wir jetzt die Strukturkonstanten ausrechnen, erhalten wir:
\begin{align*}
	\gamma_{x,y,z} &\equiv \sum_{\lambda\in\Lambda} f_\lambda^{-1} v^{3a_\lambda} \chi_\lambda(xyz) \mod\mathfrak{m} \\
	&= \sum_\lambda c_\lambda^{-1} \chi_\lambda(xyz) \\
	&= \tau(xyz)
\end{align*}
Das sind nun genau die modulo $\mathfrak{m}$ reduzierten Strukturkonstanten von $H$, da
\[ x\cdot y = \sum_{z\in B} \tau(xyz)z^\vee = \sum_{z\in B} \tau(xyz)z^\ast.\]
Mit anderen Worten: $J$ ist nichts anderes als die Reduktion von $H_0$ modulo $\mathfrak{m}$, d.\,h. die Skalarerweiterung $FH_0$. Auf diese Weise erhalten wir den klassischen modularen Fall. Die gesamte Konstruktion der asymptotischen Algebra kann also als eine Verallgemeinerung dieser Standardkonstruktion aus der modularen Darstellungstheorie verstanden werden.
\end{example}

\begin{example}[Hecke-Algebren]\label{J_alg:ex:Hecke}
Sei $(W,S)$ eine endliche Coxeter"=Gruppe und $L:W\to\Gamma$ eine Gewichtsfunktion und $H=H(W,S,L)$ die dazugehörige Hecke"=Algebra.

\medbreak
In dieser Situation stimmt unsere Definition der asymptotischen Algebra genau mit Gecks Konstruktion der asymptotischen Algebra überein, die in \cite{geckjacon} mit $\widetilde{J}$ bezeichnet wird. In \citep[2.3.16]{geckjacon} wird bewiesen, dass unter Annahme von Lusztigs Vermutungen \textbf{P1} und \textbf{P4} die Algebra $\widetilde{J}$ kanonisch isomorph zu Lusztigs $J$"~Algebra ist.
\end{example}

\begin{remark}
In der Tat lässt sich $J$ auch im Fall der Hecke"=Algebren als Reduktion modulo $\mathfrak{m}$ einer geeigneten $\mathcal{O}$-Unteralgebra von $KH$ schreiben, wenn man die Lusztig"=Vermutungen bereits zur Verfügung hat (zu deren Beweis aber natürlich die Lusztig"=Algebra verwendet wurde), denn die Basis $X_w := v^{\textbf{a}(w)} C_w$ spannt dank der Definition von $\textbf{a}(w)$ (siehe \ref{KL:def:Lusztig_a}) eine $\mathcal{O}$-Unteralgebra $X$ von $KH$ auf. Die Lusztig-Vermutungen stellen sicher, dass auf kanonische Weise $X/\mathfrak{m}X = J$ gilt, wobei die Restklasse von $X_w\in X$ mit $t_w\in J$ identifiziert wird.
\end{remark}

\subsection{Zellen in der asymptotischen Algebra}

\begin{remark}
Fixieren wir eine $\ast$"~symmetrische Basis $B\subseteq H$, so können wir die regulären $H$- bzw. $J$"~Moduln bzgl. der Basen $B$ bzw. $(t_x)_{x\in B}$ betrachten und wie in Definition \ref{KL:def:cells} die Zellen dieser Moduln-mit-Basis definieren. Während über die Zellzerlegung von $H$ wenig Allgemeines ausgesagt werden kann (etwa könnte $B\subseteq H^\times$ und damit ganz $B$ eine einzige Zelle sein), können wir für die asymptotische Algebra nichttriviale Aussagen treffen:
\end{remark}

\begin{theorem}
Für alle $x,y,z\in B$ gilt:
\begin{enumerate}
	\item $\gamma_{x,y,z}\neq 0 \implies t_x \sim_\mathcal{L} t_y^\ast$, $t_y \sim_\mathcal{L} t_z^\ast$, $t_z \sim_\mathcal{L} t_x^\ast$.
	\item $t_x \sim_\mathcal{L} t_y \implies x \sim_\mathcal{L} y$. Insbesondere ist die Zerlegung von $B$, die von der Zerlegung von $\Set{t_x | x\in B}$ in $J$-Linkszellen herkommt, eine Verfeinerung der Zerlegung in $H$-Linkszellen.
	\item Analoge Aussagen gelten auch für Rechtszellen und zweiseitige Zellen.
	\item $t_x \sim_\mathcal{LR} t_x^\ast$.
\end{enumerate}
\end{theorem}
\begin{proof}
a. Wir nutzen aus, dass zum einen $\gamma_{x,y,z}=\gamma_{y,z,x}=\gamma_{z,x,y}$ und zum anderen $\gamma_{x,y,z} = \gamma_{y^\ast,x^\ast,z^\ast}$ gilt (siehe Lemma \ref{J_alg:welldef_gamma_n}). Daraus folgt, dass $\gamma_{x,y,z}\neq 0$ nicht nur $t_y\preceq_\mathcal{L} t_z^\ast$ impliziert wie es nach Definition gilt, sondern auch die rotierten Aussagen $t_z \preceq_\mathcal{L} t_x^\ast$ und $t_x \preceq_\mathcal{L} t_y^\ast$. Aus der zweiten Gleichung folgt, dass mit $\gamma_{x,y,z}\neq 0$ auch $t_y^\ast \preceq_\mathcal{L} t_x$, $t_x^\ast \preceq_\mathcal{L} t_z$ und $t_z \preceq_\mathcal{L} t_y^\ast$ gilt. Das zeigt a.

\medbreak
b. Ist $t_z\preceq_\mathcal{L} t_y$, so gibt es ein $x\in B$ so, dass $0\neq\gamma_{x,y,z^\ast} = \sum_{\substack{\lambda\in\Lambda \\ \mathfrak{s},\mathfrak{t},\mathfrak{u}}} f_\lambda^{-1} c^\lambda(x)_\mathfrak{st} c^\lambda(y)_\mathfrak{tu} c^\lambda(z^\ast)_\mathfrak{us}$ ist. Es gibt also $\lambda\in\Lambda$ und Indizes $\mathfrak{s},\mathfrak{t}$ mit $0\neq\sum_{\mathfrak{u}} c^\lambda(y)_\mathfrak{tu} c^\lambda(z^\ast)_\mathfrak{us}$. Insbesondere muss $0\neq\sum_{\mathfrak{u}} \rho_\lambda(y)_\mathfrak{tu}\rho_\lambda(z^\ast)_\mathfrak{us} = \rho_\lambda(yz^\ast)_\mathfrak{ts}$ und daher $yz^\ast\neq 0$ sein.

Da die Spurform $\tau$ nichtentartet ist, muss es ein $w\in B$ geben mit $0\neq\tau(wyz^\ast)$. Wenn wir nun $wy = \sum_{v} h_{w,y,v} v$ schreiben für $h_{w,y,v}\in K$, dann ist $\tau(wyz^\ast)=h_{w,y,z}$, weil $z^\ast$ der zu $z$ duale Basisvektor ist. Daher ist $z \preceq_\mathcal{L} y$. Aus Symmetriegründen folgt die Behauptung.

\medbreak
c. folgt aus $x \preceq_\mathcal{L} y \iff x^\ast \preceq_\mathcal{R} y^\ast$.

\medbreak
d. folgt aus $\gamma_{x,y,z}\neq 0 \implies x \sim_\mathcal{L} y^\ast \sim_\mathcal{R} z \sim_\mathcal{L} x^\ast$.
\end{proof}

\begin{corollary}
\begin{enumerate}
	\item Die Algebra $J$ hat als Links-, Rechts- bzw. Bimodul über sich selbst die folgenden Zerlegungen in Links-, Rechts- oder Bimoduln:
	\[J = \bigoplus_\mathfrak{L} \Big(\sum_{x\in\mathfrak{L}} F t_x\Big) = \bigoplus_\mathfrak{R} \Big(\sum_{y\in\mathfrak{R}} Ft_y\Big) = \bigoplus_\mathfrak{Z} \Big(\sum_{z\in\mathfrak{Z}} Ft_z\Big)\]
	wobei $\mathfrak{L},\mathfrak{R},\mathfrak{Z}\subseteq B$ über Links-, Rechts- bzw. zweiseitige Zellen von $(J,(t_x)_{x\in B})$ läuft.
	\item $\sum_{d\in\mathcal{D}\cap\mathfrak{L}} n_d t_d$ ist ein Idempotent für jede Linkszelle $\mathfrak{L}\subseteq B$. Diese Idempotente sind paarweise orthogonal.
	\item $\sum_{d\in\mathcal{D}\cap\mathfrak{Z}} n_d t_d$ ist ein zentrales Idempotent für jede zweiseitige Zelle $\mathfrak{Z}\subseteq B$.
\end{enumerate}
\end{corollary}
\begin{proof}
a. folgt sofort aus obigem Satz, weil dieser zeigt, dass Links-, Rechts- bzw. beidseitige Multiplikation mit Basiselementen $t_w$ nicht aus den angegebenen Unterräumen herausführt, dies also wirklich Untermoduln sind.

b. und c. folgen sofort daraus, dass $1_J = \sum_d n_d t_d$ ist.
\end{proof}

\begin{definition}\label{J_alg:def:order_Lambda}
\index{symbols}{Flambda@$\mathcal{F}_\lambda$}\index{symbols}{$\preceq$}
Seien $e_\mathfrak{Z}$ die zu den zweiseitigen Zellen von $(H,B)$ assoziierten zentralen Idempotente aus dem Korollar. Für jede irreduzible Darstellung $\overline{\rho}_\lambda: J\to F^{d_\lambda\times d_\lambda}$ gibt es dann genau ein $\mathfrak{Z}$ derart, dass $\overline{\rho}_\lambda(e_\mathfrak{Z})\neq 0$ ist, da die Idempotente zentral und paarweise orthogonal sind und sich zu $1_J$ aufsummieren.

Auf diese Weise erhalten wir eine Abbildung $\Lambda\to\Set{\text{zweiseitige Zellen von $H$}}, \lambda\mapsto\mathcal{F}_\lambda$.

Da die zweiseitigen Zellen von $H$ durch $\preceq_\mathcal{LR}$ partiell geordnet sind, erhalten wir auch eine partielle Ordnung $\preceq$ auf $\Lambda$, für die $\lambda \prec \mu$ genau dann gilt, wenn $\mathcal{F}_\lambda \prec_\mathcal{LR} \mathcal{F}_\mu$ gilt.
\end{definition}

\begin{remark}
Man beachte, dass also $\lambda$ und $\mu$ unvergleichbar sind, wenn $\lambda\neq\mu$ aber $\mathcal{F}_\lambda=\mathcal{F}_\mu$ gilt. Das unterscheidet diese Definition auch von der in \citep[2.2.1]{geckjacon} getroffenen.
\end{remark}
\section{\texorpdfstring{$H$}{H} und \texorpdfstring{$J$}{J}}

\begin{remark}
Kehren wir zurück zum konkreten Fall mit einer Coxeter"=Gruppe"=mit"=Gewicht $(W,S,L)$, ihrer Hecke"=Algebra und der asymptotischen Algebra der Hecke"=Algebra (siehe \ref{J_alg:ex:Hecke}). Folgende Definition ist dafür nützlich.
\end{remark}

\begin{definition}[$L$-gute Ringe, siehe {\citep[1.5.9]{geckjacon}}]\label{J_alg:def:L_good}
Ein Ring $R$ heißt \udot{$L$"~gut}, wenn $\IZ_W\subseteq R\subseteq\IC$ und $f_\lambda\in R^\times$ für alle $\lambda\in\Lambda$ gilt.
\end{definition}

\begin{convention}
Wir nehmen für den Rest des Abschnitts an, dass $R$ ein $L$"~guter Ring ist und setzen weiterhin $F=\QuotFld(R)$ und $K=F(\Gamma)$.
\end{convention}

\begin{remark}
Als Alternative zu den Lusztig-Vermutungen \textbf{P1} bis \textbf{P15} definieren Geck und Jacon die Bedingungen $(\clubsuit)$, $(\spadesuit)$ und $(\vardiamond)$, die speziell für die Arbeit mit ihrer Variante der asymptotischen Algebra angepasst sind:
\end{remark}
\begin{conjecture}[Geck-Jacon, siehe {\citep[2.5.3]{geckjacon}}]
\index{symbols}{$(\spadesuit)$, $(\vardiamond)$, $(\clubsuit)$}
Bezeichne mit $\gamma_{x,y,z}$ die Strukturkonstanten der asymptotischen Algebra, die aus $KH$ bzgl. der $\ast$"~symmetrischen Basis $\Set{T_w | w\in W}$ gewonnen wurde.
\begin{itemize}
	\item[$(\clubsuit)$] Für alle $\lambda,\mu\in\Lambda$ soll $a_\mu<a_\lambda$ gelten, falls $\lambda\prec\mu$ ist, und $a_\lambda=a_\mu$, falls $\mathcal{F}_\lambda=\mathcal{F}_\mu$ ist. Dabei seien $\preceq$ und $\mathcal{F}_\lambda$ wie in \ref{J_alg:def:order_Lambda} definiert.
	\item[$(\spadesuit)$] Für alle $w,w',x,y\in W$ mit $x \sim_\mathcal{LR} y$ gilt:
	\[\smash{\sum_{z\in W} \gamma_{x,w',z^{-1}} h_{w,z,y} = \sum_{z\in W} h_{w,x,z} \gamma_{z,w',y^{-1}}}.\]
	\item[$(\vardiamond)$] Jede Kazhdan-Lusztig-Linkszelle $\mathfrak{C}$ enthält genau ein Element $d$ mit $n_d\neq 0$ (dabei sei $n_w$ wie in \ref{J_alg:def:J_algebra} definiert).
\end{itemize}
\end{conjecture}

\begin{remark}
Da die Charaktere von $H$ Werte in $\IZ_W[\Gamma]$ annehmen, enthält jeder $L$-gute Ring die Strukturkonstanten $\gamma_{x,y,z}$ und $n_w$ aus Definition \ref{J_alg:def:J_algebra} für alle $w,x,y,z\in W$. In der Tat ist unter gewissen Voraussetzungen sogar $\gamma_{x,y,z}, n_w\in\IZ$ (diese Voraussetzung kann beispielsweise \textbf{P1}+\textbf{P4} oder $(\clubsuit)$ sein).

$J$ kann also auch als $R$"~Algebra oder sogar $\IZ$"~Algebra mit der Basis $(t_x)_{x\in W}$ aufgefasst werden. Dies werden wir von Zeit zu Zeit tun. Da wir keine anderen $R$- oder $\IZ$-Formen dieser Algebra betrachten werden, ist dies ohne Verwechslungsgefahr möglich.
\end{remark}
\begin{remark}
Wie bereits erwähnt, implizieren die Lusztig-Vermutungen die Gültigkeit dieser drei Vermutungen. Genauer implizieren \textbf{P1}, \textbf{P4} und \textbf{P15} gemeinsam $(\clubsuit)$ und $(\spadesuit)$. Weiter implizieren \textbf{P1}, \textbf{P4}, \textbf{P13} gemeinsam $(\vardiamond)$.

Im Einparameterfall ist $(\spadesuit)$ außerdem eine Konsequenz von $(\clubsuit)$. (Siehe \citep[2.5.12]{geckjacon}.)
\end{remark}
\begin{remark}
Wir wissen bereits, dass $KH$ und $FJ$ zerfallend halbeinfache Algebren sind, gleich viele einfache Moduln besitzen und die einfachen Moduln dieselben Dimensionen haben. Als abstrakte $K$"~Algebren sind $KH$ und $KJ$ also zueinander isomorph. Dieses Argument liefert jedoch keine kanonische Möglichkeit, einen Isomorphismus konkret anzugeben. Es gibt mit Hilfe der Kazhdan-Lusztig-Theorie jedoch auch eine explizite Konstruktion eines Isomorphismus:
\end{remark}

\begin{definition}[Lusztigs Homomorphismus, siehe {\cite{lusztig2003hecke} und \citep[2.5.5]{geckjacon}}]
\index{terms}{Lusztig-Isomorphismus}
\index{symbols}{phi@$\phi$}
Sei $R$ ein $L$-guter Ring.

Definiere $\phi: R[\Gamma] H\to R[\Gamma]J$ durch $R[\Gamma]$"~lineare Fortsetzung von
\[\phi(C_w) := \sum_{\substack{z\in W,d\in\mathcal{D} \\ z \sim_\mathcal{LR} d}} n_d h_{w,d,z} t_z.\]
\end{definition}

\begin{remark}
Falls die $n_d$ ganzzahlig sind, ist $\phi$ in Wirklichkeit schon über $\IZ[\Gamma]$ definiert.
\end{remark}

\begin{theorem}[Lusztig]
Falls $(\spadesuit)$ gilt, ist $\phi$ ein Homomorphismus. Ist $k$ eine kommutative $R[\Gamma]$"~Algebra, dann ist $\ker(kH \xrightarrow{\phi} kJ)$ ein nilpotentes Ideal. Insbesondere ist $KH \xrightarrow{\phi} KJ$ ein Isomorphismus.
\end{theorem}
\begin{proof}
Siehe z.\,B. \cite{lusztig2003hecke} (verwendet die Lusztig-Vermutungen) oder \citep[2.5.5-2.5.7]{geckjacon} (verwendet nur $(\spadesuit)$).
\end{proof}

\begin{remark}
Wir haben bereits eine Konstruktion von Darstellungen für $J$ aus (balancierten) Darstellungen von $H$. Lusztigs Homomorphismus bietet eine Möglichkeit für eine Konstruktion in der umgekehrten Richtung. Die nächsten beiden Lemmata sollten diese Möglichkeit genauer untersuchen:
\end{remark}
\begin{lemma}\label{lusztig_hom:balanced_reps}
Es gelte $(\clubsuit)$. Sei $\rho: KH\to K^{d\times d}$ eine irreduzible Matrixdarstellung von $KH$ und ${f: J\to F^{d\times d}}$ eine Matrixdarstellung von $J$. Falls $\rho$ als $KH \xrightarrow{\phi} KJ \xrightarrow{f} K^{d\times d}$ faktorisiert, ist $\rho$ balanciert und $f=\overline{\rho}$.
\end{lemma}
\begin{proof}
Dazu nutzen wir, dass $C_w\in T_w + \sum_{y<w} \IZ[\Gamma_{>0}] T_y$ gilt. Somit kann Balanciertheit auch äquivalent durch die Bedingung $\nu(\rho(C_w))\geq -a_\lambda$ charakterisiert werden und ggf. gilt $v^{a_\lambda} \rho(C_w) \equiv v^{a_\lambda} \rho(T_w) \mod\mathfrak{m}$ für alle $w\in W$.

Der Rest ist Rechnen:
\begin{align*}
	v^{a_\lambda} \rho(C_w) &= v^{a_\lambda} f(\phi(C_w)) \\
	&= v^{a_\lambda} f\Big(\sum_{z,d} n_d h_{w,d,z} t_z \Big) \\
	&= \sum_{z,d} \underbrace{n_d}_{\in F} \underbrace{v^{a_\lambda} h_{w,d,z}}_{\in\IZ[\Gamma_{\geq 0}]} \underbrace{f(t_z)}_{\in F^{d\times d}} \quad\in F[\Gamma_{\geq 0}]\\
\intertext{Aufgrund von \citep[2.5.11.]{geckjacon} gilt nun $v^{a_\lambda} h_{w,d,z} \equiv \gamma_{w,d,z^{-1}} \mod F[\Gamma_{>0}]$, womit wir weiter erhalten:}
	&\equiv \sum_{z,d} n_d \gamma_{w,d,z^{-1}} f(t_z) \mod F[\Gamma_{>0}] \\ 
	&= f\Big(\sum_{z,d} n_d \gamma_{w,d,z^{-1}} t_z\Big) \\
	&= f\Big(\sum_{d} n_d t_w t_d\Big) \\
	&= f(t_w\cdot 1_J) \\
	&= f(t_w)
\end{align*}
Also ist $\rho$ tatsächlich balanciert und $f=\overline{\rho}$, wie behauptet.
\end{proof}

\begin{lemma}\label{h_vs_j:conj_class}
\index{terms}{Darstellung!balancierte}
Diejenigen irreduziblen, balancierten Matrixdarstellungen $\rho:H\to K^{d\times d}$ vom Isomorphietyp $\lambda\in\Lambda$, die als $\rho=\overline{\rho}\circ\phi$ faktorisieren, bilden eine vollständige $GL_d(F)$"~Kon\-ju\-ga\-tions\-klas\-se.
\end{lemma}
\begin{proof}
Sind $\rho_1$, $\rho_2$ zwei balancierte Matrixdarstellungen vom $KH$"~Isomorphietyp $\lambda$, die beide als $\rho_i=\overline{\rho_i}\circ\phi$ faktorisieren, so sind $\overline{\rho_1}$ und $\overline{\rho_2}$ als $J$"~Moduln vom $J$"~Isomorphietyp $\lambda$. Daher gibt es ein $A\in GL_d(\mathcal{O})$ mit $\overline{A}\overline{\rho_1}=\overline{\rho_2}\overline{A}$. Daraus folgt dann $\overline{A}\rho_1=\rho_2\overline{A}$. Die Reduktion $\overline{A}$ ist in $GL_d(F)$. Das zeigt, dass $\rho_1$ und $\rho_2$ in derselben $GL_d(F)$"~Konjugationsklasse liegen.

Sei umgekehrt $\rho_1$ balanciert, irreduzibel und faktorisiere als $\overline{\rho_1}\circ\phi$. Sei weiter $A\in GL_d(F)$ mit $A\rho_1=\rho_2 A$. Dann ist $\rho_2$ ebenfalls balanciert und irreduzibel. Durch Reduktion folgt  $A\overline{\rho_1}=\overline{\rho_2}A$ und daraus folgt $\rho_2=\overline{\rho_2}\circ\phi$.
\end{proof}

\chapter{Zelluläre Algebren}
\section{Definitionen}

\begin{remark}
$R$ sei für das gesamte Kapitel ein kommutativer Ring mit 1.

Wir geben im ersten Abschnitt zwei alternative Definitionen der Zellularität einer $R$"~Al\-ge\-bra.
\end{remark}

\subsection{Kombinatorische Definition}

\begin{definition}[Zelldatum, siehe \cite{GrahamLehrer}]\label{CellAlg:def}
\index{terms}{Zelldatum}\index{terms}{Zelluläre Algebra|(}
\index{symbols}{Lambda@$\Lambda$}\index{symbols}{$\ast$}\index{symbols}{$\leq$}\index{symbols}{Cstlambda@$C_\mathfrak{st}^\lambda$}\index{symbols}{Mlambda@$M(\lambda)$}
Sei $A$ eine $R$"~Algebra. $(\Lambda,M,C,\ast)$ heißt \udot{Zelldatum für $A$}, falls
\begin{enumerate}[label=(C\arabic*),leftmargin=35pt]
	\item $\Lambda$ ist eine endliche, partiell geordnete Menge, $M(\lambda)$ ist eine endliche, nichtleere Menge für alle $\lambda\in\Lambda$ und
	\[C: \coprod_{\lambda\in\Lambda} M(\lambda)\times M(\lambda) \to A\]
	ist eine Injektion, deren Bild eine $R$"~Basis von $A$ ist. Wir bezeichnen das Bild von $(\mathfrak{s},\mathfrak{t})\in M(\lambda)\times M(\lambda)$ unter $C$ mit $C_\mathfrak{st}^\lambda$.
	\item $\ast: A\to A$ ist ein $R$"~linearer, involutiver Antiautomorphismus mit
	\[(C_\mathfrak{st}^\lambda)^\ast = C_\mathfrak{ts}^\lambda\]
	\item Für alle $\lambda,\mathfrak{s},\mathfrak{t}$ und alle $a\in A$ ist
	\[aC_\mathfrak{st}^\lambda \equiv \sum_{\mathfrak{u}\in M(\lambda)} r_a(\mathfrak{u},\mathfrak{s}) C_\mathfrak{ut}^\lambda \mod A(<\lambda)\]
	mit von $a,\mathfrak{u},\mathfrak{s}$, aber nicht von $\mathfrak{t}$ abhängigen Koeffizienten $r_a(\mathfrak{u},\mathfrak{s})\in R$. Dabei ist $A(<\lambda)$ der von $\Set{C_\mathfrak{vw}^\mu | \mu<\lambda, \mathfrak{v},\mathfrak{w}\in M(\mu)}$ erzeugte $R$-Untermodul von $A$.
\end{enumerate}
Statt (C3) könnte man aufgrund von (C2) offenbar auch
\begin{enumerate}[resume*]
	\item[(C$3^\ast$)] Für alle $\lambda,\mathfrak{s},\mathfrak{t}$ und alle $a\in A$ ist
	\[C_\mathfrak{ts}^\lambda a^\ast \equiv \sum_{\mathfrak{u}\in M(\lambda)} C_\mathfrak{tu}^\lambda r_a(\mathfrak{u},\mathfrak{s}) \mod A(<\lambda)\]
	mit von $a,\mathfrak{u},\mathfrak{s}$, aber nicht von $\mathfrak{t}$ abhängigen Koeffizienten $r_a(\mathfrak{u},\mathfrak{s})\in R$.
\end{enumerate}
fordern.

$A$ heißt \udot{zelluläre Algebra}, falls es ein Zelldatum für $A$ gibt.
\end{definition}

\begin{remark}
Zelluläre Algebren sind insbesondere frei und endlich erzeugt als $R$"~Moduln.
\end{remark}
\begin{remark}
\index{terms}{Zelluläre Algebra!strikt zellulär}
Statt (C2) fordern einige Autoren (siehe etwa \cite{goodman2011cellularity}, \cite{goodman2011jucysmurphy} oder \cite{geetha2013wreath}) auch nur die schwächere Bedingung $(C_\mathfrak{st}^\lambda)^\ast \equiv C_\mathfrak{ts}^\lambda \mod A(<\lambda)$ und bezeichnen Algebren, die (C2) erfüllen, als "`strikt zellulär"'.
\end{remark}

\begin{example}[Matrizenringe]\label{CellAlg:example_matrix}
$R^{n\times n}$ ist zellulär mit Zelldatum 
\[\Lambda:=\Set{1}\]
\[M(1):=\Set{1,\ldots,n}\]
\[C_\mathfrak{st}^1:=E_\mathfrak{st}\]
\[a^\ast:=a^T\]
wobei $E_\mathfrak{st}$ die Standardbasis des Matrizenrings meint. Die Konstanten $r_a(\mathfrak{u},\mathfrak{s})$ sind dabei ganz einfach durch die Matrizeneinträge von $a$ gegeben: $r_a(\mathfrak{u},\mathfrak{s}) = a_{\mathfrak{u}\mathfrak{s}}$.

Eine simple Deformation dieser Algebren erhält man, indem man $X\circ Y:=XAY$ für eine feste Matrix $A\in GL_n(R)$ setzt. Die Einheit ist dann offenkundig $A^{-1}$. Transponieren ist genau dann ein Antiautomorphismus dieser Algebra, wenn $A$ symmetrisch ist. In diesem Fall liefert dieselbe Konstruktion auch ein Zelldatum für die Algebra $(R^{n\times n},\circ)$. Weil $(R^{n\times n},\circ) \to R^{n\times n}, X\mapsto AX$ ein Isomorphismus dieser $R$"~Algebren ist, ist daher
\[\Lambda:=\Set{1}\]
\[M(1):=\Set{1,\ldots,n}\]
\[C_\mathfrak{st}^1:=AE_\mathfrak{st}\]
\[a^\ast:=a^T\]
ebenfalls ein Zelldatum für $R^{n\times n}$ für alle symmetrischen $A\in GL_n(R)$.
\end{example}

\begin{example}[Polynomringe]\label{CellAlg:example_polynom}
$R[x]/(x^n)$ ist zellulär mit Zelldatum
\[\Lambda:=\Set{0,1,\ldots,n-1} \,\text{mit der Ordnung}\,n-1<\ldots<1<0\]
\[M(j):=\Set{j}\]
\[C_{jj}^j := x^j\]
\[a^\ast := a\]
\end{example}

\begin{remark}
Die beiden Beispiele befinden sich an den Extrempunkten der möglichen zellulären Algebren. In dem einen Beispiel ist die Ordnung $\Lambda$ trivial, in dem anderen die Matrizengröße $M(\lambda)$. Allgemeine zelluläre Algebren sind in gewisser Weise zwischen diesen beiden Beispielen "`interpoliert"'.

Die Definition sichert, dass die Multiplikation in $A$ sich aus (wie oben deformierten) Matrizenringen zusammensetzt, die entlang der Ordnung $\Lambda$ aufgereiht sind. Das stimmt bis auf Restterme, die kleiner in der Ordnung sind.
\end{remark}

\begin{example}[Hecke-Algebren]\label{CellAlg:example_hecke}
\index{terms}{Hecke-Algebra}\index{terms}{Kazhdan-Lusztig!-Basis}
\index{symbols}{Lambda@$\Lambda$}\index{symbols}{$\ast$}\index{symbols}{$\leq$}\index{symbols}{dlambda@$d_\lambda$}\index{symbols}{Cstlambda@$C_\mathfrak{st}^\lambda$}\index{symbols}{Mlambda@$M(\lambda)$}
Sei $(W,S,L)$ eine Coxeter"=Gruppe"=mit"=Gewicht und $R\subseteq\IR$ ein "`$L$-guter"' Ring im Sinne von \citep[1.5.9]{geckjacon}, d.\,h. $\IZ_W\subseteq R$ und $f_\lambda\in R^\times$ für alle $\lambda\in\Lambda$. Setze $F:=\QuotFld(R)$ und $K:=F(\Gamma)$.

Dann ist ein Zelldatum von $F[\Gamma]H$ wie folgt gegeben:
\begin{itemize}
	\item $\ast$ sei die Involution $T_w\mapsto T_{w^{-1}}$.
	\item $\Lambda=\Irr(W)$ sei durch $\lambda < \mu :\iff \mathcal{F}_\lambda \preceq_\mathcal{LR} \mathcal{F}_\mu, \mathcal{F}_\lambda \not\sim_\mathcal{LR} \mathcal{F}_\mu$ partiell geordnet, siehe \ref{J_alg:def:order_Lambda}.
	\item Für $\lambda\in\Lambda$ setze $M(\lambda):=\Set{1,\ldots,d_\lambda}$, wobei $d_\lambda$ der Grad des irreduziblen Charakters $\lambda$, d.\,h. die Dimension der zugehörigen Darstellung sei.
	\item Die Zellbasis wird wie folgt konstruiert:
	
	Wähle zu jedem $\lambda\in\Lambda$ eine balancierte Darstellung $\rho_\lambda: KH\to K^{d_\lambda\times d_\lambda}$, bezeichne die dazugehörige Darstellung der $J$"~Algebra wie gewohnt mit $\overline{\rho}_\lambda: J\to F^{d_\lambda\times d_\lambda}$ und wähle eine invariante Bilinearform $B_\lambda\in GL_{d_\lambda}(F)$ für $\overline{\rho}_\lambda$ (siehe \ref{balanced_reps:invariant_blf3}). Dann setze:
	\[C_\mathfrak{st}^\lambda := \sum_{w\in W} (B_\lambda\cdot\overline{\rho}_\lambda(t_w))_\mathfrak{st} C_w\]
\end{itemize}
Dass dies wirklich ein Zelldatum auf $F[\Gamma]H$ liefert, wird in \cite[2.6.8-2.6.12]{geckjacon} bewiesen.

Man erhält sogar ein Zelldatum für $R[\Gamma]H$, wenn man $\rho_\lambda$ mit $\rho_\lambda(H) \subseteq R[\Gamma]^{d_\lambda\times d_\lambda}$ und ein $B_\lambda\in GL_{d_\lambda}(R)$ wählt. Nach Lemma \ref{balanced_reps:invariant_blf3} ist dies möglich.
\end{example}

\begin{definition}[Kanonische Ideale aus einem Zelldatum, siehe {\citep[1.5+1.6]{GrahamLehrer}}]
\index{terms}{Ordnungsideal}
\index{symbols}{APhi@$A(\Phi)$, $A(\Set{\lambda})$, $A(<\lambda)$}
Sei $A$ eine zelluläre Algebra mit Zelldatum wie oben.

Sei weiter $\Phi\subseteq\Lambda$ ein Ordnungsideal, d.\,h. eine "`nach unten abgeschlossene"' Teilmenge: Für alle $\lambda\in\Lambda$ und $\phi\in\Phi$ folgt aus $\lambda\leq\phi$ automatisch auch $\lambda\in\Phi$.

Dann definiere $A(\Phi)$ als den von $\Set{C_\mathfrak{st}^\mu | \mu\in\Phi, \mathfrak{s},\mathfrak{t}\in M(\mu)}$ erzeugten $R$-Untermodul. $A(\Phi)$ ist offenbar ein zweiseitiges, $\ast$"~invariantes Ideal von $A$.

Sind $\Phi'\subseteq\Phi$ zwei Ideale von $\Lambda$, so setze $\Psi:=\Phi\setminus\Phi'$ und $A(\Psi):=A(\Phi)/A(\Phi')$.
\end{definition}

\begin{example}
Das in der Definition vorkommende $A(<\lambda)$ ist eines dieser Ideale.
\end{example}
\begin{remark}
Offenbar gilt $A(\Phi)+A(\Phi') = A(\Phi\cup\Phi')$ und $A(\Phi)\cap A(\Phi')=A(\Phi\cap\Phi')$ für alle Ordnungsideale $\Phi,\Phi'\subseteq\Lambda$.
\end{remark}
\begin{remark}
\index{symbols}{xy@$[x,y]$}
Man macht sich leicht klar, dass $A(\Psi)$ tatsächlich nur von $\Psi=\Phi\setminus\Phi'$ und nicht von der konkreten Wahl von $\Phi$ und $\Phi'$ abhängt.

Es gibt eine intrinsische Beschreibung der Teilmengen $\Psi\subseteq\Lambda$, die sich als Differenz zweier Ideale schreiben lassen: $\Psi$ ist genau dann von dieser Form, wenn mit jedem Paar $x,y\in\Psi$ auch das ganze Intervall $[x,y]:=\Set{\lambda\in\Lambda | x\leq \lambda\leq y}$ Teilmenge von $\Psi$ ist.

Insbesondere ist $A(\Set{\lambda})=A(\leq\lambda)/A(<\lambda)$ wohldefiniert.
\end{remark}

\begin{lemma}[Vererbung von Zellularität unter Standardoperationen]
Sei $A$ eine zelluläre $R$"~Algebra mit Zelldatum $(\Lambda,M,C,\ast)$. Dann gelten folgende Verträglichkeiten:
\begin{enumerate}
	\item Skalarerweiterungen:
	
	Ist $R\to S$ ein unitärer Ringhomomorphismus, dann ist $SA:=S\otimes_R A$ eine zelluläre $S$"~Algebra mit Zelldatum $(\Lambda,M,1\otimes C,\id\otimes\ast)$.
	\item Quotienten:
	
	Ist $\Phi\subseteq\Lambda$ ein Ordnungsideal, so ist $A(\Lambda\setminus\Phi)=A/A(\Phi)$ zellulär mit Zelldatum $(\Lambda\setminus\Phi,M_{|\Lambda\setminus\Phi},C_{|\Lambda\setminus\Phi},\ast)$.
	\item Duale Algebren:
	
	$A^\text{op}$ ist zellulär mit Zelldatum $(\Lambda,M,C,\ast)$.
	
	\item Produkte:
	
	Sind $A_1,\ldots,A_m$ zellulär mit Zelldaten $(\Lambda_i,M_i,C_i,\ast)$, so ist $A=A_1\times\ldots\times A_m$ zellulär mit Zelldatum $(\Lambda,M,C,\ast)$, wobei
	\[\Lambda:=\coprod_i \Lambda_i\]
	\[\forall \lambda\in\Lambda_i: M(\lambda):=M_i(\lambda)\]
	\[\forall \lambda\in\Lambda_i,\mathfrak{s},\mathfrak{t}\in M(\lambda): C(\mathfrak{s},\mathfrak{t}):=C_i(\mathfrak{s},\mathfrak{t})\]
	\[\forall a_i\in A_i: (a_1,\ldots,a_m)^\ast := (a_1^\ast,\ldots,a_m^\ast)\]
	
	\item Tensorprodukte:
	
	Sind $A_1,\ldots,A_m$ zellulär mit Zelldaten $(\Lambda_i,M_i,C_i,\ast)$, so ist $A=A_1\otimes_R \ldots\otimes_R A_m$ zellulär mit Zelldatum $(\Lambda,M,C,\ast)$, wobei
	\[\Lambda:=\prod_i \Lambda_i\]
	\[\forall (\lambda_i)\in\Lambda: M(\lambda):=\prod_i M_i(\lambda_i)\]
	\[\forall \mathfrak{s},\mathfrak{t}\in M(\lambda): C(\mathfrak{s},\mathfrak{t}):=\bigotimes_i C_i(\mathfrak{s}_i,\mathfrak{t}_i)\]
	\[\forall a_i\in A_i: (a_1\otimes\ldots\otimes a_m)^\ast := a_1^\ast\otimes\ldots\otimes a_m^\ast\]
	
\end{enumerate}
\end{lemma}

\begin{remark}
Man kann zeigen, dass (unter einer harmlosen zusätzlichen Voraussetzung) auch die Umkehrung in d. gilt: Wenn $A$ zellulär ist und sich als Produkt von zwei Algebren schreiben lässt, sind auch die Faktoren zellulär.
\end{remark}

\begin{example}[Spezialisierungen von Hecke-Algebren]
\index{terms}{Hecke-Algebra}\index{terms}{Kazhdan-Lusztig!-Basis}
Da das Beispiel \ref{CellAlg:example_hecke} eine Zellbasis für die generischen Hecke"=Algebren definiert, folgt aus der Verträglichkeit mit Spezialisierungen, dass z.\,B. auch die Gruppenalgebren $\IQ_W[W]$ ein Zelldatum haben, das durch Spezialisierung aus dem in \ref{CellAlg:example_hecke} hervorgeht. Dies ist im Wesentlichen die Kazhdan-Lusztig-Basis selbst und somit verschieden vom Zelldatum, das von einem Wedderburn"=Isomorphismus $\IQ_W[W] \isomorphic \prod_{\lambda\in\Lambda} \IQ_W^{d_\lambda\times d_\lambda}$ herkäme. Diese Eigenschaft der Kazhdan-Lusztig-Basis war das motivierende Beispiel für die Definition von zellulären Algebren durch Graham und Lehrer (und entsprechend bereits vor der Definition zellulärer Algebren bekannt).

In der Tat ist dies ja sogar ein Zelldatum für die Gruppenalgebra $R[W]$ und überträgt sich daher z.\,B. auch auf modulare Gruppenalgebren (außer für die sogenannten "`schlechten"' Primzahlen, da $R$ ja ein "`$L$-guter"' Ring zu sein hatte für die Konstruktion der Zellbasis nach Geck).
\end{example}

\subsection{Idealtheoretische Definition}

\begin{definition}[Zellideale und Zellketten, siehe \cite{KoenigXi_cellalg_1}]\label{CellAlg:def:Cellideals}
\index{terms}{Zellideal|(}\index{terms}{Zellkette}
\index{symbols}{alpha@$\alpha$}
Sei $A$ eine $R$"~Algebra mit einem involutiven Antiautomorphismus $\ast:A\to A$. Ein zweiseitiges Ideal $J\subseteq A$ heißt \udot{Zellideal} (bzgl. $\ast$), falls
\begin{enumerate}
	\item $J^\ast=J$ ist,
	\item es ein Linksideal $\Delta\subseteq J$ gibt derart, dass $\Delta$ frei und endlich erzeugt als $R$"~Modul ist und ein Isomorphismus $\alpha: \Delta\otimes_R \Delta^\ast \to J$ von $A$"~$A$"~Bimoduln existiert, sowie
	\item $\alpha$ und $\ast$ verträglich sind in dem Sinne, dass das Diagramm \ref{fig:cellideal} kommutiert.
	\begin{figure}[ht]
	\centering
		\begin{tikzpicture}
		\matrix (m) [matrix of math nodes, row sep=3em,
		column sep=6em, text height=1.5ex, text depth=0.25ex]
		{
			\Delta\otimes_R\Delta^\ast & \Delta\otimes_R\Delta^\ast \\
			J & J\\
		};
		\path[->,font=\scriptsize]
		(m-1-1) edge node[above]{$x\otimes y^\ast \mapsto y\otimes x^\ast$} (m-1-2)
		(m-2-1) edge node[below]{$\ast$} (m-2-2)
		(m-1-1) edge node[left]{$\alpha$} (m-2-1)
		(m-1-2) edge node[right]{$\alpha$} (m-2-2);
		\end{tikzpicture}
		\caption{Verträglichkeitsbedingung für Zellideale}
		\label{fig:cellideal}
	\end{figure}
	Explizit soll also $\alpha(x\otimes y^\ast)^\ast = \alpha(y\otimes x^\ast)$ gelten.
\end{enumerate}

Eine Kette von zweiseitigen Idealen
\[0=J_0 \subseteq J_1 \subseteq \ldots \subseteq J_n =A\]
heißt \udot{Zellkette}, falls es eine Zerlegung der Form
\[A=U_1 \oplus \ldots \oplus U_n\]
in $\ast$"~invariante $R$"~Untermoduln gibt derart, dass $J_i=\sum_{j\leq i} U_j$ und $J_i/J_{i-1} \unlhd A/J_{i-1}$ für alle $i=1,\ldots,n$ ein Zellideal bzgl. der induzierten Involution ist.
\end{definition}

\begin{theorem}[Äquivalenz der Definitionen, siehe {\citep[3.4]{KoenigXi_cellalg_1}}]
Sei $(A,\ast)$ eine $R$"~Algebra mit einem involutiven Antiautomorphismus. Dann sind äquivalent:
\begin{enumerate}
	\item $A$ ist zellulär mit $\ast$ als Teil des Zelldatums.
	\item Es gibt eine Zellkette
	\[0=J_0 \subseteq J_1 \subseteq \ldots \subseteq J_n = A\]
\end{enumerate}
\end{theorem}
\begin{proof}
a.$\implies$b.

Sei $\lambda\in\Lambda$ minimal. Man beachte, dass dann $A(<\lambda)=0$ gilt.

Wir behaupten, dass $J:=A(\leq\lambda)$ ein Zellideal ist. Setze dafür $\Delta:=\sum_{\mathfrak{s}\in M(\lambda)} RC_\mathfrak{st}^\lambda$ für ein festes $\mathfrak{t}\in M(\lambda)$. Aufgrund von (C3) ist $\Delta$ ein Linksideal von $A$. Wegen (C1) ist es frei über $R$. Es ist $\Delta^\ast=\sum_{\mathfrak{s}\in M(\lambda)} RC_\mathfrak{ts}^\lambda$ wegen (C2).

Wir definieren $\alpha: \Delta\otimes_R\Delta^\ast \to J$ durch $\alpha(C_\mathfrak{st}^\lambda\otimes C_\mathfrak{tu}^\lambda):=C_\mathfrak{su}^\lambda$. Nach Konstruktion ist $\alpha$ dann bijektiv. Wegen (C3) und (C$3^\ast$) ist $\alpha$ ein $A$"~$A$"~Bimodulhomomorphismus. Wegen (C2) ist die Verträglichkeit mit $\ast$ gesichert.

Indem wir nun induktiv die Existenz einer Zellkette in $A/J$ annehmen -- was ja geht, da dies eine zelluläre Algebra mit $\Lambda\setminus\Set{\lambda}$ als partieller Ordnung und den von $A$ induzierten restlichen Daten ist -- und mit $J$ kombinieren, erhalten wir eine Zellkette in $A$. Man beachte, dass dies wirklich geht: Die $\ast$"~invarianten Unterräume sind alle von der Form $\sum_{\mathfrak{s},\mathfrak{t}} RC_\mathfrak{st}^\mu$ für ein $\mu\in\Lambda$.

\medbreak
b.$\implies$a.

Wir gehen den umgekehrten Weg und konstruieren aus der Idealkette ein Zelldatum mit $\Lambda=\Set{1,\ldots,n}$, versehen mit der natürlichen (totalen!) Ordnung, sowie
\[M(i)=\Set{1,\ldots,\sqrt{\dim_R J_i/J_{i-1}}}\]
für alle $i$. Man beachte, dass $\dim_R J_i/J_{i-1}$ nach Definition eines Zellideals stets eine Quadratzahl ist.

Wir wählen eine Zerlegung $A=U_1\oplus\ldots\oplus U_n$ so, dass wie in der Definition $U_i^\ast=U_i$ und $J_i = \bigoplus_{j\leq i} U_j$ ist.

Per Induktion über $n$ ist $A/J_1$ bereits zellulär, d.\,h. wir finden für $A/J_1$ eine Zellbasis $\Set{C_\mathfrak{st}^\lambda+J_1 | \lambda>1, \mathfrak{s},\mathfrak{t}\in M(\lambda)}$. Wir wählen Urbilder $C_\mathfrak{st}^\lambda\in U_2\oplus\ldots\oplus U_n$. Weil ${U_2\oplus\ldots\oplus U_n}$ ein Komplement von $J_1$ ist, sind diese Urbilder eindeutig bestimmt. Somit ist insbesondere (C2) für diese Elemente bereits erfüllt.

(C3) ist für diese Elemente weiterhin erfüllt, weil alle Störterme, die dazukommen können, ja in $J_1$ liegen und damit nach Definition der Ordnung am unteren Ende.

Seien nun $\Delta_1\subseteq J_1$ und $\alpha$ wie in der Definition eines Zellideals. Wähle dann eine Basis $B_\mathfrak{s}$ von $\Delta_1$ und setze $C_\mathfrak{st}^1 :=\alpha(B_\mathfrak{s}\otimes B_\mathfrak{t})$ für alle $\mathfrak{s},\mathfrak{t}\in M(1)$.

Da $\alpha$ ein $A$"~$A$"~Bimodulisomorphismus ist, ist $C_\mathfrak{st}^1$ eine $R$"~Basis von $J_1$, zusammen mit den anderen Elementen haben wir damit eine $R$"~Basis von $A$. Daher ist (C1) erfüllt. (C2) ist aufgrund der Induktionsvoraussetzung bzw. der Verträglichkeitsforderung an $\alpha$ erfüllt. (C3) ist erfüllt, weil $\alpha$ $A$"~linear ist.
\end{proof}

\begin{remark}
Der Beweis zeigt deutlich die Uneindeutigkeit der Basis und der partiell geordneten Menge im Zelldatum.

Außerdem wird klar, dass Zellularität gar nicht von $\Lambda,M$ und $C$ abhängt, sondern nur von $A$ und $\ast$. Der Satz rechtfertigt also, in Zukunft "`Sei $(A,\ast)$ zellulär"' zu schreiben, wenn wir auf die Wahl eines konkreten Zelldatums verzichten wollen.

Man beachte aber, dass die Involution tatsächlich einen Einfluss auf die Zellularität hat. Es gibt Algebren, die bzgl. einer Involution zellulär sind, bzgl. einer anderen aber nicht. Beispielsweise ist $\IQ[X]/(X^2)$ bezüglich der Identität zellulär (Beispiel \ref{CellAlg:example_polynom}), jedoch nicht bezüglich $X\mapsto -X$, denn wäre dem so, dann müsste $(X)$ ein Zellideal sein (denn es ist das einzige von Null verschiedene Ideal, dessen Dimension eine Quadratzahl ist), es besitzt aber keine $\ast$"~invariante Basis. (Dieses Beispiel stammt aus \cite{KoenigXi_cellalg_1}.)
\end{remark}
\begin{remark}
Jede Zellkette der Länge $n$ liefert ein Zelldatum mit $\abs{\Lambda}=n$. Es gibt Algebren, in denen es Zellketten verschiedener Länge gibt, siehe \cite{KoenigXi_cellalg_no_cells}. Das zeigt, dass nicht nur die Ordnung selbst, sondern sogar die Größe von $\Lambda$ uneindeutig ist. Wir werden aber später sehen unter welchen Bedingungen dem abgeholfen werden kann und dass nichtsdestotrotz ein guter Teil des Zelldatums alleine durch die algebraischen Eigenschaften von $A$ festgelegt ist.
\end{remark}
\begin{remark}
Die Konstruktion zeigt weiterhin, dass jedes minimale Element $\lambda\in\Lambda$ ein Zellideal definiert. Allgemeiner ist $A(\Set{\lambda})=A(\leq\lambda)/A(<\lambda)$ ein Zellideal in der zellulären Algebra $A/A(<\lambda)$. Damit gelten alle Sätze über Zellideale auch automatisch für diese Ideale.
\end{remark}

\begin{remark}
Die Definition über Idealketten macht es möglich, eine weitere Vererbungseigenschaft von zellulären Algebren zu beweisen. Das folgende Korollar ist eine naheliegende Verallgemeinerung von \citep[4.3]{KoenigXi_cellalg_1}.
\end{remark}
\begin{corollary}[Kondensationen von zellulären Algebren]\label{CellAlg:condensation}
\index{terms}{Zelluläre Algebra|)}\index{terms}{Kondensation}\index{terms}{Zellideal|)}
$R$ habe die Eigenschaft, dass endlich erzeugte, projektive $R$"~Moduln bereits frei sind (z.\,B. $R$ lokal oder ein Polynomring in endlich vielen Unbestimmen über einem Hauptidealring). Weiter sei $(A,\ast)$ eine zelluläre $R$"~Algebra.

Ist dann $e\in A$ ein Idempotent mit $e^\ast=e$, so ist auch $eAe$ zellulär bzgl. der Einschränkung von $\ast$.
\end{corollary}
\begin{proof}
Sei $0=J_0 \subseteq J_1 \subseteq \ldots \subseteq J_n =A$ eine Zellkette in $A$. Wir müssen eine Zellkette von $eAe$ konstruieren. Dazu zeigen wir, dass für jedes Zellideal $J\subseteq A$ das Ideal $eJe\subseteq eAe$ auch ein Zellideal ist. Induktiv folgt dann die Behauptung.

Sei also $J$ ein Zellideal, d.\,h. es gibt einen Linksmodul $\Delta\subseteq J$ und einen Isomorphismus $\alpha: \Delta\otimes_R\Delta^\ast\to J$, der im Sinne von Definition \ref{CellAlg:def:Cellideals} mit $\ast$ verträglich ist. Dann ist $eJe$ ein zweiseitiges Ideal von $eAe$, welches wegen $e^\ast=e$ invariant unter $\ast$ ist, $e\Delta \subseteq eJe$ ein $eAe$-Linksideal und die Einschränkung von $\alpha$ auf $e\Delta \otimes_R \Delta^\ast e \to eJe$ immer noch ein Isomorphismus, der im selben Sinne mit $\ast$ verträglich ist.

Es gilt $\Delta=e\Delta\oplus(1-e)\Delta$ als $R$"~Moduln, d.\,h. $e\Delta$ ist $R$"~projektiv und endlich erzeugt. Nach Voraussetzung ist somit $e\Delta$ frei und endlich erzeugt als $R$"~Modul. Das zeigt, dass $eJe$ ein Zellideal von $eAe$ ist.
\end{proof}
\section{Zellideale und invariante Bilinearformen}

\begin{remark}
Bereits in der Orginalarbeit von Graham und Lehrer sind die invarianten Bilinearformen das wesentliche Hilfsmittel, um die Darstellungstheorie zellulärer Algebren zu kontrollieren. In \cite{GrahamLehrer} werden die Bilinearformen durch explizite Konstruktion definiert. Folgendes Lemma charakterisiert die invarianten Bilinearformen auf Zellidealen im abstrakten Kontext nach König und Xi. Die Idee dazu findet sich implizit bereits in \cite{KoenigXi_cellalg_3} und \cite{GrahamLehrer}.
\end{remark}

\begin{lemmadef}[Invariante Bilinearform zu Zellidealen]
\index{terms}{Zellideal}\index{terms}{Invariante Bilinearform}
Sei $(A,\ast)$ eine $R$"~Algebra mit einem involutiven Antiautomorphismus und $J\subseteq A$ ein Zellideal. Seien $\Delta\subseteq J$ ein Linksideal und $\alpha: \Delta\otimes_R\Delta^\ast\to J$ ein Isomorphismus wie in der Definition von Zellidealen \ref{CellAlg:def:Cellideals}.

Es gilt in dieser Situation:
\begin{enumerate}
	\item Es gibt genau eine $R$"~lineare Abbildung $\phi: \Delta^\ast\otimes_A \Delta\to R$ derart, dass das Diagramm \ref{CellAlg:fig:invariant_bilinearform} kommutiert.
	
	\begin{figure}[hbp]
	\centering
		\begin{tikzpicture}
		\matrix (m) [matrix of math nodes, row sep=3em,
		column sep=2.5em, text height=1.5ex, text depth=0.25ex]
		{
			\Delta\otimes_R \Delta^\ast \otimes_A \Delta \otimes_R \Delta^\ast & \Delta\otimes_R \Delta^\ast \\
			J \otimes_A J & J \\
		};
		\path[->,font=\scriptsize]
			(m-1-1) edge node[above]{$\beta$} (m-1-2)
			(m-1-1) edge node[left]{$\alpha\otimes\alpha$} (m-2-1)
			(m-2-1) edge node[below]{$x\otimes y\mapsto xy$} (m-2-2)
			(m-1-2) edge node[right]{$\alpha$} (m-2-2);
		\end{tikzpicture}
		\caption{Definition von $\phi$}
		\label{CellAlg:fig:invariant_bilinearform}
	\end{figure}
	Dabei ist $\beta(w\otimes x\otimes y\otimes z)=\phi(x\otimes y)\cdot w\otimes z$.
	
	\item Wählt man eine Basis $(C_\mathfrak{s})$ von $\Delta$, sodass $C_\mathfrak{st}=\alpha(C_\mathfrak{s}^\ast\otimes C_\mathfrak{t})$ eine Basis von $J$ ist, dann ist $\phi$ charakterisiert durch
	\[C_\mathfrak{st} C_\mathfrak{uv} = \phi(C_\mathfrak{t}^\ast \otimes C_\mathfrak{u}) C_\mathfrak{sv}.\]
	\item $\phi$ ist $A$"~invariant:
	\[\forall x,y\in\Delta, a\in A: \phi(x^\ast a \otimes y)=\phi(x^\ast \otimes ay)\]
	\item $\phi$ ist symmetrisch:
	\[\forall x,y\in\Delta: \phi(x^\ast \otimes y)=\phi(y^\ast \otimes x)\]
	\item Für alle $x,y,z\in\Delta$ gilt:
	\[\alpha(x\otimes y^\ast)z=x\phi(y^\ast\otimes z)\]
	\[x^\ast\alpha(y\otimes z^\ast)=\phi(x^\ast\otimes y)z^\ast\]
\end{enumerate}
\end{lemmadef}
\begin{proof}
Betrachte diese Basis $C_\mathfrak{s}$. Weil $\alpha$ ein $A$"~$A$"~Bimodulhomomorphismus ist, gilt
\[AC_\mathfrak{st} = A\alpha(C_\mathfrak{s}\otimes C_\mathfrak{t}^\ast) = \alpha(AC_\mathfrak{s}\otimes C_\mathfrak{t}^\ast) \subseteq \sum_\mathfrak{u} \alpha(RC_\mathfrak{u}\otimes C_\mathfrak{t}^\ast) = \bigoplus_\mathfrak{u} RC_\mathfrak{ut}\]
und analog $C_\mathfrak{st}A \subseteq\bigoplus_\mathfrak{v} RC_\mathfrak{sv}$ und die auftretenden Koeffizienten sind jeweils unabhängig von $\mathfrak{t}$ bzw. $\mathfrak{s}$.

Mit dieser Überlegung ist
\[C_\mathfrak{st} C_\mathfrak{uv} \in \Big(\bigoplus_{\mathfrak{t}'} RC_{\mathfrak{st}'}\Big)\cap\Big(\bigoplus_{\mathfrak{u}'} RC_{\mathfrak{u}'\mathfrak{v}}\Big) = RC_\mathfrak{sv}\]
und der Koeffizient ist von $\mathfrak{s}$ und $\mathfrak{v}$ unabhängig. Wenn wir diesen Koeffizient mit $\phi(C_\mathfrak{t}^\ast\otimes C_\mathfrak{u})$ bezeichnen, gilt also
\[C_\mathfrak{st} C_\mathfrak{uv} = \phi(C_\mathfrak{t}^\ast \otimes C_\mathfrak{u}) C_\mathfrak{sv},\]
wie behauptet. Das lässt sich auch schreiben als
\[\alpha(C_\mathfrak{s}\otimes C_\mathfrak{t}^\ast)\cdot\alpha(C_\mathfrak{u}\otimes C_\mathfrak{v}^\ast) = \phi(C_\mathfrak{t}^\ast \otimes C_\mathfrak{u}) \alpha(C_\mathfrak{s}\otimes C_\mathfrak{v}^\ast)\]
Wenn wir $\phi$ jetzt als die $R$-bilineare Fortsetzung dieser Gleichung auffassen, erhalten wir ganz allgemein
\[\forall w,x,y,z\in\Delta: \alpha(w\otimes x^\ast)\cdot\alpha(y\otimes z^\ast) = \phi(x^\ast \otimes y) \alpha(w\otimes z^\ast)\]
Außerdem zeigt uns das die Eindeutigkeit von $\phi$. Da $\alpha$ ein $A$"~$A$"~Bimodulhomomorphismus ist, erhalten wir weiter:
\begin{align*}
	\phi(x^\ast a\otimes y) \alpha(w\otimes z^\ast) &= \alpha(w\otimes x^\ast a)\cdot\alpha(y\otimes z^\ast) \\
	&= \alpha(w\otimes x^\ast)\cdot a\cdot\alpha(y\otimes z^\ast) \\
	&= \alpha(w\otimes x^\ast)\cdot\alpha(ay\otimes z^\ast) \\
	&= \phi(x^\ast \otimes ay) \alpha(w\otimes z^\ast)
\end{align*}
d.\,h. $\phi$ ist $A$"~invariant, wie behauptet. Das zeigt nun a., b. und c.

\bigbreak
Die Behauptung d. folgt, indem wir $\ast$ auf beide Seiten von
\[\alpha(w\otimes x^\ast)\cdot\alpha(y\otimes z^\ast) = \phi(x^\ast \otimes y) \alpha(w\otimes z^\ast)\]
anwenden. Dann erhalten wir nämlich
\begin{align*}
	\alpha(z\otimes y^\ast)\cdot\alpha(x\otimes w^\ast) &=\alpha(y\otimes z^\ast)^\ast \cdot \alpha(w\otimes x^\ast)^\ast \\
	&= \phi(x^\ast \otimes y) \alpha(w\otimes z^\ast)^\ast \\
	&= \phi(x^\ast \otimes y) \alpha(z\otimes w^\ast).
\end{align*}
Weil die linke Seite auch gleich $\phi(y^\ast \otimes x)\alpha(z\otimes w^\ast)$ ist, zeigt das die Symmetrie.

\bigbreak
e. erhalten wir ebenfalls aus der definierenden Gleichung. Seien nämlich $r_\mathfrak{v}\in R$ die Koeffizienten in
\[C_\mathfrak{st} \cdot C_\mathfrak{u} = \sum_\mathfrak{v} r_\mathfrak{v} C_\mathfrak{v}.\]
Dann folgt:
\begin{align*}
	\sum_\mathfrak{v} r_\mathfrak{v} C_\mathfrak{vz} &= \alpha\Big(\sum_\mathfrak{v} r_\mathfrak{v} C_\mathfrak{v}\otimes C_\mathfrak{z}^\ast\Big) \\
	&= \alpha(C_\mathfrak{st} C_\mathfrak{u} \otimes C_\mathfrak{z}^\ast) \\
	&= C_\mathfrak{st}\cdot\alpha(C_\mathfrak{u}\otimes C_\mathfrak{z}^\ast) \\
	&= C_\mathfrak{st} C_\mathfrak{uz} \\
	&= \phi(C_\mathfrak{t}^\ast \otimes C_\mathfrak{u}) C_\mathfrak{sz}
\end{align*}
Also ist $r_\mathfrak{v}=0$ für $\mathfrak{v}\neq\mathfrak{s}$ und $r_\mathfrak{s}=\phi(C_\mathfrak{t}^\ast \otimes C_\mathfrak{u})$. Eingesetzt in die ursprüngliche Gleichung heißt das:
\[C_\mathfrak{st} \cdot C_\mathfrak{u} = \phi(C_\mathfrak{t}^\ast,C_\mathfrak{u}) C_\mathfrak{s} = C_\mathfrak{s} \phi(C_\mathfrak{t}^\ast,C_\mathfrak{u})\]
Das heißt unsere gewünschte Gleichung
\[\alpha(x\otimes y^\ast)z = x\phi(y^\ast\otimes z)\]
ist für alle Basiselemente richtig. Weil beide Seiten $R$-trilinear sind, folgt daraus die allgemeine Gültigkeit. Die andere Gleichung folgt durch Anwenden von $\ast$.
\end{proof}

\begin{remark}
Der folgende Satz ist eine Verallgemeinerung von \citep[4.1]{KoenigXi_cellalg_1} von Körpern auf allgemeine Koeffizientenringe. Es werden dabei erneut Ideen aus \cite{GrahamLehrer} in den abstrakten Kontext von König und Xi gesetzt.
\end{remark}

\begin{theorem}[Multiplikation mit Zellidealen]\label{CellAlg:Mult_cellideals}
\index{terms}{Zellideal}
\index{symbols}{a@$\mathfrak{a}$}\index{symbols}{az@$\mathfrak{a}_z$}
Seien $A,\ast,J,\Delta,\alpha: \Delta\otimes_R \Delta^\ast\to J$ und $\phi:\Delta^\ast\otimes_A \Delta\to R$ wie oben. Dann gilt:
\begin{enumerate}
	\item Für alle $z\in\Delta$ ist $\mathfrak{a}_z:=\Set{\phi(y,z) | y\in\Delta^\ast}$ ein Ideal von $R$. Es gilt:
	\[Jz=\Delta\mathfrak{a}_z\]
	Insbesondere ist $Jz=\Delta$, falls $\mathfrak{a}_z=R$ ist.
	\item $\mathfrak{a}:=\sum_{z\in\Delta} \mathfrak{a}_z=\im(\phi)$ ist ein endlich erzeugtes Ideal von $R$ und es gilt $J\Delta=\mathfrak{a}\Delta$ und $J^2=\mathfrak{a}J$. Insbesondere ist $\mathfrak{a}$ unabhängig von der Wahl von $\Delta$ und $\alpha$.
\end{enumerate}
Ist speziell $R=K$ ein Körper, so gilt weiter:
\begin{enumerate}[resume]
	\item $J^2=\begin{cases} 0 &\text{falls}\,\phi=0 \\ J &\text{falls}\,\phi\neq0\end{cases}$.
	\item Im Fall $J^2=J$ gibt es außerdem ein primitives Idempotent $e\in J$ mit
	\begin{enumerate}
		\item $Ae \isomorphic \Delta$ und $eA\isomorphic \Delta^\ast$ als $A$"~Links- bzw. $A$"~Rechtsmoduln.
		\item $J=AeA$ und $eAe=Ke$.
		\item Die Multiplikation von $A$ ist ein Isomorphismus $Ae\otimes_K eA\to J$ von $A$"~$A$"~Bi\-mo\-duln.
	\end{enumerate}
\end{enumerate}
\end{theorem}
\begin{proof}
Dass $\mathfrak{a}_z$ und $\mathfrak{a}$ Ideale sind, ist klar. In a. folgt die behauptete Gleichung aus ${J=\im(\alpha)}$ und
\[\forall x,y\in\Delta: \alpha(x\otimes y^\ast)z=x\phi(y^\ast\otimes z)\]
Dieselbe Gleichung mit variablem $z\in\Delta$ zeigt $J\Delta=\mathfrak{a}\Delta$.

\medbreak
In b. folgt die zweite Gleichung aus
\[C_\mathfrak{st} C_\mathfrak{uv} = \phi(C_\mathfrak{t}^\ast \otimes C_\mathfrak{u}) C_\mathfrak{sv}\]
Denn daraus ergibt sich direkt $J^2 \subseteq \mathfrak{a}J$. Ist andererseits $a=\sum_{\mathfrak{t},\mathfrak{u}} r_\mathfrak{tu} \phi(C_\mathfrak{t}^\ast \otimes C_\mathfrak{u})$, so ist $aC_\mathfrak{sv} = \sum_{\mathfrak{t},\mathfrak{u}} r_\mathfrak{ru} C_\mathfrak{st} C_\mathfrak{uv}\in J^2$ und daher $\mathfrak{a}J\subseteq J^2$.

\medbreak
c. ist klar aufgrund von b. und $\mathfrak{a}\neq 0\implies \mathfrak{a}=K$, da $K$ ein Körper ist.

\medbreak
d.i. und d.ii. zeigen wir gemeinsam:

Ist $\phi\neq 0$, so gibt es $\mathfrak{u},\mathfrak{v}$ mit $k:=\phi(C_\mathfrak{v}^\ast \otimes C_\mathfrak{u})\neq 0$. Es folgt:
\[C_\mathfrak{uv} C_\mathfrak{uv} = \phi(C_\mathfrak{v}^\ast \otimes C_\mathfrak{u}) C_\mathfrak{uv} = kC_\mathfrak{uv}\]
Daher ist $e:=k^{-1} C_\mathfrak{uv}$ das gesuchte Idempotent in $J$. Wir behaupten, dass $Ae = \sum_\mathfrak{s} KC_\mathfrak{sv}$ ist. Es ist $Ae\subseteq \sum_\mathfrak{s} KC_\mathfrak{sv}$, weil rechter Hand ein Linksideal steht und $e$ in diesem enthalten ist. Umgekehrt ist wegen $C_\mathfrak{sv} e= k^{-1}C_\mathfrak{sv} C_\mathfrak{uv} = k^{-1}\phi(C_\mathfrak{v}^\ast \otimes C_\mathfrak{u}) C_\mathfrak{sv}=C_\mathfrak{sv}$ jedes der Basiselemente $C_\mathfrak{sv} \in Ae$. Das zeigt die Gleichheit $Ae=\sum_\mathfrak{s} KC_\mathfrak{sv}$ und das ist zu $\Delta$ isomorph. Die zweite Gleichheit folgt analog.

\medbreak
$AeA\subseteq J$ und $Ke\subseteq eAe$ sind klar. Da für alle $\mathfrak{s},\mathfrak{y}$ ja
\[C_\mathfrak{sv} e C_\mathfrak{uy} = k^{-1} C_\mathfrak{sv} C_\mathfrak{uv} C_\mathfrak{uy} = k^{-1} \phi(C_\mathfrak{v}^\ast \otimes C_\mathfrak{u}) \phi(C_\mathfrak{v}^\ast \otimes C_\mathfrak{u}) C_\mathfrak{sy} = kC_\mathfrak{sy}\]
und für alle $a\in A$
\[eae = k^{-2}C_\mathfrak{uv} a C_\mathfrak{uv} \in KC_\mathfrak{uv} = Ke\]
ist, gelten auch die umgekehrten Inklusionen. Insbesondere ist $K=eAe=\End_A(Ae)$ lokal, also $Ae$ unzerlegbar, also $e$ primitiv.

\medbreak
d.iii. Die Multiplikation liefert einen $A$"~$A$"~Bimodulepimorphismus $Ae\otimes_K eA\twoheadrightarrow AeA=J$. Da $\dim_K Ae=\dim_K \Delta=\dim_K \Delta^\ast = \dim_K eA$, sind die $K$-Dimensionen auf beiden Seiten gleich $(\dim_K \Delta)^2$, d.\,h. der Homomorphismus ist auch injektiv.
\end{proof}
\section{Zellmoduln}

\begin{remark}
Die Darstellungstheorie zellulärer Algebren ist unter anderem deshalb so gut kontrollierbar, weil aus einer Zellbasis bereits alle einfachen $A$"~Moduln explizit abgelesen werden können. In diesem Abschnitt wollen wir zusammenfassen, wie das geht, und fixieren dafür eine zelluläre $R$"~Algebra $A$ mit Zelldatum $(\Lambda,M,C,\ast)$.
\end{remark}

\begin{definition}[Zellmoduln, siehe {\citep[2.1]{GrahamLehrer}}]
\index{terms}{Zellmodul}
\index{symbols}{Wlambda@$W(\lambda)$}
Definiere für alle ${\lambda\in\Lambda}$ den \udot{(Links-)Zellmodul} $W(\lambda)$ als den freien $R$"~Modul über der Basis $(C_\mathfrak{s})_{\mathfrak{s}\in M(\lambda)}$ und der Multiplikation
\[aC_\mathfrak{s} = \sum_{\mathfrak{u}\in M(\lambda)} r_a(\mathfrak{u},\mathfrak{s})C_\mathfrak{u}\]
wobei $r_a(\mathfrak{u},\mathfrak{s})$ die Koeffizienten aus der Definition \ref{CellAlg:def} bezeichne. $W(\lambda)$ wird auf diese Weise ein $A$"~$R$"~Bimodul.

Der entsprechende Rechtsmodul ist $W(\lambda)^\ast$ mit derselben Basis, aber der Rechtsmultiplikation
\[C_\mathfrak{s}^\ast a := \sum_{\mathfrak{u}\in M(\lambda)} C_\mathfrak{u}^\ast r_{a^\ast}(\mathfrak{u},\mathfrak{s})\]
$W(\lambda)^\ast$ ist damit ein $R$"~$A$"~Bimodul.
\end{definition}

\begin{example}[Polynomringe]
Wir betrachten wieder $A=R[X]/(X^n)$. Für alle $\lambda\in\Set{0,\ldots,n-1}$ ist
\[W(\lambda) = (X^\lambda)/(X^{\lambda+1}) \isomorphic R[X]/(X).\]
Es gibt nur eine Operation auf diesem Modul, nämlich $X\cdot W(\lambda)=0$.
\end{example}

\begin{example}[Matrizenringe]
Ist $A=R^{d\times d}$, dann ist $W(\lambda) = R^{d\times 1}$, wobei $A$ durch Linksmultiplikation auf $W$ operiert.
\end{example}

\begin{example}[Hecke-Algebren, siehe {\citep[2.6.12]{geckjacon}}]\label{CellAlg:ex:CellModules_Hecke}
\index{terms}{Hecke-Algebra}\index{terms}{Kazhdan-Lusztig!-Basis}
Sind wir in der Situation von \ref{CellAlg:example_hecke}, dann ist abstrakt $W(\lambda)$ der irreduzible $KH$"~Modul vom Isomorphietyp $\lambda$. Ganz konkret ist jedoch die Operation gegeben durch (mit den Bezeichnungen von \ref{CellAlg:example_hecke}) die Matrixdarstellung
\[C_w \mapsto \sum_{\substack{z\in W, d\in\mathcal{D} \\ z \sim_\mathcal{LR} d}} n_d h_{w,d,z} \overline{\rho}_\lambda(t_z) = \overline{\rho}_\lambda\circ\phi(C_w)\]
d.\,h. die Zellmoduldarstellungen $KH\to K^{d_\lambda\times d_\lambda}$ faktorisieren durch den Lusztig-Iso\-mor\-phis\-mus $\phi$.
\end{example}

\begin{remark}\label{CellAlg:Hecke_cell_mod}
Aus dieser Darstellung der Zellmoduln von Hecke"=Algebren ergibt sich eine unmittelbare Folgerung:

Die irreduziblen Matrixdarstellungen $\rho: KH\to K^{d\times d}$, die als Zellmoduln bzgl. einer Zellbasis nach Geck auftreten, sind genau diejenigen balancierten Darstellungen, die als ${\rho=\overline{\rho}\circ\phi}$ faktorisieren. Insbesondere bilden sie eine vollständige $GL_d(F)$"~Kon\-ju\-ga\-tions\-klas\-se (siehe \ref{h_vs_j:conj_class}).
\end{remark}

\begin{remark}
Diese Beobachtung liefert uns eine Möglichkeit, algorithmisch zu überprüfen, ob eine gegebene Matrixdarstellung von $KH$ als Zellmodul auftritt: Zuerst wird getestet, ob es sich um eine balancierte Darstellung handelt und wenn dies der Fall ist, wird $\rho(C_s) = \overline{\rho}(\phi(C_s))$ getestet für alle $s\in S$. Siehe Algorithmus \ref{algo:cell_rep} für eine Pseudocodeimplementierung dieser Idee.
\end{remark}
\begin{remark}
\index{terms}{Zellmodul}
Die Definition sorgt dafür, dass $A(\Set{\lambda})=A(\leq\lambda)/A(<\lambda)$ geschrieben werden kann als
\[\bigoplus_{\mathfrak{t}\in M(\lambda)} \bigg(\overline{\sum_{\mathfrak{s}\in M(\lambda)} RC_\mathfrak{st}^\lambda}\bigg) \isomorphic \bigoplus_{\mathfrak{t}\in M(\lambda)} W(\lambda) \quad \text{als $A$"~Linksmoduln}\]
bzw.
\[\bigoplus_{\mathfrak{s}\in M(\lambda)} \bigg(\overline{\sum_{\mathfrak{t}\in M(\lambda)} RC_\mathfrak{st}^\lambda}\bigg) \isomorphic \bigoplus_{\mathfrak{s}\in M(\lambda)} W(\lambda)^\ast \quad \text{als $A$"~Rechtsmoduln}.\]

Wir werden die kanonischen Einbettungen
\[W(\lambda)\hookrightarrow A(\Set{\lambda}), C_\mathfrak{s}\mapsto C_\mathfrak{st}^\lambda\]
für ein festes $\mathfrak{t}$ bzw.
\[W(\lambda)^\ast\hookrightarrow A(\Set{\lambda}), C_\mathfrak{t}^\ast\mapsto C_\mathfrak{st}^\lambda\]
für ein festes $\mathfrak{s}$ von Zeit zu Zeit benutzen.
\end{remark}
\begin{remark}
\index{terms}{Zellmodul}\index{terms}{Zellideal}\index{terms}{Invariante Bilinearform}
\index{symbols}{philambda@$\phi_\lambda$}
Man beachte, dass $A(\Set{\lambda})$ ein Zellideal von $A/A(<\lambda)$ im Sinne von \ref{CellAlg:def:Cellideals} ist: $W(\lambda)$ kann als der Linksmodul $\Delta$ gewählt werden, der Isomorphismus $\alpha$ ist in Übereinstimmung mit obigen Bezeichnungen durch $C_\mathfrak{s}^\ast\otimes C_\mathfrak{t}\mapsto C_\mathfrak{st}^\lambda$ gegeben (Aus (C3) und (C$3^\ast$) folgt, dass dies wirklich ein $A$"~$A$"~Bimodulisomorphismus ist).
	
Damit sind alle obigen Überlegungen auf die Zellmoduln anwendbar. Insbesondere gibt es zu jedem $\lambda\in\Lambda$ eine invariante Bilinearform $\phi_\lambda: W(\lambda)^\ast\otimes_A W(\lambda)\to R$, die durch
\[C_\mathfrak{st}^\lambda C_\mathfrak{uv}^\lambda \equiv \phi_\lambda(C_\mathfrak{t}^\ast \otimes C_\mathfrak{u})C_\mathfrak{sv}^\lambda \mod A(<\lambda)\]
gegeben ist. Es gilt dabei nach (C3) und (C$3^\ast$):
\[\phi_\lambda(C_\mathfrak{t}^\ast \otimes C_\mathfrak{u}) = r_{C_\mathfrak{st}^\lambda}(\mathfrak{s},\mathfrak{u}) = r_{(C_\mathfrak{uv}^\lambda)^\ast}(\mathfrak{v},\mathfrak{t})\]

Unmittelbare Folgerungen aus dieser Erkenntnis sind in folgendem Lemma zu\-sam\-men\-ge\-fasst:
\end{remark}

\begin{lemma}[Annullatoren, siehe {\cite{GrahamLehrer}}]
\index{terms}{Zellmodul}
Seien $\lambda,\mu\in\Lambda$, $\mathfrak{s},\mathfrak{t}\in M(\lambda)$, $\mathfrak{u},\mathfrak{v}\in M(\mu)$ und $a\in A$ beliebig.
\begin{enumerate}
	\item Für die Multiplikation gilt:
	\[C_\mathfrak{st}^\lambda a C_\mathfrak{uv}^\mu \in A(\Set{\leq\lambda} \cap \Set{\leq\mu}) \subseteq \begin{cases} A(\leq\mu) & \mu\leq\lambda \\ A(<\mu)\cap A(<\lambda) & \mu,\lambda\,\text{unvergleichbar} \\ A(\leq\lambda) & \mu\geq\lambda \end{cases}\]
	\item Für die Annullatoren der Zellmoduln gilt:
	\[A(\not\geq\lambda) \subseteq \operatorname{Ann}(W(\lambda))\]
	\[A(\not\geq\lambda) \subseteq \operatorname{Ann}(W(\lambda)^\ast)\]
\end{enumerate}
\end{lemma}
\begin{proof}
a. Wir erhalten die Aussage durch Auswerten auf zwei verschiedene Weisen:
\[(C_\mathfrak{st}^\lambda a) C_\mathfrak{uv}^\mu = \sum_{\substack{\kappa\leq\lambda \\ \mathfrak{p},\mathfrak{q}\in M(\kappa)}} x_{a\kappa\mathfrak{p}\mathfrak{q}} C_{\mathfrak{pq}}^{\kappa}\]
\[C_\mathfrak{st}^\lambda (a C_\mathfrak{uv}^\mu) = \sum_{\substack{\kappa\leq\mu \\ \mathfrak{p},\mathfrak{q}\in M(\kappa)}} x_{a\kappa\mathfrak{p}\mathfrak{q}} C_{\mathfrak{pq}}^{\kappa}\]
Die einzigen Summanden, die auf beiden Seiten vorkommen, sind diejenigen, bei denen $\kappa\leq\lambda \wedge \kappa\leq\mu$ ist.

\medbreak
b. folgt sofort aus der Zerlegung $A(\leq\lambda)/A(<\lambda) \isomorphic \bigoplus_\mathfrak{t} W(\lambda)$ bzw. $\isomorphic \bigoplus_\mathfrak{s} W(\lambda)^\ast$.
\end{proof}

\begin{remark}
Wir haben in \ref{KL:def:cells} bereits eine Definition von Zellmoduln kennengelernt. Wir stellen nun fest, dass die beiden Definitionen dieses Begriffs in wichtigen Fällen übereinstimmen:
\end{remark}
\begin{lemma}[Zellmoduln sind Zellmoduln]\label{CellAlg:Cellmodules_are_Cells}
\index{terms}{Zellmodul}\index{terms}{Zellmodul einer Zelle}
Sei $R$ ein Körper. Betrachte dann $(A,(C_\mathfrak{st}^\lambda))$ als $A$-Links- bzw. als $A$"~$A$"~Bimodul-mit-Basis. Für die Quasiordnungen $\preceq_\mathcal{L}$ bzw. $\preceq_\mathcal{LR}$ nach \ref{KL:def:cells} gilt:
\begin{enumerate}
	\item Die partielle Ordnung $\leq$ auf $\Lambda$ ist eine Verfeinerung der Ordnung $\preceq_\mathcal{LR}$ auf der Zellbasis: $C_\mathfrak{st}^\lambda\preceq_\mathcal{LR} C_\mathfrak{uv}^\mu \implies \lambda\leq\mu$.
	\item Ist $\phi_\lambda$ nichtentartet, so ist $(W(\lambda),(C_\mathfrak{s}))$ als Modul-mit-Basis identisch mit einem Zellmodul des Linksmoduls und $(A(\Set{\lambda}),C_\mathfrak{st}^\lambda)$ ist ein Zellmodul des Bimoduls. 
\end{enumerate}
\end{lemma}
\begin{proof}
a. folgt direkt aus der Definition der Zellularität und der Ordung $\preceq_\mathcal{LR}$.

\medbreak
b. Fixiere ein Basiselement $C_\mathfrak{ab}^\lambda$ der Zellbasis.

Die Definition einer Zellbasis liefert uns nun
\[AC_\mathfrak{ab}^\lambda \subseteq \sum_\mathfrak{s} RC_\mathfrak{sb}^\lambda + A(<\lambda).\]
Alle Basiselemente $C_\mathfrak{uv}^\mu$, die $\preceq_L$-kleiner als $C_\mathfrak{ab}^\lambda$ sind, müssen $\mu\leq\lambda$ erfüllen. Diejenigen, deren oberer Index gleich $\lambda$ ist, müssen weiter den rechten unteren Index gleich $\mathfrak{b}$ haben. Und diejenigen, deren oberer Index echt kleiner als $\lambda$ ist, können nicht $\preceq_\mathcal{L}$-größer als $C_\mathfrak{ab}^\lambda$ sein, weil $A(<\lambda)$ ein Ideal ist, d.\,h. sie können nicht in derselben Zelle liegen. Also muss die Zelle von $C_\mathfrak{ab}^\lambda$ in $\Set{C_\mathfrak{sb}^\lambda | \mathfrak{s}\in M(\lambda)}$ enthalten sein.

Nun ist aber $\phi_\lambda$ nichtentartet. Daher gibt es zu jedem $\mathfrak{u}\in M(\lambda)$ ein $\mathfrak{t}\in M(\lambda)$ mit $\phi_\lambda(C_\mathfrak{t}^\ast, C_\mathfrak{u})\neq 0$. Wegen
\[C_\mathfrak{st}^\lambda C_\mathfrak{uv}^\lambda \in \phi_\lambda(C_\mathfrak{t}^\ast \otimes C_\mathfrak{u}) C_\mathfrak{sv}^\lambda + A(<\lambda)\]
folgt daraus $C_\mathfrak{sv}^\lambda \preceq_\mathcal{L} C_\mathfrak{uv}^\lambda$ für alle $\mathfrak{s}\in M(\lambda)$. Da $\mathfrak{u}$ beliebig war, ist $\Set{C_\mathfrak{sb}^\lambda | \mathfrak{s}\in M(\lambda)}$ bereits eine volle Zelle. Das analoge Argument zeigt dieselbe Behauptung auch für zweiseitige Zellen und Bimoduln.
\end{proof}

\begin{remark}
\index{symbols}{Glambda@$G^\lambda$}
Eigentlich brauchen wir gar nicht, dass $\phi_\lambda$ nichtentartet ist, sondern nur, dass die Darstellungsmatrix $G^\lambda:=(\phi_\lambda(C_\mathfrak{t}^\ast \otimes C_\mathfrak{u}))_{\mathfrak{tu}}$ keine Nullzeilen oder -spalten hat. Im generischen Fall reicht also auch $\phi_\lambda\neq 0$ als Voraussetzung.
\end{remark}
\begin{remark}
Um die Klassifikation der einfachen Moduln von zellulären Algebren vorzubereiten, untersuchen wir genauer, was uns das Zusammenspiel von $\phi_\lambda: W(\lambda)^\ast\otimes_A W(\lambda)\to R$ und $\alpha_\lambda: W(\lambda)\otimes_R W(\lambda)^\ast\to A(\Set{\lambda})$ sagt:
\end{remark}

\begin{theorem}[Homomorphismen zwischen Zellmoduln, siehe {\citep[2.6]{GrahamLehrer}}]\label{CellAlg:Hom_cellmodules}
\index{terms}{Zellmodul}
Seien $\lambda,\mu\in\Lambda$, $U\leq W(\mu)$ ein $A$-Untermodul und $f: W(\lambda)\to W(\mu)/U$ ein Homomorphismus von $A$"~Linksmoduln. Sei weiter $\phi_\lambda\neq 0$. Dann gilt:
\begin{enumerate}
	\item Ist $W(\mu)/U$ torsionsfrei über $R$ und $f\neq 0$, so folgt $\mu\leq\lambda$
	\item Falls $\lambda=\mu$ ist, gibt es $r_0\in R\setminus\Set{0}$ und $r_1\in R$ so, dass
	\[\forall x\in W(\lambda): f(x)r_0=xr_1+U\]
	\item Ist $R$ ein Integritätsbereich und $W(\lambda)/U$ frei über $R$, so ist $\End_A(W(\lambda)/U)=R$. Insbesondere ist $\End_A(W(\lambda))=R$.
\end{enumerate}
\end{theorem}
\begin{proof}
a. Wenn $\mu\not\leq\lambda$ wäre, wäre auch $C_\mathfrak{st}^\lambda \cdot W(\mu) = 0$ für alle $\mathfrak{s},\mathfrak{t}\in M(\lambda)$. Also wäre dann auch $f(C_\mathfrak{st}^\lambda W(\lambda)) = C_\mathfrak{st}^\lambda f(W(\lambda)) = 0$, da $\im(f)\subseteq W(\mu)/U$ von allen $C_\mathfrak{st}^\lambda$ annulliert wird.

Nun gibt es jedoch wegen $\phi_\lambda\neq0$ Elemente $y\in W(\lambda)^\ast, z\in W(\lambda)$ so, dass $\phi_\lambda(y \otimes z)\neq 0$. Wenn wir ein $x\in W(\lambda)$ mit $f(x)\neq 0$ dazu nehmen und die Torsionsfreiheit ausnutzen, erhalten wir $0\neq f(x)\phi_\lambda(y \otimes z) = f(x\phi_\lambda(y \otimes z)) = f(\alpha_\lambda(x\otimes y) z)= 0$, da $\alpha_\lambda(x\otimes y)$ ja eine Linearkombination der $C_\mathfrak{st}^\lambda$ ist.

\medbreak
b. Wähle $y\in W(\lambda)^\ast, z,z'\in W(\lambda)$ derart, dass $r_0:=\phi_\lambda(y \otimes z)\neq 0$ und $f(z)=z'+U$. Dann gilt für alle $x\in W(\lambda)$:
\begin{align*}
	f(x)r_0 &= f(x)\phi_\lambda(y \otimes z) \\
	&= f(x\phi_\lambda(y \otimes z)) \\
	&= f(\alpha_\lambda(x\otimes y)z) \\
	&= \alpha_\lambda(x\otimes y)f(z) \\
	&= \alpha_\lambda(x\otimes y)z' + U \\
	&= x\phi_\lambda(y \otimes z') + U
\end{align*}
d.\,h. $r_1:=\phi_\lambda(y \otimes z')$ leistet das Gewünschte.

\medbreak
c. Ist $R$ ein Integritätsbereich, so ist jeder freie $R$"~Modul auch torsionsfrei. Wir betrachten einen beliebigen Homomorphismus $f: W(\lambda)\twoheadrightarrow W(\lambda)/U\to W(\lambda)/U$ und wählen $r_0$ und $r_1$ so, dass wie in b.
\[\forall x\in W(\lambda): f(x)r_0 = xr_1 + U\]
gilt. Wähle eine Basis $(b_i)_{i\in I}$ von $W(\lambda)/U$. Es gibt dann Koeffizienten $c_{ij}\in R$ so, dass
\[f(b_i) = \sum_j c_{ij} b_j \implies f(b_i)r_0 = \sum_j c_{ij} r_0 b_j \]
gilt. Koeffizientenvergleich liefert $c_{ij}r_0 = 0$ für alle $i\neq j$ und $c_{ii}r_0 = r_1$. Da $R$ ein Integritätsbereich ist, folgt daraus, dass $c:=c_{ii}$ unabhängig von $i$ und $f(b_i)=cb_i$ für alle $i$ ist, d.\,h. $f(x)=cx+U$ für alle $x\in W(\lambda)$. Daher ist jeder Endomorphismus von $W(\lambda)/U$ eine Multiplikation mit Elementen aus $R$. $c\mapsto c\cdot id$ ist also ein Isomorphismus $R \to \End_A(W(\lambda)/U)$ (Injektivität benutzt erneut, dass $R$ ein Integritätsring und der $R$"~Modul frei ist).
\end{proof}

\subsection{Flache Moduln}

\begin{remark}
Folgender Satz ist in \cite{GrahamLehrer} bereits formuliert, dort aber unnötigerweise für projektive Moduln, obwohl nur Flachheit benutzt wird:
\end{remark}

\begin{theorem}[Zellketten in flachen Moduln, siehe {\cite{GrahamLehrer}}]\label{CellAlg:flat_modules}
\index{symbols}{Plambda@$P(\Set{\lambda})$}
Sei $P$ ein flacher $A$"~Modul. Für alle Ideale $\Phi\subseteq\Lambda$ definiere
\[P(\Phi):=A(\Phi)P\]
Dann gilt für alle Ideale $\Phi'\subseteq\Phi\subseteq\Lambda$:
\begin{enumerate}
	\item $A(\Phi)\otimes_A P \isomorphic P(\Phi)$ via $a\otimes p\mapsto ap$.
	\item $P(\Phi)/P(\Phi') \isomorphic A(\Phi\setminus\Phi') \otimes_A P$.
	\item $P$ hat eine Filtrierung
	\[0=P_0 \leq P_1 \leq \ldots \leq P_n=P\]
	derart, dass $P_i/P_{i-1} \isomorphic A(\Set{\lambda_i})\otimes_A P$ ist für eine beliebige topologische Sortierung $\lambda_1, \ldots, \lambda_n$ von $\Lambda$.
\end{enumerate}
\end{theorem}
\begin{proof}
Es ist
\[0\to A(\Phi') \to A(\Phi) \to A(\Phi\setminus\Phi') \to 0\]
exakt, also auch
\[0\to A(\Phi') \otimes_A P \to A(\Phi)\otimes_A P \to A(\Phi\setminus\Phi') \otimes_A P \to 0.\]
Setzt man $\Phi=\Lambda$ ein, erhält man, dass $A(\Phi')\otimes P$ via der induzierten Abbildung zu einem Untermodul von $A\otimes P$ isomorph ist. Zusammen mit dem Isomorphismus ${A\otimes P \to P}, {a\otimes p\mapsto ap}$ zeigt das, dass $A(\Phi')\otimes P \isomorphic A(\Phi')P$ ist. Das zeigt a. Dieselbe kurze, exakte Sequenz zeigt die Aussage über den Quotienten in b.

c. folgt aus b. nach Wahl einer topologischen Sortierung $\lambda_1\leq\ldots\leq\lambda_n$ von $\Lambda$ und $P_i:=P(\Phi_i)$ mit $\Phi_i:=\Set{\lambda_j | j\leq i}$.
\end{proof}

\begin{definition}\label{CellAlg:def:P_lambda}
\index{symbols}{Plambda@$P^\lambda$}
Sei $P$ ein beliebiger $A$"~Modul. Definiere den $R$"~Modul $P^\lambda$ dann durch
\[P^\lambda:=W(\lambda)^\ast\otimes_A P\]
\end{definition}

\begin{lemma}[Siehe {\citep[2.10]{GrahamLehrer}}]\label{CellAlg:P_lambda}
In obiger Situation gilt:
\begin{enumerate}
	\item $A(\Set{\lambda}) \otimes_A P \isomorphic W(\lambda) \otimes_R P^\lambda$ als $A$"~Moduln.
	\item Ist $R$ ein Integritätsbereich und $\phi_\lambda\neq 0$, so gilt $\Hom_A(A(\Set{\lambda})\otimes_A P,W(\lambda)) \isomorphic \Hom_R(P^\lambda,R)$ als $R$"~Moduln.
\end{enumerate}
\end{lemma}
\begin{proof}
Beides ist nur Rechnen mittels der schon bewiesenen Isomorphien:
\begin{align*}
	A(\Set{\lambda})\otimes_A P &\isomorphic (W(\lambda)\otimes_R W(\lambda)^\ast) \otimes_A P \\
	&= W(\lambda) \otimes_R (W(\lambda)^\ast \otimes_A P) \\
	&= W(\lambda) \otimes_R P^\lambda
\end{align*}
\begin{align*}
	\Hom_A(A(\Set{\lambda})\otimes_A P,W(\lambda)) &\isomorphic \Hom_A(W(\lambda)\otimes_R P^\lambda,W(\lambda)) \\
	&\isomorphic \Hom_R(P^\lambda,\Hom_A(W(\lambda),W(\lambda))) \\
	&\isomorphic \Hom_R(P^\lambda,R) \qedhere
\end{align*}
\end{proof}

\subsection{Einfache Moduln}

\begin{remark}
Mit Hilfe der invarianten Bilinearformen kann man nun explizite Konstruktionen von einfachen Moduln aus den Zellmoduln angeben:
\end{remark}

\begin{theorem}[Einfache Moduln in zellulären Algebren, siehe {\citep[Ch.\,3]{GrahamLehrer}}]\label{CellAlg:SimpleMods}
\index{terms}{Zellmodul}\index{terms}{Invariante Bilinearform!Radikal}
\index{symbols}{Lambda0@$\Lambda_0$}\index{symbols}{Lplambda@$L_\mathfrak{p}(\lambda)$}
Sei $R=K$ ein Körper.

Definiere dann $\Lambda_0 := \Set{\lambda\in\Lambda | \phi_\lambda\neq 0}$ und $L_0(\lambda) := W(\lambda) / \rad(\phi_\lambda)$. Dabei ist das Radikal der Bilinearform wie üblich $\rad(\phi_\lambda) := \Set{x\in W(\lambda) | \forall y\in W(\lambda)^\ast: \phi_\lambda(y \otimes x)=0}$.

\begin{enumerate}
	\item Für alle $\lambda\in\Lambda_0$ ist $\rad(\phi_\lambda)$ ein $A$-Untermodul von $W(\lambda)$ und stimmt mit $\rad(W(\lambda))$ überein.
	\item Für alle $\lambda\in\Lambda_0$ ist $L_0(\lambda)$ absolut irreduzibel.
	\item Für alle $\lambda,\mu\in\Lambda_0$ mit $\lambda\neq\mu$ ist $L_0(\lambda) \not\isomorphic L_0(\mu)$.
	\item Jeder irreduzible $A$"~Modul ist zu einem der $L_0(\lambda)$ isomorph.
\end{enumerate}
\end{theorem}
\begin{proof}
a. und b. Dass $\rad(\phi_\lambda)$ ein $A$-Untermodul ist, folgt daraus, dass $\phi_\lambda$ $A$"~invariant ist. Ist $z\notin\rad(\phi_\lambda)$, so gibt es ein $y\in W(\alpha)^\ast$ mit $\phi_\lambda(y \otimes z)\neq 0$, d.\,h. das in \ref{CellAlg:Mult_cellideals} definierte Ideal ${\mathfrak{a}_z=\Set{\phi_\lambda(y \otimes z) | y\in W(\lambda)^\ast}}$ ist ungleich $0$ und somit gleich $K$, da $K$ ein Körper ist. Also ist
\[W(\lambda) = W(\lambda)\mathfrak{a}_z \overset{\text{\ref{CellAlg:Mult_cellideals}}}{=} A(\Set{\lambda})\cdot z \subseteq A\cdot z\]
Das zeigt, dass die einzigen echten Untermoduln in $\rad(\phi_\lambda)$ enthalten sein müssen. Daher ist $\rad(\phi_\lambda)$ der größte echte Untermodul und somit gleich $\rad(W(\lambda))$. Das zeigt außerdem, dass $L_0(\lambda)$ einfach ist. \ref{CellAlg:Hom_cellmodules} zeigt nun, dass $\End_A(L_0(\lambda))=K$ ist, also ist $L_0(\lambda)$ absolut irreduzibel.

\medbreak
c. Angenommen, $L_0(\lambda) \xrightarrow{f} L_0(\mu)$ wäre ein Isomorphismus. Aus \ref{CellAlg:Hom_cellmodules} und $f\neq 0$ folgt $\lambda\geq\mu$. Aus $f^{-1}\neq 0$ folgt genauso $\mu\geq\lambda$. Daraus folgt die Behauptung.

\medbreak
d. Wähle ein primitives Idempotent $e\in A$ so, dass $P:=Ae$ eine projektive Überdeckung des einfachen $A$"~Moduls $L$ ist. Dann gilt aufgrund \ref{CellAlg:flat_modules} einerseits
\[0\neq P/P(<\lambda) = P(\leq\lambda)/P(<\lambda) \isomorphic A(\Set{\lambda})\otimes_A P\]
für ein geeignetes $\lambda$ und andererseits
\[A(\Set{\lambda})\otimes_A P \isomorphic A(\Set{\lambda}) \otimes_A Ae \isomorphic A(\Set{\lambda})e\]
als $A$"~Linksmoduln. Nun ist $\bigoplus_\mathfrak{t} W(\lambda)\isomorphic A(\Set{\lambda}) = A(\Set{\lambda})e \oplus A(\Set{\lambda})(1-e)$, d.\,h. wir erhalten einen surjektiven Homomorphismus
\[\bigoplus_\mathfrak{t} W(\lambda) \twoheadrightarrow A(\Set{\lambda})e \xrightarrow{\isomorphic} P/P(<\lambda) \twoheadrightarrow P/\rad(P)=L\]
Also ist $L$ ein einfacher Quotient von $\bigoplus_\mathfrak{t} W(\lambda)$, d.\,h. ein Quotient von $\bigoplus_\mathfrak{t} L_0(\lambda)$, und daher zu $L_0(\lambda)$ isomorph.
\end{proof}

\begin{remark}
\index{terms}{Zerfällungskörper}\index{terms}{Invariante Bilinearform!Radikal}
Man erhält als unmittelbares Korollar: $K$ ist ein Zerfällungskörper für $A$, d.\,h. alle irreduziblen $A$"~Moduln sind absolut irreduzibel. Das wird auch klar, wenn man sich alle durchgeführten Konstruktionen vor Augen hält: Die Zellbasis und ihre Strukturkonstanten ändern sich nicht bei Skalarerweiterung, also sind auch die Darstellungsmatrizen aller invarianten Bilinearformen sowie die der Operationen auf allen Zellmoduln gleich. Somit bleiben die Dimensionen der Radikale der Bilinearformen gleich, wenn man zu Erweiterungskörpern übergeht. Die Größe dieser Radikale bestimmt, wie eben bewiesen, ob ein Modul einfach ist oder nicht. Wenn also ein Modul nach Skalarerweiterung nichteinfach ist, muss er auch vorher schon nichteinfach gewesen sein.
\end{remark}
\begin{definition}
\index{symbols}{Glambda@$G^\lambda$}\index{symbols}{Lambdap@$\Lambda_\mathfrak{p}$, $\Lambda_\mathfrak{p}^\times$}
Wähle für alle $\lambda\in\Lambda$ eine Basis $(C_\mathfrak{s})_{\mathfrak{s}\in M(\lambda)}$ von $W(\lambda)$ und definiere $G^\lambda$ als Darstellungsmatrix von $\phi_\lambda$ bzgl. dieser Basis, d.\,h. $G_\mathfrak{st}^\lambda=\phi_\lambda(C_\mathfrak{s}^\ast \otimes C_\mathfrak{t})$.

Sei $\mathfrak{p}\in\operatorname{Spec}(R)$ ein Primideal. Definiere dann:
\[\Lambda_\mathfrak{p} := \Set{\lambda\in\Lambda | G^\lambda \not\equiv 0 \mod\mathfrak{p}}\]
\[\Lambda_\mathfrak{p}^\times := \Set{\lambda\in\Lambda | \det(G^\lambda) \not\equiv 0 \mod\mathfrak{p}}\]
\end{definition}

\begin{remark}
\index{symbols}{Lambda0@$\Lambda_0$}
Man macht sich leicht klar, dass $\Lambda_\mathfrak{p}$ und $\Lambda_\mathfrak{p}^\times$ nicht von der Wahl der Basis abhängen.

Mit diesen Bezeichnungen entspricht das Radikal von $\phi_\lambda$ dem Kern von $G^\lambda$. Insbesondere ist $\Lambda_{(0)}$ identisch mit $\Lambda_0$ wie in Satz \ref{CellAlg:SimpleMods} definiert.

Hat man einen allgemeinen Ring $R$ statt eines Körpers, so liefert der Satz also auch eine Parametrisierung aller modularer Darstellungen, indem man die Algebra modulo eines $\mathfrak{p}\in\operatorname{Spec}(R)$ reduziert und den obigen Satz anwendet.
\end{remark}

\begin{corollary}[Modulare Darstellungen]
\index{terms}{Zellmodul}\index{terms}{Zerfällungskörper}
\index{symbols}{Lplambda@$L_\mathfrak{p}(\lambda)$}
Sei $R\to K$ ein Homomorphismus in einen Körper mit Kern $\mathfrak{p}$. Dann gilt:
\begin{enumerate}
	\item Die $KA$"~Moduln $KL_\mathfrak{p}(\lambda) := KW(\lambda) / \ker(KW(\lambda)\xrightarrow{G^\lambda} KW(\lambda))$ für $\lambda\in\Lambda_\mathfrak{p}$ sind ein vollständiges Repräsentantensystem der einfachen $KA$"~Moduln.
	\item $K$ ist ein Zerfällungskörper für $KA$.
\end{enumerate}
\end{corollary}

\begin{corollary}[Halbeinfachheit zellulärer Algebren, siehe {\citep[3.8]{GrahamLehrer}}]\label{CellAlg:Semisimplicity}
\index{terms}{Satz von!Tits' Deformationssatz}\index{terms}{Halbeinfach}
Sei $R\to K$ ein Homomorphismus in einen Körper mit Kern $\mathfrak{p}$. Dann gilt:
\begin{enumerate}
	\item $KA$ ist halbeinfach $\iff \Lambda=\Lambda_\mathfrak{p}^\times$.
\end{enumerate}
Nehmen wir nun $KA$ als halbeinfach an, dann gilt weiter:
\begin{enumerate}[resume]
	\item $\abs{\Lambda}=\abs{\Lambda_\mathfrak{p}}=\abs{\Irr(KA)}$ und $\forall\lambda\in\Lambda: \abs{M(\lambda)}=\dim_K KL_p(\lambda)$.
	\item Tits' Deformationssatz für zelluläre Algebren:

	Sei $R$ ein Integritätsbereich und $Q:=\QuotFld(R)$ sein Quotientenkörper. Ist die Spezialisierung $KA$ halbeinfach, so ist auch $QA$ halbeinfach.
\end{enumerate}
\end{corollary}
\begin{proof}
\index{terms}{Satz von!Wedderburn}
Ist $KA$ halbeinfach, dann ist es zerfallend halbeinfach und aufgrund obigen Satzes sind die einfachen $KA$"~Mo\-duln durch $\Lambda_\mathfrak{p}$ parametrisiert. Es gilt wegen der Wedderburn"=Zerlegung:
\begin{align*}
	\dim_R A &= \dim_K KA \\
	&= \sum_{\lambda\in\Lambda_\mathfrak{p}} (\dim_K KL_\mathfrak{p}(\lambda))^2 \\
	&\leq \sum_{\lambda\in\Lambda} (\dim_K KW(\lambda))^2 \\
	&= \sum_{\lambda\in\Lambda} \abs{M(\lambda)}^2 \\
	&= \dim_R A
\end{align*}
Daher gilt $\Lambda=\Lambda_\mathfrak{p}$ und $\dim_K KL_\mathfrak{p}(\lambda)=\dim_K KW(\lambda)$ für alle $\lambda\in\Lambda$. Das zeigt die Behauptung b.

Weil deshalb aber auch $KW(\lambda)=KL_\mathfrak{p}(\lambda)$ gilt, muss ${\ker(G^\lambda\mod\mathfrak{p})=0}$ und somit, wie behauptet, $\lambda\in\Lambda_\mathfrak{p}^\times$ sein. Das zeigt die Richtung $\implies$ in a.

\medbreak
Für die Umkehrung sei $\Lambda=\Lambda_\mathfrak{p}^\times$. Dann ist $\ker(G^\lambda\mod\mathfrak{p})=0$, d.\,h. ${KW(\lambda)=KL_\mathfrak{p}(\lambda)}$ ist einfach für alle $\lambda$. Der Quotient $KA/\rad(KA)$ ist halbeinfach und seine Dimension ist die Summe der Quadrate der Dimensionen aller einfachen Moduln. Weil das nun aber $\sum_\lambda \abs{M(\lambda)}^2=\dim_K KA$ ist, muss $\rad(KA)=0$ und $KA$ halbeinfach sein.

\bigbreak
c. ist einfach: Wegen $0\subseteq\mathfrak{p}$ ist $\Lambda_0^\times \supseteq \Lambda_\mathfrak{p}^\times = \Lambda$. Wegen a. ist $QA$ halbeinfach.
\end{proof}

\begin{remark}
Wenn man sich also fragt, ob eine Algebra zellulär sein könnte, ist man durch dieses Korollar in der Wahlfreiheit des potentiellen Zelldatums schon sehr eingeschränkt.

Man beachte auch, dass $\Irr(KA)$ gar nicht von der zellulären Struktur abhängt. Wenn es also eine halbeinfache Spezialisierung gibt, ist nicht nur in einer, sondern in jeder zellulären Struktur $\abs{\Lambda}$ festgelegt. Insbesondere haben alle Zellketten dieselbe Länge.

Das erklärt z.\,B., wieso bei Geck die Konstruktion eines Zelldatums für die Hecke"=Algebra $H(W,S,L)$ von vornherein mit $\abs{\Lambda}=\abs{\Irr(W)}$ und $\abs{M(\lambda)}=d_\lambda$ beginnt (zur Erinnerung: $d_\lambda$ ist der Grad des irreduziblen Charakters vom Typ $\lambda$). Das muss so sein, weil die Spezialisierung $\IZ[\Gamma]\to F, v^\gamma\mapsto 1$ der Hecke"=Algebra gleich der Gruppenalgebra $F[W]$ und diese halbeinfach ist.
\end{remark}

\chapter{Darstellungen II: \texorpdfstring{$W$}{W}-Graphen}
\section{Definition und Beispiele}

\begin{convention}
\index{terms}{Coxeter!-Gruppe!-mit-Gewicht}
In diesem Kapitel fixieren wir eine endliche Coxeter"=Gruppe"=mit"=Gewicht $(W,S,L)$, bezeichnen die Gewichtegruppe mit $\Gamma$ und die dazugehörige Hecke"=Algebra mit $H$. Wir nehmen außerdem ab jetzt an, dass $L(s)>0$ für alle $s\in S$ gilt.

Außerdem sei $\IQ_W\subseteq F\subseteq\IR$ und $K:=F(\Gamma)$ der Körper der rationalen Funktionen mit Exponenten aus $\Gamma$ und Koeffizienten aus $F$.
\end{convention}

\begin{definition}
\index{symbols}{kC@$k^{(\mathfrak{C})}$}
Ist $k$ ein kommutativer Ring und $\mathfrak{C}$ eine beliebige Menge, dann bezeichnen wir mit $k^{(\mathfrak{C})}$ den freien $k$-Modul mit Basis $\mathfrak{C}$.
\end{definition}

\begin{remark}
\index{terms}{Matrix!spaltenfinite}
Es sei darauf hingewiesen, dass alle Endomorphismen von $k^{(\mathfrak{C})}$ Darstellungsmatrizen bzgl. der Basis $\mathfrak{C}$ haben. Im Gegensatz zum endlichdimensionalen Fall sind jedoch nicht alle Elemente des Matrizenraums $k^{\mathfrak{C}\times\mathfrak{C}}$ auch tatsächlich Darstellungsmatrizen von Endomorphismen. Stattdessen ist $\End(k^{(\mathfrak{C})})$ zum Raum der \udot{spaltenfiniten} Matrizen isomorph. Spaltenfinit heißen diejenigen Matrizen, welche in jeder Spalte endlich viele von Null verschiedene Einträge haben.
\end{remark}

\begin{remark}
Seit der ursprünglichen Definition von $W$"~Graphen in Kazhdan und Lusztigs Originalarbeit \cite{KL} sind verschiedene Varianten der Definition benutzt worden. Wir entscheiden uns für die folgende, die eine einfache darstellungstheoretische Interpretation von $W$"~Graphen zulässt, und die wir direkt danach spezialisieren werden:
\end{remark}

\begin{definition}[$W$-Graphen]\label{def:W_graph}
\index{terms}{Matrix!spaltenfinite}\index{terms}{W-Graph@$W$-Graph}\index{terms}{W-Graph@$W$-Graph!mit Gewichten in $k$}\index{terms}{W-Graph@$W$-Graph!-Modul}
\index{symbols}{Ix@$I(x)$}\index{symbols}{mxys@$m_{xy}^s$}\index{symbols}{omegaTs@$\omega(T_s)$}
Sei $k$ eine kommutative $\IZ[\Gamma]$"~Algebra (beispielsweise $k=K$).

Ein \udot{$W$"~Graph mit Kantengewichten in $k$} ist eine (auch unendliche) Eckenmenge $\mathfrak{C}$ zusammen mit Eckenlabeln $I:\mathfrak{C}\to 2^S$ und spaltenfiniten Matrizen $m^s\in k^{\mathfrak{C}\times\mathfrak{C}}$ für alle $s\in S$ derart, dass:
\begin{enumerate}
	\item $m_{xy}^s \neq 0 \implies s\in I(x)\setminus I(y)$.
	\item Die Matrizen
	\[\omega(T_s)_{xy} := \begin{cases}
			-v_s^{-1}\cdot 1_k & \text{falls}\,x=y \ \text{und}\ s\in I(x) \\
			v_s\cdot 1_k & \text{falls}\,x=y \ \text{und}\ s\notin I(x) \\
			m_{xy}^s & \text{sonst}
		\end{cases}\]
		induzieren eine Matrixdarstellung $H\to\End_k(k^{(\mathfrak{C})})$.
\end{enumerate}
Wir fassen dies als gerichteten Graph mit den Ecken $\mathfrak{C}$ auf, bei dem genau dann eine Kante $y\to x$ existiert, wenn eines der Gewichte $m_{xy}^s\neq 0$ ist. Man beachte, dass eine Kante mehrere Gewichte haben kann, nämlich bis zu $\abs{I(x)\setminus I(y)}$ viele.

Der \udot{$W$"~Graph"=Modul} des $W$"~Graphen $(\mathfrak{C},I,m)$ ist der $k$-Modul-mit-Basis $(k^{(\mathfrak{C})},\mathfrak{C})$, auf dem $H$ durch $\omega$ operiert. Wir unterscheiden nicht immer präzise zwischen dem $W$"~Gra\-phen und dem $W$"~Graph"=Modul. Offenkundig lässt sich aber das eine aus dem anderen rekonstruieren.
\end{definition}

\begin{remark}
Die Bedingung, dass $\omega$ eine Matrixdarstellung von $H$ induziert, fordert im Wesentlichen nur, dass die $\omega(T_s)$ die Zopfrelationen erfüllen, weil die quadratische Relation bereits durch Bedingung a. und die spezielle Form der Matrizen sichergestellt ist, wie man leicht nachrechnen kann. (Siehe auch Lemma \ref{wgraph_alg:iota_morphism}.)
\index{terms}{Zopfrelationen}
\end{remark}
\begin{remark}
Wir haben Links"=$W$"~Graphen definiert. Natürlich kann man auch Rechts"=$W$"~Gra\-phen definieren. Zweiseitige $W$"~Graphen sind dasselbe wie $W\times W^\textrm{op}$-Graphen und wurden ebenfalls bereits von Kazhdan und Lusztig betrachtet (siehe \citep[1.3]{KL}).
\end{remark}

\begin{definition}[Geck-$W$-Graphen, siehe {\citep[1.4.11]{geckjacon}}]
\index{terms}{W-Graph@$W$-Graph!Geck-}\index{terms}{palindromisch}
\index{symbols}{$\overline{\phantom{m}}$}
Ein $W$"~Graph $(\mathfrak{C},I,m)$ heiße \udot{Geck-$W$"~Graph}, falls er Gewichte in $k=A[\Gamma]$ hat für einen kommutativen Ring $A$, falls zusätzlich die Kantengewichte palindromisch sind (d.\,h. $\overline{m_{xy}^s}=m_{xy}^s$) und sie außerdem die Gradschranke $v_s m_{xy}^s \in A[\Gamma_{>0}]$ erfüllen (d.\,h. alle in $m_{xy}^s$ vorkommenden Exponenten sind echt größer als $-L(s)$ und echt kleiner als $+L(s)$) für alle $x,y\in\mathfrak{C}$ und alle $s\in S$.
\end{definition}

\begin{definition}[$W$-Graphen mit konstanten Koeffizienten]
\index{terms}{W-Graph@$W$-Graph!mit konstanten Gewichten}
Sei $(\mathfrak{C},I,m)$ ein $W$"~Graph mit Gewichten in $k$. Wir sagen, der $W$"~Graph habe \udot{konstante Gewichte}, falls $(v^\gamma\cdot 1_k)_{\gamma\in\Gamma}$ linear unabhängig über dem von $\Set{m_{xy}^s | x,y\in\mathfrak{C}, s\in S}$ erzeugten Teilring von $k$ ist.
\end{definition}

\begin{remark}
Ist in obiger Situation $A$ der von den $m_{x,y}^s$ erzeugte Teilring von $k$, dann ist ${A[\Gamma]\to k}$ eine Einbettung. Wenn wir $A[\Gamma]$ als Ring von Laurent"=Polynomen auffassen, dann sind $W$"~Graphen mit konstanten Gewichten mit anderen Worten also genau diejenigen $W$"~Graphen, deren Kantengewichte konstante Laurent"=Polynome sind. Es ist dann jeder $W$"~Graph mit konstanten Gewichten ein Geck-$W$"~Graph, da konstante Laurent"=Polynome palindromisch sind und die Gradschranke erfüllen.

Im Einparameterfall, d.\,h. $\Gamma=\IZ$ und $L(s)=1$ für alle $s\in S$, hat umgekehrt jeder Geck-$W$"~Graph mit Kantengewichten in $A[v^{\pm 1}]$ automatisch auch konstante Gewichte, denn die Bedingung, dass die in $m_{xy}^s$ vorkommenden Exponenten echt größer als $-L(s)$ und echt kleiner $+L(s)$ sind für alle $s\in S$, stellt $m_{x,y}^s\in A$ sicher.
\index{terms}{Einparameterfall}
\end{remark}

\begin{example}[Kazhdan-Lusztig-$W$-Graphen, siehe \cite{KL}]
\index{terms}{Kazhdan-Lusztig!-$W$-Graph}\index{terms}{W-Graph@$W$-Graph!mit konstanten Gewichten}\index{terms}{W-Graph@$W$-Graph!Geck-}\index{terms}{Kazhdan-Lusztig!-$\mu$}
Das motivierende Beispiel für die Definition von $W$"~Graphen sind die Kazhdan-Lusztig-$W$"~Graphen: Der reguläre $H$"~Linksmodul ist eine $W$"~Graph"=Darstellung bzgl. der Kazh\-dan"=Lusz\-tig-Basis $(C_w)_{w\in W}$, d.\,h. die Darstellungsmatrizen bzgl. dieser Basis sind die Matrizen eines $W$"~Graphen. Die Eckenlabel sind dabei durch die Linksabstiegsmengen
\[D_L(w) := \Set{s\in S | sw < w}\]
gegeben. Die Kantengewichte sind dabei durch die Kazhdan-Lusztig-$\mu$ gegeben: Für ${s\in D_L(x)\setminus D_L(y)}$ hat dieser $W$"~Graph die Kantengewichte
\[m_{xy}^s := \begin{cases}1 & \text{falls}\, x=sy, \\ (-1)^{l(x)+l(y)+1}\mu_{xy}^s & \text{falls}\, x<y, \\ 0 & \text{sonst}. \end{cases}\]

Diese $W$"~Graphen erfüllen auch die zusätzliche Bedingung der Palindromität der Kantengewichte und die Gradbeschränkungen, es sind also Geck-$W$"~Graphen (siehe \citep[2.1.8+2.1.10]{geckjacon}).

Sie haben jedoch im Allgemeinen keine konstanten Gewichte. Im Einparameterfall muss das zwar der Fall sein, wie wir eben festgestellt haben, im Multiparameterfall wird es jedoch fast immer fehlschlagen:

Betrachten wir die Coxeter"=Gruppe vom Typ $I_2(m)$ mit $m\in\Set{4,6,8,\ldots}\cup\Set{\infty}$, wobei die Erzeuger $s$ und $t$ sowie die Gewichte $L(s)=a \neq b=L(t)$ seien. Für die Elemente ${y:=s}$, ${w:=ts}$ gilt dann $\mu_{y,w}^s = v^{a-b}+v^{b-a}$, wie man sich leicht mittels der Rekursionsformeln in \ref{KL:def:KL_poly} überzeugen kann. Insbesondere ist $\Set{v^{a-b},v^0,v^{b-a}}$ nicht linear unabhängig über dem von den $\mu$-Werten erzeugten Teilring, der Kazhdan-Lusztig-$W$"~Graph ist also kein $W$"~Graph mit konstanten Kantengewichten.

Alle irreduziblen Coxeter"=Gruppen, für die der Multiparameterfall überhaupt eintreten kann, enthalten eine parabolische Untergruppe von dieser Gestalt, sodass auch in all diesen Fällen der Kazhdan-Lusztig-$W$"~Graph keine konstanten Kantengewichte hat.

Für reduzible Coxeter"=Gruppen $W=W_1\times W_2$ gilt
\[\mu_{y_1 y_2, w_1 w_2}^s =\begin{cases} \mu_{y_1,w_1}^s & \text{falls }s\in S_1\text{ und }y_2=w_2 \\ \mu_{y_2,w_2}^s & \text{falls } s\in S_2\text{ und }y_1=w_1 \\ 0 & \text{sonst}\end{cases}\]
für alle $y_i,w_i\in W_i$ und alle $s\in S$ mit $s(y_1 y_2)<y_1 y_2 < w_1 w_2 < s (w_1 w_2)$. Wenn $L$ nun eine Gewichtsfunktion auf $W$ ist, dann können folgende Fälle eintreten: $L$ ist auf jeder irreduziblen Komponente konstant oder es gibt eine parabolische Untergruppe von obigem Typ. Im ersten Fall sind alle $\mu$-Werte ganzzahlig und der Kazhdan-Lusztig-$W$"~Graph hat konstante Kantengewichte. Im zweiten Fall greift erneut das obige Argument.

Ausgangspunkt für dieses Beispiel war ein eMail"=Austausch mit Meinolf Geck (\cite{geck2013kl_wgraphs}).
\end{example}

\begin{example}[Spiegelungsdarstellung]\label{wgraphs:ex:refl_rep}
\index{terms}{Spiegelungsdarstellung!einer Hecke-Algebra}\index{terms}{Dynkin-Diagramm}
$F$ enthält alle Zahlen der Form $\zeta_m+\zeta_m^{-1}$, wobei $\zeta_m$ eine primitive $m$-te Einheitswurzel ist und $m=\ord(st)$ für $s,t\in S$.

Wähle nun $c_{st}\in K$ so, dass folgende Gleichungen erfüllt sind:
\begin{alignat*}{2}
	c_{ss} &= v_s+v_s^{-1} \\
	c_{st} &= c_{ts} = 0 & \text{falls}\,m=2 \\
	c_{st}c_{ts} &= \frac{v_s}{v_t}+(\zeta_m+\zeta_m^{-1})+\frac{v_t}{v_s} \quad & \text{falls}\,m>2
\end{alignat*}
Dann definiert $e_s\mapsto v_s e_s - c_{st} e_t$ eine Darstellung von $H$ auf $K^S$, die sogenannte \udot{Spie\-ge\-lungs\-dar\-stel\-lung}. (Siehe \citep[8.1.11+11.1.3]{geckpfeiffer}. Man beachte, dass dort die $\dot{T}$"~Basis verwendet wird.)

Dies ist eine $W$"~Graph"=Darstellung. Der dazugehörige $W$"~Graph hat das Dynkin"=Diagramm von $(W,S)$ als zugrundeliegenden (ungerichteten) Graphen, die Eckenlabel $I(s)=\Set{s}$ und die Kantenlabel
\[m_{xy}^s := \begin{cases}
-c_{xy} & \text{falls}\,\ord(xy)>2 \text{ und } s=x \\
0 & \text{sonst}
\end{cases}\]
für alle $x,y,s\in S$.

Da die Wahl der $c_{st}$ nicht eindeutig festgelegt ist, sehen wir insbesondere, dass die $W$"~Graphen durch den Isomorphietyp des $W$"~Graph-Moduls nicht vollständig festgelegt sind, weil für alle Wahlen von $c_{st}$ isomorphe $H$"~Moduln herauskommen.
\end{example}

\begin{example}[Eindimensionale Darstellungen]\label{wgraphs:ex:1D_reps}
\index{terms}{Darstellung!eindimensionale}\index{terms}{Signum}
Alle eindimensionalen Darstellungen von $H$ über $K$ sind $W$"~Graph"=Darstellungen.

Sei also $\rho: H\to K^{1\times 1}$ eine eindimensionale Matrixdarstellung. Aufgrund der quadratischen Relationen muss $\rho(T_s)\in\Set{v_s,-v_s^{-1}}$ für alle $s\in S$ gelten. Sei $I$ definiert als $\Set{s\in S | \rho(T_s)=-v_s^{-1}}$. Der $W$"~Graph mit einer Ecke und dem Eckenlabel $I$ definiert dann genau die Darstellung $\rho$ (Kantengewichte müssen wir nicht definieren, da laut Definition sowieso nur Kanten zwischen Ecken mit verschiedenen Labeln vorkommen).

Insbesondere definieren $I:=\emptyset$ die triviale und $I:=S$ die Signumsdarstellung.
\end{example}

\begin{example}[Äußere Potenzen der Spiegelungsdarstellung]\label{wgraphs:ex:ext_powers_refl_rep}
\index{terms}{Spiegelungsdarstellung!äußere Potenzen}\index{terms}{Signum}\index{terms}{Hecke-Algebra!$\IZ[q]$-Form}
Alle äußeren Potenzen der Spiegelungsdarstellung sind $H$"~Moduln via
\[T_s \diamond (v_1 \wedge \ldots \wedge v_r) := v_s^{-(r-1)} (T_s v_1 \wedge \ldots \wedge T_s v_r).\]
(Dies ist einer der Fälle, in denen die $\dot{T}$"~Basis die Formeln schöner macht, weil dort wirklich $\dot{T}_s \diamond (v_1\wedge\ldots\wedge v_r)=\dot{T}_s v_1 \wedge\ldots\wedge \dot{T}_s v_r$ gilt)

Dies kann zu einem $W$"~Graphen gemacht werden, indem man $\mathfrak{C}:=\Set{A\subseteq S | \abs{A}=r}$, $I(A):=A$ sowie $m_{AB}^s := (-1)^r c_{st}$ setzt, falls $A\setminus B=\Set{s}$ und $B\setminus A=\Set{t}$ ist, bzw. $m_{AB}^s=0$ andernfalls. Dabei sei $c_{st}$ wie in \ref{wgraphs:ex:refl_rep} gewählt. Für $r=0$ ist das die triviale Darstellung, für $r=1$ die Spiegelungsdarstellung, für $r=\abs{S}-1$ das duale der Spiegelungsdarstellung und für $r=\abs{S}$ schließlich die Signumsdarstellung.
\end{example}

\begin{remark}
Eine umfangreichere Liste expliziter Beispiele für $W$"~Graphen für $H_3$, $H_4$, $F_4$ ist in \citep[11.1-11.3]{geckpfeiffer} zu finden.
\end{remark}

\begin{lemma}[Konstruktionen mit $W$-Graphen]\label{wgraphs:constructions}
\index{terms}{Zellmodul einer Zelle}\index{terms}{Darstellung!duale}\index{terms}{W-Graph@$W$-Graph!dualer}\index{terms}{Parabolische Untergruppe}\index{terms}{Restriktion}

Sei $(\mathfrak{C},I,m)$ ein $W$"~Graph mit Kantengewichten in $k$, $V=k^{(\mathfrak{C})}$ der zugehörige $W$"~Graph"=Modul und $\omega:kH\to\End_k(V)$ die Matrixdarstellung.

Die folgenden Konstruktionen liefern dann ebenfalls $W$"~Graphen:
\begin{enumerate}
	\item Zellen:
	
	Ist $\mathfrak{Z}\subseteq\mathfrak{C}$ eine Zelle (im Sinne von \ref{KL:def:cells}) von $(V,\mathfrak{C})$, dann ist $(\mathfrak{Z},I_{|\mathfrak{Z}},m_{|\mathfrak{Z}\times\mathfrak{Z}})$ ein $W$"~Graph. Der zugehörige $W$"~Graph"=Modul stimmt als Modul-mit-Basis mit dem Zellmodul $V(\mathfrak{Z})$ überein.
	\item Algebraische Konjugation:
	
	Ist $\alpha:k\to k'$ ein Homomorphismus von $\IZ[\Gamma]$"~Algebren, dann ist $(\mathfrak{C},I,\alpha m)$ ein $W$"~Graph mit Kantengewichten in $k'$.
	\item Dualität:
	
	Falls $\abs{\mathfrak{C}}<\infty$ ist, dann ist $(\mathfrak{C},I',m')$ ein $W$"~Graph, wobei $I'(x):=S\setminus I(x)$ und $(m_{xy}^s)' := -m_{yx}^s$. Der zugehörige Modul ist $V^\dagger$ und die Matrixdarstellung durch ${\omega^\dagger(T_s) = \omega(-T_s^{-1})^{Tr}}$ gegeben.
	\item Parabolische Restriktion:
	
	Ist $J\subseteq S$ beliebig, dann ist  $(\mathfrak{C},I_J,m)$ ein $W_J$"~Graph, wobei $I_J(x):=I(x)\cap J$.
\end{enumerate}
\end{lemma}
\begin{proof}
a. ergibt sich sofort aus den Definitionen: Bei geeigneter Anordnung haben die Matrizen die Gestalt
\[\omega(T_s) = \begin{pmatrix}
	\ast & \ast & \ast \\
	0 & A_s & \ast \\
	0 & 0 & \ast
\end{pmatrix}\]
wobei die mittleren Indizes genau zu den Elementen von $\mathfrak{Z}$ gehören. Das folgt direkt aus der Definition von Zellen. Aus der Definition von $W$"~Graphen folgt außerdem, dass die $\omega_{|\mathfrak{Z}}(T_s):=A_s$ genau die Matrizen zum eingeschränkten Graphen sind. Wenn nun die $\omega(T_s)$ die Zopfrelationen erfüllen, tun es die $\omega_{|\mathfrak{Z}}(T_s)$ auch. Also ist der eingeschränkte Graph tatsächlich selbst ein $W$"~Graph.

\medbreak
b. Die Matrizen von $(\mathfrak{C},I,\alpha m)$ ergeben sich einfach, indem man $\alpha$ auf alle Einträge von $\omega(T_s)$ anwendet, da $\alpha(v_s)=v_s$ gilt. Da $k^{\mathfrak{C}\times\mathfrak{C}}\xrightarrow{\alpha} k'^{\mathfrak{C}\times\mathfrak{C}}$ sich zu einem Ringhomomorphismus $\End_k(k^{(\mathfrak{C})})\to\End_{k'}(k'^{(\mathfrak{C})})$ einschränkt, erfüllen mit den $\omega(T_s)$ auch die $\alpha\omega(T_s)$ die Zopfrelationen, also ist der Graph tatsächlich ein $W$"~Graph.

\medbreak
c. Wenn man die Indizes nach $s\in I(x)$ und $s\notin I(x)$ sortiert, hat $\omega(T_s)$ die folgende Dreiecksgestalt
\[\omega(T_s)=\begin{pmatrix} -v_s^{-1} & m \\ 0 & v_s \end{pmatrix}\]
und die Darstellungsmatrix von $(\mathfrak{C},I',m')$ hat die Form
\[\begin{pmatrix} v_s & 0 \\ -m & -v_s^{-1} \end{pmatrix} = -(\omega(T_s)-(v_s-v_s^{-1}))^T.\]
Nun ist $T_s-(v_s-v_s^{-1})=T_s^{-1}$, d.\,h. $\omega^\dagger(T_s) = \omega(-T_s^{-1})^T$. Weiter ist $T_s\mapsto -T_s^{-1}$ ein Automorphismus von $H$ und Transponieren ein Antiautomorphismus von $k^{\mathfrak{C}\times\mathfrak{C}}$, d.\,h. wenn die $\omega(T_s)$ die Zopfrelationen erfüllen, dann tun es auch die $\omega^\dagger(T_s)$. In der Tat folgt aus dieser Darstellung, dass $\omega^\dagger$ den mit dem obigen Automorphismus getwisteten Dualmodul $\Hom_k(V,k)$ realisiert.

\medbreak
d. Die Matrizen $\omega(T_s)$ für den $W_J$-Graphen sind exakt dieselben wie diejenigen für den $W$"~Graphen, wenn $s\in J$ ist. Daher gelten die Zopfrelationen.
\end{proof}

\begin{remark}
\index{terms}{Kazhdan-Lusztig!-Zellen}\index{terms}{Kazhdan-Lusztig!-$W$-Graph}\index{terms}{Darstellung!eindimensionale}
Die Kazhdan-Lusztig-Linkszellen sind genau die Zellen des Moduls-mit-Basis $(H,(C_w)_{w\in W})$, d.\,h. die Linkszellmoduln sind $W$"~Graph"=Moduln. Das liefert die Hauptquelle von Beispielen für $W$"~Graphen.

$\Set{w_0}$ und $\Set{1}$ sind (zweiseitige) Kazhdan-Lusztig-Zellen. Sie realisieren die Signums- bzw. die triviale Darstellung.
\end{remark}
\begin{remark}
\index{terms}{Parabolische Unteralgebra}\index{terms}{Restriktion}\index{terms}{Induktion}
Nicht nur parabolische Restriktion $\text{Res}_{H_J}^H$, auch Induktion $\text{Ind}_{H_J}^H$ entlang von parabolischen Unteralgebren liefert wieder $W$"~Graphen. Die entsprechende Konstruktion ist aber komplex, siehe z.\,B. \cite{howlett2003inducingI} und \cite{howlett2004inducingII}. Howlett und Yin haben in diesen beiden Artikeln die Kombination aus Induktion und Zellen-Extraktion genutzt, um $W$"~Graphen für viele irreduzible Darstellungen zu finden.
\end{remark}
\section{\texorpdfstring{$W$}{W}-Graph-Algebren}

\begin{definition}[Zopfkommutator]
\index{terms}{Zopfkommutator}\index{terms}{Zopfrelationen}
\index{symbols}{Deltam@$\Delta_m$}
Für $m\in\IN$ und zwei Elemente $x,y$ eines beliebigen Ringes $A$ definiere den \udot{$m$-ten Zopfkommutator} als
\[\Delta_m(x,y) := \underbrace{xyx\ldots}_{m \text{ Faktoren}}-\underbrace{yxy\ldots}_{m\text{ Faktoren}}\]
Insbesondere ist also $\Delta_0(x,y)=1-1=0$, $\Delta_1(x,y)=x-y$, $\Delta_2(x,y)=xy-yx=[x,y]$. Der Bequemlichkeit halber definieren wir außerdem $\Delta_{\infty}(x,y):=\Delta_0(x,y)=0$.
\end{definition}

\newcommand{\OmegaGy}{\Omega^\text{Gy}}
\newcommand{\OmegaGJ}{\Omega^\text{GJ}}

\begin{definition}[$W$-Graph-Algebren, siehe {\citep[2.4]{Gyoja}}]
\index{terms}{W-Graph@$W$-Graph!-Algebra}
\index{symbols}{Xi@$\Xi$, $\Xi_\Gamma$, $\Xi_G$}\index{symbols}{es@$e_s$}\index{symbols}{xs@$x_s$}
\index{symbols}{Omega@$\Omega$, $\OmegaGJ_G$, $\OmegaGy$}
Sei $S$ eine endliche Menge. Definiere die $\IZ$"~Algebra $\Xi$ wie folgt: Erzeuger seien Symbole $e_s, x_s$ für $s\in S$. Die Relationen seien zunächst die folgenden:
\begin{enumerate}
	\item $e_s^2 = e_s$, $e_s e_t = e_t e_s$,
	\item $e_s x_s = x_s$ und $x_s e_s = 0$
\end{enumerate}
für alle $s,t\in S$.

\medbreak
Sei nun $(W,S,L)$ eine Coxeter"=Gruppe"=mit"=Gewicht. Betrachte die $\Gamma$"~graduierte Algebra $\IZ[\Gamma]\otimes_\IZ \Xi$. Darin definieren wir das Element
\[\iota(T_s) := -v_s^{-1} e_s + v_s (1-e_s) + x_s.\]
Die \udot{universelle $W$"~Graph"=Algebra} $\Omega=\Omega(W,S,L)$ ist definiert als der Quotient von $\IZ[\Gamma]\otimes_\IZ\Xi$ nach den Zopfrelationen $\Delta_{m_{st}}(\iota(T_s),\iota(T_t))=0$ für alle $s,t\in S$.

\bigbreak
Wir definieren eine weitere $W$"~Graph"=Algebra. Dazu fangen wir erneut mit einer $\IZ$"~Algebra $\Xi_\Gamma$ an, die von $e_s, x_{s,\gamma}$ für $s\in S$ und $\gamma\in\Gamma$ erzeugt wird und folgende Relationen hat:
\begin{enumerate}
	\item $e_s^2 = e_s$,\quad $e_s e_t = e_t e_s$,
	\item $e_s x_{s,\gamma} = x_{s,\gamma}$,\quad $x_{s,\gamma} e_s = 0$,
	\item $x_{s,\gamma} = x_{s,-\gamma}$ und
	\item $x_{s,\gamma} = 0$ falls $\gamma\notin (-L(s),+L(s))$
\end{enumerate}
für alle $s,t\in S$, $\gamma\in\Gamma$. Für jede \textit{endliche} Menge $G\subseteq\Gamma$ definieren wir $\Xi_G$ als den Quotient von $\Xi_\Gamma$ nach den zusätzlichen Relationen
\begin{enumerate}[resume]
	\item $x_{s,\gamma} = 0$ falls $\gamma\notin G$.
\end{enumerate}

Wir definieren ebenfalls ein Element $\iota_G(T_s)\in\IZ[\Gamma]\otimes_{\IZ}\Xi_G$ durch
\[\iota_G(T_s) := -v_s^{-1} e_s + v_s (1-e_s) + \sum_{\gamma\in G} x_{s,\gamma} v^\gamma.\]
Wir nehmen nun die Zopfkommutatoren und trennen sie nach $\Gamma$-homogenen Komponenten, d.\,h. wir schreiben
\[\Delta_{m_{st}}(\iota_G(T_s),\iota_G(T_t)) = \sum_{\gamma\in\Gamma} y_G^\gamma(s,t) v^\gamma\]
mit $y_G^\gamma(s,t)\in\Xi_G$ und definieren $\OmegaGJ_G$ als den Quotienten von $\Xi_G$ nach den Relationen $y_G^\gamma(s,t)=0$ für alle $s,t\in S$ und alle $\gamma\in\Gamma$.

\medbreak
Den Spezialfall $G=\Set{0}$ nennen wir \udot{Gyojas $W$"~Graph"=Algebra} und bezeichnen diese Algebra als $\OmegaGy := \OmegaGJ_{\Set{0}}$.
\end{definition}

\begin{lemma}[Universelle Eigenschaft]\label{wgraph_alg:univ_prop_OmegaGJ}
\index{terms}{W-Graph@$W$-Graph!-Algebra!universelle Eigenschaft}
Sei $k$ ein kommutativer Ring und $G\subseteq\Gamma$ eine feste, endliche Teilmenge. Dann hat $k\OmegaGJ_G$ die folgende universelle Eigenschaft in der Kategorie der $k$"~Algebren:
\begin{align*}
\Hom(k\OmegaGJ_G, Y) \isomorphic \big\lbrace & f\in\Hom(k\Xi_G, Y) \big| \\
	& \forall s,t\in S: (\id\otimes f)\big(\Delta_{m_{st}}(\iota_G(T_s),\iota_G(T_t))\big) = 0 \:\text{in}\:k[\Gamma]\otimes_k Y \big\rbrace
\end{align*}
wobei der natürliche Isomorphismus durch Zurückziehen entlang von $k\Xi_G\to k\OmegaGJ_G$ gegeben ist. Dabei ist $\id\otimes f$ der von $f$ induzierte Homomorphismus $k[\Gamma]\otimes_k k\Xi_G \to k[\Gamma]\otimes_k Y$.
\end{lemma}
\begin{proof}
Wir schreiben 
\[\Delta_{m_{st}}(\iota_G(T_s),\iota_G(T_t)) = \sum_{\gamma\in\Gamma} v^\gamma \otimes y_G^\gamma(s,t)\]
mit $y_G^\gamma(s,t)\in\Xi_G$ wie in der Definition.

Ist nun $f:k\OmegaGJ_G\to Y$ ein Homomorphismus, dann können wir $f$ als Homomorphismus $k\Xi_G\to Y$ auffassen, für den $f(y_G^\gamma(s,t))=0$ gilt. Es folgt dann:
\begin{align*}
	0 &= \sum_{\gamma} v^\gamma\otimes f(y_G^\gamma(s,t)) \\
	&= (\id\otimes f)\Big(\sum_\gamma v^\gamma \otimes y_G^\gamma(s,t)\Big) \\
	&= (\id\otimes f)\big(\Delta_{m_{st}}(\iota_G(T_s),\iota_G(T_t))\big)
\end{align*}
wie gewünscht.

Ist umgekehrt $f: k\Xi_G\to Y$ mit $(\id\otimes f)\big(\Delta_{m_{st}}(\iota_G(T_s),\iota_G(T_t))\big) =0$ in $k[\Gamma]\otimes_k Y$, dann gilt
\[ 0 = (\id\otimes f)\Big(\sum_\gamma v^\gamma\otimes y_G^\gamma(s,t) \Big) = \sum_{\gamma} v^\gamma\otimes f(y_G^\gamma(s,t)). \]
Da in $k[\Gamma]\otimes_k Y$ nun die $v^\gamma\otimes 1$ linear unabhängig über $1\otimes Y$ sind, folgt $1\otimes f(y_G^\gamma(s,t)) = 0$ und somit (da $y\mapsto 1\otimes y$ injektiv ist) auch $f(y_G^\gamma(s,t))=0$ für alle $s,t\in S$, $\gamma\in\Gamma$. Daher steigt $f$ zu einem Homomorphismus $k\OmegaGJ_G\to Y$ ab, wie gewünscht.
\end{proof}

\begin{definition}[$W$-Graph-Algebren, Teil II]
\index{terms}{W-Graph@$W$-Graph!-Algebra}
\index{symbols}{kOmega@$k\OmegaGJ$}
Sind $\Gamma\supseteq G_1\supseteq G_2$ zwei endliche Teilmengen, so existiert ein kanonischer Projektionsmorphismus $\Xi_{G_1}\twoheadrightarrow\Xi_{G_2}$ mit
\[e_s\mapsto e_s \quad\text{und}\quad x_{s,\gamma}\mapsto\begin{cases} x_{s,\gamma} & \text{falls }\gamma\in G_2 \\ 0 &\text{sonst}\end{cases}.\]
Dieser induziert aufgrund der universellen Eigenschaft einen Morphismus ${\OmegaGJ_{G_1} \twoheadrightarrow \OmegaGJ_{G_2}}$.

\medbreak
Ist $k$ ein kommutativer Ring, so definieren wir die \udot{Geck-Jacon-$W$"~Graph"=Algebra} $k\OmegaGJ$ als den projektiven Limes
\[k\OmegaGJ := \lim_{\substack{\longleftarrow \\ G\subseteq\Gamma \,\text{endl.}}} (k\otimes_\IZ \OmegaGJ_G)\]
in der Kategorie der topologischen $k$"~Algebren. Dabei versehen wir $k$ und $k\otimes_\IZ \OmegaGJ_G$ jeweils mit der diskreten Topologie.
(Man beachte, dass die Bezeichnung ein Missbrauch von Notation ist, denn i.\,A. vertauschen Skalarerweiterungen und Limites nicht miteinander, d.\,h. es gibt a priori keinen Grund $k\otimes_\IZ \IZ\OmegaGJ_G = k\OmegaGJ$ zu erwarten!)

\bigbreak
Die Morphismen $\Omega\to\IZ[\Gamma]\OmegaGJ_G$, $e_s\mapsto e_s$, $x_s\mapsto\sum_{\gamma\in G} x_{s,\gamma}v^\gamma$ sind wohldefiniert und mit den Projektionen ${\OmegaGJ_{G_1}\twoheadrightarrow\OmegaGJ_{G_2}}$ verträglich, induzieren also einen kanonischen Morphismus ${\Omega\to \IZ[\Gamma]\OmegaGJ}$.

\bigbreak
Ebenso ist die Familie der Homomorphismen ${\Xi_\Gamma\twoheadrightarrow\Xi_G\twoheadrightarrow\OmegaGJ_G}$ mit den kanonischen Projektionen ${\OmegaGJ_{G_1}\twoheadrightarrow\OmegaGJ_{G_2}}$ verträglich, d.\,h. es gibt einen Morphismus ${k\Xi_\Gamma\to k\OmegaGJ}$.

\bigbreak
Man beachte ebenso, dass ${e_s\mapsto e_s, x_s\mapsto x_{s,0}}$ einen Isomorphismus $\Xi\to\Xi_{\Set{0}}$ definiert und allgemeiner $e_s\mapsto e_s, x_s\mapsto \sum_{\gamma\in G} x_{s,\gamma} v^\gamma$ einen Morphismus $\IZ[\Gamma]\Xi\to\IZ[\Gamma]\Xi_G$ für alle endlichen $G\subseteq\Gamma$ definiert.

\bigbreak
Wir werden all diese kanonischen Morphismen sowie auch die Projektionen ${k\OmegaGJ\to k\OmegaGJ_G}$ von Zeit zu Zeit ohne weitere Kennzeichnung benutzen (siehe auch Abbildung \ref{fig:wgraph_alg:can_morphs}). Insbesondere werden wir die Bezeichnungen $e_s$, $x_s$ und $x_{s,\gamma}$ gleichermaßen für die entsprechenden Elemente von $\Xi$, $\Xi_\Gamma$, $\Xi_G$, $\OmegaGJ_G$ und $k\OmegaGJ$ verwenden. Insbesondere gilt etwa $x_s = \sum_{\gamma\in G} x_{s,\gamma} v^\gamma$ in $\IZ[\Gamma]\Xi_G$ und $\IZ[\Gamma]\OmegaGJ_G$, $x_s=x_{s,0}$ in $\Xi_{\Set{0}}$ und $\OmegaGy$ etc.
\begin{figure}[ht]
	\centering
	\begin{tikzpicture}[
		crossing line/.style = {preaction={draw=white,-,line width=6pt}}
	]
	\matrix (m) [matrix of math nodes, row sep=3em,
	column sep=3.5em, text height=1.5ex, text depth=0.25ex]
	{
	           & k\Xi_\Gamma &             &                \\
		k\Xi    &             & k\Xi_G      & k\Xi_{\Set{0}} \\
		k\Omega & k\OmegaGJ   & k\OmegaGJ_G & k\OmegaGy      \\
	};
	\path[->,dashed]
		(m-2-1) edge[crossing line] (m-2-3)
		(m-3-1) edge (m-3-2);
	\path[->>,dashed]
		(m-2-1) edge (m-3-1);
	\path[->]
		(m-1-2) edge (m-3-2);
	\path[->>]
		(m-1-2) edge (m-2-3)
		(m-2-3) edge (m-3-3)
		(m-2-4) edge (m-3-4)
		(m-3-2) edge (m-3-3)
		(m-3-1) edge[bend right=20] (m-3-4);
	\path[->>,dotted,thick]
		(m-3-3) edge[dotted,thick] (m-3-4)
		(m-2-3) edge[dotted,thick] (m-2-4);
	\path[right hook->>]
		(m-2-1) edge[bend left=20,crossing line] (m-2-4);
	\end{tikzpicture}
	\setcapwidth[c]{0.80\textwidth}
	\caption{Kommutatives Diagramm der kanonischen Morphismen zwischen den $W$"~Graph"=Algebren. Gestrichelte Pfeile existieren, wenn $k$ eine $\IZ[\Gamma]$"~Algebra ist, und gepunktete Pfeile, wenn $0\in G$ ist.}
	\label{fig:wgraph_alg:can_morphs}
\end{figure}
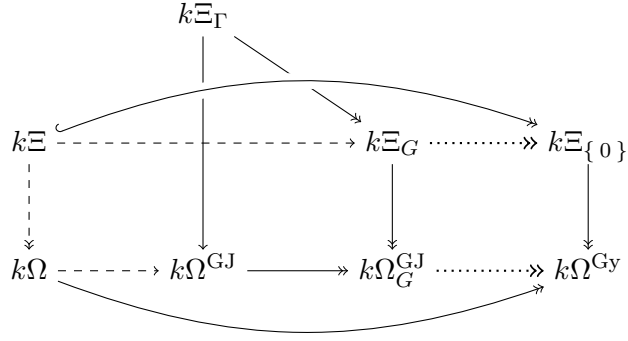
\end{definition}

\begin{remark}
Man beachte, dass $\Omega$ eine $\IZ[\Gamma]$"~Algebra ist, während $\OmegaGy$ und $\OmegaGJ_G$ als $\IZ$"~Algebren definiert wurden.
\end{remark}
\begin{remark}
$k\OmegaGJ\to k\OmegaGy$ ist ein Quotient und auch ${\Omega\to\IZ[\Gamma]\OmegaGy}$ ist surjektiv. Jedoch ist ${k\Omega\to k\OmegaGJ}$ i.\,A. kein Quotient. Für $\abs{S}=1$ und $\Gamma=\IQ$ ist $\IC\OmegaGJ$ beispielsweise unendlichdimensional, während $\IC\Omega=\IC\Xi$ dreidimensional ist, siehe Beispiel \ref{wgraph_alg:small_ex}.
\end{remark}
\begin{remark}
Da jede endliche Teilmenge $G$ in der endlichen Teilmenge ${G\cup -G}$ enthalten ist, sind die Mengen mit $G=-G$ kofinal in allen endlichen Teilmengen von $\Gamma$. Daher kann man den projektiven Limes auch nur über solche Mengen laufen lassen.

Da wir die Relationen $x_{s,\gamma}=x_{s,-\gamma}$ und $x_{s,\gamma}=0$ für $\gamma\notin G$ gefordert haben, ist außerdem $\OmegaGJ_G = \OmegaGJ_{G\cap -G}$.

Wir werden, wann immer es nützlich erscheint, darauf zurückgreifen, dass o.\,B.\,d.\,A. ${G=-G}$ angenommen werden kann.
\end{remark}
\begin{remark}
Wenn das Intervall $(-L(s),L(s))$ für alle $s\in S$ endlich ist (was z.B. für $\Gamma\isomorphic\IZ$ der Fall ist), dann gibt es eine maximale, endliche Teilmenge derart, dass $x_{s,\gamma}=0$ für alle $\gamma\notin G$ gilt, nämlich $G=\bigcup_{s\in S} (-L(s),L(s))$. Dann gilt $k\OmegaGJ=k\OmegaGJ_G$ für diese maximale Teilmenge.
\end{remark}
\begin{remark}
Da die Zopfrelationen nur endlich viele von Null verschiedene $\Gamma$-Komponenten haben und nur endlich viele Zopfrelationen existieren, sind alle $\OmegaGJ_G$ endlich präsentiert für endliche $G\subseteq\Gamma$.
\end{remark}

\begin{example}\label{wgraph_alg:small_ex}
Für $\abs{S}=0$ ist mangels Erzeugern einfach $k\OmegaGJ_G=k\Xi_G=k$ für alle endlichen $G\subseteq\Gamma$. Daher ist auch $k\OmegaGJ=\lim\limits_{\longleftarrow} k\Xi_G=k$.

\medbreak
Für $\abs{S}=1$ ist die Situation bereits komplizierter. Es gilt ebenfalls $k\OmegaGJ_G=k\Xi_G$, weil keine verschiedenen $s,t\in S$ existieren, die nichttriviale Zopfrelationen verursachen könnten.

Es reicht im Limes $\lim\limits_{\longleftarrow} k\Xi_G$ o.\,B.\,d.\,A nur solche $G\subseteq\Gamma$ mit $G=-G$ zu betrachten, da diese Teilmengen kofinal in der Menge aller endlichen Teilmengen von $\Gamma$ sind. $\Xi_G$ ist dann eine Pfadalgebra über dem Köcher mit den zwei Knoten $\emptyset$ und $S$ sowie $\abs{G\cap[0,L(s))}$ Kanten von $\emptyset$ nach $S$. Als $k$"~Modul ist dies also $k E_\emptyset \times k E_S \times \prod_{\substack{\gamma\in G \\ 0\leq\gamma<L(s)}} k X_{S,\emptyset}^{s,\gamma}$. Durch Übergang zum projektiven Limes gilt somit
\[k\OmegaGJ \isomorphic k E_\emptyset \times k E_S \times \prod_{0\leq\gamma<L(s)} k X_{S,\emptyset}^{s,\gamma}.\]
Insbesondere ist $k\OmegaGJ$ i.\,A. weder frei als $k$"~Modul noch endlich erzeugt.
\end{example}

\begin{lemma}[Diskrete Darstellungen von $\OmegaGJ$]\label{wgraph_alg:continuous_reps}
Sei $k$ ein kommutativer Ring und $A$ eine $k$"~Algebra.

\begin{enumerate}
	\item Es sei $V$ ein $k\OmegaGJ$"~$A$"~Bimodul, der als $A$"~Modul endlich erzeugt ist.
	
	Wenn wir $V$ als diskreten topologischen Raum betrachten und $V$ so ein stetiger $k\OmegaGJ$"~Modul wird, faktorisiert die Darstellung $k\OmegaGJ\to\End_{A}(V)$ durch einen Quotienten $k\OmegaGJ\to k\OmegaGJ_G$.
	\item Ein Morphismus $k\OmegaGJ\to A$ von $k$-Algebren ist genau dann stetig für die diskrete Topologie auf $A$, wenn er durch einen Quotienten $k\OmegaGJ\to k\OmegaGJ_G$ faktorisiert.
\end{enumerate}
\end{lemma}
\begin{proof}
Wir führen die Bezeichnung $\pi_G$ für die Projektionen $k\OmegaGJ\to k\OmegaGJ_G$ ein und halten zunächst ein paar Fakten fest:

\medbreak
Fakt 1: In einem diskreten Raum $X$ konvergiert ein Netz $(x_i)_{i\in I}$ genau dann gegen $x$, wenn ein $i\in I$ existiert derart, dass für alle $j\geq i$ stets $x_j=x$ gilt. (Siehe \cite{kelley1975general} für eine Einführung in Netze und ihre Konvergenzeigenschaften.)

\medbreak
Fakt 2: Der Limes $\smash{k\OmegaGJ=\lim\limits_{\longleftarrow} k\OmegaGJ_G}$ ist als topologische $k$"~Algebra gleich
\[\Big\lbrace (x_G)\in \smash{\prod_{G\subseteq\Gamma\ \text{endl.}}} k\OmegaGJ_G \mid \pi_G(z_H) = z_G \ \text{für alle}\ G\subseteq H\Big\rbrace \subseteq \prod_{G\subseteq\Gamma\ \text{endl.}} k\OmegaGJ_G\]
d.\,h. $k\OmegaGJ$ trägt die Topologie der punktweisen Konvergenz: $(x_i)_{i\in I}$ konvergiert gegen $x\in k\OmegaGJ$ genau dann, wenn für alle endlichen $G\subseteq\Gamma$ die Projektionen $\pi_G(x_i)$ gegen $\pi_G(x)$ konvergieren.

\medbreak
Fakt 3: Wenn wir für endliche $G\subseteq\Gamma$ die Partialsummen
\[z_{G} := \sum_{\gamma\in G} x_{s,\gamma} \in k\OmegaGJ\]
definieren, so erhalten wir ein Netz $(z_G)_{G\subseteq\Gamma}$ in $k\OmegaGJ$, wobei die endlichen Teilmengen durch Inklusion partiell geordnet seien. Weiter definieren wir ${z:=\big(\sum_{\gamma\in G} x_{s,\gamma}\big)_G\in k\OmegaGJ}$. Da $\pi_G(z)=\pi_G(z_H)$ für alle $H\supseteq G$ gilt, konvergiert das Netz $(z_G)$ aufgrund von Fakt 2 gegen $z$.

\medbreak
Wir folgern daraus nun a.: Ist $V$ ein stetiger $k\OmegaGJ$"~Modul, dann ist insbesondere für jedes $w\in V$ die Abbildung $k\OmegaGJ\to V, a\mapsto aw$ stetig. Das heißt bedeutet wegen Fakt~3 insbesondere, dass $\lim\limits z_G \cdot w = z\cdot w$ gelten muss. Da $V$ jedoch diskret ist, muss ein endliches $G\subseteq\Gamma$ existieren mit $z_H\cdot w = z\cdot w$ für alle $H\supseteq G$. Indem wir Mengen der Form $H=G\dot{\cup}\Set{\gamma}$ einsetzen, erkennen wir, dass $0 = zw-zw = (z_H-z_G)w = x_{s,\gamma}w$ gelten muss. Deshalb gilt $x_{s,\gamma}w=0$ für alle $\gamma\notin G$.

Indem wir für $w$ nun endlich viele Erzeuger des $A$"~Moduls $V$ einsetzen, erhalten wir eine endliche Menge $G$ mit $x_{s,\gamma}w=0$ für alle $w\in V$ und alle $s\in S, \gamma\notin G$. Das heißt, dass $f$ durch $\pi_G$ faktorisiert.

\medbreak
b. folgt, indem man in a. den regulären $A$"~Rechtsmodul für $V$ einsetzt, denn ist etwa ${f:k\OmegaGJ\to A}$ ein stetiger Morphismus für die diskrete Topologie auf $A$, so ist auch die Multiplikation $m: k\OmegaGJ\times A\to A$ aufgrund von
\[m^{-1}(\Set{x}) = \bigcup_{\substack{y,z\in A \\ x=yz}} f^{-1}(\Set{y})\times\Set{z}\]
stetig und $A$ somit ein stetiger $k\OmegaGJ$"~Modul. Da $\End_{\textbf{Mod}-A}(A) = A$ gilt, erhalten wir daher aus a., dass $f$ durch einen der Quotienten $k\OmegaGJ\to k\OmegaGJ_G$ faktorisiert. Die umgekehrte Richtung ist trivial.
\end{proof}

\begin{remark}
Sei $k$ nun speziell eine $\IZ[\Gamma]$"~Algebra. Analoge Überlegungen wie zu Fakt 3 in obigem Beweis zeigen, dass das Netz $\big(\sum_{\gamma\in G} x_{s,\gamma}v^\gamma\big)_{G\subseteq\Gamma\ \text{endl.}} \in k\OmegaGJ$ gegen $x_s\in k\OmegaGJ$ konvergiert. Wir können also in $k\OmegaGJ$ guten Gewissens $x_s = \sum_{\gamma\in\Gamma} x_{s,\gamma} v^\gamma$ schreiben.
\end{remark}

\begin{remark}
Obwohl man beliebige stetige $k\OmegaGJ$"~Moduln betrachten könnte, werden wir aufgrund des eben bewiesenen Lemmas bevorzugen, ausschließlich diskrete Moduln zu benutzen, da schon ein Beweis der Stetigkeit für beliebige Topologien viel komplizierter wäre. Für diskrete Moduln kann man hingegen einfach dieses Lemma zitieren.
\end{remark}

\subsection{\texorpdfstring{$W$}{W}-Graphen und \texorpdfstring{$\Omega$-Moduln}{Moduln von W-Graph-Algebren}}

\begin{remark}
Die $W$"~Graph"=Algebren sind so definiert, dass in gewisser Weise die $W$"~Graph"=Moduln bis auf eine geeignete Basiswahl $\Omega$"~Moduln entsprechen, dass $W$"~Graphen mit konstanten Gewichten $\OmegaGy$"~Moduln entsprechen und dass Geck-$W$"~Graphen $k\OmegaGJ$"~Mo\-duln entsprechen.
	
Ist etwa $(\mathfrak{C},I,m)$ ein $W$"~Graph mit $W$"~Graph"=Modul $V=k^{(\mathfrak{C})}$, dann realisiert $e_s$ die Projektion auf den von $\Set{x\in\mathfrak{C} | s\in I(x)}$ aufgespannten Unterraum von $V$ und $x_s$ die Anteile von $\omega(T_s)$, die nicht von den Projektionen herkommen, d.\,h. die Nichtdiagonalelemente der Matrix. Mit dieser Interpretation ist intuitiv einleuchtend, dass man aus $W$"~Graphen $\Omega$"~Moduln konstruieren kann. Die nächsten Lemmata haben den Zweck dies zu beweisen.
\index{terms}{W-Graph@$W$-Graph!-Modul}\index{terms}{W-Graph@$W$-Graph!mit konstanten Gewichten}\index{terms}{W-Graph@$W$-Graph!Geck-}

Der kanonische Homomorphismus $k\OmegaGJ\to k\OmegaGy$ ist in dieser Sichtweise eine Manifestation der Tatsache, dass $W$"~Graphen mit konstanten Kantengewichten automatisch Geck-$W$"~Graphen sind.

Wenn $L(s)=1$ für alle $s\in S$ gilt, ist $(-L(s),+L(s))=\Set{0}$. Daher ist in dieser Situation $k\OmegaGJ=k\OmegaGJ_{\Set{0}}=k\OmegaGy$. Das entspricht der Tatsache, dass Geck-$W$"~Graphen und $W$"~Graphen mit konstanten Gewichten im Einparameterfall identisch sind.
\end{remark}

\begin{lemma}[Siehe {\citep[2.4.3]{Gyoja}}]\label{wgraph_alg:iota_morphism}
\index{symbols}{iota@$\iota$}
Die Elemente $\iota(T_s)\in\Omega$ induzieren einen Morphismus von $\IZ[\Gamma]$"~Algebren $\iota: H\to\Omega$.
\end{lemma}
\begin{proof}
Die Definition stellt sicher, dass die Zopfrelationen erfüllt sind. Wir müssen also nur die quadratischen Relationen nachprüfen. Dazu stellen wir zunächst fest, dass $\smash{x_s^2 = (e_s x_s)^2 = e_s \underbrace{x_s e_s}_{=0} x_s = 0}$ gilt. Jetzt vergleichen wir die Terme der quadratischen Relation:
\begin{align*}
	\iota(T_s)^2 &= (-v_s^{-1} e_s + v_s(1-e_s) + x_s)(-v_s^{-1} e_s + v_s(1-e_s) + x_s) \\
	&= v_s^{-2} \underbrace{e_s^2}_{=e_s} - \underbrace{e_s(1-e_s)}_{=0} + (-v_s^{-1}) \underbrace{e_s x_s}_{=x_s} 
	                    -\underbrace{(1-e_s)e_s}_{=0} + v_s^2 \underbrace{(1-e_s)^2}_{=1-e_s} \\
	&\phantom{\text{=}} + v_s \underbrace{(1-e_s)x_s}_{=0} -v_s^{-1} \underbrace{x_s e_s}_{=0} + v_s \underbrace{x_s(1-e_s)}_{=x_s} + \underbrace{x_s^2}_{=0} \\
	&= v_s^{-2} e_s + v_s^2 (1-e_s) + (v_s-v_s^{-1})x_s
\end{align*}
\begin{align*}
	1+(v_s-v_s^{-1})\iota(T_s) &= 1+(-1+v_s^{-2})e_s + (v_s^2 - 1)(1-e_s) + (v_s-v_s^{-1})x_s \\
	&= 1-e_s-(1-e_s) + v_s^{-2} e_s + v_s^2 (1-e_s) + (v_s-v_s^{-1})x_s \\
	&= v_s^{-2} e_s + v_s^2 (1-e_s) + (v_s - v_s^{-1})x_s \qedhere
\end{align*}
\end{proof}

\begin{remark}
Wir werden in Kürze noch sehen, dass $\iota$ in der Tat eine Einbettung ist, $H$ also als eine Unteralgebra von $\Omega$ aufgefasst werden kann. In der Tat ist auch $H\to\Omega\to\IZ[\Gamma]\OmegaGJ$ eine Einbettung.
\end{remark}

\begin{lemmadef}[Siehe {\citep[2.5]{Gyoja}}]\label{wgraph_alg:E_I_and_X_IJ}
\index{terms}{Pfadalgebra}
\index{symbols}{EI@$E_I$}\index{symbols}{XIJs@$X_{IJ}^s$}
In $\Xi_\Gamma$ und seinen Quotienten definieren wir folgende Elemente: Für alle $I\subseteq S$ setze
\[E_I:=\Big( \prod_{s\in I} e_s \Big)\Big(\prod_{s\notin I} (1-e_s)\Big)\]
sowie für $I,J\subseteq S$ setze
\[X_{IJ}^s := E_I x_s E_J\quad\text{bzw.}\]
\[X_{IJ}^{s,\gamma} := E_I x_{s,\gamma} E_J.\]

Es gilt mit diesen Bezeichnungen:
\begin{enumerate}
	\item Die $E_I$ sind paarweise orthogonale Idempotente mit
	\[1 = \sum_{I\subseteq S} E_I \qquad\text{und}\qquad e_s = \smash{\sum_{\substack{I\subseteq S \\ s\in I}} E_I}.\]
	\item Es gilt weiterhin
	\[s\notin I\setminus J \implies X_{IJ}^s = 0, X_{IJ}^{s,\gamma}=0\quad\text{sowie}\]
	\[x_s = \sum_{\substack{I,J\subseteq S \\ s\in I\setminus J}} X_{IJ}^s \quad\text{bzw.}\quad x_{s,\gamma} = \sum_{\substack{I,J\subseteq S \\ s\in I\setminus J}} X_{IJ}^{s,\gamma}.\]
	\item $\Xi_G$ ist die Pfadalgebra $\IZ\mathcal{Q}_G$, wobei $\mathcal{Q}_G$ der Köcher mit den Ecken $I\subseteq S$ und $\abs{\Set{(s,\gamma)\in (I\setminus J) \times(G\cap -G) | 0\leq \gamma < L(s)}}$ Kanten $I \leftarrow J$ sei. Insbesondere ist $\Xi\isomorphic\Xi_{\Set{0}}$ die Pfadalgebra auf dem Köcher mit Eckenmenge $2^S$ und $\abs{I\setminus J}$ Kanten $I\leftarrow J$.
\end{enumerate}
\end{lemmadef}
\begin{proof}
a. Dass sie paarweise orthogonal sind, sieht man der Definition an. Dass sie zur $1$ summieren, folgt durch geschicktes Ausmultiplizieren:
\[1 = 1^\abs{S} = \prod_{s\in S} (e_s + (1-e_s)) = \sum_{I\subseteq S} E_I\]
Die Gleichung für $e_s$ folgt durch Einsetzen in $e_s\cdot 1$.

\medbreak
b. Die zweite Behauptung folgt sofort aus den definierenden Relationen: Wenn $s\notin I$ ist, enthält $E_I x_{s,\gamma}$ den Faktor $(1-e_s)x_{s,\gamma}=0$. Wenn $s\in J$ ist, enthält $x_{s,\gamma} E_J$ den Faktor $x_{s,\gamma} e_s=0$. Die Darstellung für $x_{s,\gamma}$ folgt jetzt durch Einsetzen:
\[x_{s,\gamma} = 1\cdot x_{s,\gamma}\cdot 1 = \sum_{I,J\subseteq S} E_I x_{s,\gamma} E_J = \sum_{\substack{I,J\subseteq S \\ s\in I\setminus J}} X_{IJ}^{s,\gamma}\]

c. Es gilt nach Definition $E_I E_J = \delta_{IJ} E_I$ und $E_I X_{JK}^{s,\gamma} E_L= \delta_{IJ} X_{JK}^{s,\gamma} \delta_{KL}$.

Die Pfadalgebra $\IZ\mathcal{Q}_G$ hat genau diese Präsentation durch Erzeuger und Relationen:
\[\IZ\mathcal{Q}_G=\frac{\IZ\left\langle \widetilde{E}_I, \widetilde{X}_{IJ}^{s,\gamma} \mid I,J\subseteq S, s\in I\setminus J, \gamma\in G\cap -G\cap(-L(s),+L(s))\right\rangle}{\big(\widetilde{E}_I \widetilde{E}_J = \delta_{IJ} \widetilde{E}_I,\; 1=\sum_I \widetilde{E}_I,\; \widetilde{E}_I \widetilde{X}_{JK}^{s,\gamma} \widetilde{E}_L = \delta_{IJ} \widetilde{X}_{JK}^{s,\gamma} \delta_{KL}, X_{IJ}^{s,\gamma}=X_{IJ}^{s,-\gamma}\big)}\]
Daher ist $\smash[t]{\widetilde{E}_I\mapsto E_I, \widetilde{X}_{IJ}^{s,\gamma} \mapsto X_{IJ}^{s,\gamma}}$ ein wohldefinierter Homomorphismus $\IZ\mathcal{Q}\to\Xi$. Man überzeugt sich leicht davon, dass umgekehrt die Elemente 
\[\widetilde{e}_s:=\sum_{\substack{I\subseteq S \\ s\in I}} \widetilde{E}_I \quad\text{und}\quad {\widetilde{x}_{s,\gamma}:=\sum_{\substack{I,J\subseteq S \\ s\in I\setminus J}} \widetilde{X}_{IJ}^{s,\gamma}}\]
die definierenden Relationen von $\Xi_G$ erfüllen und daher den inversen Morphismus induzieren. $\Xi_G$ und $\IZ\mathcal{Q}_G$ sind also isomorph.
\end{proof}

\begin{theorem}[$W$-Graphen und $\Omega$-Moduln, siehe {\citep[2.7]{Gyoja}}]
\index{terms}{W-Graph@$W$-Graph!-Modul}
Es gibt die folgenden Entsprechungen zwischen $W$"~Graphen und Moduln von $W$"~Graph-Algebren:
\begin{enumerate}
	\item Von $W$"~Graphen zu $\Omega$"~Moduln:
	
	Es sei $k$ eine kommutative $\IZ[\Gamma]$"~Algebra und $(\mathfrak{C},I,m)$ ein $W$"~Graph mit Kantengewichten in $k$ und $W$"~Graph"=Modul $V=k^{(\mathfrak{C})}$. Definiere
	\[\omega(e_s) y := \begin{cases} y & s\in I(y) \\ 0 & s\notin I(y) \end{cases} \quad\text{und}\quad \omega(x_s) y := \sum\limits_{\substack{x \in\mathfrak{C} \\ s\in I(x)\setminus I(y)}} m_{xy}^s\cdot x\]
	für alle $s\in S$ und alle $y\in\mathfrak{C}$.

	Mit diesen Bezeichnungen gilt:
	\begin{enumerate}
		\item $\omega$ ist eine Matrixdarstellung $\Omega\to\End_k(V)$ und
		\[\smash{H\xrightarrow{\iota}\Omega\xrightarrow{\omega} \End_k(V)}\]
		ist genau die von $(\mathfrak{C},I,m)$ induzierte Matrixdarstellung von $H$.
		
		\item Ist $A\subseteq k$ ein Teilring und $(\mathfrak{C},I,m)$ ein $W$"~Graph mit konstanten Kantengewichten in $A\subseteq k$, so induziert $\omega$ eine Matrixdarstellung ${A\OmegaGy\to\End_A(A^{(\mathfrak{C})})}$.
		
		\item Ist $A\subseteq k$ ein Teilring mit $A[\Gamma]=k$, $(\mathfrak{C},I,m)$ ein Geck-$W$"~Graph und $\abs{\mathfrak{C}}<\infty$, dann faktorisiert $\omega$ als $k\Omega\to k\OmegaGJ\xrightarrow{\varpi}\End_k(V)$: Schreibe $m_{xy}^s = \sum_{\gamma} m_{xy}^{s,\gamma} \cdot v^\gamma$ mit $m_{xy}^{s,\gamma}\in A$ und setze
		\[\varpi(e_s) y := \begin{cases} y & s\in I(y) \\ 0 & s\notin I(y) \end{cases} \quad\text{und}\quad\varpi(x_{s,\gamma}) y := \smash{\sum\limits_{\substack{x \in\mathfrak{C} \\ s\in I(x)\setminus I(y)}}} m_{xy}^{s,\gamma}\cdot x.\]
		In dieser Situation ist $V$ mit der diskreten Topologie ein stetiger $k\OmegaGJ$-Modul und ${A\OmegaGJ\xrightarrow{\varpi}\End_A(A^{(\mathfrak{C})})}$ macht $A^{(\mathfrak{C})}$ ebenfalls zu einem stetigen, diskreten $A\OmegaGJ$"~Modul.
	\end{enumerate}
	\item Von $\Omega$"~Moduln zu $W$"~Graphen:

	Sei umgekehrt $k$ ein beliebiger kommutativer Ring und $V$ ein $k\Xi$"~Modul mit Darstellung $\omega: k\Xi\to\End_k(V)$.
	
	Dann erhalten wir aus der orthogonalen Zerlegung $1=\sum_I E_I$ eine Zerlegung
	\[V=\bigoplus_{I\subseteq S} E_I V.\]
	Wir nehmen an, dass $E_I V$ als $k$"~Modul frei ist, wählen eine $k$"~Basis $\mathfrak{C}_I\subseteq V_I$ und setzen $\mathfrak{C} := \bigcup_{I\subseteq S} \mathfrak{C}_I$ sowie $I(x):=I$ für $x\in\mathfrak{C}_I$.
	
	Mit diesen Bezeichnungen gilt:
	
	\begin{enumerate}
		\item Ist $k$ speziell eine kommutative $\IZ[\Gamma]$"~Algebra und $V$ ein $k\Omega$"~Modul, dann ist $(\mathfrak{C},I,m)$ ein $W$"~Graph mit Kantengewichten in $k$, wenn wir die Kantengewichte durch die Darstellungsmatrizen $(m_{xy}^s)_{x,y\in\mathfrak{C}}$ von $\omega(x_s)$ bzgl. der Basis $\mathfrak{C}$ definieren. Der zugehörige $W$"~Graph"=Modul ist $(V,\mathfrak{C})$.
		
		\item Ist $V$ ein $k\OmegaGy$"~Modul, dann ist $(\mathfrak{C},I,m)$ ein $W$"~Graph mit konstanten Kantengewichten in $k\subseteq k[\Gamma]$, wenn wir die Kantengewichte durch die Darstellungsmatrizen $(m_{xy}^s)_{x,y\in\mathfrak{C}}$ von $\omega(x_s)$ bzgl. der Basis $\mathfrak{C}$ definieren. Der zugehörige $W$"~Graph"=Modul ist $(k[\Gamma]\otimes_k V,1\otimes\mathfrak{C})$.
		
		\item Ist $V$ ein stetiger, diskreter $k\OmegaGJ$"~Modul und als $k$"~Modul endlich erzeugt, dann ist $(\mathfrak{C},I,m)$ ein Geck-$W$"~Graph mit Kantengewichten in $k[\Gamma]$, wenn wir zunächst $(m_{xy}^{s,\gamma})_{x,y\in\mathfrak{C}}$ als Darstellungsmatrizen von $\omega(x_{s,\gamma})$ bzgl. der Basis $\mathfrak{C}$ und die Kantengewichte dann als $m_{xy}^s := \sum_{\gamma\in\Gamma} m_{xy}^{s,\gamma} \cdot v^\gamma$ definieren. Der zugehörige $W$"~Graph"=Modul ist $(k[\Gamma]\otimes_k V,1\otimes\mathfrak{C})$.
	\end{enumerate}	
\end{enumerate}
\end{theorem}
\begin{proof}
a.i. Dass die $\omega(e_s)$ und $\omega(x_s)$ die Relationen von $\Omega$ erfüllen, ist aus der Definition von $W$"~Graphen klar.

\medbreak
Für a.ii. müssen wir zeigen, dass $\omega(e_s)$ und $\omega(x_s)$ unter den angegebenen Bedingungen die Relationen von $A\OmegaGy$ ebenfalls erfüllen. Das folgt aus der universellen Eigenschaft von $A\OmegaGy=A\OmegaGJ_{\Set{0}}$ (siehe Lemma \ref{wgraph_alg:univ_prop_OmegaGJ}), da $\omega(x_s)$ nach Voraussetzung in $A^{\mathfrak{C}\times\mathfrak{C}}$ ist.

\medbreak
Für a.iii. sei also $k=A[\Gamma]$ und $m_{x,y}^s = \sum_{\gamma\in\Gamma} m_{x,y}^{s,\gamma} v^\gamma$ mit $m_{xy}^{s,\gamma}\in A$ wie in der Behauptung. Ist nun $\mathfrak{C}$ endlich, dann gibt es nur endlich viele nichtverschwindende $m_{x,y}^{s,\gamma}$, d.\,h. wir können eine endliche Menge $G\subseteq\Gamma$ mit $\gamma\notin G \implies m_{x,y}^{s,\gamma} = 0$ finden.

Weil die $\varpi(\iota_G(T_s))=\omega(\iota_G(T_s))$ die Zopf-Relationen erfüllen, folgt aus der universellen Eigenschaft \ref{wgraph_alg:univ_prop_OmegaGJ}, dass $\varpi$ ein wohldefinierter Morphismus $A\OmegaGJ_G\to\End_A(A^{(\mathfrak{C})})$ ist.

Daher faktorisiert $\omega$ als $k\Omega\to k\OmegaGJ\to k\OmegaGJ_G\xrightarrow{\varpi}\End_k(k^{(\mathfrak{C})})$ wie behauptet. Gleichermaßen faktorisiert die Multiplikation $k\OmegaGJ\times V\to V$ als $k\OmegaGJ\times V\to k\OmegaGJ_G\times V\to V$ und ist daher stetig. Somit ist $V$ ein stetiger, diskreter $k\OmegaGJ$"~Modul. Analog ist $A^{(\mathfrak{C})}$ ein diskreter $A\OmegaGJ$"~Modul.

\bigbreak
b. Hier muss geprüft werden, dass die Konstruktionen wirklich $W$"~Graphen liefern. Die Eigenschaft $m_{xy}^s \neq 0 \implies s\in I(x)\setminus I(y)$ folgt dabei aus ${X_{IJ}^s \neq 0 \implies s\in I\setminus J}$. Die Konstruktion der Kantengewichte sichert außerdem zu, dass die durch $(\mathfrak{C},I,m)$ definierten Matrizen gleich $\omega(\iota(T_s))$ sind und sie daher die Zopfrelationen erfüllen.

\medbreak
In b.ii. ist offenkundig, dass ein $W$"~Graph mit konstanten Kantengewichten konstruiert wurde.

\medbreak
Für b.iii. halten wir zum einen fest, dass $V$ als diskreter $k\OmegaGJ$"~Modul wegen Lemma \ref{wgraph_alg:continuous_reps} sogar ein $k\OmegaGJ_G$"~Modul ist für ein hinreichend großes, aber endliches $G\subseteq\Gamma$. Die Relationen von $\OmegaGJ_G$ sichern dann zu, dass $\omega(x_{s,\gamma})=0$ für $\gamma\notin G\cap(-L(s),+L(s))$ gilt. Zum einen ist daher die Summe $\sum_{\gamma\in\Gamma} m_{xy}^{s,\gamma} \cdot v^\gamma$ wohldefiniert. Zum anderen folgt daraus, dass die Kantengewichte die Gradschranke erfüllen, die in der Definition von Geck-$W$"~Graphen gefordert wird. Die Relation $x_{s,\gamma}=x_{s,-\gamma}$ sichert weiterhin, dass auch die Symmetriebedingung erfüllt ist.
\end{proof}

\begin{corollary}
\index{terms}{Kazhdan-Lusztig!-Basis}\index{terms}{Kazhdan-Lusztig!-$W$-Graph}\index{terms}{Einparameterfall}
Sei $(W,S,L)$ eine Coxeter"=Gruppe"=mit"=Gewicht und $k$ eine kommutative $\IZ[\Gamma]$"~Algebra.
\begin{enumerate}
	\item $\iota: kH\to k\Omega$ ist injektiv.
	\item Die Idempotente $E_I$ sind, aufgefasst als Elemente von $k\Omega$, von Null verschieden.
	\item a. und b. gelten auch, wenn $\abs{W}<\infty$ ist und $k\Omega$ durch $k\OmegaGJ$ oder $k\OmegaGJ_G$ für hinreichend großes $G\subseteq\Gamma$ ersetzt wird.
	\item a. und b. gelten auch im Einparameterfall, wenn $k\Omega$ durch $k\OmegaGy$ ersetzt wird.
\end{enumerate}
\end{corollary}
\begin{proof}
Wir betrachten den Kazhdan"=Lusztig"=$W$"~Graphen $(kH,(C_w))$. Da es sich um einen $W$"~Graphen handelt, können wir $kH\to\End_k(kH)$ dank des eben bewiesenen Satzes durch $\Omega$ faktorisieren. Wenn $W$ endlich ist, können wir sogar durch $k\OmegaGJ$ faktorisieren, da es sich um einen Geck-$W$"~Graphen handelt, und, wenn wir im Einparameterfall sind, durch $k\OmegaGy$. Da $kH\to\End_k(kH)$ injektiv ist, erhalten wir dann, dass $\iota$ in allen drei Fällen injektiv sein muss.

\medbreak
Für b. genügt die Feststellung, dass die $I$-Invarianten des $W$"~Graphen durch die Linksabstiegsmengen $I(w):=D_L(w)$ gegeben sind. Daher kommt in der Tat jede Teilmenge $I\subseteq S$ auch in diesem $W$"~Graphen vor und somit kann nicht $E_I=0$ sein.
\end{proof}

\begin{convention}
Insbesondere werden wir $kH$ von nun an als Unteralgebra von $k\Omega$ bzw. $k\OmegaGJ$ auffassen, ohne $\iota$ explizit zu erwähnen, wann immer das nützlich erscheint.
\end{convention}

\subsection{Verschiedene Morphismen}

\begin{lemmadef}[Funktorialität von $W$-Graph-Algebren]\label{wgraph:functors}
\index{terms}{W-Graph@$W$-Graph!-Algebra}\index{terms}{W-Graph@$W$-Graph!Zurückziehen von}
Es seien $(W_1,S_1,L_1)$ und $(W_2,S_2,L_2)$ zwei Coxeter"=Gruppen"=mit"=Gewicht mit gemeinsamer Gewichtsgruppe $\Gamma$. Weiter sei $k$ ein kommutativer Ring.

$\phi: W_1\to W_2$ sei ein Gruppenhomomorphismus, der sowohl ${\phi(S_1)\subseteq S_2\cup\Set{1}}$ als auch ${L_2(\phi(s_1))=L_1(s_1)}$ für alle $s_1\in S_1$ erfüllt.

\medbreak
Wir verwenden die Kurzschreibweisen $\Omega_i$ für $\Omega(W_i,S_i,L_i)$, $H_i$ für $H(W_i,S_i,L_i)$ usw. und bezeichnen mit $f$ die von $\phi$ induzierte Abbildung $S_1\cup\Set{1}\to S_2\cup\Set{1}$.

Mit diesen Bezeichnungen gilt:
\begin{enumerate}
	\item Wenn $k$ eine $\IZ[\Gamma]$"~Algebra ist, existiert ein $k$"~Algebramorphismus $k\Omega_1 \to k\Omega_2$, und wenn $k$ ein beliebiger Ring ist, existiert ein Morphismus $k\OmegaGJ_{1,G}\to k\OmegaGJ_{2,G}$ für alle endlichen $G\subseteq\Gamma$ sowie ein stetiger Morphismus $k\OmegaGJ_1\to k\OmegaGJ_2$ mit
	\[e_s\mapsto e_{f(s)} \quad\text{und}\quad x_s \mapsto x_{f(s)} \quad\text{bzw.}\quad x_{s,\gamma} \mapsto x_{f(s),\gamma}\]
	für alle $s\in S_1$ und $\gamma\in\Gamma$, wobei wir die Konvention $e_1=x_1=x_{1,\gamma}=0$ verwenden wollen.
	
	Der Einfachheit halber seien diese Morphismen ebenfalls mit $\phi$ bezeichnet. Es gilt:
	\begin{enumerate}
		\item 		
		$\displaystyle \phi(E_I) = \begin{cases} \sum_{\substack{K\subseteq S_2 \\ f^{-1}(K)=I}} E_K & \text{falls }f^{-1}(f(I))=I \\ 0 & \text{sonst}\end{cases}$
		
		und
		
		$\displaystyle \phi(X_{IJ}^s) = \begin{cases} \sum_{\substack{K,L\subseteq S_2 \\ f^{-1}(K)=I \\ f^{-1}(L) = J}} X_{KL}^{f(s)} & \text{falls }f^{-1}(f(I))=I, f^{-1}(f(J))=J \\ 0 &\text{sonst}\end{cases}$
		
		bzw.
		
		$\displaystyle \phi(X_{IJ}^{s,\gamma}) = \begin{cases} \sum_{\substack{K,L\subseteq S_2 \\ f^{-1}(K)=I \\ f^{-1}(L) = J}} X_{KL}^{f(s),\gamma} & \text{falls }f^{-1}(f(I))=I, f^{-1}(f(J))=J \\ 0 &\text{sonst}\end{cases}$.
		\item Falls $k$ eine $\IZ[\Gamma]$"~Algebra ist, dann gilt weiter $\phi(\iota_1(T_w)) = v^{g(w)} \iota_2(T_{\phi(w)})$ für alle $w\in W_1$, wobei die Gewichtsfunktion $g:W_1\to\Gamma$ durch $g(s):=L(s)$ für $f(s)=1$ und $g(s):=0$ sonst definiert sei.
	\end{enumerate}
	\item Ist $k$ eine $\IZ[\Gamma]$"~Algebra und $(\mathfrak{C},I,m)$ ein $W_2$"~Graph mit Gewichten in $k$, so ist $(\mathfrak{C},f^\ast I,f^\ast m)$ ein $W_1$"~Graph mit Gewichten in $k$, wobei
	\[(f^\ast I)(x) := f^{-1}(I(x)) \quad\text{und}\quad (f^\ast m)_{xy}^{s_1} := \begin{cases} m_{xy}^{f(s_1)} & \text{falls } f(s_1)\neq 1 \\ 0 &\text{falls }f(s_1)=1\end{cases}\]
	für alle $x,y\in\mathfrak{C}$ und $s_1\in S_1$ sei. Die $W_1$"~Graph-Darstellungen, die auf diese Weise entstehen, sind genau diejenigen, die durch $k\Omega_1\xrightarrow{\phi}k\Omega_2$ faktorisieren.
\end{enumerate}
\end{lemmadef}
\begin{proof}
$\phi$ ist auf jeden Fall als Morphismus $k\Xi_1\to k\Xi_2\to k\Omega_2$ wohldefiniert. Zunächst halten wir fest, dass für $\IZ[\Gamma]$"~Algebren $k$ stets
\begin{align*}
	\phi(\iota_1(T_s)) &= \phi(-v_s^{-1} e_s + v_s (1-e_s) + x_s) \\
	&= -v^{-L(s)} e_{f(s)} + v^{L(s)} (1-e_{f(s)}) + x_{f(s)} \\
	&= -v^{-L(f(s))} e_{f(s)} + v^{L(f(s))} (1-e_{f(s)}) + x_{f(s)} \\
	&= \begin{cases}\iota_2(T_{f(s)}) & \text{falls }f(s)\in S_2 \\ v_s & \text{falls }f(s)=1\end{cases} \tag{$\ast$}\label{eq:wgraph_alg:morphisms1}
\end{align*}
für alle $s\in S_1$ gilt.

\medbreak
Seien nun $s,t\in S_1$ beliebig, aber fest. Wir nehmen zunächst $f(s),f(t)\in S_2$ an. Da $W_1\xrightarrow{\phi} W_2$ ein Gruppenhomomorphismus ist, folgt, dass $m_2:=\ord(f(s)f(t))$ ein Teiler von $m_1:=\ord(st)$ ist. Die Gültigkeit der Zopfrelation $\Delta_{m_2}(x,y)=0$ impliziert daher die Gültigkeit von $\Delta_{m_1}(x,y)=0$. Daraus folgt
\[\phi(\Delta_{m_1}(\iota_1(T_s),\iota_1(T_t))) = \Delta_{m_1}(\phi(\iota_1(T_s)),\phi(\iota_1(T_t))) \overset{\eqref{eq:wgraph_alg:morphisms1}}{=} \Delta_{m_1}(\iota_2(T_{f(s)}),\iota_2(T_{f(t)}))\]
und das ist Null in $k\Omega_2$.

Wenn $f(s)=1$, aber $f(t)\neq 1$ gilt, ist $\phi(\iota_1(T_s))=v_s$ zentral in $k\Omega_2$ und es gilt ${m_1=\ord(st)\in 2\IN\cup\Set{\infty}}$.  Daraus ergibt sich ebenfalls $0=\phi(\Delta_{m_1}(\iota_1(T_s),\iota_2(T_t)))$ für diese $s,t$. Dies gilt analog, wenn $f(s)\neq 1$ und $f(t)=1$ ist. Im letzten Fall $f(s)=f(t)=1$ ist entweder $m_1$ gerade und
\[\Delta_{m_1}(\phi(\iota_1(T_s)),\phi(\iota_1(T_t))) = \Delta_{m_1}(v_s,v_t) = v_s^{m_1/2} v_t^{m_1/2} - v_t^{m_1/2} v_s^{m_1/2} = 0\]
oder $m_1$ ist ungerade und somit $v_s=v_t$, weshalb auch in diesem Fall $\Delta_{m_1}(v_s,v_t) = 0$ folgt.

\medbreak
Daraus folgt zum einen, dass $\phi$ als Homomorphismus $k\Omega_1\to k\Omega_2$ wohldefiniert ist. Wenn wir nun speziell die $\IZ[\Gamma]$"~Algebra $k[\Gamma]$ betrachten für einen beliebigen kommutativen Ring $k$, dann folgt aus dieser Überlegung auch, dass $\phi: k[\Gamma]\Xi_1 \to k[\Gamma]\Xi_2 \to k[\Gamma]\OmegaGJ_{2,G}$ die Zopfrelationen respektiert. Damit ergibt sich aus der universellen Eigenschaft \ref{wgraph_alg:univ_prop_OmegaGJ}, dass $\phi$ ein wohldefinierter Morphismus $k\OmegaGJ_{1,G}\to k\OmegaGJ_{2,G}$ ist, der offenbar mit den Projektionen verträglich ist und daher einen wohldefinierten, stetigen Morphismus $k\OmegaGJ_1\to k\OmegaGJ_2$ induziert.

\bigbreak
Die weiteren Aussagen in a. sind leicht nachzuprüfen. Es gilt etwa:
\begin{align*}
	\phi(E_I) &= \phi\Big(\prod_{s\in I} e_s \prod_{s\in S_1\setminus I} (1-e_s)\Big) \\
	&= \prod_{s\in I} e_{f(s)} \prod_{s\in S_1\setminus I} (1-e_{f(s)}) \\
	&= \prod_{t\in f(I)} e_t \prod_{t\in f(S_1\setminus I)} (1-e_t) \quad\text{da $e_t$ und $1-e_t$ Idempotente sind} \\
\intertext{Falls nun $f^{-1}(f(I))\supsetneq I$ ist, also ein $t\in f(I)\cap f(S_1\setminus I)$ existiert oder $1\in f(I)$ ist, dann ist $\phi(E_I)=0$. Ansonsten formen wir weiter um:}
	&= \prod_{t\in f(I)} e_t \prod_{t\in f(S_1)\setminus f(I)} (1-e_t) \prod_{t\in S_2\setminus f(S_1)} (e_t + (1-e_t)) \\
	&= \sum_{K'\subseteq S_2\setminus f(S_1)} \prod_{t\in f(I)} e_t \prod_{t\in f(S_1)\setminus f(I)} (1-e_t) \prod_{t\in K'} e_t \prod_{t\in S_2 \setminus (K'\cup f(S_1))} (1-e_t) \\
	&= \sum_{\substack{K\subseteq S_2 \\ f^{-1}(K)=I}} \prod_{t\in K} e_t \prod_{t\in S_2\setminus K} (1-e_t) \\
	&= \sum_{\substack{K\subseteq S_2 \\ f^{-1}(K)=I}} E_K
\end{align*}
Daraus folgt auch die Darstellung von $\phi(X_{IJ}^s)$.

\medbreak
Dass über $\IZ[\Gamma]$"~Algebren $k$ nun $\phi(T_s)=v^{g(s)} T_{f(s)}$ gilt, ist gerade die Aussage \eqref{eq:wgraph_alg:morphisms1}. Per Induktion nach der Länge folgt $\phi(T_w) = v^{g(w)} T_{\phi(w)}$ für alle $w\in W_1$, wie behauptet.

\bigbreak
Für b. müssen wir nur nachprüfen, dass die Matrizen $\omega_1(T_s)$, die zum neuen Graphen $(\mathfrak{C},f^\ast I,f^\ast m)$ gehören, die Zopfrelationen erfüllen, oder äquivalent, dass $\omega_1(e_s)$ und $\omega_1(x_s)$ die Relationen von $\Omega$ erfüllen. Es gilt aber $s\in (f^\ast I)(x) \iff f(s)\in I(x)$, d.\,h. $\omega_1(e_s) = \omega_2(e_{f(s)}) = \omega_2(\phi(e_s))$. Ebenso gilt 
\[s\in (f^\ast I)(x)\setminus (f^\ast I)(y) = f^{-1}(I(x)\setminus I(y)) \iff f(s)\in I(x)\setminus I(y),\]
woraus $\omega_1(x_s) = \omega_2(x_{f(s)}) = \omega_2(\phi(x_s))$ folgt.

Es ist aber $\omega_2\circ\phi$ ein Algebrahomomorphismus, also erfüllen $\omega_1(e_s)$ und $\omega_1(x_s)$ die Relationen von $\Omega_1$, d.\,h. der Graph ist wirklich ein $W_1$"~Graph und die Darstellung ${k\Omega_1\to\End_k(k^{(\mathfrak{C})})}$ faktorisiert durch $\phi$. Umgekehrt wird durch $\omega_2\circ\phi$ eine $W_1$"~Graph-Darstellung mit $W_1$"~Gra\-phen $(\mathfrak{C},f^\ast I,f^\ast m)$ definiert.
\end{proof}

\begin{example}[Graphautomorphismen]\label{wgraph_alg:def:graph_auto}
\index{terms}{Graphautomorphismus}
\index{symbols}{alpha@$\alpha$}
Wenn $\alpha:S\to S$ ein Graphautomorphismus des Dynkin"=Diagramms ist, d.\,h. für alle $s,t\in S$ gilt ${\ord(\alpha(s)\alpha(t))=\ord(st)}$, und wenn zusätzlich $L(\alpha(s))=L(s)$ erfüllt ist, setzt sich $\alpha$ zu einem Automorphismus auf den $W$"~Graph"=Algebren fort, der
\begin{align*}
	e_s &\mapsto e_{\alpha(s)} \\
	x_{s,\gamma} &\mapsto x_{\alpha(s),\gamma} \\
	E_I &\mapsto E_{\alpha(I)} \\
	X_{IJ}^{s,\gamma} &\mapsto X_{\alpha(I)\alpha(J)}^{\alpha(s),\gamma} \\
	T_s &\mapsto T_{\alpha(s)}
\end{align*}
abbildet.
\end{example}

\begin{example}[Parabolische $W$-Graph-Algebren]\label{wgraph_alg:def:parabolic}
\index{terms}{Parabolische Unteralgebra}\index{terms}{Restriktion}
\index{symbols}{jI@$j_I$}
Ist $I\subseteq S$, dann setzt sich die Inklusionsabbildung $I\hookrightarrow S$ zum \udot{parabolischen Morphismus} $j_I: \Omega(W_I,I,L_{|I})\to\Omega(W,S,L)$ (und analog für die anderen $W$-Graph-Algebren) fort. Dabei wird wie folgt abgebildet:
\begin{align*}
	e_s &\mapsto e_s \\
	x_{s,\gamma} &\mapsto x_{s,\gamma} \\
	E_K &\mapsto \sum_{\substack{K'\subseteq S \\ K'\cap I = K}} E_{K'} \\
	X_{KL}^{s,\gamma} &\mapsto \sum_{\substack{K',L'\subseteq S \\ K'\cap I=K, L'\cap I=L}} X_{K'L'}^{s,\gamma} \\
	T_s &\mapsto T_s
\end{align*}

Zurückziehen von $\Omega$"~Moduln entlang von $j_I$ entspricht genau der in Lemma \ref{wgraphs:constructions} definierten parabolischen Restriktion von $W$"~Graphen.
\end{example}

\begin{conjecture}
Es liegt auf der Hand, zu vermuten, dass diese parabolischen Morphismen in der Tat Einbettungen sind. Auf der Ebene der Hecke"=Algebren ist dies ja immer der Fall. In Spezialfällen kann man das tatsächlich mit Hilfe des obigen Satzes auch für die $W$"~Graph"=Algebren beweisen, wie folgendes Beispiel zeigt.
\end{conjecture}

\begin{example}[Inflation von $W$-Graphen]\label{wgraph_alg:inflation}
\index{terms}{Inflation}
Ist $I\subseteq S$ derart, dass alle Kanten zwischen $I$ und $S\setminus I$ im Dynkin"=Diagramm gerades oder unendliches Kantengewicht haben (mit anderen Worten ist kein Element von $I$ zu einem Element von $S\setminus I$ konjugiert), dann induziert die Identität auf $S$ einen Homomorphismus $W \to W_I \times W_{S\setminus I}$. Indem wir mit der Projektion auf $W_I$ verketten, erhalten wir somit eine Faktorisierung von $W_I\xrightarrow{\id} W_I$ als
\[W_I \hookrightarrow W \to W_I \times W_{S\setminus I} \twoheadrightarrow W_I,\]
die sich auf Ebene der $W$"~Graph"=Algebren dank des obigen Satzes zu einer Faktorisierung der Identität $\Omega(W_I)\xrightarrow{\id}\Omega(W_I)$ als
\[\Omega(W_I) \to \Omega(W) \to \Omega(W_I\times W_{S\setminus I}) \to \Omega(W_I)\]
fortsetzt. Insbesondere ist unter den gegebenen Voraussetzungen der parabolische Morphismus $\Omega(W_I)\to\Omega(W)$ injektiv.

\medbreak
Das Zurückziehen entlang von $\pi_I: W\twoheadrightarrow W_I$ entspricht auf Ebene der Gruppen der Inflation von Darstellungen entlang von $\pi_I$. Auf Ebene der $W$"~Graph"=Algebren sagt uns Teil b. des obigen Satzes, dass $\pi_I$ uns erlaubt, $W_I$"~Graphen als $W$"~Graphen aufzufassen. Indexmengen oder Kantengewichte werden dabei nicht verändert.

\medbreak
Beispielsweise ist so jeder $A_n$-Graph automatisch auch ein $B_{n+1}$-Graph.
\end{example}

\begin{lemmadef}[Ein Antiautomorphismus von $W$-Graph-Algebren]\label{wgraph_alg:duality}
\index{terms}{W-Graph@$W$-Graph!dualer}
\index{symbols}{delta@$\delta$}
Sei $(W,S,L)$ eine Coxeter"=Gruppe"=mit"=Gewicht und $k$ ein kommutativer Ring.

Wenn $k$ eine kommutative $\IZ[\Gamma]$"~Algebra ist, existiert ein Antiautomorphismus von $k\Omega(W,S,L)$, und wenn $k$ ein beliebiger kommutativer Ring ist, existiert ein Antiautomorphismus von $k\OmegaGJ_G(W,S,L)$ für alle endlichen $G\subseteq\Gamma$ sowie ein stetiger Antiautomorphismus von $k\OmegaGJ(W,S,L)$ mit
\[e_s \mapsto 1-e_s \quad\text{und}\quad x_s \mapsto -x_s \quad\text{bzw.}\quad x_{s,\gamma} \mapsto -x_{s,\gamma}\]
für alle $s\in S$ und $\gamma\in\Gamma$. Wir bezeichnen diese Antiautomorphismen jeweils mit $\delta$. Für diese gilt
\begin{enumerate}
	\item $\delta(E_I) = E_{I^c}$ und $\delta(X_{IJ}^s)=X_{J^c I^c}^s$ bzw. $\delta(X_{IJ}^{s,\gamma}) = -X_{J^c I^c}^{s,\gamma}$, wobei ${}^c:2^S\to 2^S$ die Komplementoperation bezeichne.
\end{enumerate}
sowie, falls $k$ eine $\IZ[\Gamma]$"~Algebra ist, außerdem
\begin{enumerate}[resume]
	\item $\delta(\iota(T_w))=(-1)^{l(w)} \iota(T_w)^{-1}$ für alle $w\in W$.
	\item Ist $(\mathfrak{C},I,m)$ ein $W$"~Graph mit Gewichten in $k$, $\abs{\mathfrak{C}}<\infty$, $W$"~Graph"=Modul $V=k^{\mathfrak{C}}$ und $W$"~Graph-Darstellung $\omega:\Omega\to k^{\mathfrak{C}\times\mathfrak{C}}$, so ist $\omega^\dagger = \omega(\delta(\cdot))^\text{Tr}$ ebenfalls eine $W$"~Graph"=Darstellung, siehe \ref{wgraphs:constructions}.
\end{enumerate}
\end{lemmadef}
\begin{proof}
Dass $\delta$ als Abbildung $k\Xi\to k\Xi$ bzw. $k\Xi_G\to k\Xi_G$ wohldefiniert ist, ist leicht einzusehen. Dass sie auf der Heckealgebra die angegebene Form hat, kann man wie folgt nachrechnen:
\begin{align*}
	\delta(\iota(T_s)) &= \delta(-v_s^{-1}e_s+v_s(1-e_s)+x_s) \\
	&=-v_s^{-1}(1-e_s) + v_s e_s - x_s \\
	&=+v_s^{-1}e_s-v_s(1-e_s)-x_s + (v_s - v_s^{-1}) \\
	&=-\iota(T_s)+(v_s-v_s^{-1}) \\
	&=-\iota(T_s)^{-1}
\end{align*}

Es gilt nun ganz allgemein $\Delta_m(x,y) = 0 \iff \Delta_m(-x^{-1},-y^{-1}) = 0$ für invertierbare Elemente $x,y$ eines Ringes. Da die Zopfrelationen von Antihomomorphismen erhalten werden, ist $\delta$ ein wohldefinierter Antihomomorphismus $k\Omega\to k\Omega$.

\medbreak
Dass $\delta$ auch auf $k\OmegaGJ_G$ wohldefiniert ist, wenn $k$ ein beliebiger Ring ist, folgt wie zuvor aus der universellen Eigenschaft \ref{wgraph_alg:univ_prop_OmegaGJ}, indem wir die vorherige Überlegung auf $k[\Gamma]$ statt $k$ anwenden. Da $\delta$ offenbar mit den Projektionen $k\OmegaGJ_{G_1}\to k\OmegaGJ_{G_2}$ verträglich ist, ist auch $\delta: k\OmegaGJ\to k\OmegaGJ$ wohldefiniert.

\medbreak
Alle anderen Aussagen folgen nun unmittelbar durch Einsetzen in die Definitionen.
\end{proof}

\begin{lemmadef}\label{wgraph_alg:tensor_prod}
\index{symbols}{tau@$\tau$}
Sei $(W,S,L)$ eine Coxeter"=Gruppe"=mit"=Gewicht und $(W,S)$ reduzibel, etwa $S=S_1 \coprod S_2$ und ${W=W_1\times W_2}$. Definiere dann $L_i:=L_{|S_i}$ und schreibe zur Abkürzung entsprechend $H_i:=H(W_i,S_i,L_i)$, $\Omega_i:=\Omega(W_i,S_i,L_i)$ usw.
	
Sei nun $k$ ein kommutativer Ring. Wenn $k$ eine $\IZ[\Gamma]$"~Algebra ist, gibt es einen Morphismus $k\Omega\to k\Omega_1\otimes_k k\Omega_2$, und wenn $k$ ein beliebiger Ring ist, gibt es einen Morphimus $k\OmegaGJ_G\to k\OmegaGJ_{1,G}\otimes_k k\OmegaGJ_{2,G}$ für alle endlichen $G\subseteq\Gamma$ mit
\begin{alignat*}{3}
	e_s &\mapsto \begin{cases} e_s \otimes 1 & \text{falls }s\in S_1 \\ 1\otimes e_s & \text{falls }s\in S_2\end{cases} \quad&\text{und} \\
	x_s &\mapsto \begin{cases} x_s \otimes 1 & \text{falls }s\in S_1 \\ 1\otimes x_s & \text{falls }s\in S_2\end{cases} \quad&\text{bzw.} \\
	x_{s,\gamma}&\mapsto\begin{cases} x_{s,\gamma} \otimes 1 & \text{falls }s\in S_1 \\ 1\otimes x_{s,\gamma} & \text{falls }s\in S_2\end{cases}.
\end{alignat*}
Diese Morphismen seien jeweils mit $\tau$ bezeichnet. Dann gilt:
\begin{enumerate}
	\item $\tau(E_I) = E_{I_1} \otimes E_{I_2}$, wobei $I_i:=I\cap S_i$ sei,
	\item $\displaystyle\tau(X_{IJ}^{s}) = \begin{cases} X_{I_1 J_1}^{s} \otimes E_{I_2}E_{J_2} & \text{falls }s\in S_1 \\ E_{I_1}E_{J_1} \otimes X_{I_2 J_2}^{s} & \text{falls }s\in S_2\end{cases}$
	
	bzw.
	
	$\displaystyle\tau(X_{IJ}^{s,\gamma}) = \begin{cases} X_{I_1 J_1}^{s,\gamma} \otimes E_{I_2}E_{J_2} & \text{falls }s\in S_1 \\ E_{I_1}E_{J_1} \otimes X_{I_2 J_2}^{s,\gamma} & \text{falls }s\in S_2\end{cases}$.
\end{enumerate}
sowie, falls $k$ eine $\IZ[\Gamma]$"~Algebra ist, weiterhin auch
\begin{enumerate}[resume]
	\item $\displaystyle\tau(\iota(T_s)) = \begin{cases} \iota_1(T_s)\otimes 1 & \text{falls }s\in S_1 \\ 1\otimes\iota_2(T_s) & \text{falls }s\in S_2\end{cases}$,
	
	d.\,h. $\tau$ schränkt sich zum kanonischen Isomorphismus $kH \to kH_1 \otimes_k kH_2$ ein.
\end{enumerate}
\end{lemmadef}
\begin{proof}
Es ist erneut klar, dass $\tau: k\Xi\to k\Xi_1\otimes_k k\Xi_2$ bzw. $\tau: k\Xi_G\to k\Xi_{1,G}\otimes_k k\Xi_{2,G}$ wohldefiniert ist. Es ist ebenfalls aus der Definition klar, dass
\[\tau(\iota(T_s)) = \begin{cases} \iota_1(T_s)\otimes 1 & \text{falls } s\in S_1 \\ 1\otimes\iota_2(T_s) &\text{falls }s\in S_2\end{cases}\]
gilt.

Wenn nun $s,t\in S_1$ oder $s,t\in S_2$ sind, dann sind die Zopfrelationen für $\tau(\iota(T_s))$ und $\tau(\iota(T_t))$ eine Konsequenz der Zopfrelationen für $\iota(T_s)$ und $\iota(T_t)$. Wenn $s\in S_1$ und $t\in S_2$ oder umgekehrt ist, dann ist $m_{st}=2$ und $\Delta_2(\tau(\iota(T_s)),\tau(\iota(T_t))) = \Delta_2(T_s\otimes 1,1\otimes T_t)=0$. Also induziert $\tau$ einen Morphismus $k\Omega\to k\Omega_1\otimes_k k\Omega_2$, wie behauptet.

\medbreak
Erneut folgt aus dieser Überlegung angewandt auf $k[\Gamma]$ für einen beliebigen kommutativen Ring $k$ und der universellen Eigenschaft, dass $\tau$ als Abbildung $k\OmegaGJ_G\to k\OmegaGJ_{1,G}\otimes_k k\OmegaGJ_{2,G}$ wohldefiniert ist.

Alle anderen Aussagen folgen wieder durch Einsetzen in die Definitionen.
\end{proof}

\begin{remark}
Es folgt aus diesen Aussagen zwar, dass $\tau$ einen Homomorphismus
\[k\OmegaGJ \to \lim_{\substack{\longleftarrow \\ G\subseteq\Gamma\,\text{endl.}}} (k\Omega_{1,G}\otimes_k k\Omega_{2,G})\]
induziert. Da Tensorprodukte jedoch i.\,A. nicht mit Limites vertauschen, folgt hieraus noch nicht, dass $\tau$ auch einen Homomorphismus $k\OmegaGJ\to k\OmegaGJ_1\otimes k\OmegaGJ_2$ induziert.

Wenn man nur an endlichdimensionalen, stetigen Darstellungen interessiert ist, dann haben wir jedoch bereits festgestellt (Lemma \ref{wgraph_alg:continuous_reps}), dass jede solche Darstellung durch eine der Projektionen $k\OmegaGJ\to k\OmegaGJ_G$ faktorisiert. Daher ist der Homomorphismus $\tau$ trotzdem nützlich.
\end{remark}
\section{Lusztigs Homomorphismus und \texorpdfstring{$W$}{W}-Graph-Algebren}

\begin{convention}
Wir fixieren für diesen Abschnitt einen $L$-guten Ring (siehe Definition \ref{J_alg:def:L_good}) $\IZ_W\subseteq R\subseteq\IC$ und setzen $F:=\QuotFld(R)$ sowie $K:=F(\Gamma)$.
\end{convention}

\begin{remark}
Folgender Satz ist wesentlich für die Verbindung von $W$"~Graph"=Algebren und der Hecke"=Algebra. Der Beweis kombiniert die Ideen von {\citep[2.7.11]{geckjacon}} und \citep[2.9]{Gyoja}.
\end{remark}

\begin{theorem}[Faktorisierung von Lusztigs Homomorphismus durch $\OmegaGJ$]
\index{terms}{Lusztig-Isomorphismus}\index{terms}{Kazhdan-Lusztig!-$\mu$}\index{terms}{Kazhdan-Lusztig!-$W$-Graph}
Es gelte $(\spadesuit)$. Wir führen die Bezeichnung $m_{xy}^s$ für die Kantengewichte des Kazhdan-Lusztig-$W$"~Gra\-phen ein, d.\,h. wir setzen
\[m_{xy}^s := \begin{cases}
1 & \text{falls}\, y<sy=x \\
(-1)^{l(x)+l(y)+1}\mu_{xy}^s & \text{falls}\, sx<x<y<sy \\
0 & \text{sonst}
\end{cases}.\]
Weiterhin seien mit $m_{xy}^{s,\gamma}$ die Koeffizienten von $m_{xy}^s\in\IZ[\Gamma]$ bezeichnet. Mit diesen Bezeichnungen gilt:
\begin{enumerate}
	\item Lusztigs Homomorphismus $R[\Gamma]H \xrightarrow{\phi} R[\Gamma]J$ faktorisiert durch $R[\Gamma]H\xhookrightarrow{\iota}R[\Gamma]\OmegaGJ$. Genauer wird durch
	\[q(e_s) = \sum_{\substack{d\in\mathcal{D} \\ sd<d}} n_d t_d\]
	\[q(x_{s,\gamma}) = \sum_{\substack{z\in W,d\in\mathcal{D} \\ z \sim_\mathcal{LR} d \\ sz<z, d<sd}} n_d m_{zd}^{s,\gamma} \cdot t_z\]
	ein Morphismus $q: R\OmegaGJ_G\to RJ$ definiert für ein hinreichend großes, endliches $G\subseteq\Gamma$, welcher $\phi=q\circ\iota$ erfüllt.
	\item Sind $n_d$ und $\gamma_{x,y,z}\in\IZ$ (was z.B. unter Annahme von $(\clubsuit)$ der Fall ist), so ist $q$ bereits über $\IZ$ definiert.
\end{enumerate}

\begin{figure}[hbp]
\centering
	\begin{tikzpicture}
	\matrix (m) [matrix of math nodes, row sep=3em,
	column sep=2.5em, text height=1.5ex, text depth=0.25ex]
	{
		& R[\Gamma]\OmegaGJ & \\
		R[\Gamma]H & & R[\Gamma]J \\
	};
	\path[->,font=\scriptsize]
		(m-2-1) edge node[left]{$\iota$} (m-1-2)
		(m-1-2) edge node[right]{$q$} (m-2-3)
		(m-2-1) edge node[below]{$\phi$} (m-2-3);
	\end{tikzpicture}
	\caption{Faktorisierung von Lusztigs Homomorphismus $\phi$}
	\label{fig:lusztigs_hom}
\end{figure}
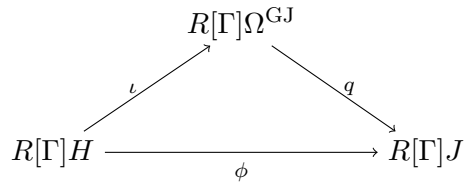
\end{theorem}
\begin{proof}
Sei $E$ der freie $R[\Gamma]$"~Modul mit Basis $(b_w)_{w\in W}$. Wenn wir $[\mathfrak{Z}]$ für den zur zweiseitigen Zelle $\mathfrak{Z}\subseteq W$ gehörigen Zellmodul schreiben, dann können wir mit der Identifikation $b_w \leftrightarrow C_w$ die $H$"~Linksmodulstruktur auf
\[\bigoplus_{\mathfrak{Z}\subseteq W \,\text{LR-Zelle}} [\mathfrak{Z}]\]
auf $E$ zurückziehen. Explizit gilt also
\[C_x \cdot b_w = \sum_{\substack{z\in W \\ z \sim_\mathcal{LR} w}} h_{x,w,z} b_z\]
für alle $x,w\in W$. Da dies als Modul-mit-Basis eine direkte Summe von Kazhdan-Lusztig-Linkszellmoduln ist, handelt es sich dabei um einen endlichen Geck-$W$"~Graph"=Modul \index{terms}{W-Graph@$W$-Graph!Geck-} mit Gewichten in $\IZ[\Gamma]\subseteq R[\Gamma]$, d.\,h. wir können diese Operation auf $R[\Gamma]\OmegaGJ_G$ fortsetzen für ein hinreichend großes, aber endliches $G\subseteq\Gamma$. Explizit heißt das, dass für alle LR-Zellen $\mathfrak{Z}\subseteq W$ und alle $w\in\mathfrak{Z}$ gilt:
\[e_s \cdot b_w = \begin{cases} b_w & \text{falls } s\in D_L(w) \\ 0 & \text{sonst} \end{cases}\]
\[x_{s,\gamma} \cdot b_w = \sum_{z\in\mathfrak{Z}} m_{zw}^{s,\gamma} \cdot b_z\]

Andererseits können wir $E$ via $b_w\leftrightarrow t_w$ mit $R[\Gamma]J$ identifizieren und erhalten so eine $R[\Gamma]J$"~Rechtsmodulstruktur auf $E$. Explizit ist sie durch:
\[b_w \cdot t_y = \sum_{z\in W} \gamma_{w,y,z^{-1}} b_z\]
für alle $w,y\in W$ gegeben. Da $J$ die direkte Summe der Rechtszellen ist, läuft die Summe hier in Wirklichkeit nur über die Rechtszelle von $w$. (Aufgrund von $(\spadesuit)$ müssen wir nicht zwischen $H$- und $J$"~Zellen unterscheiden, siehe \citep[2.5.9]{geckjacon}.)

\medbreak
Wir zeigen jetzt, dass diese beiden Modulstrukturen kommutieren, $E$ also ein $\OmegaGJ_G$"~$J$"~Bi\-mo\-dul ist. Es gilt zunächst für die Idempotente:
\begin{align*}
	e_s \cdot (b_w \cdot t_y) &= e_s \sum_{z\sim_\mathcal{R} w} \gamma_{w,y,z^{-1}} b_z \\
	&= \sum_{z\sim_\mathcal{R} w} \gamma_{w,y,z^{-1}} e_s \cdot b_z
\intertext{Nun ist jedoch $D_L(z)=D_L(w)$, da $z \sim_\mathcal{R} w$ (siehe \citep[2.1.16.]{geckjacon} und \citep[8.6]{lusztig2003hecke}), d.\,h. ob $e_s \cdot b_z=b_z$ oder $e_s \cdot b_z=0$ ist, hängt gar nicht von $z$ ab, sondern nur von $w$:}
	\ldots &= \begin{cases}
		\sum_{z\sim_\mathcal{R} w} \gamma_{w,y,z^{-1}} b_z & \text{falls } s\in D_L(w) \\
		0 & \text{falls } s\notin D_L(w)
	\end{cases} \\
	&= \begin{cases}
		b_w \cdot t_y & \text{falls } s\in D_L(w) \\
		0 & \text{falls } s\notin D_L(w)
	\end{cases} \\
	&= (e_s \cdot b_w)\cdot t_y
\end{align*}

Nun gilt die Bedingung $(\spadesuit)$. Dies ist exakt die Behauptung, dass die $H$"~Linksmodul- und die $J$"~Rechtsmodulstruktur auf $E$ kommutieren (siehe \citep[2.5.4.]{geckjacon}), d.\,h. $(\spadesuit)$ ist äquivalent dazu, dass für alle $h\in H, w,y\in W$
\[(h\cdot b_w)\cdot t_y = h\cdot(b_w \cdot t_y)\]
gilt. Da nun die Idempotente $e_s$ und die $T_s\in H$ zusammen bereits ganz $\Omega$ erzeugen, folgt, dass $E$ schon ein $R[\Gamma]\Omega$"~$R[\Gamma]J$"~Bimodul ist. Es gilt also auch
\[(x_s \cdot b_w)\cdot t_y = x_s\cdot (b_w\cdot t_y)\]
für alle $s\in S$ und alle $w,y\in W$. Wir wollen aber mehr als das. Wir wollen, dass nicht nur die $\Omega$-Operation mit der $J$-Operation kommutiert, sondern auch die $\OmegaGJ_G$-Operation. Dazu schreiben wir $x_s = \sum_{\alpha\in G} x_{s,\alpha} v^\alpha$.

Wenn wir nun ausnutzen, dass $E$ sich als $R$-Modul als $\bigoplus_{\alpha\in\Gamma} (\sum_z Rb_z)v^\alpha$ zerlegt, dann vergleichen wir $(x_{s,\alpha} \cdot b_w) \cdot t_y$ und $x_{s,\alpha} \cdot (b_w \cdot t_y)$ wie folgt:
\begin{align*}
	(x_s \cdot b_w) \cdot t_y &= \Big(\sum_{\alpha\in G} x_{s,\alpha} v^\alpha \cdot b_w \Big) \cdot t_y \\
	&= \sum_{\alpha\in G} \big((x_{s,\alpha} \cdot b_w)\cdot t_y\big) v^\alpha \\
	x_s \cdot (b_w \cdot t_y) &= \Big(\sum_{\alpha\in G} x_{s,\alpha} v^\alpha \Big) \cdot (b_w \cdot t_y) \\
	&=\sum_{\alpha\in G} \big(x_{s,\alpha} \cdot (b_w \cdot t_y)\big)v^\alpha
\end{align*}
Weil nun sowohl $(x_{s,\alpha} \cdot b_w) \cdot t_y$ als auch $x_{s,\alpha} \cdot (b_w \cdot t_y)$ in $E_0:=\sum_z Rb_z$ liegen nach Konstruktion der $\OmegaGJ_G$-Operation, können wir einen Koeffizientenvergleich durchführen und so auf die gewünschte Gleichheit $(x_{s,\alpha} \cdot b_w) \cdot t_y=x_{s,\alpha} \cdot (b_w \cdot t_y)$ schließen. Da die $x_{s,\alpha}$ zusammen mit den $e_s$ ganz $\OmegaGJ_G$ erzeugen, ist $E$ ein $R[\Gamma]\OmegaGJ_G$"~$R[\Gamma]J$-Bimodul. In der Tat ist schon $E_0$ ein $R\OmegaGJ_G$"~$RJ$"~Bimodul.

\medbreak
Aufgrund dieser Bimodulstruktur auf $E_0$ gibt es nun einen natürlichen Algebrahomomorphismus $R\OmegaGJ_G\to\End_{\textbf{Mod}-RJ}(E_0)$, nämlich $a\mapsto (b_w\mapsto a\cdot b_w)$. Weiter ist $E_0$ als $RJ$"~Rechtsmodul via $b_w \leftrightarrow t_w$ kanonisch isomorph zum regulären $RJ$"~Rechtsmodul und $\End_{\textbf{Mod}-RJ}(RJ)$ kanonisch isomorph zu $RJ$ selbst mittels $f\mapsto f(1_J) = f(\sum_{d\in\mathcal{D}} n_d t_d)$.

Indem wir diese Homomorphismen komponieren, erhalten wir den gewünschten Homomorphismus $q: R\OmegaGJ_G\to RJ$. Wir zeigen nun, dass $q\circ\iota=\phi$ ist. Dazu setzen wir die Kazhdan-Lusztig-Basis ein und rechnen:
\begin{align*}
	\iota(C_x)\cdot \sum_{d\in\mathcal{D}} n_d b_d &= \sum_{d\in\mathcal{D}} n_d C_x\cdot b_d \\
	&= \sum_{d\in\mathcal{D}} \sum_{z\in W,z \sim_\mathcal{LR} d} n_d h_{x,d,z} b_z \\
	\implies q(\iota(C_x)) &= \sum_{d\in\mathcal{D},z\in W,z \sim_\mathcal{LR} d} n_d h_{x,d,z} t_z \\
	&= \phi(C_x) \qedhere
\end{align*}
\end{proof}

\begin{remark}
In \cite{Gyoja} wird die analoge Aussage aus der $W\times W^\textrm{op}$-Graph-Struktur auf den zweiseitigen Kazhdan-Lusztig-Zellen, d.\,h. aus \textbf{P15} gefolgert.
\end{remark}

\begin{remark}
Der folgende Beweis ist von \cite{Gyoja} und \citep[2.7.12]{geckjacon} inspiriert und kombiniert das Beste aus beiden Welten. Wir erhalten so einen Beweis, der auf die Annahme $(\vardiamond)$ verzichtet, die in \cite{geckjacon} noch auftaucht.
\end{remark}

\begin{corollary}[Existenz von Geck-$W$-Graphen]\label{lusztig_hom:ex_of_w_graphs}
\index{terms}{W-Graph@$W$-Graph!Geck-}
Es gelten $(\spadesuit)$ und $(\clubsuit)$. Sei weiter $R\subseteq F$ ein Hauptidealring mit $F=\operatorname{Quot}(R)$.

Ist $R$ ein $L$-guter Ring, so kann jeder Isomorphietyp von einfachen $KH$"~Moduln durch einen Geck-$W$"~Graph mit Kantengewichten in $R[\Gamma]$ realisiert werden, dessen Matrixdarstellung balanciert ist.
\end{corollary}
\begin{proof}
Sei $\lambda\in\Irr(W)$ beliebig, aber fest, und $f:FJ\to F^{d_\lambda\times d_\lambda}$ eine Matrixdarstellung, die den $FJ$-Isomorphietyp $\lambda$ hat. Weil $R$ ein Hauptidealring ist, können wir die Darstellung sogar über $R$ realisieren, d.\,h. $f(t_w)\in R^{d_\lambda\times d_\lambda}$ für alle $w\in W$. Dann ist $\rho:=f\circ\phi$ eine Matrixdarstellung $R[\Gamma]H\to R[\Gamma]^{d_\lambda\times d_\lambda}$, die den $KH$"~Isomorphietyp $\lambda$ hat. Weil $(\spadesuit)$ gilt, faktorisiert $\phi$ und somit auch $\rho$ durch $K\OmegaGJ$. Also ist $\rho$ sogar eine Geck-$W$"~Graph"=Darstellung.

Lemma \ref{lusztig_hom:balanced_reps} zeigt, dass $\rho$ sogar balanciert und $f=\overline{\rho}$ ist. Das zeigt uns die Faktorisierung $\rho=\overline{\rho}\circ\phi$.
\end{proof}

\begin{remark}
Für den Typ $I_2(m)$ mit $m\geq 7$ ist $\IZ_W$ i.\,A. kein Hauptidealring mehr. Trotzdem können alle einfachen Darstellungen als Geck-$W$"~Graphen mit Kantengewichten in $\IZ_W[\Gamma]$ realisiert werden. Das folgt aus den expliziten Konstruktionen in \ref{wgraphs:ex:1D_reps}, \ref{wgraphs:ex:refl_rep} und \ref{wgraphs:constructions}, da jede Darstellung von $W$ entweder eindimensional oder algebraisch konjugiert zur Spiegelungsdarstellung ist.
\end{remark}

\begin{theorem}
\index{terms}{Vermutung!von Gyoja}
Es gelte $(\spadesuit)$. Dann sind folgende Aussagen äquivalent:
\begin{enumerate}
	\item Sind $X,Y$ zwei halbeinfache $K\Omega$"~Moduln, dann ist jede $KH$"~lineare Abbildung $X\to Y$ bereits $K\Omega$"~linear.
	\item Je zwei halbeinfache $K\Omega$"~Moduln, deren Einschränkungen auf $KH$ als $KH$"~Moduln isomorph sind, sind auch als $K\Omega$"~Moduln isomorph.
	\item $\abs{\Irr(K\Omega)} = \abs{\Irr(KH)}$.
	\item $\rad K\Omega = \ker(K\Omega \xrightarrow{q} KJ)$.
	\item $K\Omega / \rad K\Omega \isomorphic KJ$.
	\item $\dim_K K\Omega/\rad K\Omega = \abs{W}$.
\end{enumerate}
Diese Äquivalenz gilt mutatis mutandis auch für $K\OmegaGJ$.
\end{theorem}
\begin{proof}
a.$\implies$b. ist trivial.

b.$\implies$c.
Sei $Z\in\Omega-\textbf{Mod}$ einfach und $\operatorname{Res}_H^\Omega Z=\bigoplus_{\alpha\in A} M_\alpha$ eine Zerlegung in einfache \mbox{$H$"~Moduln}. Da es zu jedem einfachen $H$"~Modul einen $W$"~Graphen gibt, können wir o.\,B.\,d.\,A. annehmen, dass die $M_\alpha$ Einschränkungen von $Z_\alpha\in\Omega-\textbf{Mod}$ sind. Aus der Voraussetzung und der Einfachheit von $Z$ folgt $\abs{A}=1$, d.\,h. $\operatorname{Res}_H^\Omega$ induziert eine Abbildung $\Irr(\Omega)\to\Irr(H)$, die aufgrund der Voraussetzung injektiv ist. Da jede einfache $H$-Darstellung als $W$"~Graph"=Darstellung realisiert werden kann, ist die Verkettung $\Irr(H)\to\Irr(\Omega)\to\Irr(H)$ die Identität, d.\,h. es handelt sich sogar um eine Bijektion.

\bigbreak
c.$\implies$d.

Gilt c., dann sind die einfachen $K\Omega$"~Moduln, die durch $K\Omega\to KJ$ faktorisieren, bereits ein vollständiges Vertretersystem der einfachen $K\Omega$"~Moduln. Also ist
\[\rad(K\Omega) = \bigcap_{M\in\Irr(K\Omega)} \operatorname{Ann}(M) = \bigcap_{M\in\Irr(KJ)} \operatorname{Ann}(M) \supseteq \ker(K\Omega\to KJ)\]
Da $KJ$ halbeinfach ist, ist $\rad(K\Omega)$ jedoch sowieso in $\ker(K\Omega\to KJ)$ enthalten, also gilt Gleichheit.

\bigbreak
d.$\implies$e.$\implies$f. ist dann offensichtlich. f.$\implies$d. ist auch klar, da $KJ$ halbeinfach ist und somit die Inklusion $\rad(K\Omega)\subseteq\ker(K\Omega\to KJ)$ immer gilt.

\bigbreak
d.$\implies$a. $\rad(K\Omega)$ liegt im Annullator jedes einfachen und daher auch jedes halbeinfachen Moduls, daher sind halbeinfache $K\Omega$"~Moduln auch automatisch $K\Omega/\rad(K\Omega)$"~Moduln. Wenn $K\Omega/\rad(K\Omega)$ nun zu der endlichdimensionalen Algebra $KJ$ isomorph ist, sind umgekehrt auch alle $K\Omega/\rad(K\Omega)$"~Moduln halbeinfach und somit auch als $K\Omega$"~Moduln halbeinfach, d.\,h. $\Set{M \in K\Omega-\textbf{Mod} | M\text{ halbeinfach}} = K\Omega/\rad(K\Omega)-\textbf{Mod}$.

\begin{figure}[ht]
\centering
	\begin{tikzpicture}
	\matrix (m) [matrix of math nodes, row sep=3em,
	column sep=2.5em, text height=1.5ex, text depth=0.25ex]
	{
		& \Set{M \in K\Omega-\textbf{Mod} | M\text{ halbeinfach}} & \\
		KH-\textbf{Mod} & & KJ-\textbf{Mod} \\
	};
	\path[<-,font=\scriptsize]
		(m-2-1) edge node[desc]{$\iota^\ast$} (m-1-2)
		(m-1-2) edge node[desc]{$q^\ast$} (m-2-3)
		(m-2-1) edge node[desc]{$\phi^\ast$} (m-2-3);
	\end{tikzpicture}
	\caption{$\phi=q\circ\iota$ und Auswirkungen auf Modul-Kategorien.}
	\label{fig:equivalences_mod_categories}
\end{figure}
Die Restriktion entlang der natürlichen Homomorphismen liefert das kommutative Diagramm \ref{fig:equivalences_mod_categories}, wobei der untere Pfeil sowieso eine Äquivalenz ist, da er vom Lusztig-Isomorphismus induziert wird. Der rechte Pfeil ist eine Äquivalenz, weil er vom Isomorphismus $\overline{q}: K\Omega / \rad K\Omega \to KJ$ induziert wird. Also muss auch der linke Pfeil $\iota^\ast$ eine Äquivalenz sein. Dies ist genau der Restriktionsfunktor. Insbesondere folgt also $\Hom_\Omega(X,Y)=\Hom_H(X,Y)$.

Alle Beweise funktionieren offenbar genauso, wenn man $K\Omega$ durch $K\OmegaGJ$ ersetzt.
\end{proof}

\begin{conjecture}[Gyojas Vermutung, siehe {\citep[2.18]{Gyoja}} im Einparameterfall]
\index{terms}{Vermutung!von Gyoja}
Die Aussagen a.) bis f.) in obigem Satz treffen für $K\OmegaGJ(W,S,L)$ zu, wenn $(W,S,L)$ eine endliche Coxeter"=Gruppe"=mit"=Gewicht ist.
\end{conjecture}

\begin{remark}
Da wir wissen, dass $R\OmegaGJ\xrightarrow{q} RJ$ durch ein $R\OmegaGJ_G$ mit hinreichend großem $G\subseteq\Gamma$ faktorisiert, kann man problemlos in der Vermutung und im Äquivalenzbeweis $K\OmegaGJ$ durch solch ein $K\OmegaGJ_G$ ersetzen.
\end{remark}
\begin{remark}
Dass der Körper gerade $K$ war, spielte offenbar für den Beweis keine größere Rolle. Man könnte mit demselben Beweis also auch alle analogen Behauptungen für $K\subseteq K'$ beliebig als äquivalent erkennen. Insbesondere könnte man eine $\IC$-Variante der Vermutung formulieren.

Aus \citep[5.13-5.17]{lam2001firstcourse} folgt, dass in Charakteristik $0$ die Kodimension des Radikals bei Körpererweiterungen höchstens wachsen kann:
	
Ist $K\subseteq K'$ ein Körperturm in Charakteristik $0$ und $A$ eine beliebige $K$"~Algebra, so impliziert 5.14. $A\cap\rad(K'A)\subseteq\rad(A)$. Für rein transzendente Erweiterungen impliziert 5.13 $\rad(K'A)=K'(A\cap\rad(K'A))$, also $\frac{K'A}{\rad(K'A)}=\frac{K'A}{K'(A\cap\rad(K'A))}=K'\frac{A}{A\cap\rad(K'A)}$, was $K'\frac{A}{\rad(A)}$ als Quotienten hat. Also folgt in diesem Fall
\[\dim_{K'} K'A/\rad(K'A) \geq \dim_K A/\rad(A).\]
Aus \citep[5.17]{lam2001firstcourse} folgt weiterhin die Gleichheit $\rad(K'A) = K'\rad(A)$ bei algebraischen Erweiterungen $K\subseteq K'$, die wegen $\CharFld(K)=0$ hier ja automatisch separabel sind.
	
Haben wir also einen Körperturm $K\subseteq K'\subseteq K''$, dann impliziert die Richtigkeit von Gyojas Vermutung für $K''$ die Richtigkeit der Vermutung für $K'$ und die beiden Varianten sind äquivalent, wenn $K'\subseteq K''$ algebraisch ist.
\end{remark}
\begin{remark}
Da $K\OmegaGJ\xrightarrow{q} KJ$ surjektiv ist, muss bereits $F\OmegaGJ\xrightarrow{q}FJ$ surjektiv gewesen sein. Insbesondere erhalten wir dann die Abschätzung $\dim F\OmegaGJ / \rad(F\OmegaGJ) \geq \abs{W}$ und aufgrund der Monotonie der Radikal-Kodimension bzgl. Körpererweiterungen impliziert die Gültigkeit von Gyojas Vermutung für $K$ (oder irgendeinen größeren Körper) dann auch hierbei die Gleichheit.
	
Da aber $FJ$ zerfallend halbeinfach ist, wäre dann $FJ = F\OmegaGJ/\rad(F\OmegaGJ)$ und somit $F=\End_{FJ}(X)=\End_{F\OmegaGJ}(X)$ für jeden einfachen $F\OmegaGJ$"~Modul $X$. Mit anderen Worten: $F$ wäre bereits ein Zerfällungskörper für $F\OmegaGJ$.
\end{remark}
\begin{remark}
Für $\abs{S}=0$ ist Gyojas Vermutung trivialerweise wahr, weil dann $k\OmegaGJ=k$ gilt (siehe Beispiel \ref{wgraph_alg:small_ex}) und die Kodimension des Radikals von $K\OmegaGJ$ daher $1=\abs{W}$ ist.

\medbreak
Für $\abs{S}=1$ ist Gyojas Vermutung ebenfalls wahr. Es gilt, wie in \ref{wgraph_alg:small_ex} gesehen,
\[k\OmegaGJ = k E_\emptyset \times k E_S \times \prod_{0\leq\gamma<L(s)} k X_{S,\emptyset}^{s,\gamma}\]
als $k$"~Moduln. Nun ist aber $X_{S,\emptyset}^{s,\gamma_1} \cdot X_{S,\emptyset}^{s,\gamma_2} = 0$ für alle $\gamma_1,\gamma_2\in\Gamma$, d.\,h. der rechte Faktor in diesem Produkt ist ein nilpotentes Ideal von $k\OmegaGJ$. Somit ist die Kodimension des Radikals von $K\OmegaGJ$ höchstens 2 und daher gilt Gyojas Vermutung.
\end{remark}

\begin{remark}
Im Zusammenhang mit ihren Berechnungen von Zerlegungszahlen, um die James"=Vermutung zu überprüfen, stellten Meinolf Geck und Jürgen Müller außerdem die folgende Vermutung auf, die die zelluläre Struktur der Hecke"=Algebren aus Kapitel Zwei mit $W$"~Gra\-phen in Verbindung setzt:
\end{remark}

\begin{conjecture}[Geck-Müller, siehe {\citep[2.7.13]{geckjacon}} und {\citep[4.5]{GeckMueller}}]\label{conj:geck_mueller}
\index{terms}{Vermutung!von Geck-Müller}\index{terms}{Darstellung!balancierte}
Jeder irreduzible Geck-$W$"~Graph $\omega: KH\to K^{d\times d}$ ist balanciert und faktorisiert als $\omega=\overline{\omega}\circ\phi$ oder äquivalent (siehe Lemma \ref{h_vs_j:conj_class} und \ref{CellAlg:Hecke_cell_mod}): Jeder Geck-$W$"~Graph kommt als Zellmodul in Gecks Zellbasiskonstruktion vor für geeignet gewählte Inputdaten.
\end{conjecture}

\begin{remark}
Selbst die Teilaussage, dass Geck-$W$"~Graphen immer balanciert sind, ist allerdings noch unbewiesen im allgemeinen Fall.
\end{remark}

\begin{conjecture}[Geck-Jacon, siehe {\citep[1.4.14]{geckjacon}}]
\index{terms}{Vermutung!von Geck-Jacon}\index{terms}{Darstellung!balancierte}
Jeder irreduzible Geck-$W$"~Graph $\omega: KH\to K^{d\times d}$ definiert eine balancierte Matrixdarstellung.
\end{conjecture}

\begin{remark}
Es stellt sich nun heraus, dass die Richtigkeit von Gyojas Vermutung die Gültigkeit der Geck"=Müller"=Vermutung nach sich zieht:
\end{remark}

\begin{theorem}
\index{terms}{Vermutung!von Gyoja}\index{terms}{Vermutung!von Geck-Müller}
Es gelten $(\spadesuit)$ und $(\clubsuit)$. Ist Gyojas Vermutung für $K\OmegaGJ(W,S,L)$ richtig, so ist auch die Geck"=Müller"=Vermutung richtig.
\end{theorem}
\begin{proof}
Es sei $(\mathfrak{C},I,m)$ ein endlicher Geck-$W$"~Graph mit Kantengewichten in $F[\Gamma]$ und die zugehörige Matrixdarstellung $\omega: KH\to K^{\mathfrak{C}\times\mathfrak{C}}$ sei irreduzibel.

Dann faktorisiert $\omega$ als Darstellung eines endlichen Geck-$W$"~Graphen durch ${KH \xrightarrow{\iota} K\OmegaGJ_G}$ für eine hinreichend große, endliche Teilmenge $G\subseteq\Gamma$ und definiert daher auch eine einfache Darstellung von $K\OmegaGJ_G$. Wir erhalten daher das Diagramm \ref{fig:w_graphs_fact_over_phi}.
\begin{figure}[ht]
\centering
	\begin{tikzpicture}
	\matrix (m) [matrix of math nodes, row sep=3em,
	column sep=2.5em, text height=1.5ex, text depth=0.25ex]
	{
		K^{\mathfrak{C}\times\mathfrak{C}} & & K\OmegaGJ_G/\rad(K\OmegaGJ_G) \\
		 & K\OmegaGJ_G & \\
		KH & & KJ \\
	};
	\path[->,font=\scriptsize]
		(m-1-3) edge (m-1-1)
		(m-1-3) edge node[right]{$\isomorphic$} node[left]{$\overline{q}$} (m-3-3)
		(m-2-2) edge (m-1-1)
		(m-2-2) edge (m-1-3)
		(m-2-2) edge node[desc]{$q$} (m-3-3)
		(m-3-1) edge node[left]{$\omega$} (m-1-1)
		(m-3-1) edge node[desc]{$\iota$} (m-2-2)
		(m-3-1) edge node[desc]{$\phi$} node[below]{$\isomorphic$} (m-3-3);
	\end{tikzpicture}
	\caption{Faktorisierung von $\omega$ über $\phi$.}
	\label{fig:w_graphs_fact_over_phi}
\end{figure}

Das untere Dreieck kommutiert nach Konstruktion von $q$. Das rechte Dreieck kommutiert, weil $KJ$ halbeinfach ist. Das linke Dreieck kommutiert nach Konstruktion. Das obere Dreieck kommutiert, weil $\omega$ als $K\OmegaGJ_G$"~Modul einfach ist.

Weil $\smash{K\OmegaGJ_G/\rad(K\OmegaGJ_G) \xrightarrow[\isomorphic]{\overline{q}} KJ}$ ein Isomorphismus ist, faktorisiert $\omega$ also als
\[\smash[t]{KH \xrightarrow{\phi} \underbrace{KJ \xrightarrow{\overline{q}^{-1}} K\OmegaGJ_G/\rad(K\OmegaGJ_G) \to K^{\mathfrak{C}\times\mathfrak{C}}}_{=:f}.}\]

Das rechte Dreieck in Diagramm \ref{fig:w_graphs_fact_over_phi} ist schon über $F$ definiert, da $\OmegaGJ_G\xrightarrow{q} J$ schon über $\IQ_W\subseteq F$ definiert ist. Der $W$"~Graph hat Kantengewichte in $F[\Gamma]$, d.\,h. das obere Dreieck ist auch über $F$ definiert (siehe Diagramm \ref{fig:w_graphs_fact_over_phi2}). Beides zusammen liefert, dass $f(t_w)$ nicht nur in $K^{\mathfrak{C}\times\mathfrak{C}}$, sondern schon in $F^{\mathfrak{C}\times\mathfrak{C}}$ liegt für alle $w\in W$. Aus \ref{lusztig_hom:balanced_reps} folgt dann, dass $\omega$ balanciert ist und als $\overline{\omega}\circ\phi$ faktorisiert.\qedhere

\begin{figure}[ht]
\centering
	\begin{tikzpicture}
	\matrix (m) [matrix of math nodes, row sep=3em,
	column sep=2.5em, text height=1.5ex, text depth=0.25ex]
	{
		F^{\mathfrak{C}\times\mathfrak{C}} & & F\OmegaGJ_G/\rad(F\OmegaGJ_G) \\
		 & F\OmegaGJ_G & \\
		 & & FJ \\
	};
	\path[->,font=\scriptsize]
		(m-1-3) edge (m-1-1)
		(m-1-3) edge node[right]{$\isomorphic$} node[left]{$\overline{q}$} (m-3-3)
		(m-2-2) edge (m-1-1)
		(m-2-2) edge (m-1-3)
		(m-2-2) edge node[desc]{$q$} (m-3-3);
	\end{tikzpicture}
	\caption{$f$ schickt $FJ$ nach $F^{\mathfrak{C}\times\mathfrak{C}}$.}
	\label{fig:w_graphs_fact_over_phi2}
\end{figure}
\end{proof}

\begin{remark}
\index{terms}{Gröbner-Basis}\index{terms}{Darstellung!balancierte}\index{terms}{Vermutung!von Geck-Jacon}\index{terms}{Vermutung!von Geck-Müller}\index{terms}{Vermutung!$W$-Graph-Zerlegungs-}
Um Balanciertheit algorithmisch zu testen, kann man direkt die Definition verwenden. So kann die Vermutung von Geck"=Jacon für jeden konkreten $W$"~Graphen computergestützt getestet werden.

Lemma \ref{h_vs_j:conj_class} und Beobachtung \ref{CellAlg:Hecke_cell_mod} zeigen uns auch eine algorithmische Möglichkeit, die Vermutung von Geck"=Müller für jeden konkreten $W$"~Graphen zu überprüfen: Man berechne $\overline{\omega}\circ\phi$ und vergleiche mit $\omega$.

Ob auch Gyojas Vermutung algorithmisch zu überprüfen ist, ist mir indes nicht bekannt. Eine Analyse aller einfachen Moduln erscheint aussichtslos, da a priori nicht einmal klar ist, dass es nur endlich viele einfache $K\OmegaGJ$"~Moduln gibt. Wir haben bereits gesehen, dass bereits für $\abs{S}=1$ die Algebra $K\OmegaGJ$ unendlichdimensional sein kann. Daher muss schon in diesem Fall ein nichttriviales Argument her, um so eine Endlichkeitsaussage zu beweisen.
	
Im Einparameterfall, also $k\OmegaGJ=k\OmegaGy$, ist es einfacher, solche Ergebnisse zu erzielen. Im nächsten Abschnitt werden wir eine neue Präsentation durch Erzeuger und Relationen der Algebra $\OmegaGy$ herleiten. In einigen Fällen lässt sich mit deren Hilfe beispielsweise beweisen, dass $\IQ_W\OmegaGy$ endlichdimensional ist. Wenn man bereits weiß, dass es sich um eine endlichdimensionale Algebra handelt, könnten Techniken aus der Theorie nichtkommutativer Gröbner"=Basen (siehe z.\,B. \cite{mora1994introduction} und \cite{green1993introduction} für eine Einführung) es ermöglichen, die Kodimension des Radikals zu bestimmen. Dann könnte Gyojas Vermutung tatsächlich algorithmisch angreifbar werden, etwa durch die Ideen in \cite{king2012loewy}.
\end{remark}
\section{\texorpdfstring{$\OmegaGy$}{Gyojas W-Graph-Algebra} als Pfadalgebra-Quotient}

\begin{remark}
In Lemma \ref{wgraph_alg:E_I_and_X_IJ} wurde gezeigt, dass $\Xi_G$ für jedes $G\subseteq\Gamma$ eine Pfadalgebra ist. Wir können also $\Omega$ und $\OmegaGJ_G$ nicht nur als Quotienten einer freien Algebra betrachten (was ja schon die Definition durch Erzeuger und Relationen impliziert), sondern auch als Quotient einer speziellen Pfadalgebra. 

Zunächst stellt sich natürlich die Frage, ob wir explizitere Relationen für den Quotienten ${\Xi\xrightarrow{\isomorphic}\Xi_{\Set{0}}\twoheadrightarrow\OmegaGy}$ bestimmen können als bisher, wenn wir sie mit Hilfe der Pfadalgebra ausdrücken. Dieser Abschnitt hat genau dieses Ziel.
\end{remark}

\subsection{Relationen}

\begin{lemma}[vgl. {\citep[Prop.\,3.1]{Stembridge2008admissble}}]\label{wgraph_alg:braid_commutator}
\index{terms}{Zopfkommutator}\index{terms}{Zopfrelationen}
\index{symbols}{taur@$\tau_r$}
Definiere $\tau_r\in\IZ[X]$ durch folgende Rekursion:
\[\left\lbrace
\begin{aligned}
	\tau_{-1} &= 0 \\
	\tau_0 &= 1 \\
	\tau_r &= X\tau_{r-1}-\tau_{r-2}
\end{aligned}\right.\]

Mit dieser Definition gilt:
\begin{enumerate}
	\item Für $r\in\IN$ ist $\tau_r$ normiert und vom Grad $r$. Insbesondere ist $\Set{\tau_0,\ldots,\tau_r}$ eine $\IZ$"~Basis von $\Set{f\in\IZ[X] | \deg(f)\leq r}$.
	\item $\tau_r$ ist ein gerades Polynom für gerade $r$ und ein ungerades Polynom sonst, d.\,h. $\tau_r(-X) = (-1)^r \tau_r(X)$.
	\item $A$ sei ein Ring und $x, y\in A$ seien zwei Elemente, die die Gleichung $T^2 = 1+\zeta T$ erfüllen für ein festes $\zeta\in A$. Dann gilt für ihre Zopfkommutatoren
	\[\Delta_{r+1}(x,y) = (-1)^r\tau_r(x+y-\zeta)\cdot(x-y).\]
\end{enumerate}
\end{lemma}
\begin{proof}
Dass $\tau_r$ normiert und vom Grad $r$ ist, sieht man der Rekursion sofort an. Ebenso folgt b. sofort aus der Rekursionsformel.

c. ist für $r=-1$ und $r=0$ klar. Damit ist der Induktionsanfang gemacht. Es gilt weiter:
\begin{align*}
	(x+y)\Delta_{r+1}(x,y) &= x^2 \underbrace{y x \ldots}_{r} - \underbrace{x y x \ldots}_{r+2} + \underbrace{y x y\ldots}_{r+2} - y^2 \underbrace{x y\ldots}_{r}  \\
	&= 1\cdot\underbrace{y x \ldots}_{r} - 1\cdot\underbrace{x y \ldots}_{r} + \zeta x \underbrace{y x \ldots}_{r} - \zeta y \underbrace{x y \ldots}_{r} \\
	&\hspace{1em} -\big(\underbrace{x y x \ldots}_{r+2} - \underbrace{y x y\ldots}_{r+2}\big) \\
	&= -\Delta_r(x,y) + \zeta\Delta_{r+1}(x,y) - \Delta_{r+2}(x,y) \\
	\implies \Delta_{r+2}(x,y) &= (-1)\big((x+y-\zeta)\Delta_{r+1}(x,y) - \Delta_r(x,y)\big)
\end{align*}
Daraus folgt die Aussage.
\end{proof}

\begin{remark}
Die $\tau_r$ sind nur leicht modifizierte Tschebyscheff-Polynome zweiter Art. Es gilt genauer: $U_n(X) = \tau_n(2X)$ (siehe \citep[Ch.\,22]{abramowitz1964handbook}).
\end{remark}

\begin{remark}
Gewappnet mit diesem Rüstzeug können wir es in Angriff nehmen, die Relationen von $\Xi\to\OmegaGy$ auszurechnen.

\medbreak
Betrachten wir dazu die $\IZ[\Gamma]$"~Algebra $V:=\IZ[\Gamma]\otimes_\IZ \Xi$ und wählen $s,t\in S$ beliebig, aber fest. Dann definieren wir
\[V_{00} := \bigoplus_{\substack{I\subseteq S \\ s\notin I,t\notin S}} E_I V,\quad
  V_{01} := \bigoplus_{\substack{I\subseteq S \\ s\in I,t\notin S}} E_I V,\quad
  V_{10} := \bigoplus_{\substack{I\subseteq S \\ s\notin I,t\in I}} E_I V,\quad
  V_{11} := \bigoplus_{\substack{I\subseteq S \\ s\in I, t\in I}} E_I V\]
und beschreiben die Wirkung von Elementen aus $\Xi$ auf $V$ entsprechend dieser Zerlegung $V=V_{00}\oplus V_{01}\oplus V_{10}\oplus V_{11}$ durch $4\times 4$-Matrizen.

\medbreak
Die Matrizen von $\iota(T_s)=-v_s^{-1}e_s + v_s(1-e_s)+x_s$ und analog auch von $\iota(T_t)$ sind durch
\[\iota(T_s) = \begin{pmatrix}
v_s & & & \\
B_1 & -v_s^{-1} & A_1 & \\
& & v_s & \\
D_1 & & C_1 & -v_s^{-1}
\end{pmatrix}\quad\text{und}\quad
\iota(T_t)=\begin{pmatrix}
v_t & & & \\
& v_t & & \\
B_2 & A_2 & -v_t^{-1} & \\
D_2 & C_2 & & -v_t^{-1}
\end{pmatrix}\]
gegeben, wobei
\[A_1 = \sum_{\substack{I,J\subseteq S \\ s\in I, s\notin J \\ t\notin I, t\in J}} X_{IJ}^s\quad\text{und}\quad
  A_2 = \sum_{\substack{I,J\subseteq S \\ s\notin I, s\in J \\ t\in I, t\notin J}} X_{IJ}^t,\]
\[B_1 = \sum_{\substack{I,J\subseteq S \\ s\in I, s\notin J \\ t\notin I, t\notin J}} X_{IJ}^s\quad\text{und}\quad
  B_2 = \sum_{\substack{I,J\subseteq S \\ s\notin I, s\notin J \\ t\in I, t\notin J}} X_{IJ}^t,\]
\[C_1 = \sum_{\substack{I,J\subseteq S \\ s\in I, s\notin J \\ t\in I, t\in J}} X_{IJ}^s\quad\text{und}\quad
  C_2 = \sum_{\substack{I,J\subseteq S \\ s\in I, s\in J \\ t\in I, t\notin J}} X_{IJ}^t\quad\text{sowie}\]
\[D_1 = \sum_{\substack{I,J\subseteq S \\ s\in I, s\notin J \\ t\in I, t\notin J}} X_{IJ}^s\quad\text{und}\quad
  D_2 = \sum_{\substack{I,J\subseteq S \\ s\in I, s\notin J \\ t\in I, t\notin J}} X_{IJ}^t\]
ist.
\end{remark}

\begin{theorem}[Relationen für $\Xi\to\OmegaGy$ im Einparameterfall]\label{wgraph_alg:relations}
\index{terms}{Zopfkommutator}\index{terms}{Zopfrelationen}\index{terms}{Pfadalgebra}\index{terms}{W-Graph@$W$-Graph!-Algebra}\index{terms}{Einparameterfall}
\index{symbols}{alphabetagamma@$(\alpha), (\beta), (\gamma)$}
Sei $(W,S,L)$ eine Coxeter"=Gruppe"=mit"=Gewicht.

Für $k\in\IN$, $s,t\in S$ und $I,J\subseteq S$ definiere dann die \udot{Pfad-Summen} als
\[P_{IJ}^k(s,t) :=  E_I \underbrace{x_s x_t x_s\ldots}_{k\,\text{Faktoren}} E_J = \begin{cases} 0 & k=0, I\neq J \\ E_I & k=0, I=J \\ \sum\limits_{I_1,\ldots,I_{k-1}\subseteq S} X_{I I_1}^s X_{I_1 I_2}^t X_{I_2 I_3}^s \ldots X_{I_{k-1} J}^s & k>0, 2\nmid k \\ \sum\limits_{I_1,\ldots,I_{k-1}\subseteq S} X_{I I_1}^s X_{I_1 I_2}^t X_{I_2 I_3}^s \ldots X_{I_{k-1} J}^t & k>0, 2\mid k \end{cases}.\]

Ist $m:=\ord(st)<\infty$ und $L(s)=L(t)$, so sind die Relationen von $\OmegaGy(W,S,L)$, die von der Zopf-Relation $\Delta_m(T_s,T_t)=0$ herkommen, äquivalent zu den folgenden:
\begin{enumerate}
	\item[$(\alpha)$] Die Relationen:
	\[0 = P_{IJ}^{m-1}(s,t) + a_{m-2} P_{IJ}^{m-2}(s,t) + \ldots + a_2 P_{IJ}^2(s,t) + a_1 P_{IJ}^1(s,t) + a_0 P_{IJ}^0(s,t)\]
	Dabei seien $a_i$ die Koeffizienten von $\tau_{m-1}\in\IZ[Y]$, d.\,h.
	\[\tau_{m-1}(Y) = Y^{m-1} + a_{m-2} Y^{m-2} + \ldots + a_2 Y^2 + a_1 Y^1 + a_0.\]
	Dabei laufen $I,J$ über alle Teilmengen mit $s\in I, t\notin I$, die außerdem $s\in J, t\notin J$ für ungerade $m$ bzw. $s\notin J, t\in J$ für gerade $m$ erfüllen.
	\item[$(\beta)$] Für alle $I,J\subseteq S$ mit $s,t\in I\setminus J$ die Relationen:
	\[X_{IJ}^s = X_{IJ}^t\]
	\item[$(\gamma)$] Für alle $I,J\subseteq S$ mit $s,t\in I$ und $s,t\notin J$ die Relationen:
	\[\forall 2\leq r\leq m: P_{IJ}^r(s,t) = P_{IJ}^r(t,s)\]
\end{enumerate}
Insbesondere wird im Einparameterfall der Kern von $\Xi\to\OmegaGy$ von diesen Relationen für alle $s,t\in S$ mit $\ord(st)<\infty$ erzeugt.
\end{theorem}
\begin{proof}
Schritt 1: Wir setzen zur Abkürzung $v:=v_s=v_t$ sowie $z:=v+v^{-1}$. Wir behaupten, dass mit den obigen Bezeichnungen für alle $r\in\IN$
\begin{equation}
\Delta_{r+1}(\iota(T_s),\iota(T_t)) = (-1)^r\begin{pmatrix}
0 & 0 & 0 \\
\tau_r(A) JB & \tau_r(A)J(A-z) & 0 \\
X_r &  -C\tau_r(A)J & 0 
\end{pmatrix}
\label{eq:wgraph_alg:braid_commutator2}\tag{$\ast$}
\end{equation}
gilt, wobei
\[A:=\begin{pmatrix} 0 & A_1 \\ A_2 & 0 \end{pmatrix},\quad B:=\begin{pmatrix} B_1 \\ B_2 \end{pmatrix},\quad C:=\begin{pmatrix} C_2 & C_1 \end{pmatrix},\quad J:=\begin{pmatrix} 1 & 0 \\ 0 & -1 \end{pmatrix} \quad\text{und}\]
\[X_r := \sum_{i=0}^{r-1} (-1)^iC\tau_i(z)\tau_{r-1-i}(A)JB+(-1)^r\tau_r(z)(D_1-D_2)\]
sei.

\medbreak
Um diese Behauptung zu beweisen, definieren wir
\begin{align*}
	E &:= \iota(T_s)+\iota(T_t)-(v-v^{-1}) = \begin{pmatrix}
		z & 0 & 0 \\
		B & A & 0 \\
		D_1+D_2 & C & -z
	\end{pmatrix} \quad\text{und} \\
	F &:= \iota(T_s)-\iota(T_t) = \begin{pmatrix}
		0 & 0 & 0 \\
		JB & J(A-z) & 0 \\
		D_1-D_2 & -CJ & 0 
	\end{pmatrix}.
\end{align*}
Aus Lemma \ref{wgraph_alg:braid_commutator} folgt $\Delta_{r+1}(\iota(T_s),\iota(T_t))=(-1)^r \tau_r(E)F$. Wir müssen also nur zeigen, dass $\tau_r(E)F$ der Matrix in \eqref{eq:wgraph_alg:braid_commutator2} entspricht. Für $r=-1$ und $r=0$ ist das offenkundig. Für den Induktionsschritt gilt
\begin{align*}
	\tau_{r+1}(E)F &= E\tau_r(E)F - \tau_{r-1}(E)F \\
	&\overset{\text{IV}}{=} \begin{pmatrix}
	z & 0 & 0 \\ B & A & 0 \\ D_1+D_2 & C & -z
	\end{pmatrix}\cdot\begin{pmatrix}
	0&0&0 \\ \tau_r(A)JB & \tau_r(A)J(A-z) & 0 \\ X_r & -C\tau_r(A)J & 0
	\end{pmatrix} \\
	&\phantom{\text{= }} - \begin{pmatrix}
	0&0&0 \\ \tau_{r-1}(A)JB & \tau_{r-1}(A)J(A-z) & 0 \\ X_{r-1} & -C\tau_{r-1}(A)J & 0
	\end{pmatrix} \\
	&=\begin{pmatrix}
	0 & 0 & 0 \\ A\tau_r(A)JB-\tau_{r-1}(A)JB & H & 0 \\
	C\tau_r(A)JB-zX_r-X_{r-1} & K & 0
	\end{pmatrix}
\end{align*}
wobei wir die Abkürzungen
\begin{align*}
	H &:= A\tau_r(A)J(A-z)-\tau_{r-1}(A)J(A-z) \quad\text{und} \\
	K &:= C\tau_r(A)J(A-z)+zC\tau_r(A)J+C\tau_{r-1}(A)J
\end{align*}
verwendet haben. An den Positionen $(2,1)$ und $(2,2)$ ist klar, dass das Gewünschte herauskommt. Für Position $(3,2)$ nutzt man $JA=-AJ$ und vereinfacht wie folgt:
\begin{align*}
	K &= C\tau_r(A)JA-C\tau_r(A)Jz+zC\tau_r(A)J+C\tau_{r-1}(A)J \\
	&= -C\tau_r(A)AJ+C\tau_{r-1}(A)J \\
	&= -C\tau_{r+1}(A)J
\end{align*}
Nun bleibt noch die Position $(3,1)$. Wir rechnen wie folgt:
\begin{align*}
	... &= C\tau_r(A)JB-z\sum_{i=0}^{r-1}(-1)^iC\tau_i(z)\tau_{r-1-i}(A)JB-(-1)^r z\tau_r(z)(D_1-D_2) \\
	&\phantom{\text{= }} -\sum_{i=0}^{r-2}(-1)^iC\tau_i(z)\tau_{r-2-i}(A)JB-(-1)^{r-1}\tau_{r-1}(z)(D_1-D_2) \\
	&= C\tau_r(A)JB \\
	&\phantom{\text{= }} +C\Big(\sum_{i=0}^{r-1}(-1)^{i+1}z\tau_i(z)\tau_{r-1-i}(A)-\sum_{i=1}^{r-1}(-1)^{i-1}\tau_{i-1}(z)\tau_{r-1-i}(A)\Big)JB \\
	&\phantom{\text{= }} +(-1)^{r+1}(z\tau_r(z)-\tau_{r-1}(z))(D_1-D_2) \\
	&= C\tau_r(A)JB + C\Big(\sum_{i=0}^{r-1}(-1)^{i+1}(z\tau_i(z)-\tau_{i-1}(z))\tau_{r-1-i}(A)\Big)JB \\
	&\phantom{\text{= }} +(-1)^{r+1}\tau_{r+1}(z)(D_1-D_2) \\
	&= (-1)^0C\tau_0(z)\tau_r(A)JB+C\Big(\sum_{i=0}^{r-1}(-1)^{i+1}\tau_{i+1}(z)\tau_{r-1-i}(A)\Big)JB \\
	&\phantom{\text{= }} +(-1)^{r+1}\tau_{r+1}(z)(D_1-D_2) \\
	&= \sum_{i=0}^r (-1)^i C\tau_i(z)\tau_{r-i}(A)JB+(-1)^{r+1}\tau_{r+1}(z)(D_1-D_2)
\end{align*}
Das zeigt Behauptung \eqref{eq:wgraph_alg:braid_commutator2}.

\bigbreak
Schritt 2: Vereinfachen.

Sei nun $\mathfrak{I}=\ker(\Xi\to\OmegaGy)$ das Ideal, für welches wir uns die ganze Zeit interessieren. Nach Definition wird es von den Koeffizienten vor den einzelnen $v$"=Potenzen in $\Delta_m(\iota(T_s),\iota(T_t))\in\IZ[v^{\pm 1}]\otimes\Xi$ aufgespannt. Wir schauen uns also die vier Terme
\begin{enumerate}
	\item $R_1:=\tau_{m-1}(A)JB$,
	\item $R_2:=\tau_{m-1}(A)J(A-z)$,
	\item $R_3:=\sum_{i=0}^{m-2} (-1)^iC\tau_i(z)\tau_{m-2-i}(A)JB+(-1)^{m-1}\tau_{m-1}(z)(D_1-D_2)$ und
	\item $R_4:=C\tau_{m-1}(A)J$
\end{enumerate}
genauer an. Der höchste Koeffizient in $R_2$ ist $-\tau_{m-1}(A)J$. Das ist in $\mathfrak{I}$ genau dann, wenn $\tau_{m-1}(A)\in\mathfrak{I}$ ist, da $J$ invertierbar ist. Daraus erhalten wir umgekehrt auch, dass $R_1$, $R_2$ und $R_4$ in $\mathfrak{I}$ liegen.

Schauen wir uns $R_3$ genauer an. $\tau_r$ hat den Grad $r$. Der höchste Koeffizient in $R_3$ ist $(-1)^{m-1}(D_1-D_2)$, also erhalten wir $D_1-D_2\in\mathfrak{I}$. $R_3$ ist also genau dann in $\mathfrak{I}$, wenn $D_1-D_2\in\mathfrak{I}$ und $R_3'=\sum_{i=0}^{m-2} (-1)^iC\tau_i(z)\tau_{r-2-i}(A)JB\in\mathfrak{I}$ ist. Indem wir immer wieder auf den höchsten Koeffizient schauen und den Ausdruck iteriert verkürzen, erhalten wir, dass $R_3'$ genau dann in $\mathfrak{I}$ ist, wenn $C\tau_0(A) JB, C\tau_1(A)JB, \ldots, C\tau_{m-2}(A)JB\in\mathfrak{I}$ sind. Da $\tau_0,\ldots,\tau_{m-2}$ eine $\IZ$"~Basis von $\Set{f\in\IZ[Y] | \deg(f)\leq m-2}$ bilden, sind diese Terme in $\mathfrak{I}$ genau dann, wenn $CA^0 JB, CA^1 JB, \ldots, CA^{m-2} JB$ in $\mathfrak{I}$ sind.

Wir erhalten also das Erzeugendensystem
\begin{enumerate}
	\item[$(\alpha)$] $R_\alpha:=\tau_{m-1}(A)$,
	\item[$(\beta)$]  $R_\beta:=D_1-D_2$ und
	\item[$(\gamma)$] $R_{\gamma,k}:=CA^k JB$ für $0\leq k \leq m-2$
\end{enumerate}
für $\mathfrak{I}$.

\bigbreak
Schritt 3: Die Relationen.

Nun zerlegen wir $V$ wieder als $\bigoplus_I E_I V$ und nehmen Relationen $R_\alpha$, $R_\beta$ und $R_{\gamma,k}$ komponentenweise auseinander. Dabei nutzen wir, dass $R$ genau dann in $\mathfrak{I}$ liegt, wenn für alle $I,J\subseteq S$ auch $E_I R E_J$ in $\mathfrak{I}$ ist.

\medbreak
Um $E_I R_\alpha E_J$ zu bestimmen, schauen wir auf $E_I A^k E_J$. Es gilt natürlich $A^0 =1$ und $1=\sum_I E_I$, sodass $E_I A^0 E_J=\delta_{IJ} E_I=P_{IJ}^0$ gilt. Für $k>0$ erhalten wir:
\[A^k = \begin{cases}
\begin{pmatrix}	(A_1 A_2)^{\frac{k}{2}} & 0 \\ 0 & (A_2 A_1)^{\frac{k}{2}} \end{pmatrix} & \text{falls}\,2\mid k \\
\begin{pmatrix}	0 & (A_1 A_2)^{\frac{k-1}{2}} A_1 \\ (A_2 A_1)^{\frac{k-1}{2}} A_2 & 0 \end{pmatrix} & \text{falls}\,2\nmid k
\end{cases}\]
Wir setzen
\[A_1 = \sum_{\substack{I,J\subseteq S \\ s\in I, s\notin J \\ t\notin I, t\in J}} X_{IJ}^s,\quad
  A_2 = \sum_{\substack{I,J\subseteq S \\ s\notin I, s\in J \\ t\in I, t\notin J}} X_{IJ}^t\]
ein und erhalten dafür Terme der Gestalt
\[\sum_{I_0,I_1,\ldots,I_k\subseteq S} X_{I_0 I_1}^s X_{I_1 I_2}^t X_{I_2 I_3}^s \ldots\]
Dabei wird für $A_1 A_2 A_1 \ldots$ nur über solche $I_i$ summiert, die $s\in I_{2i}\setminus I_{2i+1}$ und $t\in I_{2i+1}\setminus I_{2i}$ für alle $i$ erfüllen. Da aber $X_{IJ}^s = 0$ ist, wenn $s\notin I\setminus J$ gilt, sind davon nur die Bedingungen für $I=I_0$ und $I_k=J$ nichttrivial. Wir können also auch einfach über alle möglichen Pfade der Länge $k$ summieren, die zwischen diesen $I$ und $J$ verlaufen. Wir erhalten somit:
\[\underbrace{A_1 A_2 \ldots}_{k} = \begin{cases}
	\sum\limits_{\substack{I,J\subseteq S \\ s\in I, s\in J \\ t\notin I, t\notin J}} P_{IJ}^k(s,t) & \text{falls}\,2\mid k \\
	\sum\limits_{\substack{I,J\subseteq S \\ s\in I, s\notin J \\ t\notin I, t\in J}} P_{IJ}^k(s,t) & \text{falls}\,2\nmid k
	\end{cases}\]
Für das andere Produkt erhalten wir analog:
\[\underbrace{A_2 A_1 \ldots}_{k} = \begin{cases}
	\sum\limits_{\substack{I,J\subseteq S \\ s\notin I, s\notin J \\ t\in I, t\in J}} P_{IJ}^k(t,s) & \text{falls}\,2\mid k \\
	\sum\limits_{\substack{I,J\subseteq S \\ s\notin I, s\in J \\ t\in I, t\notin J}} P_{IJ}^k(t,s) & \text{falls}\,2\nmid k
	\end{cases}\]
Wenn wir dies von links mit $E_I$ und von rechts mit $E_J$ multiplizieren, erhalten wir also entweder $0$ oder $P_{IJ}^k(s,t)$ und $P_{IJ}^k(t,s)$. Das Element $E_I \tau_{m-1}(A) E_J\in\mathfrak{I}$ ist, wenn es nicht Null ist, also genau von der behaupteten Gestalt
\[P_{IJ}^{m-1}(s,t) + a_{m-2} P_{IJ}^{m-2}(s,t) + \ldots + a_2 P_{IJ}^2(s,t) + a_1 P_{IJ}^1(s,t) + a_0 P_{IJ}^0(s,t)\]
wobei $\tau_{m-1}(Y) = Y^{m-1} + a_{m-2} Y^{m-2} + \ldots + a_2 Y^2 + a_1 Y^1 + a_0$ ist, oder es ist gleich der symmetrischen Variante, in der $s$ und $t$ vertauscht sind.

\medbreak
Der zweite Typ von Relationen ist einfacher: $R_\beta$ ist gleich
\[\sum_{\substack{I,J\subseteq S \\ s\in I, s\notin J \\ t\in I, t\notin J}} X_{IJ}^s - X_{IJ}^t.\]
Für $I$ und $J$, die in dieser Summe nicht vorkommen, ist $E_I R_\beta E_J=0$. Für die restlichen Teilmengen ergeben sich die Relatoren $X_{IJ}^s - X_{IJ}^t$.

\medbreak
Kommen wir zum letzten Typ von Relationen $R_{\gamma,k} = CA^k JB$. Wir haben oben bereits die Potenzen von $A$ berechnet und erhalten daraus:
\[CA^kJB = \begin{cases}
C_2 (A_1 A_2)^{\frac{k}{2}} B_1 - C_1 (A_2 A_1)^{\frac{k}{2}} B_2 & \text{falls}\,2\mid k \\
- C_2 (A_1 A_2)^{\frac{k-1}{2}} A_1 B_2 + C_1 (A_2 A_1)^{\frac{k-1}{2}} A_2 B_1 & \text{falls}\,2\nmid k
\end{cases}\]
Wir benutzen nun die Definitionen
\[B_1 = \sum_{\substack{I,J\subseteq S \\ s\in I, s\notin J \\ t\notin I, t\notin J}} X_{IJ}^s,\quad
  B_2 = \sum_{\substack{I,J\subseteq S \\ s\notin I, s\notin J \\ t\in I, t\notin J}} X_{IJ}^t\]
\[C_1 = \sum_{\substack{I,J\subseteq S \\ s\in I, s\notin J \\ t\in I, t\in J}} X_{IJ}^s,\quad
  C_2 = \sum_{\substack{I,J\subseteq S \\ s\in I, s\in J \\ t\in I, t\notin J}} X_{IJ}^t\]
und setzen ein: Falls $s,t\in I$ oder $s,t\notin J$ verletzt ist, ist $E_I CA^k JB E_J=0$, weil $E_I C=0$ bzw. $BE_J=0$ ist. Für die anderen Fälle ist:
\begin{align*}
	E_I CA^kJB E_J &= \sum_{I_0,\ldots,I_k\subseteq S} X_{I I_0}^t X_{I_0 I_1}^s \ldots X_{I_{k-1} I_k}^t X_{I_k J}^s - X_{I, I_0}^s X_{I_0 I_1}^t \ldots X_{I_{k-1} I_k}^s X_{I_k J}^t \\
	&= P_{IJ}^{k+2}(t,s) - P_{IJ}^{k+2}(s,t)
\end{align*}
falls $2\mid k$ und
\begin{align*}
	E_I CA^kJB E_J &= \sum_{I_0,\ldots,I_k\subseteq S} - X_{I I_0}^t X_{I_0 I_1}^s \ldots X_{I_{k-1} I_k}^s X_{I_k J}^t + X_{I I_0}^s X_{I_0 I_1}^t \ldots X_{I_{k-1} I_k}^t X_{I_k J}^s \\
	&= -P_{IJ}^{k+2}(t,s) + P_{IJ}^{k+2}(s,t)
\end{align*}
falls $2\nmid k$ ist. Wir erhalten also die Relationen dritten Typs. Damit ist der Satz vollständig bewiesen.
\end{proof}

\begin{convention}
\index{terms}{Einparameterfall}
\index{symbols}{XIJ@$X_{IJ}$}
Aufgrund der Relationen vom Typ $(\beta)$ erlauben wir es uns, von jetzt an im Einparameterfall $X_{IJ}$ zu schreiben, wenn wir den gemeinsamen Wert von $X_{IJ}^s\in\OmegaGy$ für alle $s\in I\setminus J$ meinen.
\end{convention}

\subsection{Anwendungen}

\begin{corollary}[Eine $\IZ/2$-Graduierung auf $\OmegaGy$]\label{wgraph_alg:grading_OmegaGy}
\index{terms}{Graduierung}
\index{symbols}{OmegaGy@$\OmegaGy_{\pm}$}
Im Einparameterfall ist $\OmegaGy$ auf natürliche Weise $\IZ/2$-graduiert mittels
\[\OmegaGy_{+} := \text{span}_\IZ\big(\Set{E_I | I\subseteq S}\cup\Set{X_{I_0 I_1} X_{I_1 I_2} \ldots X_{I_{k-1} I_k} | I_j\subseteq S, 2\mid k}\big)\]
und
\[\OmegaGy_{-} := \text{span}_\IZ\Set{X_{I_0 I_1} X_{I_1 I_2} \ldots X_{I_{k-1} I_k} | I_j\subseteq S, 2\nmid k}.\]
\end{corollary}
\begin{proof}
$\IZ\mathcal{Q}$ besitzt die kanonische $\IZ$"~Graduierung durch die Pfadlänge, also auch eine $\IZ/2$"~Graduierung. Da die Polynome $\tau_r$ nur gerade bzw. nur ungerade Potenzen besitzen, sind die Relationen vom Typ $(\alpha)$ homogen bzgl. dieser Graduierung. Die Relationen vom Typ $(\beta)$ und $(\gamma)$ sind sowieso homogen (schon bzgl. der $\IZ$"~Graduierung). Also ist $\OmegaGy$ ein homogener Quotient und übernimmt damit die $\IZ/2$"~Graduierung.
\end{proof}

\begin{remark}
Diese Graduierung könnte man auch realisieren durch den Automorphismus $\sigma: e_s\mapsto e_s, x_s\mapsto -x_s$ von $\Omega$ und $\OmegaGy$, indem man $\OmegaGy_{\pm}:=\operatorname{Eig}_{\pm1}(\sigma)$ setzt. Das funktioniert allerdings nur, wenn man $\operatorname{Eig}_{+1}(\sigma)\cap\operatorname{Eig}_{-1}(\sigma)=0$ zeigen kann (und wenn man einen unabhängigen Beweis der Wohldefiniertheit von $\sigma$ findet). Hinreichend dafür wäre etwa, dass $\OmegaGy$ keine $2$-Torsion hat. Die Vermutung, dass dies tatsächlich so sein muss, liegt nahe, aber bisher ist mir dafür kein Argument eingefallen. Natürlich könnte man auch einfach $\IZ[\frac{1}{2}]\OmegaGy$ betrachten und das Problem so umgehen.
\end{remark}
\begin{remark}
Es gibt eine Interpretation dieser Graduierung in der Sprache der $W$"~Gra\-phen: Ein $\OmegaGy$"~Modul $V$ ist genau dann ein graduierbarer Modul, wenn er sich (wieder bis auf Wahl einer Basis) durch einen \textit{bipartiten} $W$"~Graphen realisieren lässt. Dies ist Gegenstand des folgenden Lemmas. Die Beweisideen tauchen bereits in \cite{Gyoja} auf, werden dort aber nicht als Aussage über graduierte Moduln verstanden.
\end{remark}

\begin{lemma}[Graduierte Moduln]\label{wgraph_alg:graded_mods}
\index{terms}{W-Graph@$W$-Graph!-Modul!graduierter}\index{terms}{Graduierung}\index{terms}{Signum}\index{terms}{W-Graph@$W$-Graph!gerader}\index{terms}{W-Graph@$W$-Graph!bipartiter}
Sei $k$ ein kommutativer Ring mit $2\in k^\times$. Für einen $k\OmegaGy$"~Modul $V$ sind äquivalent:
\begin{enumerate}
	\item $V$ besitzt eine Zerlegung $V=V_+\oplus V_-$, mit der $V$ zu einem graduierten $k\OmegaGy$"~Modul wird bzgl. der obigen $\IZ/2$"~Graduierung von $k\OmegaGy$.
	\item Die Darstellung $k\OmegaGy\to\End_k(V)$ lässt sich auf $k\OmegaGy\rtimes\langle\sigma\rangle$ fortsetzen, wobei $\sigma$ der obige Automorphismus $e_s\mapsto e_s, x_s\mapsto -x_s$ ist.
\end{enumerate}
Nun habe $k$ zusätzlich die Eigenschaft, dass (endlich erzeugte) projektive $k$"~Moduln stets frei sind. Weiter sei $V$ als $k$"~Modul frei (und endlich erzeugt). Dann sind auch folgende Aussagen zu den obigen äquivalent:
\begin{enumerate}[resume]
	\item $V$ lässt sich durch einen bipartiten $W$"~Graphen mit konstanten Kantengewichten in $k\subseteq k[\Gamma]$ realisieren.
	\item $V$ lässt sich durch einen \udot{geraden} $W$"~Graphen $(\mathfrak{C},I,m)$ mit konstanten Kantengewichten in $k\subseteq k[\Gamma]$ realisieren, d.\,h. es gibt eine Funktion $\sgn:\mathfrak{C}\to\Set{\pm 1}$ mit
	\[\forall x,y\in\mathfrak{C}: \sgn(x)\sgn(y)m_{xy}^s = -m_{xy}^s\]
\end{enumerate}
\end{lemma}
\begin{proof}
a.$\iff$b. ist allgemeingültig: Für jede $k$"~Algebra $A$, jeden $kA$"~Modul $V$ und jede abelsche Gruppe $G\leq\operatorname{Aut}(A)$ mit der Eigenschaft, dass alle $g\in G$ diagonalisierbar auf $A$ wirken, wird $A$ eine $\Hom(G,k^\times)$"~graduierte Algebra durch
\[A_\chi := \Set{x\in A | \forall g\in G: gx = \chi(g)x}\]
und ein $A$"~Modul $V$ ist genau dann graduierbar, wenn $A\to\End(V)$ so auf $A\rtimes G$ fortsetzbar ist, dass alle $g\in G$ diagonalisierbar auf $V$ wirken. Die homogenen Komponenten der Graduierung $V=\bigoplus_{\chi} V_\chi$ sind in diesem Fall gegeben durch die Eigenräume, also ${V_\chi := \Set{v\in V | \forall g\in G: gv = \chi(g)v}}$.

Da $2\in k^\times$ ist und $\sigma$ Ordnung $2$ hat, operiert $\sigma$ immer diagonalisierbar auf $\OmegaGy$ und $V$. Das zeigt die Behauptung.
\medbreak

a.$\implies$c. folgt, weil $E_I\in\OmegaGy_+$ ist, denn deshalb gilt $E_I V_{\pm} \subseteq V_{\pm}$ und man kann die Zerlegung $V=\bigoplus E_I V$ somit weiter verfeinern zu $V_{\pm} = \bigoplus E_I V_{\pm}$. Wählt man eine Basis $\mathfrak{C}$, die an diese Zerlegung angepasst ist, dann definieren $V_+$ und $V_-$ eine Partition der Eckenmenge des $W$"~Graphen. Da $X_{IJ}\in\OmegaGy_-$ ist, gilt $X_{IJ} V_{J,\pm} \subseteq V_{I,\mp}$, d.\,h. die Kantengewichte $m_{xy}$ sind höchstens dann ungleich Null, wenn $x\in V_+$ und $y\in V_-$ ist oder umgekehrt. Das heißt nichts anderes, als dass der Graph bipartit ist.

\medbreak
c.$\implies$d. ist klar, indem man eine Bipartition des $W$"~Graphen wählt und $\sgn$ auf der einen Teilmenge konstant gleich $+1$ und auf der anderen konstant gleich $-1$ wählt.

\medbreak
d.$\implies$b. $\sigma$ operiert auf dem $W$"~Graph"=Modul von $(\mathfrak{C},I,m,\sgn)$ durch $\sigma x := \sgn(x)x$. Das definiert die Fortsetzung.
\end{proof}

\begin{remark}
\index{terms}{W-Graph@$W$-Graph!gerader}\index{terms}{Hecke-Algebra!$\IZ[q]$-Form}
Die geraden $W$"~Graphen sind im Einparameterfall genau diejenigen Moduln, die sich bereits über der $\IZ[v^2]$-Form von $H$ realisieren lassen: Durch
\[\dot{x} := \begin{cases} x & \sgn(x)=+1 \\ vx & \sgn(x)=-1 \end{cases}\]
für alle $x\in\mathfrak{C}$ wird eine Basis von $V$ definiert, die $\dot{T}_s \dot{x} \in \operatorname{span}_{\IZ[v^2]} \mathfrak{C}$ für alle $s\in S$ erfüllt.
\end{remark}

\begin{example}[Kazhdan-Lusztig-Zellen]
\index{terms}{Kazhdan-Lusztig!-$W$-Graph}\index{terms}{Kazhdan-Lusztig!-$\mu$}
Wir haben in \ref{KL:lemma:symmetry_KL_poly} gesehen, dass für alle Homomorphismen $\Gamma\to\Set{\pm1}$ (mit den dortigen Bezeichnungen) $\theta(\mu_{xy}^s) = (-1)^{L(x)+L(y)+L(s)}\mu_{xy}^s$ gilt. 

Im Einparameterfall ist $\mu_{xy}^s$ ganzzahlig und daher unter allen Automorphismen $\sigma$ invariant. Es ergibt sich hieraus u.\,a. $\mu_{xy}^s \neq 0 \implies l(x)+l(y)+1 \equiv 0\mod 2$ für alle $x,y\in W$ und alle $s\in S$, d.\,h. $\sgn(x):=(-1)^{l(x)}$ definiert eine Signumsfunktion, die die Kazhdan-Lusztig-Zellen zu geraden $W$"~Graphen macht. Die Bipartition des $W$"~Graphen ist durch die Einteilung in Elemente $x\in W$ gerader und ungerader Länge gegegeben.
\end{example}

\subsection{Ein einfacherer Köcher}

\begin{remark}
Wir destillieren eine sehr einfache Folgerung aus den Relationen in Satz \ref{wgraph_alg:relations}:
\end{remark}

\begin{lemmadef}[Kompatibilitätsgraph, siehe \cite{Stembridge2008admissble}]
\index{terms}{Kompatibilitätsgraph}\index{terms}{Inklusionskante}\index{terms}{Transversale Kante}
\index{symbols}{QW@$\mathcal{Q}_{W}$}
Es sei $(W,S)$ eine Coxeter"=Gruppe. Wir betrachten den Einparameterfall.

Es gilt: Falls $X_{IJ} \neq 0$ in $\OmegaGy(W,S)$ ist, muss jedes $s\in I\setminus J$ mit jedem $t\in J\setminus I$ durch eine Kante im Dynkin"=Diagramm von $(W,S)$ verbunden sein.

\bigbreak
Der \udot{Kompatibilitätsgraph} $\mathcal{Q}_W$ sei wie folgt definiert: Seine Ecken seien die Teilmengen $I\subseteq S$. Eine gerichtete Kante $I\leftarrow J$ existiere genau dann, wenn die beiden folgenden Bedingungen erfüllt sind:
\begin{enumerate}
	\item $I\setminus J\neq\emptyset$.
	\item Alle $s\in I\setminus J$ sind mit allen $t\in J\setminus I$ im Dynkin"=Diagramm von $(W,S)$ durch je eine Kante verbunden.
\end{enumerate}

Als \udot{Inklusionskanten} bezeichnen wir alle Kanten $I\leftarrow J$, für die $I\supseteq J$ ist. Im Gegensatz bezeichnen wir Kanten mit $I \not\supseteq J$ als \udot{transversale} Kanten. Diese treten immer paarweise auf: Ist $I\to J$ transversal, dann existiert auch die transversale Kante $J\to I$.
\end{lemmadef}
\begin{proof}
Wenn $m=\ord(st)=2$ ist, dann ist die Relation vom Typ $(\alpha)$ die Gleichung $X_{IJ}^s = 0$ für alle $I,J\subseteq S$ mit $s\in I\setminus J$ und $t\in J\setminus I$, da $\tau_{2-1}(X)=X$ gilt.
\end{proof}

\begin{remark}
Die Bedingung im Lemma ist natürlich trivialerweise wahr, wenn $J\setminus I=\emptyset$ ist. Sie liefert also nur ein Kriterium für die Kanten $J\to I$, die nicht von einer Inklusion herkommen.
\end{remark}
\begin{remark}
Es folgt insbesondere, dass für transversale Kanten $I \leftrightarrows J$ von der Teilmenge ${I \Delta J\subseteq S}$ des Dynkin"=Diagramms ein vollständiger bipartiter Graphen induziert wird mit $I\setminus J$ und $J\setminus I$ als den beiden Teilen der Bipartition. Unter den endlichen Coxeter"=Gruppen sind nur $I_2(m)$, $A_3$, $B_3$, $H_3$ und $D_4$ von dieser Gestalt. Insbesondere folgt $\abs{I \Delta J}\leq 4$. Das liefert eine wesentliche Einschränkung an die Gestalt des Kompatibilitätsgraphen.
\end{remark}
\begin{remark}
Die Abbildung $\Xi=\IZ\mathcal{Q}\to\OmegaGy$ faktorisiert als $\IZ\mathcal{Q}\to\IZ\mathcal{Q}_W\to\OmegaGy$, wobei der erste Pfeil durch $\widetilde{E}_I\mapsto E_I$ und $\widetilde{X}_{IJ}^s \mapsto X_{IJ}$ für alle $I,J\subseteq S$, $s\in I\setminus J$ definiert ist. Wir können und werden $\OmegaGy$ daher als Quotienten von $\IZ\mathcal{Q}_W$ auffassen.
\end{remark}

\begin{example}
In Abbildung \ref{fig:wgraph_alg:comp_graphs} sind die Kompatibilitätsgraphen der endlichen, irreduziblen Coxeter"=Gruppen vom Rang $\leq 4$ dargestellt.

Der Übersichtlichkeit halber sind Inklusionskanten nur für Rang 2 und Rang 3 und dort auch nur die Inklusionskanten zwischen Mengen mit $\abs{I\setminus J}=1$ dargestellt. Außerdem sind Paare transversaler Kanten zusammengefasst. Statt zwei gerichteter Kanten $I\to J$ und $J\to I$ ist nur eine (fette) ungerichtete Kante eingezeichnet.

\begin{figure}[htp]
	\centering
	\begin{tabular}{c|c}
		\begin{minipage}[c][ 75pt]{ 75pt}
		\begin{tikzpicture}

\node[Vertex] (E_empty) at (0,0) {$\emptyset$};

\node[Vertex] (E_1) at (-1,1) {$1$};
\node[Vertex] (E_2) at (+1,1) {$2$};

\node[Vertex] (E_12) at (0,2) {$12$};

\path[EdgeI]
	(E_empty) edge (E_1)
	(E_empty) edge (E_2)
	(E_1) edge (E_12)
	(E_2) edge (E_12);

\path[EdgeT]
	(E_1) edge (E_2);
\end{tikzpicture}
		\end{minipage}	& 
		\begin{minipage}[c][100pt]{100pt}
		\begin{tikzpicture}

\node[Vertex] (E_empty) at (0,0) {$\emptyset$};

\node[Vertex] (E_1) at (-1.5,1) {$1$};
\node[Vertex] (E_2) at ( 0.0,1) {$2$};
\node[Vertex] (E_3) at (+1.5,1) {$3$};

\node[Vertex] (E_23) at (-1.5,2) {$23$};
\node[Vertex] (E_13) at ( 0.0,2) {$13$};
\node[Vertex] (E_12) at (+1.5,2) {$12$};

\node[Vertex] (E_123) at (0,3) {$123$};

\path[EdgeI]
	(E_empty) edge (E_1)
	(E_empty) edge (E_2)
	(E_empty) edge (E_3)
	(E_1) edge (E_12)
	(E_1) edge (E_13)
	(E_2) edge (E_12)
	(E_2) edge (E_23)
	(E_3) edge (E_13)
	(E_3) edge (E_23)
	(E_12) edge (E_123)
	(E_13) edge (E_123)
	(E_23) edge (E_123);

\path[EdgeT]
	(E_1) edge (E_2)
	(E_2) edge (E_3)
	      edge (E_13)
	(E_23) edge (E_13)
	(E_13) edge (E_12);

\end{tikzpicture}
		\end{minipage} \\
		\hline
		\begin{minipage}[c][320pt]{200pt}
		\begin{tikzpicture}

\node[Vertex] (E_empty) at (0,0) {$\emptyset$};

\node[Vertex] (E_1) at (-3.0,1) {$1$};
\node[Vertex] (E_2) at (-1.0,1) {$2$};
\node[Vertex] (E_3) at (+1.0,1) {$3$};
\node[Vertex] (E_4) at (+3.0,1) {$4$};

\node[Vertex] (E_12) at (-3.0,2.5) {$12$};
\node[Vertex] (E_13) at (-1.0,2.5) {$13$};
\node[Vertex] (E_23) at ( 0.0,2.0) {$23$};
\node[Vertex] (E_14) at ( 0.0,3.0) {$14$};
\node[Vertex] (E_24) at (+1.0,2.5) {$24$};
\node[Vertex] (E_34) at (+3.0,2.5) {$34$};

\node[Vertex] (E_123) at (-3.0,4) {$123$};
\node[Vertex] (E_124) at (-1.0,4) {$124$};
\node[Vertex] (E_134) at (+1.0,4) {$134$};
\node[Vertex] (E_234) at (+3.0,4) {$234$};

\node[Vertex] (E_1234) at (0,5) {$1234$};

%
%
%
%
%
%
%
%
%
%
%
%

\path[EdgeT]
	(E_1) edge (E_2)
	(E_2) edge (E_3) edge (E_13)
	(E_3) edge (E_4) edge (E_24)
	
	(E_13) edge (E_23) edge (E_14) edge (E_12)
	(E_24) edge (E_23) edge (E_14) edge (E_34)
	
	(E_123) edge (E_124)
	(E_124) edge (E_134) edge (E_13)
	(E_134) edge (E_234) edge (E_24);

\end{tikzpicture}
		\end{minipage} &
		\begin{minipage}[c][320pt]{150pt}
		\begin{tikzpicture}[
	Vertex/.style = {inner sep = 1pt,outer sep = 2pt,minimum size=15pt,
	                 circle,draw=black!70,thick,
	                 font=\scriptsize},
	EdgeT/.style = {ultra thick,black},
	EdgeI/.style = {black!80,->}
]

\node[Vertex] (E_empty) at (0,0) {$\emptyset$};

\node[Vertex] (E_0) at (-2,1) {$0$};
\node[Vertex] (E_1) at ( 0,1) {$1$};
\node[Vertex] (E_3) at (+2,1) {$3$};

\node[Vertex] (E_2) at ( 0,2) {$2$};

\node[Vertex] (E_13) at (-2,4) {$13$};
\node[Vertex] (E_03) at ( 0,4) {$03$};
\node[Vertex] (E_01) at (+2,4) {$01$};

\node[Vertex] (E_02) at (-2,6) {$02$};
\node[Vertex] (E_12) at ( 0,6) {$12$};
\node[Vertex] (E_23) at (+2,6) {$23$};

\node[Vertex] (E_013) at (0,8) {$013$};

\node[Vertex] (E_123) at (-2,9) {$123$};
\node[Vertex] (E_023) at ( 0,9) {$023$};
\node[Vertex] (E_012) at (+2,9) {$012$};

\node[Vertex] (E_1234) at (0,10) {$0123$};

%
%
%
%
%
%
%
%
%
%
%

\path[EdgeT]
	(E_2) edge (E_0)
	      edge (E_1)
	      edge (E_3)
	      
	      edge (E_13)
	      edge (E_03)
	      edge (E_01)
	(E_13) edge (E_02)
	       edge (E_12)
	(E_03) edge (E_02)
	       edge (E_23)
	(E_01) edge (E_12)
	       edge (E_23)

	(E_013) edge (E_02)
	        edge (E_12)
	        edge (E_23)
	        
	        edge (E_123)
	        edge (E_023)
	        edge (E_012)

	(E_2) edge[bend right=30] (E_013);

\end{tikzpicture} 
		\end{minipage}
	\end{tabular}
	
	\setcapwidth[c]{0.80\textwidth}
	\caption{Kompatibilitätsgraphen im Einparameterfall; oben links für $I_2(m)$, oben rechts für $A_3$, $B_3$ und $H_3$, unten links für $A_4$, $B_4$ und $F_4$ sowie unten rechts für $D_4$.}
	\label{fig:wgraph_alg:comp_graphs}
\end{figure}
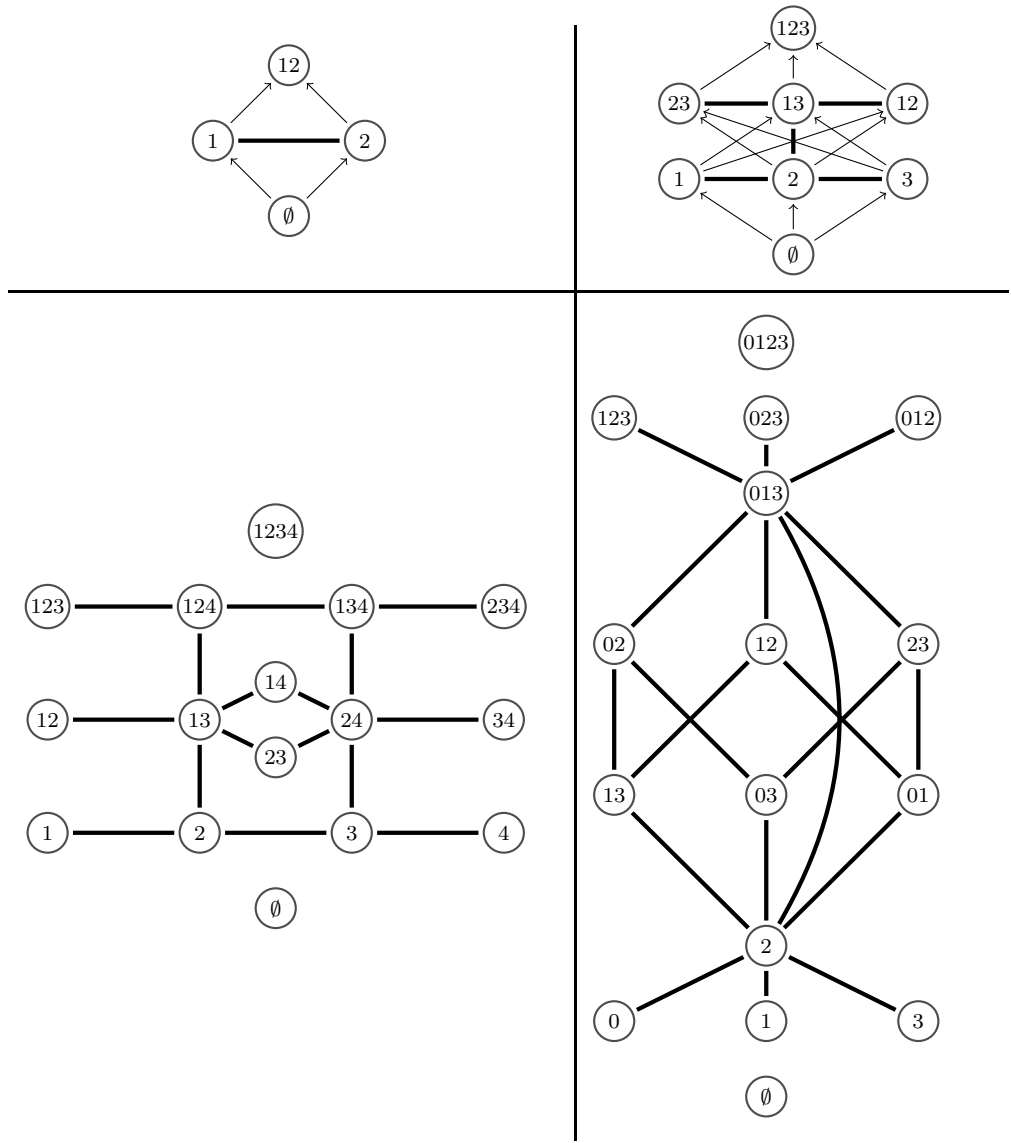
\end{example}
\section{Gyojas Vermutung in Spezialfällen}

\subsection{Vorbereitungen}

\begin{remark}
Gyojas Vermutung stellt eine Verbindung zwischen $H$- und $\Omega$"~Moduln her. Besonders eine der äquivalenten Formulierungen ist interessant: Sind halbeinfache $\Omega$"~Moduln, die als Fortsetzungen von isomorphen $H$"~Moduln entstehen, bereits als $\Omega$"~Moduln isomorph? Oder anders formuliert ist dies eine Frage nach der Rigidität von $W$"~Graphen: Wie viel der Struktur eines $W$"=Graphen wird alleine durch den von ihm induzierten $H$"~Modul bestimmt? Folgendes Lemma gibt eine nützliche Teilantwort auf diese Frage.
\end{remark}

\begin{lemma}\label{wgraphs:unique_I_inv}
Sei $k$ eine kommutative $\IZ[\Gamma]$"~Algebra, in der $(v_s+v_s^{-1})\cdot 1_k$ ein Nichtnullteiler ist für alle $s\in S$. Weiter sei $V$ ein $k\Omega$"~Modul, der als $k$"~Modul frei ist. Dann gilt:
\[e_s V = \Set{x\in V | T_s x=-v_s^{-1} x} = \operatorname{Eig}_{-v_s^{-1}}(T_s;V)\]
Ist speziell $V$ der $W$"~Graph"=Modul eines $W$"~Graphen $(\mathfrak{C},I,m)$, so sind die Vielfachheiten der Eckenlabel, d.\,h. die Dimensionen $\dim_k E_I V$, eindeutig durch den $kH$"~Isomorphietyp von $V$ bestimmt.
\end{lemma}
\begin{proof}
Ist $u\in e_s V$, dann ist:
\begin{align*}
	T_s u &= (-v_s^{-1} e_s+v_s (1-e_s)+x_s)(e_s u) \\
	&= (-v_s^{-1} \underbrace{e_s^2}_{=e_s}+v_s \underbrace{(1-e_s)e_s}_{=0}+\underbrace{x_s e_s}_{=0}) u \\
	&= -v_s^{-1} e_s u \\
	&= -v_s^{-1} u
\end{align*}
und daher $u\in\operatorname{Eig}_{-v_s^{-1}}(T_s;V)$. Wenn umgekehrt $T_s u = -v_s^{-1} u$ gilt, folgt:
\begin{align*}
	-v_s^{-1} u &= -v_s^{-1} e_s u + v_s(1-e_s)u + x_s u \\
	\implies 0 &= (v_s+v_s^{-1}) (1-e_s) u + x_s u \\
	\implies 0 &= (v_s+v_s^{-1}) (1-e_s)^2 u + \smash{\underbrace{(1-e_s) x_s}_{=0}} u \\
	&= (v_s+v_s^{-1}) (1-e_s) u \\
	\implies 0 &= (1-e_s) u \\
	\implies u &\in e_s V
\end{align*}
Aus $e_s = \sum_{s\in I} E_I$ folgt, dass sich $\bigcap_{s\in I} e_s V = \bigoplus_{J\supseteq I} E_J V$ eindeutig durch die Operation der $T_s$ charakterisieren lässt. Wir erhalten insbesondere, dass die Vielfachheit
\[\abs{\Set{x\in\mathfrak{C} | I(x)=I}} = \dim E_I V = \dim \frac{\bigcap_{s\in I} \operatorname{Eig}_{-v_s^{-1}}(T_s; V) }{ \sum_{s\notin I} \bigcap_{t\in I\cup\Set{s}} \operatorname{Eig}_{-v_t^{-1}}(T_t;V) }\]
eindeutig durch die $H$-Operation auf $V$ bestimmt ist für alle $I\subseteq S$.
\end{proof}

\begin{remark}
\index{terms}{Signum}\index{terms}{Charakter}
Im endlichdimensionalen Fall können wir die Vielfachheiten in der Tat alleine aus dem Charakter $\chi_V$ von $V$ bestimmen: Wenn wir $V=\bigoplus_{\chi\in\Irr(H)} V_\chi$ für die Zerlegung in homogene Komponenten schreiben, ist die Signumskomponente von $V$ gleich ${V_{\textrm{sgn}}=\bigcap_{s\in S} \operatorname{Eig}_{-v_s^{-1}}(T_s;V)}$. Allgemeiner ist $\bigcap_{s\in I} \operatorname{Eig}_{-v_s^{-1}}(T_s;V)$ genau die Signumskomponente von $\text{Res}_{H_I}^H (V)$. Die Dimensionen sind also genau gleich den Vielfachheiten des Signumscharakters in $\chi_V$.

In der Tat können wir diese Vielfachheiten schon aus dem Charakter des $F[W]$"~Mo\-duls bestimmen, welcher entsteht, wenn wir $V$ mit $v^\gamma\mapsto 1$ spezialisieren. Es kann also ausschließlich mit der Coxeter"=Gruppe gerechnet werden, es ist keine Rechnung in $H$ nötig.
\end{remark}

\begin{remark}
Wir werden auch das folgende einfache Lemma benutzen, um Aussagen über einfache Moduln von Pfadalgebren oder verallgemeinerten Pfadalgebren zu treffen. Wir definieren dafür zunächst einmal verallgemeinerte Pfadalgebren.
\end{remark}

\begin{definition}[Verallgemeinerte Pfadalgebren, {siehe \cite{zhang2007structures}}]
\index{terms}{Pfadalgebra!verallgemeinerte}
Sei $\mathcal{K}$ ein Köcher mit endlicher Eckenmenge $V$. Sei weiter $k$ ein kommutativer Ring und für jedes $v\in V$ sei $A_v$ eine $k$"~Algebra.

Dann ist die \udot{verallgemeinerte Pfadalgebra auf $\mathcal{K}$ mit Knotenalgebren $A_v$} wie folgt definiert: Für jeden Pfad $P : v_0 \xleftarrow{e_1} v_1 \xleftarrow{e_2} \ldots \xleftarrow{e_{n-1}} v_{n-1} \xleftarrow{e_n} v_n$ in $\mathcal{K}$ definiere
\[ A_P := A_{v_0} \otimes_k A_{v_1} \otimes_k \ldots \otimes_k A_{v_{n-1}} \otimes_k A_{v_n}.\]
Insbesondere ist $A_P = A_{v_0}$, wenn $n=0$ ist. Der zugrundeliegende $k$"~Modul der Algebra ist dann definiert als
\[ A := \bigoplus_{P \text{ Pfad in }\mathcal{K}} A_P.\]
Die Multiplikation ist darauf wie folgt definiert: Ist $P$ wie oben und
\[Q=v_0' \xleftarrow{e_1'} v_1' \xleftarrow{e_2'} \ldots \xleftarrow{e_{m-1}'} v_{m-1}' \xleftarrow{e_m'} v_m'\]
ein weiterer Pfad in $\mathcal{K}$, dann definiere zum einen $A_P \cdot A_Q := 0$, falls $v_n \neq v_0'$ ist, und zum anderen definiere die Multiplikation $A_P \otimes_k A_Q \to A_{PQ}$ durch
\[(a_0 \otimes \ldots \otimes a_n)\cdot(a_0' \otimes \ldots \otimes a_m') := a_0 \otimes \ldots \otimes (a_n a_0') \otimes \ldots \otimes a_m',\]
falls $v_n=v_0'$ gilt.
\end{definition}

\begin{remark}
Das Einselement dieser Algebra ist $\sum_{v\in V} 1_{A_v}$. Daher war die Forderung einer endlichen Eckenmenge wichtig.
\end{remark}
\begin{remark}
Zu jeder Kante $v \xleftarrow{e} w$ in $\mathcal{K}$ existiert ein \udot{Kantenelement} $x_e:=1_{A_v} \otimes 1_{A_w}\in A_{e}\subseteq A$. Die Kantenelemente erzeugen zusammen mit den Knotenalgebren $A_v$ die gesamte Algebra $A$: Die Elemente $a_1\otimes\ldots\otimes a_n$ von $A_P$ kann man jetzt nämlich eindeutig als $a_0 x_{e_1} a_1 \ldots a_{n-1} x_{e_n} a_n$ mit $a_i\in A_{v_i}$ schreiben.
\end{remark}
\begin{example}
Die gewöhnliche Pfadalgebra $k\mathcal{K}$ ergibt sich aus dieser Definition, indem man $A_v:=k$ für alle $v\in V$ einsetzt.
\end{example}

\begin{example}\label{path_algs:ex:generic_examples}
Es sei $A$ eine $k$-Algebra, $1=\sum_{i=1}^n f_i$ eine Zerlegung der Eins in paarweise orthogonale Idempotente, wobei wir o.\,B.\,d.\,A. $f_i\neq 0$ annehmen, und $X\subseteq A$ eine Teilmenge derart, dass $X\cup\bigcup_{i=1}^n f_i A f_i$ ein Erzeugendensystem von $A$ ist.

Dann ist $A$ auf kanonische Weise Quotient einer verallgemeinerte Pfadalgebra $\widetilde{A}$: Der Köcher hat die Ecken $1,\ldots,n$, die Menge der Kanten $i\leftarrow j$ ist durch $\Set{x\in X | f_i x f_j \neq 0}$ gegeben und die Algebra an der Ecke $i$ ist $f_i A f_i$.

Der Quotient $\widetilde{A}\to A$ ist an den Ecken durch die Inklusionen $f_i A f_i \to A$ und auf den Kantenelementen durch $(i \xleftarrow{x} j) \mapsto f_i x f_j$ für alle $x\in X$ gegeben. Da $x=\sum_{i,j} f_i x f_j$ im Bild liegt für alle $x\in X$, ist dies wirklich ein surjektiver Homomorphismus.
\end{example}

\begin{lemmadef}\label{wgraph:def:Q_V}
\index{terms}{stark zusammenhängend}\index{terms}{Zusammenhangskomponente!starke}\index{terms}{Pfadalgebra!verallgemeinerte}
\index{symbols}{KM@$\mathcal{K}_M$}
Sei $\mathcal{K}=(V,E)$ irgendein Köcher mit endlich vielen Knoten, $k$ ein kommutativer Ring, $A$ eine verallgemeinerte Pfadalgebra auf diesem Köcher und $M$ ein $A$"~Modul. Dann definiere einen Teilgraphen $\mathcal{K}_M=(V_M,E_M)$ von $\mathcal{K}$ wie folgt: $V_M$ sei $\Set{v | 1_{A_v} M\neq 0}$ und die Kante $v \xleftarrow{e} u$ sei in $E_M$ genau dann, wenn $x_e V_M\neq 0$.

Mit diesen Bezeichnungen gilt:
\begin{enumerate}
	\item Ist $M$ unzerlegbar, so ist $\mathcal{K}_M$ zusammenhängend.
	\item Ist $M$ einfach, so ist $\mathcal{K}_M$ stark zusammenhängend.
	\end{enumerate}
\end{lemmadef}
\begin{proof}
a. Ist $V_M=U\dot{\cup} W$ eine Zerlegung in nichtleere Teilmengen, sodass keine Kanten zwischen $U$ und $W$ in $\mathcal{K}_M$ verlaufen, dann betrachte die Idempotente $1_U:=\sum_{v\in U} 1_{A_v}$ und $1_W:=\sum_{v\in W} 1_{A_v}$. Da $U$ und $W$ disjunkt sind, sind es orthogonale Idempotente. Da alle $v\in V\setminus V_M$ aufgrund der Definition schon $1_{A_v} M = 0$ erfüllen, operiert $1_U+1_W$ als Identität auf $M$. Dann ist also $M=1_U M \oplus 1_W M$ als $k$"~Modul.

Da alle von $U$ ausgehenden Kanten als Null auf $M$ operieren, folgt $x_e (1_U M) \subseteq 1_U M$ für alle Kanten $e\in E$. Analog gilt das für $W$, also sind $1_U M$ und $1_W M$ nichttriviale $A$"~Untermoduln von $M$.

\medbreak
b. Es gibt stets eine starke Zusammenhangskomponente $U\subseteq V_M$, die nur eingehende, aber keine ausgehenden Kanten in $\mathcal{K}_M$ hat. Definiere dann $1_U$ wie zuvor.

Da alle von $U$ ausgehenden Kanten als Null auf $M$ operieren, folgt erneut, dass $1_U M$ ein nichttrivialer $A$"~Untermodul von $M$ ist. Da $M$ einfach ist, muss $1_U M=M$ sein, d.\,h. $U=V_M$, d.\,h. $\mathcal{K}_M$ muss stark zusammenhängend sein.
\end{proof}

\begin{remark}
Wir benötigen das folgende Lemma, um Kantengewichte von Kazhdan-Lusztig-$W$"~Graphen zu kontrollieren:
\end{remark}

\begin{lemma}[Kantengewichte für ADE-Gruppen, siehe {\citep[3.3(c)]{Stembridge2008admissble}}]\label{wgraphs:weights_simply_laced}
\index{terms}{W-Graph@$W$-Graph!mit natürlichen Gewichten}\index{terms}{Matrix!Permutations-}
Sei $(W,S)$ eine Coxeter"=Gruppe, deren Coxetermatrix $m_{st}=\ord(st)\in\Set{1,2,3}$ für alle $s,t\in S$ erfüllt (d.\,h. im endlichen Fall, dass $W$ nur Komponenten vom Typ $A$, $D$ bzw. $E$ hat).

Ist dann $(\mathfrak{C},I,m)$ ein $W$"~Graph mit $m_{xy}^s\in\IN$ für alle $x,y\in\mathfrak{C}$ und $s\in S$, dann gilt $m_{xy}^s=m_{yx}^s \in\Set{0,1}$ für alle $x,y\in\mathfrak{C}$, für die $I(x)$ und $I(y)$ im Kompatibilitätsgraphen durch eine transversale Kante verbunden sind.
\end{lemma}
\begin{proof}
Schritt 1: Wir zeigen, dass die einzigen Matrizen $A,B\in\mathbb{N}^{d\times d}$ mit $AB=1=BA$ zueinander inverse Permutationsmatrizen sind.

Aus $1 = \sum_{j=1}^d A_{ij} B_{ji}$ folgt nämlich, dass es für alle $i$ genau ein $j=f(i)$ mit $A_{ij}=B_{ji}=1$ gibt und $A_{ik}=0$ oder $B_{ki}=0$ für $k\neq j$. Aus Symmetriegründen gibt es auch für alle $j$ genau ein $i=g(j)$ mit $A_{ji}=B_{ij}=1$ und $A_{jk}=0$ oder $B_{kj}=0$ für $k\neq i$. Weil nun $A_{i,f(i)}=B_{f(i),i}=1$ ist, muss $g(f(i))=i$ sein und symmetrisch $f(g(j))=j$, d.\,h. $f$ und $g$ sind zueinander inverse Permutationen von $\Set{1,\ldots,d}$.

Fixiere nun $i$. Dann gilt $0 = \sum_{j=1}^d A_{ij} B_{jk}$ für alle $k\neq i$. Da für $j=f(i)$ ja $A_{ij}=1$ ist, muss $B_{jk}=0$ sein, d.\,h. es gilt $B_{f(i),k}=0$ für alle $k\neq i$. Mit anderen Worten gibt es in $B$ in jeder Zeile ($f$ ist surjektiv) genau einen von Null verschiedenen Wert und dieser ist $1$. Aus Symmetriegründen gilt das auch für $A$, d.\,h. $A$ und $B$ sind zueinander inverse Permutationsmatrizen, wie behauptet.

\medbreak
Schritt 2: Wir zeigen die Behauptung für $W=A_2$.

Wenn der $W$"~Graph natürliche Kantengewichte hat, dann hat er insbesondere konstante Kantengewichte und wir können ihn als Matrixdarstellung $\omega: \OmegaGy\to\End_\IZ(\IZ^{(\mathfrak{C})})$ auffassen. Die einzige transversale Kante im Kompatibilitätsgraphen von $A_2$ ist $\Set{1}\leftrightarrows\Set{2}$ und es gelten daher in $\OmegaGy$ die $(\alpha)$-Relationen
\[E_1 = X_{1,2} X_{2,1}\quad\text{und}\quad E_2 = X_{2,1} X_{1,2}.\]
Ordnen wir $\mathfrak{C}$ derart, dass die Ecken mit Label $\emptyset$, $\Set{1}$, $\Set{2}$ bzw. $\Set{1,2}$ in dieser Reihenfolge auftauchen, dann sind also
\[\omega(X_{1,2})=\begin{pmatrix}
0 &   &   &   \\
0 & 0 & A &   \\
0 & 0 & 0 &   \\
0 & 0 & 0 & 0 
\end{pmatrix}
\quad\text{und}\quad
\omega(X_{2,1})=\begin{pmatrix}
0 &   &   &   \\
0 & 0 & 0 &   \\
0 & B & 0 &   \\
0 & 0 & 0 & 0 
\end{pmatrix}\]
Matrizen über $\IN$ mit $AB=1$ und $BA=1$. Aus den Überlegungen in Schritt $1$ folgt also, dass $A$ und $B$ 0-1-Matrizen sind und $B=A^T$ gilt. Die Kantengewichte von Kanten $\Set{1} \leftarrow \Set{2}$ stehen in $A$, die von Kanten $\Set{2}\leftarrow\Set{1}$ in $B$, also gilt die Behauptung.

\medbreak
Schritt 3: Der allgemeine Fall.

Wenn eine transversale Kante $I(x) \leftrightarrows I(y)$ im Kompatibilitätsgraphen existiert, dann gibt es $s,t\in S$ mit ${s\in I(x)\setminus I(y)}$, $t\in I(y)\setminus I(x)$ und $m_{st}=3$. Dann ist die parabolische Untergruppe $W_{\Set{s,t}}$ vom Typ $A_2$. Wir können dann die parabolische Restriktion $(\mathfrak{C},I',m')$ des $W$"~Graphen zu einem $W_{\Set{s,t}}$-Graphen betrachten (siehe \ref{wgraphs:constructions}). Dies ist ein $A_2$-Graph mit nichtnegativen Kantengewichten und $I'(x)=\Set{s}$, $I'(y)=\Set{t}$. Aus dem vorherigen Schritt folgt $m_{xy}=m_{yx}\in\Set{0,1}$.
\end{proof}

\begin{remark}
Man beachte, dass Kazhdan-Lusztig-$W$"~Graphen von irreduziblen Coxeter"=Gruppen im Einparameterfall (welcher für ADE-Gruppen automatisch gegeben ist für alle irreduziblen Komponenten von $(W,S)$) natürliche Kantengewichte haben. Dies ist gerade die Aussage von Lusztigs Positivitätsvermutung (Vermutung \ref{KL:conj:positivity}).
\end{remark}

\begin{definition}
\index{symbols}{$\leq$}
Ist $A$ eine Algebra, so definieren wir die partielle Ordnung $\leq$ auf der Menge der Idempotente von $A$ durch $e\leq f :\iff e=ef=fe$.
\end{definition}

\begin{remark}
Wenn $e,f$ Idempotente mit $e\leq f$ sind, dann ist auch $e':=f-e$ ein Idempotent, $e'\leq f$ mit $e'e=ee'=0$ und $f=e+e'$. Jedes Idempotent $\leq f$ liefert also automatisch eine orthogonale Zerlegung von $f$.
\end{remark}

\begin{convention}
\index{terms}{Einparameterfall|(}
Wir nehmen ab jetzt für den Rest des Kapitels an, dass wir uns im Einparameterfall befinden, dass also $\Gamma=\IZ$ und $L(s)=1$ für alle $s\in S$ ist.
\end{convention}

\begin{remark}
Der folgende Beweis ist Vorbild für alle restlichen Beweise dieses Kapitels. Er zeigt die Grundidee, aus den in Satz \ref{wgraph_alg:relations} hergeleiteten Relationen von $\OmegaGy$ Strukturinformationen über $W$"~Graphen zu destillieren, um so Gyojas Vermutung zu beweisen.
\end{remark}

\begin{theorem}
\index{terms}{Vermutung!von Gyoja}\index{terms}{Vermutung!$W$-Graph-Zerlegungs-}
Sei $(W,S)$ vom Typ $A_3$ und $V$ ein einfacher $\OmegaGy(W,S)$"~Modul.
\begin{enumerate}
	\item Der von $V$ im Sinne von Definition \ref{wgraph:def:Q_V} induzierte Teilgraph des Kompatibilitätsgraphen $\mathcal{Q}_{A_3}$ (siehe Abbildung \ref{fig:wgraph_alg:comp_graphs}) ist einer der folgenden:
	\begin{enumerate}
		\item 
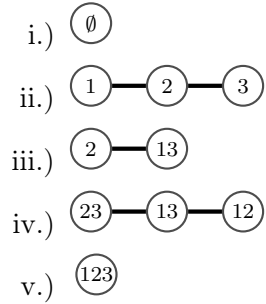
\begin{tikzpicture}
			\node[circle,inner sep=0,minimum size=15pt,draw=black!70,thick] (E_empty) at (1,0) {};
			\node[font=\scriptsize] at (E_empty) {$\emptyset$};
		\end{tikzpicture}
		\item \begin{tikzpicture}
			\node[circle,inner sep=0,minimum size=15pt,draw=black!70,thick] (E_1) at (0,0) {};
			\node[circle,inner sep=0,minimum size=15pt,draw=black!70,thick] (E_2) at (1,0) {};
			\node[circle,inner sep=0,minimum size=15pt,draw=black!70,thick] (E_3) at (2,0) {};
			\node[font=\scriptsize] at (E_1) {$1$};
			\node[font=\scriptsize] at (E_2) {$2$};
			\node[font=\scriptsize] at (E_3) {$3$};
			\path[ultra thick,black] (E_2) edge (E_1) edge (E_3);
		\end{tikzpicture}
		\item \begin{tikzpicture}
			\node[circle,inner sep=0,minimum size=15pt,draw=black!70,thick] (E_2) at (0,0) {};
			\node[circle,inner sep=0,minimum size=15pt,draw=black!70,thick] (E_13) at (1,0) {};
			\node[font=\scriptsize] at (E_2) {$2$};
			\node[font=\scriptsize] at (E_13) {$13$};
			\path[ultra thick,black] (E_2) edge (E_13);
		\end{tikzpicture}
		\item \begin{tikzpicture}
			\node[circle,inner sep=0,minimum size=15pt,draw=black!70,thick] (E_23) at (0,0) {};
			\node[circle,inner sep=0,minimum size=15pt,draw=black!70,thick] (E_13) at (1,0) {};
			\node[circle,inner sep=0,minimum size=15pt,draw=black!70,thick] (E_12) at (2,0) {};
			\node[font=\scriptsize] at (E_23) {$23$};
			\node[font=\scriptsize] at (E_13) {$13$};
			\node[font=\scriptsize] at (E_12) {$12$};
			\path[ultra thick,black] (E_13) edge (E_23) edge (E_12);
		\end{tikzpicture}
		\item \begin{tikzpicture}
			\node[circle,inner sep=0,minimum size=15pt,draw=black!70,thick] (E_123) at (0,0) {};
			\node[font=\scriptsize] at (E_123) {$123$};
		\end{tikzpicture}
	\end{enumerate}
	\item Gyojas Vermutung ist für Gruppen vom Typ $A_3$ und über allen Körpern $\IQ(v)\subseteq K$ richtig.
\end{enumerate}
\end{theorem}
\begin{proof}
Die Relationen aus \ref{wgraph_alg:relations} induzieren nichttriviale Einschränkungen an die Gestalt eines $W$"~Graphen. Für den Typ $A_3$ lautet z.B. die Relation des Typs $(\alpha)$ mit $I:=\Set{1}$ und $J:=\Set{2}$:
\[X_{1,2} X_{2,1} - E_1 = 0\]
Das heißt $E_1 V \xrightarrow{X_{2,1}} E_2 V \xrightarrow{X_{1,2}} E_1 V = \id_{E_1 V}$, d.\,h. es muss in jedem $A_3$"~Graphen mindestens so viele Ecken mit Label $\Set{2}$ geben wie es Ecken mit Label $\Set{1}$ gibt und $F_2':=X_{2,1} X_{1,2}$ ist ein Idempotent $\leq E_2$ derart, dass $\dim F_2' V$ die Anzahl der Ecken mit Label $\Set{1}$ angibt.

Indem wir den nichttrivialen Graphautomorphismus von $A_3$ anwenden, erhalten wir, dass dasselbe gelten muss, wenn wir $1$ durch $3$ ersetzen.

Zwei weitere Relationen vom Typ $(\alpha)$ lauten
\[X_{2,1} X_{1,2} + X_{2,13} X_{13,2} - E_2 = 0 \quad\text{und}\quad X_{2,3} X_{3,2} + X_{2,13} X_{13,2} - E_2 = 0,\]
woraus wir durch Vergleich der beiden Relationen $F_2'=X_{2,1} X_{1,2}=X_{2,3} X_{3,2}$ erhalten. Insbesondere sind also die Dimensionen von $E_1 V$ und $E_3 V$ gleich, d.\,h. in jedem $A_3$"~Graphen muss es gleich viele Ecken mit Label $\Set{1}$ und $\Set{3}$ geben.

Weiter muss $F_2'':=X_{2,13} X_{13,2}=E_2-F_2'$ ein weiteres Idempotent sein, das zu $F_2'$ orthogonal ist. Aus der Zerlegung $E_2 = F_2' + F_2''$ erkennen wir, dass die Ecken mit Label $\Set{2}$, ggf. nach einem Basiswechsel, in zwei disjunkte, eventuell leere Klassen zerfallen: Diejenigen, die eingehende Kanten von und zu Ecken haben, die mit $\Set{1}$ bzw. $\Set{3}$ gelabelt sind, und diejenigen, die eingehende Kanten von mit $\Set{1,3}$ gelabelten Ecken haben.

Aus der $(\alpha)$-Relation
\[X_{13,2} X_{2,1} + X_{13,23} \underbrace{X_{23,1}}_{=0} = 0\]
folgt zusätzlich $X_{13,2} F_2' =0$, d.\,h. dass auch keine Kanten von Ecken in $F_2' V$ zu Ecken in $E_{13} V$ gehen.

Indem wir den Antiautomorphismus $\delta$ anwenden, erhalten wir analoge Aussagen für die zweielementigen Mengen: $E_{13}$ zerfällt in zwei orthogonale Idempotente, nämlich ${F_{13}' := X_{13,12} X_{12,13} = X_{13,23} X_{23,13}}$ und $F_{13}'':=X_{13,2} X_{2,13}$. Die Ecken mit Label $\Set{1,3}$ zerfallen ggf. nach einem Basiswechsel ebenfalls in zwei disjunkte Klassen: Diejenigen, die Kanten von und nach mit $\Set{2,3}$ und $\Set{1,2}$ gelabelten Ecken haben, sowie diejenigen, die Kanten von und nach mit $\Set{2}$ gelabelten Ecken haben.

\medbreak
Indem wir die Zerlegung $E_2=F_2'+F_2''$ und $E_{13} = F_{13}'+F_{13}''$ benutzen, können wir die Pfadalgebra-Struktur von $\OmegaGy$, die im Kompatibilitätsgraphen kodiert ist, verbessern und genauer aufschlüsseln. Wir erhalten den verbesserten Kompatibilitätsgraphen $\mathcal{Q}'$ in Graphik \ref{fig:gyoja:better_compatibility_graph_A3} (wobei aus Gründen der Übersichtlichkeit nur die Inklusionskanten zwischen ein- und zweielementigen Teilmengen dargestellt sind).

\begin{figure}[htp]
	\index{terms}{Kompatibilitätsgraph!verbesserter}
	\centering
	\begin{tikzpicture}

\node[Vertex] (E_empty) at (0,0) {$\emptyset$};

\node[Vertex] (E_1) at (-1.5,1) {$1$};
\node[Vertex] (F_2a)at ( 0.0,1) {$2$};
\node[Vertex] (E_3) at ( 1.5,1) {$3$};

\node[Vertex] (F_2b) at (0,2) {$2$};
\node[Vertex] (F_13a) at (0,5) {$13$};

\node[Vertex] (E_23) at (-1.5,6) {$23$};
\node[Vertex] (F_13b)at ( 0.0,6) {$13$};
\node[Vertex] (E_12) at ( 1.5,6) {$12$};

\node[Vertex] (E_123) at (0,7) {$123$};

\path[EdgeI]
	(E_1) edge (E_12)
	(E_1) edge (F_13a)
	(E_1) edge (F_13b)
	(F_2a) edge (E_12)
	(F_2a) edge (E_23)
	(F_2b) edge (E_12)
	(F_2b) edge (E_23)
	(E_3) edge (F_13a)
	(E_3) edge (F_13b)
	(E_3) edge (E_23);

\path[EdgeT]
	(E_1) edge (F_2a)
	(F_2a) edge (E_3)
	(F_2b) edge (F_13a)
	(E_23) edge (F_13b)
	(F_13b) edge (E_12);
	
\node (lambda_4)   at (4,0) {$(4)$};
\node (lambda_31)  at (4,1) {$(3,1)$};
\node (lambda_22)  at (4,3.5) {$(2,2)$};
\node (lambda_211) at (4,6) {$(2,1^2)$};
\node (lambda_1111)at (4,7) {$(1^4)$};

\end{tikzpicture}
	\caption{Verbesserte Zerlegung von $\OmegaGy(A_3)$ als Pfadalgebra}
	\label{fig:gyoja:better_compatibility_graph_A3}
\end{figure}

Aus Lemma \ref{wgraph:def:Q_V} folgt nun Behauptung a., denn ein einfacher Modul induziert einen stark zusammenhängenden Teilgraphen dieses verbesserten Kompatibilitätsgraphen. Dafür kommen jetzt nur noch die fünf in der Behauptung genannten Teilgraphen in Frage. In der Tat kommen alle diese Möglichkeiten wirklich bei den einfachen $A_3$"~Graphen vor. Davon kann man sich überzeugen, indem man für die fünf irreduziblen Charaktere von $A_3$ mittels Lemma \ref{wgraphs:unique_I_inv} die Vielfachheiten der Indexmengen in den dazugehörigen $A_3$"~Graphen mit diesen Charakteren ausrechnet.

\bigbreak
Gyojas Vermutung für $A_3$-Graphen können wir mit dieser Erkenntnis beweisen. Wir wissen bereits, dass $\OmegaGy$ ein Quotient von $\IZ\mathcal{Q}'$ ist. Nun beweisen wir, dass die (nicht unitären) Unterringe von $\OmegaGy$, die von den stark zusammenhängenden Teilgraphen induziert werden, zu Matrixalgebren isomorph sind. In der einen Richtung sind die fünf $W$"~Graph"=Darstellungen der Kazhdan-Lusztig-Linkszellen zusammengenommen ein Algebrahomomorphismus $\OmegaGy \to \IZ \times \IZ^{3\times 3} \times \IZ^{2\times 2} \times \IZ^{3\times 3} \times \IZ$.

Für die andere Richtung benutzen wir, dass sich $\IZ^{d\times d}$ durch Erzeuger und Relationen schreiben lässt:
\[\IZ^{d\times d} \isomorphic \frac{\IZ\langle e_{ij} | 1\leq i,j\leq d\rangle}{(e_{ij} e_{kl} = \delta_{jk} e_{il}, 1={\textstyle\sum_i} e_{ii})}\]

Durch die Abbildungsvorschriften
\[e_{11} \mapsto E_{123}\]
\[\begin{pmatrix}
e_{11} & e_{12} & e_{13} \\
e_{21} & e_{22} & e_{23} \\
e_{31} & e_{32} & e_{33}
\end{pmatrix} \mapsto \begin{pmatrix}
E_{23} & X_{23,13} & X_{23,13} X_{13,12} \\
X_{13,23} & F_{13}' & X_{13,12} \\
X_{12,13} X_{13,23} & X_{12,13} & E_{12}
\end{pmatrix}\]
\[\begin{pmatrix}
e_{11} & e_{12} \\
e_{21} & e_{22} \\
\end{pmatrix} \mapsto \begin{pmatrix}
F_2'' & X_{2,13} \\
X_{13,2} & F_{13}''
\end{pmatrix}\]
\[\begin{pmatrix}
e_{11} & e_{12} & e_{13} \\
e_{21} & e_{22} & e_{23} \\
e_{31} & e_{32} & e_{33}
\end{pmatrix} \mapsto \begin{pmatrix}
E_1 & X_{1,2} & X_{1,2} X_{2,3} \\
X_{2,1} & F_2' & X_{2,3} \\
X_{3,2} X_{2,1} & X_{3,2} & E_3
\end{pmatrix}\]
\[e_{11} \mapsto E_\emptyset\]
wird ein Algebrahomomorphismus $\IZ \times \IZ^{3\times 3} \times \IZ^{2\times 2} \times \IZ^{3\times 3} \times \IZ \to \OmegaGy$ gegeben, der linksinvers zum obigen ist. Dabei haben wir Lemma \ref{wgraphs:weights_simply_laced} benutzt. Daher induzieren die in Abbildung \ref{fig:gyoja:better_compatibility_graph_A3} mit $\lambda \vdash 4$ gekennzeichneten starken Zusammenhangskomponenten Matrixalgebren $\IZ^{d_\lambda\times d_\lambda}$ in $\OmegaGy$.

Indem wir nun mit $K$ tensorieren, ergibt sich, dass die fünf gewählten $W$"~Graph-Dar\-stel\-lung\-en eine Surjektion $\alpha: K\OmegaGy\twoheadrightarrow K\times K^{3\times 3}\times K^{2\times 2}\times K^{3\times 3}\times K$ induzieren.

Außerdem erhalten wir, dass $\ker(\alpha)$ als Ideal von den Kanten in $\mathcal{Q}'$ erzeugt wird, die zwischen verschiedenen starken Zusammenhangskomponenten verlaufen. Die Komponenten sind in Abbildung \ref{fig:gyoja:better_compatibility_graph_A3} mit den Partitionen $\lambda \vdash 4$ gekennzeichnet worden. Da nur Kanten von $\lambda$ zu $\mu$ verlaufen, wenn $\lambda \unrhd \mu$ ist, ist $\ker(\alpha)^m$ von Elementen der Form $X_{I_0 I_1} a_1 X_{I_1 I_2} a_2 \ldots a_{m-1} X_{I_{m-1} I_m}$ erzeugt, wobei stets $I_j$ in Komponente $\lambda_j$ liegt und $\lambda_0 \triangleleft \lambda_1 \triangleleft \ldots \triangleleft \lambda_m$ gilt. Daher muss $\ker(\alpha)^5=0$ sein, weil es keine solche Ketten der Länge fünf gibt. Insbesondere ist daher $\ker(\alpha)$ im Jacobson"=Radikal von $K\OmegaGy$ enthalten. Das Radikal hat also maximal die Kodimension $1+3^2+2^3+3^2+1=\abs{W}$ und somit gilt Gyojas Vermutung.
\end{proof}

\subsection{Die Zerlegungsvermutung}
\index{terms}{Vermutung!$W$-Graph-Zerlegungs-|(}

\begin{remark}
Die Elemente
\begin{align*}
	F^{(1^4)}   &:= E_{123} \\
	F^{(2,1^2)} &:= E_{23}+F_{13}'+E_{12} \\
	F^{(2^2)}   &:= F_{13}''+F_2'' \\
	F^{(3,1)}   &:= E_1+F_2'+E_3 \\
	F^{(4)}     &:= E_\emptyset
\end{align*}
aus dem vorangegangenen Beweis sind paarweise orthogonale Idempotente von $\OmegaGy(A_3)$, die zu $1$ summieren und ${F^\lambda \OmegaGy F^\lambda \isomorphic \IZ^{d_\lambda\times d_\lambda}}$ erfüllen.

Mindestens für Coxeter"=Gruppen vom Typ $A$ könnte es möglich sein, mit Hilfe geschickter Anwendung der Relationen in Satz \ref{wgraph_alg:relations} eine Zerlegung $1=\sum_{\lambda\in\Lambda} F^\lambda$ in orthogonale Idempotente zu konstruieren, für die analog $F^\lambda \OmegaGy F^\mu\neq 0 \implies \lambda \unlhd \mu$ und $F^\lambda \OmegaGy F^\lambda \isomorphic \IZ^{d_\lambda\times d_\lambda}$ gelten sollte. In Kürze werden wir das etwa für $A_2$ und $A_4$ auch tun.

\medbreak
Eine völlig analoge Aussage für alle (endlichen) Coxeter"=Gruppen zu vermuten, ist wohl hingegen zu viel verlangt, da der Beweis nur deshalb funktioniert, weil für $A_3$ die Isomorphietypen $\lambda\in\Irr(W)$ bereits eindeutig durch die Mengen $\Set{I | E_I V_\lambda \neq 0}$ bestimmt sind. Bei den exzeptionellen Typen stimmt das nicht mehr. Die beiden vierdimensionalen, irreduziblen Darstellungen von $H_3$ haben beispielsweise die gleichen Eckenlabel (siehe \cite{geckpfeiffer}). Das lässt eine explizite Konstruktion der entsprechenden Idempotente $F^{4_r}$ und $F^{4_r'}$ aus den Pfadalgebra"=Erzeugern von $\OmegaGy(H_3)$ allein mit dem obigen Ansatz unmöglich erscheinen.

Dazu kommt, dass selbst im Typ $A$ Indexmengen auch mit höheren Vielfachheiten als Eins auftreten können. Damit das mit dem Isomorphismus $F^\lambda \OmegaGy F^\lambda \isomorphic \IZ^{d_\lambda\times d_\lambda}$ verträglich sein kann, sollte $(E_I F^\lambda) \OmegaGy (E_I F^\lambda)$ selbst schon eine Matrixalgebra $\IZ^{k\times k}$ sein, wobei $k$ die Vielfachheit von $I$ in $W$"~Graphen mit Isomorphietyp $\lambda$ ist. Wie das zustande kommen soll, ist auch a priori unklar. Da jedoch in keinem $W$"~Graphen Kanten zwischen Ecken mit demselben Eckenlabel existieren, bleibt die Hoffnung, verschiedene Ecken mit gleichen Labeln durch die Label benachbarter Ecken (und, sofern nötig, deren Nachbarn und deren Nachbarn ...) zu unterscheiden.
\end{remark}

\begin{conjecture}[$W$-Graph-Zerlegungsvermutung]
\index{symbols}{Flambda@$F^\lambda$}\index{symbols}{Z1-Z7}
Trotz der erwähnten Probleme kann man natürlich hoffen, dass $\OmegaGy(W,S)$ für eine möglichst große Klasse der endlichen Coxeter"=Gruppen $(W,S)$ die folgenden (oder zumindest ähnliche) Eigenschaften erfüllt.

Wenn $k\subseteq\IC$ ein guter Ring im Sinne von \ref{J_alg:def:L_good} ist, sollte es Elemente $F^\lambda\in k\OmegaGy$ für $\lambda\in\Lambda$ geben, welche zumindest die ersten vier der folgenden Eigenschaften haben. Die zusätzlichen Eigenschaften könnten sich als nützlich bei der Konstruktion dieser Elemente erweisen:
\begin{enumerate}[label=(Z\arabic*),leftmargin=35pt]
	\item Die $F^\lambda$ sind eine Zerlegung der $1$ in orthogonale Idempotente:
	\[\smash{1=\sum_\lambda F^\lambda} \quad\text{und}\quad F^\lambda F^\mu = \delta_{\lambda,\mu} F^\lambda\]
	\item Es gibt eine gemeinsame Verfeinerung dieser Idempotentzerlegung und der Zerlegung $1=\sum_I E_I$:
	\[F^\lambda E_I = E_I F^\lambda =: F_I^\lambda\]
	\item Zwischen den $F^\lambda$ existieren nur Kanten "`nach unten"': $F^\lambda x_s F^\mu \neq 0 \implies \lambda \preceq \mu$ für alle $s\in S$ (siehe \ref{J_alg:def:order_Lambda} für die Definition der Ordnung auf $\Irr(W)$).
	\item Es gibt surjektive Homomorphismen $\psi_\lambda: k^{d_\lambda\times d_\lambda}\twoheadrightarrow F^\lambda k\OmegaGy F^\lambda$ von $k$"~Algebren.
	\item Für alle $I,J$ existieren $\sigma_\lambda\in k^\times\cup\Set{0}$ mit $X_{IJ} X_{JI} = \sum_{\lambda\in\Lambda} \sigma_\lambda F_I^\lambda$.
	\item $\delta(F^\lambda) = F^{\lambda^\dagger}$.
	\item $\alpha(F^\lambda) = F^{\lambda^\alpha}$ für alle Graphautomorphismen $\alpha$, wobei $\lambda^\alpha$ der Isomorphietyp der mit $\alpha$ getwisteten Darstellung sei.
\end{enumerate}
\end{conjecture}

\begin{remark}
Z5 ist eine sehr starke Aussage, da es in gewisser Weise die Spektralzerlegung der Endomorphismen $X_{IJ} X_{JI}$ auf allen $W$"~Graph"=Moduln simultan beschreibt. Vor allem ist darin die Aussage enthalten, dass alle diese Endomorphismen miteinander kommutieren und auf $\OmegaGy$ selbst diagonalisierbar (durch Linksmultiplikation) operieren.
	
In den wenigen Beispielen, für die ich bisher einen Satz von solchen Idempotenten $F^\lambda$ explizit kenne, trifft dies zu, daher taucht diese Vermutung in der obigen Liste auf. Ich habe jedoch bei diesem Punkt der Liste die größten Zweifel, ob er wirklich allgemeingültig sein kann.

Sollte dies aber tatsächlich gelten, könnte es sein, dass die $F_I^\lambda$ auf diese Weise eindeutig bestimmt werden. Das hieße im Umkehrschluss, dass irreduzible $W$"~Graphen und ihre Isomorphietypen durch die Spektralzerlegung dieser "`Schleifenendomorphismen"' $X_{IJ} X_{JI}$ mehr oder weniger eindeutig erkannt werden können.
\end{remark}
\begin{remark}
Für spätere Überlegungen halten wir fest, dass die Algebren $F^\lambda \OmegaGy F^\lambda$ von den Elementen $F_I^\lambda$ und $F^\lambda X_{IJ}F^\lambda$ erzeugt werden, wenn Z3 gilt. Das folgt daraus, dass nur Kanten $F^\lambda \leftarrow F^\mu$ existierten, wenn $\lambda\preceq\mu$ ist. Wenn wir also $F^\lambda X_{I_0 I_1} \ldots X_{I_{k-1} I_k} F^\lambda$ als 
\[\sum_{\lambda_i} F^\lambda X_{I_0 I_1} F^{\lambda_1} X_{I_1 I_2} \ldots X_{I_{k-2} I_{k-1}} F^{\lambda_{k-1}} X_{I_{k-1} I_k} F^\lambda\]
schreiben, gilt in allen von Null verschiedenen Summanden $\lambda \preceq \lambda_1 \preceq \ldots \preceq \lambda_{k-1} \preceq \lambda$. Daraus folgt, dass $F^\lambda \OmegaGy F^\lambda$ wirklich von den $F^\lambda X_{IJ} F^\lambda$ erzeugt wird. Dies wird auch zum Nachweis von Z4 nützlich sein, wenn wir die Surjektivität prüfen wollen.
\end{remark}

\begin{remark}
Wir halten zunächst einige Folgerungen aus den Vermutungen Z1 bis Z7 fest, bevor wir zu den Beweisen der Zerlegungsvermutung in den Fällen $I_2(m)$, $B_3$ und $A_4$ übergehen.
\end{remark}

\begin{lemma}
\index{terms}{Vermutung!von Gyoja}\index{terms}{Pfadalgebra!verallgemeinerte}
Es sei $(W,S)$ eine Coxeter"=Gruppe, $k\subseteq\IQ_W$ sei ein guter Ring und es gelten Z1, Z2 und Z3. Dann gilt auch:
\begin{enumerate}
	\item $k\OmegaGy$ ist ein Quotient einer verallgemeinerten Pfadalgebra auf dem azyklischen Köcher $\mathcal{L}$, dessen Knotenmenge $\Lambda=\Irr(W)$ ist und der $\abs{S}$ Kanten $\lambda\leftarrow\mu$ besitzt, falls $\lambda\prec\mu$ ist. Die Algebren an den Knoten sind $F^\lambda k\OmegaGy F^\lambda$.
	\item Jeder $W$"~Graph $(\mathfrak{C},I,m)$ mit konstanten Koeffizienten in $k$, der einen einfachen $k\OmegaGy$"~Modul $V$ definiert, erfüllt $\mathcal{Q}_V\subseteq\Set{I\subseteq S | F_I^\lambda\neq 0}$ für ein $\lambda\in\Lambda$. (siehe \ref{wgraph:def:Q_V} für die Definition von $\mathcal{Q}_V$)
\end{enumerate}
Gilt zusätzlich Z4, so gelten weiterhin:
\begin{enumerate}[resume]
	\item Die in Z4 auftauchenden Surjektionen $\psi_\lambda$ sind Isomorphismen.
	\item $\rad(K\OmegaGy)$ ist nilpotent und Gyojas Vermutung ist für $K\OmegaGy(W,S)$ wahr, wobei $K$ eine beliebige Körpererweiterung von $\IQ_W(v)$ ist.
	\item $k\OmegaGy$ als $k$"~Modul endlich erzeugt.
\end{enumerate}
\end{lemma}
\begin{proof}
a. Es sei $\widetilde{\Omega}$ die verallgemeinerte Pfadalgebra auf dem Köcher $\mathcal{L}$ mit den angegebenen Knotenalgebren. Wegen Z1 sind die $F^\lambda$ eine orthogonale Zerlegung der 1 in $\OmegaGy$. Wegen Z2 gilt $E_I=\sum_\lambda F^\lambda E_I F^\lambda\in\sum_\lambda F^\lambda \OmegaGy F^\lambda$. Da die $E_I$ zusammen mit den $x_s$ ganz $\OmegaGy$ erzeugen, folgt aus den Überlegungen in \ref{path_algs:ex:generic_examples} die Behauptung.

\medbreak
b. ergibt sich aus a., denn jeder $\OmegaGy$"~Modul $V$ ist nun auch ein $\widetilde{\Omega}$"~Modul. Weil $\widetilde{\Omega}$ eine verallgemeinerte Pfadalgebra ist, muss $V$ einen stark zusammenhängenden Teilgraphen $\mathcal{L}_V\subseteq\mathcal{L}$ induzieren (siehe \ref{wgraph:def:Q_V}). $\mathcal{L}$ ist aber azyklisch, d.\,h. die einzigen stark zusammenhängenden Teilgraphen sind einelementig. Mit anderen Worten: Es gibt ein $\lambda\in\Lambda$ derart, dass $F^\lambda V\neq 0$ und $F^\mu V=0$ für alle $\mu\neq\lambda$ ist. Für alle $I\subseteq S$ mit $E_I V\neq 0$ muss also $E_I F^\lambda \neq 0$ sein.

\bigbreak
%
%
c. und d. beweisen wir gemeinsam. Wenn wir Surjektionen $\psi_\lambda: k^{d_\lambda\times d_\lambda} \twoheadrightarrow F^\lambda k\OmegaGy F^\lambda$ haben, dann erhalten wir durch Tensorieren mit $K$ Surjektionen $K^{d_\lambda\times d_\lambda} \twoheadrightarrow F^\lambda K\OmegaGy F^\lambda$. Da Matrizenringe über Körpern einfach sind, ist somit $F^\lambda K\OmegaGy F^\lambda$ entweder gleich $0$ oder zu $K^{d_\lambda\times d_\lambda}$ isomorph für alle $\lambda$.

Da wir wissen, dass $K\OmegaGy$ ein Quotient von $K\widetilde{\Omega}$ ist, können wir die Kodimension des Radikals abschätzen:
\[\dim_K K\OmegaGy / \rad(K\OmegaGy) \leq \dim_K K\widetilde{\Omega} / \rad(K\widetilde{\Omega}) \]
$K\widetilde{\Omega}$ ist nun eine verallgemeinerte Pfadalgebra auf einem azyklischen Köcher, d.\,h. das von den Kanten erzeugte Ideal ist nilpotent und somit im Radikal enthalten. Es ist in der Tat gleich dem Radikal, weil der Quotient isomorph zu $\prod_{\lambda} F^\lambda K\OmegaGy F^\lambda$ und daher halbeinfach ist. Wir erhalten:
\begin{align*}
	\dim_K K\OmegaGy / \rad(K\OmegaGy) &\leq \sum_{\substack{\lambda\in\Lambda \\ F^\lambda\neq 0}} d_\lambda^2 \\
	&\leq \sum_{\lambda} d_\lambda^2 \\
	&= \abs{W}
\end{align*}
Nun wissen wir aber, dass das Radikal von $K\OmegaGy$ mindestens die Kodimension $\abs{W}$ haben muss. Es gelten also überall die Gleichheiten.

Das zeigt c., denn nun muss für alle $\lambda$ gelten, dass $K^{d_\lambda\times d_\lambda} \xrightarrow{\psi_\lambda} F^\lambda K\OmegaGy F^\lambda$ ein Isomorphismus ist. Insbesondere muss $\psi_\lambda: k^{d_\lambda\times d_\lambda} \to F^\lambda k\OmegaGy F^\lambda$ schon injektiv gewesen sein, ist also auch ein Isomorphismus.

Wir erhalten aber auch d., denn wir haben gezeigt, dass das Radikal von $K\widetilde{\Omega}$ nilpotent ist und durch die Quotientenabbildung $K\widetilde{\Omega}\to K\OmegaGy$ genau auf das Radikal von $K\OmegaGy$ abgebildet wird. Also ist auch $\rad(K\OmegaGy)$ nilpotent. Wir haben außerdem gezeigt, dass die Kodimension des Radikals gleich $\abs{W}$ ist, also Gyojas Vermutung für $K\OmegaGy$ wahr ist.

\medbreak
e. Nun ist $F^\lambda k\OmegaGy F^\lambda\isomorphic k^{d_\lambda\times d_\lambda}$, also als $k$"~Modul endlich erzeugt. Da es nur endlich viele Pfade im Köcher $\mathcal{L}$ gibt, ist auch $k\widetilde{\Omega}$ als $k$"~Modul endlich erzeugt. Als Quotient ist daher $k\OmegaGy$ ebenfalls als $k$"~Modul endlich erzeugt.
\end{proof}

\begin{remark}
Wir werden verschiedene Techniken benutzen, um die Idempotente $F_I^\lambda$ zu konstruieren. Zum einen werden wir einige Idempotente ad hoc konstruieren. Das folgende Lemma hilft dann dabei, einmal konstruierte Idempotente weiter zu propagieren.
\end{remark}

\begin{lemma}[Idempotenttransport]\label{gyoja:idempotenttransport}
\index{terms}{Idempotent!-transport}\index{terms}{Idempotent!Rest-}
Es seien $I,J\subseteq S$ beliebig, aber fest. Es seien $A$ eine endliche Indexmenge und $(e_\alpha)_{\alpha\in A}$ paarweise orthogonale Idempotente $\leq E_I$ mit $X_{IJ} X_{JI} = \sum_{\alpha\in A} \sigma_\alpha e_\alpha$ für $\sigma_\alpha\in k^\times$. Bezeichne das \udot{Restidempotent} $E_I-\sum_\alpha e_\alpha$ weiterhin mit $e_0$. Dann gilt:
\begin{enumerate}
	\item $\widetilde{e}_\alpha := \sigma_\alpha^{-1} X_{JI} e_\alpha X_{IJ}$ und das Restidempotent $\widetilde{e}_0:=E_J-\sum_{\alpha\in A} \widetilde{e}_\alpha$ sind paarweise orthogonale Idempotente $\leq E_J$.
	\item Es gilt $X_{IJ} \widetilde{e}_\alpha = e_\alpha X_{IJ}$ und $X_{JI} e_\alpha = \widetilde{e}_\alpha X_{JI}$ für $\alpha\in A\cup\Set{0}$.
	\item $r:=X_{JI} e_0 X_{IJ}$ erfüllt $r^2=0$, $r = \widetilde{e}_0 r \widetilde{e}_0$ sowie $X_{JI} X_{IJ} = \sum_{\alpha\in A} \sigma_\alpha\widetilde{e}_\alpha+r$. Insbesondere ist $r=0$, falls $X_{JI} X_{IJ}$ selbst ein Idempotent ist.
	\item $X_{IJ} \widetilde{e}_\alpha X_{JI} = \sigma_\alpha e_\alpha$ für alle $\alpha\in A$, d.\,h. erneuter Transport der Idempotente liefert die ursprünglichen Idempotente zurück.
\end{enumerate}
\end{lemma}
\begin{proof}
Für $\alpha,\beta\in A$ gilt:
\begin{align*}
	\widetilde{e}_\alpha \widetilde{e}_\beta &= (\sigma_\alpha^{-1} X_{JI} e_\alpha X_{IJ})(\sigma_\beta^{-1} X_{JI} e_\beta X_{IJ}) \\
	&= \sigma_\alpha^{-1} \sigma_\beta^{-1} X_{JI} e_\alpha \Big(\smash{\sum_{\gamma\in A} \sigma_\gamma e_\gamma}\Big) e_\beta X_{IJ} \\
	&= \sum_\gamma \tfrac{\sigma_\gamma}{\sigma_\alpha \sigma_\beta} X_{JI} e_\alpha e_\gamma e_\beta X_{IJ} \\
	&= \delta_{\alpha\beta} X_{JI} \sigma_\alpha^{-1} e_\alpha X_{IJ} \\
	&= \delta_{\alpha\beta} \widetilde{e}_\alpha
\end{align*}
Das zeigt a.

\medbreak
b. ergibt sich aus der Rechnung
\begin{align*}
	X_{IJ} \widetilde{e}_\alpha &= \sigma_\alpha^{-1} X_{IJ} X_{JI} e_\alpha X_{IJ} \\
	&= \sum_{\beta\in A} \sigma_\alpha^{-1} \sigma_\beta \smash{\underbrace{e_\beta e_\alpha}_{\delta_{\alpha\beta}}} X_{IJ} \\
	&= e_\alpha X_{IJ}
\end{align*}
und analog ergibt sich auch $\widetilde{e}_\alpha X_{JI} = X_{JI} e_\alpha$.

\medbreak
c. folgt aus einer ähnlichen Rechnung wie in a., denn für alle $\alpha\in A\cup\Set{0}$ folgt
\begin{align*}
	r \cdot X_{JI} e_\alpha X_{IJ} &= X_{JI} e_0 X_{IJ} X_{JI} e_\alpha X_{IJ} \\
	&= \sum_{\beta\in A} X_{JI} \sigma_\beta \smash{\underbrace{e_0 e_\beta}_{=0}} e_\alpha X_{IJ} \\
	&= 0
\end{align*}
sowie analog auch die umgekehrte Gleichung $X_{JI} e_\alpha X_{IJ} r = 0$. Speziell für $\alpha=0$ ergibt sich $r^2=0$ und für $\alpha\in A$ ergibt sich $\widetilde{e}_\alpha r = r \widetilde{e}_\alpha = 0$.

Aus der Definition von $\widetilde{e}_\alpha$ ergibt sich außerdem unmittelbar
\begin{align*}
	X_{JI} X_{IJ} &= X_{JI} E_I X_{IJ} \\
	&= \sum_{\alpha\in A\cup\Set{0}} X_{JI} e_\alpha X_{IJ} \\
	&= \sum_{\alpha\in A} \sigma_\alpha \widetilde{e}_\alpha + r.
\end{align*}

\medbreak
d. folgt ebenso direkt aus der Definition der $\widetilde{e}_\alpha$ und der Orthogonalität der $e_\alpha$.
\end{proof}

\begin{remark}
Wir werden in den folgenden Beweisen das Ziel verfolgen, die Kompatibilitätsgraphen verschiedener Coxeter"=Gruppen zu verfeinern und die Idempotente $E_I$ weiter aufzuspalten, wie wir es für $A_3$ bereits getan haben. In dieser Sichtweise sagt die Teilaussage b. des eben bewiesenen Lemmas etwa aus, dass bei solch einer Aufspaltung von $E_I$ in die Idempotente $e_\alpha$, $\alpha\in A\cup\Set{0}$ und beim Transport dieser Aufspaltung nach $E_J$ aus der transversalen Kante $E_I \leftrightarrows E_J$ keine wilden Kanten $e_\alpha \leftrightarrows \widetilde{e}_\beta$ mit $\alpha\neq\beta$ entstehen können, sondern die Kante stattdessen in $\leq\abs{A}+1$ parallele Kantenpaare $e_\alpha\leftrightarrows \widetilde{e}_\alpha$ ($\alpha\in A\cup\Set{0}$) gespalten wird.

\medbreak
d. zeigt dann insbesondere, dass die $\abs{A}$ Kanten $e_\alpha \leftrightarrows \widetilde{e}_\alpha$ ($\alpha\in A$) wirklich existieren, wenn mit $e_\alpha\neq 0$ begonnen wurde (was natürlich o.\,B.\,d.\,A. angenommen werden darf).

\medbreak
Ob von den zwei möglichen Kanten $e_0 \leftrightarrows \widetilde{e}_0$ keine, eine oder zwei existieren, ist a priori nicht klar. Dass keine oder eine Kante existiert, kommt in den folgenden Beispielen vor. Ein Beispiel, in dem beide Kanten existieren, habe ich bisher nicht.
\end{remark}

\subsection{Die Zerlegungsvermutung für \texorpdfstring{$I_2(m)$}{Coxeter-Gruppen vom Rang Zwei}}

\begin{lemma}
Sei $m\in\IN_{\geq 1}$ sowie $R:=\IZ[2\cos(\frac{2\pi}{m}),\frac{1}{m}]$. Dann gilt:
\begin{enumerate}
	\item $2\cos(k\frac{2\pi}{m}), 4\cos(k\frac{\pi}{m})^2\in R$ für alle $k\in\IZ$.
	\item $4\cos(k\frac{\pi}{m})^2\in R^\times$ für alle $k\in\IZ\setminus \tfrac{m}{2}\IZ$.
	\item $4\cos(k\frac{\pi}{m})^2-4\cos(l\frac{\pi}{m})^2\in R^\times$ für alle $1\leq k<l\leq\floor{\frac{m}{2}}$.
\end{enumerate}
\end{lemma}
\begin{proof}
Wir setzen $\zeta_n := \exp(\frac{2\pi i}{n})$ für alle $n\in\IN_{>0}$. Es gilt dann $2\cos(k\tfrac{2\pi}{n}) = \zeta_n^k+\zeta_n^{-k}$. Da nun jedoch $X^k+X^{-k}\in\IZ[X+X^{-1}]$ ist, folgern wir $2\cos(k\tfrac{2\pi}{n})\in R$.

Da $2\cos(\tfrac{\theta}{2})^2 = \cos(\theta)+1$ gilt, folgern wir auch $4\cos(k\frac{\pi}{m})^2\in R$. Damit haben wir a. gezeigt.

\medbreak
Wir benutzen für b. und c., dass
\[\IZ[2\cos(\tfrac{2\pi}{n})] = \IZ[\zeta_n+\zeta_n^{-1}] \subseteq \IZ[\zeta_n] \subseteq \IZ[\zeta_{nl}]\]
ganze Ringerweiterungen sind für alle $l\in\IN_{\geq 1}$. Für ganze Erweiterungen $R\subseteq S$ gilt $R\cap S^\times = R^\times$, daher genügt es zu zeigen, dass unsere Elemente Einheiten in $\IZ[\zeta_{ml}, \frac{1}{m}]$ sind für irgendein $l\in\IN_{\geq 1}$.

\medbreak
Schritt 1: $2\sin(a\frac{\pi}{m})$ ist invertierbar für $a\in\IZ\setminus m\IZ$.

Dies folgt aus
\begin{align*}
	\prod_{a=1}^{m-1} 2\sin(a\tfrac{\pi}{m}) &= \prod_{a=1}^{m-1} \tfrac{1}{i}(\zeta_{2m}^a-\zeta_{2m}^{-a}) \\
	&= (-i)^{m-1} \prod_{a=1}^{m-1} \zeta_{2m}^a \cdot\prod_{a=1}^{m-1} (1-\zeta_{2m}^{-2a}) \\
	&= (-i)^{m-1}\zeta_{2m}^{\tfrac{1}{2}m(m-1)} \prod_{a=1}^{m-1} (1-\zeta_m^a) \\
	&= \left.\frac{X^m - 1}{X-1}\right|_{X=1} \\
	&= m.
\end{align*}
Wir erhalten, dass alle $2\sin(k\frac{\pi}{m})$ Einheiten sind für $k\in\IZ\setminus m\IZ$. Nun ist c. eine einfache Folgerung, da $4\cos(k\tfrac{\pi}{m})^2 - 4\cos(l\tfrac{\pi}{m})^2 = 2\sin((k+l)\tfrac{\pi}{m})\cdot 2\sin((k-l)\tfrac{\pi}{m})$ gilt.

\medbreak
Schritt 2: $4\cos(a\tfrac{\pi}{m})^2$ ist invertierbar für $a\in\IZ \setminus \tfrac{m}{2} \IZ$.

Die folgt aus
\begin{align*}
	\prod_{1\leq a<\tfrac{m}{2}} (2\cos(a\tfrac{\pi}{m}))^2 &= \prod_{1\leq a<\tfrac{m}{2}} (\zeta_{2m}^a + \zeta_{2m}^{-a})^2 \\
	&= \prod_{1\leq a <\tfrac{m}{2}} (\zeta_{2m}^{2a} + 1)\zeta_{2m}^{-a}\cdot\zeta_{2m}^a (1 + \zeta_{2m}^{-2a}) \\
	&= \prod_{\substack{-\tfrac{m}{2} < a < \tfrac{m}{2} \\ a\neq 0}} (1+\zeta_m^a) \\
	&= \begin{cases}
	(-1)^{m-1} \prod_{\zeta^m=1, \zeta\neq 1} (-1-\zeta) & \text{falls}\,2\nmid m \\
	(-1)^{m-2} \prod_{\zeta^m=1, \zeta\neq\pm1} (-1-\zeta) & \text{falls}\,2\mid m
	\end{cases} \\
	&= \begin{cases}
	\left.\frac{X^m-1}{X-1}\right|_{X=-1} & \text{falls}\,2\nmid m \\
	\left.\frac{X^m-1}{X^2-1}\right|_{X=-1} & \text{falls}\,2\mid m
	\end{cases} \\
	&= \begin{cases}
	1 & \text{falls}\,2\nmid m \\
	\tfrac{m}{2} & \text{falls}\,2\mid m
	\end{cases}.
\end{align*}
Daher ist auch $4\cos(k\tfrac{\pi}{m})^2$ eine Einheit.
\end{proof}

\begin{theorem}
Es sei $m\in\IN_{\geq 3}$. Die Zerlegungsvermutung gilt für Coxeter"=Gruppen vom Typ $I_2(m)$.
\end{theorem}
\begin{proof}
Die Idee des Beweises ist, genau die in Z5 behauptete Eigenraumzerlegung zu benutzen und so eine Zerlegung des Kompatibilitätsgraphen wie in Abbildung \ref{fig:gyoja:better_compatibility_graph_I2} zu konstruieren.

\begin{figure}[ht]
	\centering
	\begin{tabular}{c|c}
		\begin{tikzpicture}[
		crossing line/.style = {preaction={draw=white,-,line width=6pt}}
	]

\node[Vertex] (E_empty) at (0,0) {$\emptyset$};

\node[Vertex] (F_1a) at (-1,1) {$1$};
\node[Vertex] (F_2a) at ( 1,1) {$2$};

\node[Vertex] (F_1b) at (-1,2) {$1$};
\node[Vertex] (F_2b) at ( 1,2) {$2$};

\node[Vertex] (F_1c) at (-1,4) {$1$};
\node[Vertex] (F_2c) at ( 1,4) {$2$};

\node[Vertex] (E_12) at (0,5) {$12$};

\node at (-1,3){$\vdots$};
\node at ( 1,3){$\vdots$};

\path[EdgeI]
	(E_empty)
		edge (F_1a)
		edge (F_2a)
		edge (F_1b)
		edge (F_2b)
		edge (F_1c)
		edge (F_2c)
		edge (E_12)
	(F_1a) edge[bend right=10] (E_12)
	(F_2a) edge[bend left=10] (E_12)
	(F_1b) edge (E_12)
	(F_2b) edge (E_12)
	(F_1c) edge (E_12)
	(F_2c) edge (E_12);
	
\path[EdgeT]
	(F_1a) edge[crossing line] (F_2a)
	(F_1b) edge[crossing line] (F_2b)
	(F_1c) edge[crossing line] (F_2c);

\node at (3,5) {$\sgn$};
\node at (3,4) {$\lambda_1$};
\node at (3,2) {$\lambda_{\tfrac{m-3}{2}}$};
\node at (3,1) {$\lambda_{\tfrac{m-1}{2}}$};
\node at (3,0) {$1$};

\end{tikzpicture} & \begin{tikzpicture}[
		crossing line/.style = {preaction={draw=white,-,line width=6pt}}
	]

\node[Vertex] (E_empty) at (0,0) {$\emptyset$};

\node[Vertex] (F_1a) at (-1,1) {$1$};
\node[Vertex] (F_2a) at ( 1,1) {$2$};

\node[Vertex] (F_1b) at (-1,2) {$1$};
\node[Vertex] (F_2b) at ( 1,2) {$2$};

\node[Vertex] (F_1c) at (-1,4) {$1$};
\node[Vertex] (F_2c) at ( 1,4) {$2$};

\node[Vertex] (E_12) at (0,5) {$12$};

\node at (-1,3){$\vdots$};
\node at ( 1,3){$\vdots$};

\path[EdgeI]
	(E_empty)
		edge (F_1a)
		edge (F_2a)
		edge (F_1b)
		edge (F_2b)
		edge (F_1c)
		edge (F_2c)
		edge (E_12)
	(F_1a) edge[bend right=10] (E_12)
	(F_2a) edge[bend left=10] (E_12)
	(F_1b) edge (E_12)
	(F_2b) edge (E_12)
	(F_1c) edge (E_12)
	(F_2c) edge (E_12);
	
\path[EdgeT]
	(F_1b) edge[crossing line] (F_2b)
	(F_1c) edge[crossing line] (F_2c);

\node at (3,5) {$\sgn$};
\node at (3,4) {$\lambda_1$};
\node at (3,2) {$\lambda_{\tfrac{m-2}{2}}$};
\node at (2.5,1) {$\epsilon_1$};
\node at (3.5,1) {$\epsilon_2$};
\node at (3,0) {$1$};

\end{tikzpicture}
	\end{tabular}
	\setcapwidth[c]{0.80\textwidth}
	\caption{Verbesserte Zerlegung von $\OmegaGy(I_2(m))$ als Pfadalgebra; links für ungerade $m$, rechts für gerade $m$.}
	\label{fig:gyoja:better_compatibility_graph_I2}
\end{figure}
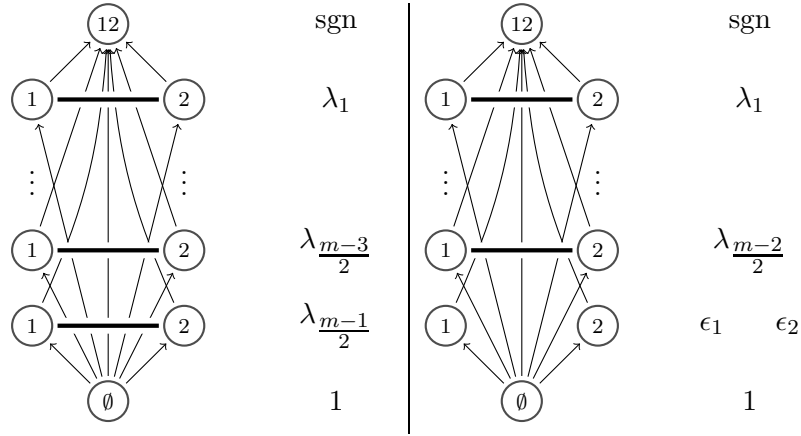

Zunächst halten wir fest, dass gute Ringe für $I_2(m)$ stets $\IZ_W=\IZ[2\cos(\tfrac{2\pi}{m}),\tfrac{1}{m}]$ enthalten. Wir können also das obige Lemma über Elemente dieses Rings anwenden.

Eine weitere wichtige Beobachtung ist, dass es nur sehr wenig transversale Kanten im Kompatibilitätsgraphen gibt, wenn der Rang Zwei ist (siehe Abbildung \ref{fig:wgraph_alg:comp_graphs}), nämlich nur $X_{1,2}$ und $X_{2,1}$.

Die einzigen $(\alpha)$-Relationen sind also
\[0 = \sum_{j=0}^{m-1} a_j \underbrace{X_{1,2} X_{2,1} \ldots}_{j}\quad\text{und}\quad 0 = \sum_{j=0}^{m-1} a_j \underbrace{X_{2,1} X_{1,2} \ldots}_{j},\]
wobei $a_j$ die Koeffizienten von $\tau_{m-1}$ seien, also $\tau_{m-1}(X) = \sum_j a_j X^j$.

\medbreak
Schritt 1: Vorbereitungen.

Wir definieren $\widetilde{\tau}_n\in\IZ[X]$ durch
\[\widetilde{\tau}_n := \begin{cases} \tau_n(\sqrt{X}) & \text{falls } 2\mid n \\ \tau_n(\sqrt{X})\sqrt{X} &\text{falls } 2\nmid n\end{cases}.\]
Man beachte, dass $\tau_n$ ein gerades Polynom ist, wenn $n$ gerade ist, und ein ungerades, wenn $n$ ungerade ist. Daher ist $\widetilde{\tau}_n$ tatsächlich wieder ein Polynom in $X$. Es hat den Grad $\ceil{\frac{n}{2}}$ und ist normiert. Da die $n$ Nullstellen von $\tau_n$ durch $2\cos(\frac{k}{n+1}\pi)$ ($k=1,\ldots,n$) gegeben sind, sind die Nullstellen von $\widetilde{\tau}_n$ durch $4\cos(\tfrac{k}{n+1}\pi)^2$ ($k=1,\ldots,\ceil{\frac{n}{2}})$ gegeben (siehe \citep[22.16]{abramowitz1964handbook}). Insbesondere sind die Nullstellen von $\widetilde{\tau}_{m-1}$ genau $\sigma_k:=4\cos(k\tfrac{\pi}{m})^2$ für $k=1,\ldots,\floor{\frac{m}{2}}$.

\medbreak
Schritt 2: Konstruktion der Idempotente.

Wenn $m$ ungerade ist, dann sind die $\alpha$-Relationen bereits in der Form $\widetilde{\tau}_{m-1}(X_{1,2}X_{2,1})=0$ bzw. $\widetilde{\tau}_{m-1}(X_{2,1}X_{1,2})=0$. Wenn $m$ gerade ist, dann können wir die Relation mit $X_{1,2}$ bzw. $X_{2,1}$ multiplizieren und erhalten ebenfalls diese Gleichungen.

Indem wir jetzt für $k=1,\ldots,\floor{\frac{m}{2}}$
\begin{equation}
	F_{1,k} := \prod_{\substack{l=1,\ldots,\floor{\frac{m}{2}} \\ l\neq k}} \frac{X_{1,2} X_{2,1} - \sigma_l}{\sigma_k - \sigma_l} \qquad\text{und}\qquad
	F_{2,k} := \prod_{\substack{l=1,\ldots,\floor{\frac{m}{2}} \\ l\neq k}} \frac{X_{2,1} X_{1,2} - \sigma_l}{\sigma_k - \sigma_l}
	\label{eq:gyoja_I2:1}\tag{1}
\end{equation}
setzen, haben wir Idempotente $F_{1,k}, F_{2,k}\in R\OmegaGy$ mit
\[E_1 = \sum_{k=1}^{\floor{\frac{m}{2}}} F_{1,k} \qquad\text{und}\qquad X_{1,2} X_{2,1} = \sum_{k=1}^{\floor{\frac{m}{2}}} \sigma_k F_{1,k}\]
bzw.
\[E_2 = \sum_{k=1}^{\floor{\frac{m}{2}}} F_{2,k} \qquad\text{und}\qquad X_{2,1} X_{1,2} = \sum_{k=1}^{\floor{\frac{m}{2}}} \sigma_k F_{2,k}\]
gefunden. Wir bezeichnen die zweidimensionalen Darstellungen von $I_2(m)$ mit $\lambda_k$ für $k=1,\ldots,\tfrac{m-1}{2}$, falls $m$ ungerade ist, bzw. $k=1,\ldots,\frac{m-2}{2}$, falls $m$ gerade ist, und mit $\epsilon_1$ und $\epsilon_2$ die beiden zusätzlichen eindimensionalen Darstellungen, wenn $m$ gerade ist, und definieren die Idempotente $(F^\lambda)_{\lambda\in\Irr(W)}$ als
\begin{align*}
	F^{1} &= E_\emptyset, \\
	F^{\lambda_k} &= F_{1,k}+F_{2,k}, \\
	F^{\sgn} &= E_{\Set{1,2}} \\
\intertext{sowie, falls $m$ gerade ist, außerdem noch}
	F^{\epsilon_1} &= F_{1,\tfrac{m}{2}}\quad\text{und} \\
	F^{\epsilon_2} &= F_{2,\tfrac{m}{2}}.
\end{align*}
Nach Konstruktion gelten dann Z1, Z2, Z5, Z6 und Z7. Es bleiben Z3 und Z4 zu prüfen.

\medbreak
Schritt 3: Wir zeigen Z3.
Das folgt aus Lemma \ref{gyoja:idempotenttransport}, denn die $F_{1,k}$ ergeben sich durch Idempotenttransport aus den $F_{2,k}$ und umgekehrt: Man beachte, dass $\sigma_k$ für $1\leq k<\frac{m}{2}$ invertierbar ist. Für $k=\frac{m}{2}$ ist $\sigma_k=0$, d.\,h. $F_{1,m/2}$ und $F_{2,m/2}$ sind die Restidempotente. Daher ist das Lemma für den Idempotenttransport anwendbar und es gilt
\[\sum_{k=0}^{\floor{\tfrac{m}{2}}} \sigma_k X_{1,2} F_{2,k} X_{2,1} = X_{1,2} \Big(\underbrace{\sum_k \sigma_k F_{2,k}}_{=E_2}\Big) X_{2,1} = X_{1,2} X_{2,1} = \sum_{k=0}^{\floor{\tfrac{m}{2}}} \sigma_k F_{1,k}.\]
Da $X_{1,2} F_{2,k} X_{2,1}$ ein Idempotent ist für $1\leq k<\frac{m}{2}$, beschreiben also beide Seiten der Gleichung $\sum_k \sigma_k X_{1,2} F_{2,k} X_{2,1} = \sum_k \sigma_k F_{1,k}$ die Spektralzerlegung von $X_{1,2} X_{2,1}$. Da die $\sigma_k$ paarweise verschieden sind, gilt somit insbesondere $F_{1,k} = X_{1,2} F_{2,k} X_{2,1}$ sowie aus Symmetriegründen auch $X_{2,1} F_{1,k} X_{1,2} = F_{2,k}$ für alle $1\leq k<\frac{m}{2}$. Aufgrund des Lemmas \ref{gyoja:idempotenttransport} kann es daher nur Kanten $F_{1,k} \leftrightarrows F_{2,k}$ geben für $1\leq k\leq\floor{\tfrac{m}{2}}$, aber keine Kanten $F_{1,k} \leftrightarrows F_{2,l}$ für $k\neq l$.

Wir zeigen außerdem, dass für gerade $m$ außerdem keine Kanten $F_{1,m/2} \leftrightarrows F_{2,m/2}$ existieren. Nach Konstruktion ist
\[\prod_{1\leq l<\frac{m}{2}} (X^2-\sigma_l) = \frac{\widetilde{\tau}_{m-1}(X^2)}{X^2} = \frac{\tau_{m-1}(X)}{X} = \sum_{j=1}^{m-1} a_j X^{j-1}.\]
Indem wir $X_{2,1}X_{1,2}$ für $X^2$ einsetzen und mit $X_{1,2}$ multiplizieren, ergibt sich daraus
\[X_{1,2} \prod_{1\leq l<\frac{m}{2}} (X_{2,1}X_{1,2}-\sigma_l) = \sum_{j=0}^{m-1} a_j \underbrace{X_{1,2} X_{2,1}\ldots}_{j\,\text{Faktoren}} \overset{(\alpha)}{=} 0.\]
Durch Multiplikation mit dem Nenner von \eqref{eq:gyoja_I2:1} folgt $X_{1,2} F_{2,m/2} = 0$. Es gibt also keine Kante von $F_{2,m/2}$ zu irgendeiner mit $\Set{1}$ gelabelten Ecke und aus Symmetriegründen auch keine Kante von $F_{1,m/2}$ zu einer mit $\Set{2}$ gelabelten Ecke.

%
%
%

\medbreak
Schritt 4: Wir zeigen Z4.
%
		%
%

Wir definieren $\psi_k: \IZ_W^{2\times 2}\to F^\lambda \IZ_W\OmegaGy F^\lambda$ durch
\[\begin{pmatrix} e_{11} & e_{12} \\ e_{21} & e_{22} \end{pmatrix} \mapsto \begin{pmatrix} F_1^{\lambda_k} & F^{\lambda_k} X_{1,2} F^{\lambda_k} \\ \sigma_k^{-1} F^{\lambda_k} X_{2,1} F^{\lambda_k} & F_2^{\lambda_k} \end{pmatrix}.\]
Das ist ein wohldefinierter Algebrahomomorphismus nach Konstruktion der $F^{\lambda_k}$. Er ist surjektiv, da $F^{\lambda_k} \IZ_W \OmegaGy F^{\lambda_k}$ von den Elementen $F_I^{\lambda_k}$ und $F^{\lambda_k} X_{IJ} F^{\lambda_k}$ erzeugt wird, die alle nach Konstruktion im Bild von $\psi_k$ liegen.
\end{proof}

\subsection{Die Zerlegungsvermutung für \texorpdfstring{$A_4$}{Typ A und Rang Vier}}

\begin{lemma}[Relationen für Matrizenringe]\label{gyoja:matrixrings_pathrelations}
Sei $k$ ein kommutativer Ring und $n\in\IN$. Es sei $\mathcal{M}$ der Köcher
\[v_1 \leftrightarrows v_2 \leftrightarrows \cdots \leftrightarrows v_{n-1} \leftrightarrows v_n.\]
Der Matrizenring $k^{n\times n}$ ist isomorph zum Quotienten von $k\mathcal{M}$ nach den Relationen 
\[(v_i\leftarrow v_{i+1} \leftarrow v_i) = e_{v_i} \text{ für }1\leq i<n\quad\text{und}\]
\[(v_i\leftarrow v_{i-1} \leftarrow v_i) = e_{v_i} \text{ für }1<i\leq n,\]
wobei $e_v$ das Idempotent der Ecke $v$ bezeichne.
\end{lemma}
\begin{proof}
Wir werden die Standardbasis des Matrizenringes mit $e_{ij}$ bezeichnen und den Pfadalgebraquotienten mit $X$. Wir definieren zunächst den Homomorphismus $\alpha: k\mathcal{M}\to k^{n\times n}$ durch $\alpha(e_{v_i}) := e_{ii}$ und $\alpha(v_i \leftarrow v_{i\pm 1}):=e_{i,i\pm 1}$. Dies ist wohldefiniert, da $e_{ii}^2=e_{ii}$, $1=\sum_{i=1}^n e_{ii}$ und $e_{i,i\pm 1} = e_{ii} e_{i,\pm 1} e_{i\pm 1,i\pm 1}$ gilt, also die Relationen der Pfadalgebra erfüllt sind. 
In der Tat gilt auch $e_{i,i\pm 1}e_{i\pm 1,i} = e_{ii}$, sodass $\alpha$ einen Homomorphismus $X\to k^{n\times n}$ induziert.

\medbreak
Umgekehrt definieren wir $\beta: k^{n\times n}\to X$ indem wir $\beta(e_{ii}):=e_{v_i}$ sowie
\[\beta(e_{ij}) := (v_i \leftarrow v_{i+1} \leftarrow \cdots \leftarrow v_j)\]
für $i<j$ und
\[\beta(e_{ij}) := (v_i \leftarrow v_{i-1} \leftarrow \cdots \leftarrow v_j)\]
für $i>j$ definieren. Wir müssen die Relationen $e_{ij} e_{kl} = \delta_{jk} e_{il}$ des Matrizenrings nachprüfen. Nun gilt $\beta(e_{ij})\beta(e_{kl})=0$, falls $j\neq k$ ist, da $\beta(e_{ij})$ und $\beta(e_{kl})$ in diesem Fall nicht verkettbare Wege im Köcher sind. Für $i\leq j=k\leq l$ gilt $\beta(e_{ij})\beta(e_{kl}) = \beta(e_{il})$ nach Konstruktion. Ebenso für $i\geq j=k\geq l$. Falls jedoch $i<j$, $j>l$ gilt, dann erhalten wir
\begin{align*}
	\beta(e_{ij})\beta(e_{jl}) &= (v_i \leftarrow \cdots \leftarrow \underbrace{v_{j-1} \leftarrow v_j)(v_j \leftarrow v_{j-1}}_{\equiv e_{v-1}} \leftarrow \cdots \leftarrow v_l) \\
	&= (v_i \leftarrow \cdots \leftarrow v_{j-1})(v_{j-1} \leftarrow \cdots \leftarrow v_l) \\
	&=\beta(e_{i,j-1})\beta(e_{j-1},l)
\end{align*}
woraus auch in diesem Fall induktiv die Behauptung $\beta(e_{ij})\beta(e_{jl})=\beta(e_{il})$ folgt. Der letzte verbleibende Fall $i>j$, $j<l$ folgt analog.
\end{proof}

\begin{theorem}
Die Zerlegungsvermutung gilt für Coxeter"=Gruppen vom Typ $A_4$.
\end{theorem}
\begin{proof}
Wir benutzen eine analoge Strategie wie zuvor für $A_3$ und nutzen die Relationen vom Typ $(\alpha)$ aus. Wir schreiben $(\alpha^{st})$, wenn wir die Relation vom Typ $(\alpha)$ meinen, die zur Kante $s - t$ des Dynkin"=Diagramms gehört.

Der Kompatibilitätsgraph soll letzten Endes wie in Abbildung \ref{fig:gyoja:better_compatibility_graph_A4} zerlegt werden (Inklusionskanten sind erneut weggelassen worden).

\begin{figure}[ht]
	\index{terms}{Kompatibilitätsgraph!verbesserter}
	\centering
	\begin{tikzpicture}

\node[Vertex] (E_empty) at (0,0) {$\emptyset$};

\node[Vertex] (E_1a) at (-3.0,1) {$1$};
\node[Vertex] (E_2a) at (-1.0,1) {$2$};
\node[Vertex] (E_3a) at (+1.0,1) {$3$};
\node[Vertex] (E_4a) at (+3.0,1) {$4$};

\node[Vertex] (E_2b)  at (-1.0,2.0) {$2$};
\node[Vertex] (E_13b) at (-1.0,3.5) {$13$};
\node[Vertex] (E_14b) at ( 0.0,3.0) {$14$};
\node[Vertex] (E_24b) at (+1.0,3.5) {$24$};
\node[Vertex] (E_3b)  at (+1.0,2.0) {$3$};

\node[Vertex] (E_12c) at (-3.0,4.5) {$12$};
\node[Vertex] (E_13c) at (-1.0,4.5) {$13$};
\node[Vertex] (E_14c) at ( 0.0,4.0) {$14$};
\node[Vertex] (E_23c) at ( 0.0,5.0) {$23$};
\node[Vertex] (E_24c) at (+1.0,4.5) {$24$};
\node[Vertex] (E_34c) at (+3.0,4.5) {$34$};

\node[Vertex] (E_124d) at (-1.0,7.0) {$124$};
\node[Vertex] (E_13d)  at (-1.0,5.5) {$13$};
\node[Vertex] (E_23d)  at ( 0.0,6.0) {$23$};
\node[Vertex] (E_24d)  at (+1.0,5.5) {$24$};
\node[Vertex] (E_134d) at (+1.0,7.0) {$134$};

\node[Vertex] (E_123e) at (-3.0,8.0) {$123$};
\node[Vertex] (E_124e) at (-1.0,8.0) {$124$};
\node[Vertex] (E_134e) at (+1.0,8.0) {$134$};
\node[Vertex] (E_234e) at (+3.0,8.0) {$234$};

\node[Vertex] (E_1234) at (0,9.0) {$1234$};

%
%
%
%
%
%
%
%
%
%
%
%

\path[EdgeT]
	(E_1a) edge (E_2a)
	(E_2a) edge (E_3a)
	(E_3a) edge (E_4a)
	
	(E_2b) edge (E_13b)
	(E_13b) edge (E_14b)
	(E_14b) edge (E_24b)
	(E_24b) edge (E_3b)
	
	(E_13c) edge (E_23c) edge (E_14c) edge (E_12c)
	(E_24c) edge (E_23c) edge (E_14c) edge (E_34c)
	
	(E_124d) edge (E_13d)
	(E_13d) edge (E_23d)
	(E_23d) edge (E_24d)
	(E_24d) edge (E_134d)
	
	(E_123e) edge (E_124e)
	(E_124e) edge (E_134e)
	(E_134e) edge (E_234e);

\node (lambda_5)    at (5,0.0) {$(5)$};
\node (lambda_41)   at (5,1.0) {$(4,1)$};
\node (lambda_32)   at (5,3.0) {$(3,2)$};
\node (lambda_311)  at (5,4.5) {$(3,1^2)$};
\node (lambda_221)  at (5,6.0) {$(2^2,1)$};
\node (lambda_2111) at (5,8.0) {$(2,1^3)$};
\node (lambda_11111)at (5,9.0) {$(1^5)$};

\end{tikzpicture}
	\caption{Verbesserte Zerlegung von $\OmegaGy(A_4)$ als Pfadalgebra}
	\label{fig:gyoja:better_compatibility_graph_A4}
\end{figure}

\medbreak
Schritt 1: Konstruktion der $F_I^\lambda$, Verifizieren von Z1 und Z2. Unsere Konstruktion wird auch gleich Z5, Z6 und Z7 zeigen.

Wir beginnen mit $F^{(5)}:=E_\emptyset$ und $F^{(1^5)}:=E_{1234}$.

\medbreak
Wir setzen $F_1^{(4,1)}:=E_1$ und $F_4^{(4,1)}:=E_4$.

Aus der $(\alpha^{12})$-Relation
\begin{align*}
E_1 &= X_{1,2} X_{2,1}
\end{align*}
folgt, dass $F_2^{(4,1)} := X_{2,1} X_{1,2}$ ein Idempotent $\leq E_2$ ist. Aus der folgenden $(\alpha^{12})$-Relation bzw. $(\alpha^{23})$-Relation
\begin{align*}
E_2 &= X_{2,1} X_{1,2} + X_{2,13} X_{13,2} \\
E_2 &= X_{2,3} X_{3,2} + X_{2,13} X_{13,2}
\end{align*}
folgt, dass $F_2^{(4,1)} = X_{2,3} X_{3,2}$ ist. Durch Anwenden des nichttrivialen Graphautomorphismus erhalten wir symmetrisch das Idempotent $F_3^{(4,1)}:=X_{3,4} X_{4,3} = X_{3,2} X_{2,3} \leq E_3$.

\medbreak
Durch Anwenden von $\delta$ erhalten wir die entsprechenden Idempotente
\begin{align*}
	F_{234}^{(2,1^3)} &:= E_{234}, \\
	F_{134}^{(2,1^3)} &:= X_{134,234} X_{234,134}, \\
	F_{124}^{(2,1^3)} &:= X_{124,123} X_{123,124}, \\
	F_{123}^{(2,1^3)} &:= E_{123}.
\end{align*}

\medbreak
Man beachte, dass $F_2^{(4,1)}$ hier durch Idempotenttransport entlang von $\Set{1}\to\Set{2}$ entsteht und $F_2^{(3,2)} := X_{2,13} X_{13,2}$ das Restidempotent $\leq E_2$ für diesen Transport ist. Indem wir $F_2^{(3,2)}$ entlang von $\Set{2}\to\Set{13}$ transportieren, erhalten wir somit das Idempotent
\begin{align*}
	F_{13}^{(3,2)} &:= X_{13,2} F_2^{(3,2)} X_{2,13} \\
	&= X_{13,2} E_2 X_{2,13} - X_{13,2}\smash{\underbrace{F_2^{(4,1)} X_{2,13}}_{=0}} \\
	&= X_{13,2} X_{2,13}.
\end{align*}

Durch Anwenden von $\delta$ erhalten wir das Idempotent $F_{13}^{(2^2,1)} := X_{13,124} X_{124,13} \leq E_{13}$. Aus der $(\alpha^{23})$-Relation
\[E_{12} = X_{12,13} X_{13,12}\]
folgt, dass $F_{13}^{(3,1^2)} := X_{13,12} X_{12,13} = X_{13,12} E_{12} X_{12,13}$ ein Idempotent $\leq E_{13}$ ist.

Wenn wir uns nun die $(\alpha^{32})$-Relation
\[E_{13} = X_{13,2} X_{2,13} + X_{13,12} X_{12,13} + X_{13,124} X_{124,13} = F_{13}^{(3,2)} + F_{13}^{(3,1^2)} + F_{13}^{(2^2,1)}\]
anschauen, dann erkennen wir, dass die Summe von je zwei dieser Idempotente das Restidempotent $\leq E_{13}$ des Idempotenttransports ist, mit welchem wir das jeweils dritte definiert haben. Insbesondere sind diese drei Idempotente paarweise orthogonal.

\medbreak
Indem wir den Graphautomorphismus anwenden, erhalten wir die Idempotente $F_{24}^{(3,1^2)}$, $F_{24}^{(3,2)}$ und $F_{24}^{(2^2,1)}$.

\medbreak
Wir brauchen noch die Idempotente an den Ecken $\Set{14}$ und $\Set{23}$. Das tun wir ebenfalls durch Idempotenttransport. Es gilt die $(\alpha^{34})$-Relation
\[E_{13} = X_{13,14} X_{14,13} + X_{13,124} X_{124,13},\]
d.\,h. $X_{13,14}X_{14,13} = F_{13}^{(3,2)}+ F_{13}^{(3,1^2)}$. Daher erhalten wir durch Idempotenttransport entlang von $\Set{13}\to\Set{14}$ die beiden Idempotente
\[F_{14}^{(3,2)} := X_{14,13} F_{13}^{(3,2)} X_{13,14}\quad\text{und}\]
\[F_{14}^{(3,1^2)} := X_{14,13} F_{13}^{(3,1^2)} X_{13,14}\]
sowie aus Symmetriegründen auch
\[F_{23}^{(3,1^2)} := X_{23,13} F_{13}^{(3,1^2)} X_{13,23}\quad\text{und}\]
\[F_{23}^{(2^2,1)} := X_{23,13} F_{13}^{(2^2,1)} X_{13,23}.\]

Aus der $(\alpha^{43})$-Relation
\[E_{14} = X_{14,13} X_{13,14}\]
und der $(\alpha^{21})$-Relation
\[E_{23} = X_{23,13} X_{13,23}\]
folgt, dass die beiden Restidempotente für diese beiden Idempotenttransporte gleich Null sind, d.\,h. wir haben orthogonale Zerlegungen
\[E_{14} = F_{14}^{(3,2)} + F_{14}^{(3,1^2)}\quad\text{und}\]
\[E_{23} = F_{23}^{(3,1^2)} + F_{23}^{(2^2,1)}.\]

\medbreak
Damit haben wir alle Idempotente bereits gefunden. Die noch nicht definierten $F_I^\lambda$ definieren wir als Null.

\bigbreak
Schritt 2: Verifizieren von Z3.

Wir beweisen dazu, dass nur Inklusionskanten zwischen den verschiedenen in Abbildung \ref{fig:gyoja:better_compatibility_graph_A4} dargestellten Komponenten existieren. Da wir alle Idempotente durch Idempotenttransport definiert haben, werden transversale Kanten in parallele Kanten aufgespalten. Das eliminiert bereits fast alle potentiellen transversalen Kanten zwischen den Komponenten.

\medbreak
Einzig entlang der Kanten $\Set{14} \leftrightarrows \Set{24}$ und $\Set{23} \leftrightarrows \Set{24}$ ist dies noch nicht klar, da wir entlang dieser Kante keinen Idempotenttransport durchgeführt haben, sondern entlang der Kanten $\Set{13} \leftrightarrows \Set{14}$ und $\Set{13} \leftrightarrows \Set{23}$ gearbeitet haben.

\medbreak
Jedoch liefert die symmetrische Variante
\[\widetilde{F_{14}^\lambda} := X_{14,24} F_{24}^\lambda X_{24,14} = \alpha(F_{14}^\lambda)\]
\[\widetilde{F_{23}^\lambda} := X_{23,24} F_{24}^\lambda X_{24,23} = \alpha(F_{23}^\lambda)\]
für alle $\lambda$ die gleichen Idempotente, wobei $\alpha$ den nichttrivialen Graphautomorphismus bezeichne. Sei $\Set{\lambda,\mu}=\Set{(3,2),(3,1^2)}$. Dann gilt
\begin{align*}
	F_{14}^\lambda \widetilde{F_{14}^\mu} &= (X_{14,13} F_{13}^\lambda X_{13,14})\cdot(X_{14,24} F_{24}^\mu X_{24,14}) \\
	&= X_{14,13} F_{13}^\lambda (X_{13,14} X_{14,24}) F_{24}^\mu X_{24,14} \\
	&\overset{\mathclap{(\gamma^{13})}}{=} X_{14,13} F_{13}^\lambda (X_{13,23} X_{23,24} - X_{13,124}X_{124,24} + X_{13,3} X_{3,24}) F_{24}^\mu X_{24,14} \\
	&= X_{14,13} F_{13}^\lambda (X_{13,23} X_{23,24}) F_{24}^\mu X_{24,14} \\
	&\quad- X_{14,13} F_{13}^\lambda (X_{13,124}X_{124,24}) F_{24}^\mu X_{24,14} \\
	&\quad+ X_{14,13} F_{13}^\lambda (X_{13,3} X_{3,24}) F_{24}^\mu X_{24,14}.
\end{align*}
Da nun $F_{13}^{(3,2)} X_{13,23}=0$ und $X_{23,24} F_{24}^{(3,2)}=0$ ist und stets entweder $\lambda$ oder $\mu$ gleich $(3,2)$ ist, verschwindet der erste Summand. Der zweite Summand verschwindet, da für $\lambda\in\Set{(3,2),(3,1^2)}$ in beiden Fällen $F_{13}^\lambda X_{13,124} = 0$ gilt.

Auch der dritte Summand verschwindet. Für $\mu=(3,1^2)$ folgt dies aus $0=X_{3,24} F_{24}^{(3,1^2)}$. Falls $\lambda=(3,1^2)$ und $\mu=(3,2)$ ist, wenden wir die $(\gamma^{12})$-Relation
\[X_{12,2}^1 X_{2,3}^2 = X_{12,13}^2 X_{13,3}^1\]
an. Daraus ergibt sich nämlich
\begin{alignat*}{3}
	F_{13}^{(3,1^2)} X_{13,3} X_{3,24} &= (X_{13,12} X_{12,13}) X_{13,3} X_{3,24} \quad&\text{nach Konstruktion}  \\
	&= X_{13,12} (X_{12,13} X_{13,3}) X_{3,24} \\
	&= X_{13,12} (X_{12,2} X_{2,3}) X_{3,24} &\text{aufgrund von }(\gamma^{12}) \\
	&= X_{13,12} X_{12,2} (X_{2,3} X_{3,24}) \\
	&= 0 &\text{aufgrund von }(\alpha^{23}).
\end{alignat*}

Da somit alle drei Summanden verschwinden, folgt $F_{14}^\lambda \widetilde{F_{14}^\mu}=0$ und durch Anwenden des Graphautomorphismus auch $\widetilde{F_{14}^\lambda} F_{14}^\mu=0$. Durch Vertauschen von $\lambda$ und $\mu$ gilt auch $F_{14}^\mu \widetilde{F_{14}^\lambda}=0$. Also ist $\widetilde{F_{14}^\lambda} \leq E_{14}-F_{14}^\mu = F_{14}^\lambda$. Durch erneutes Anwenden des Graphautomorphismus folgt $F_{14}^\lambda\leq\widetilde{F_{14}^\lambda}$ und somit die Gleichheit. Durch Anwenden von $\delta$ folgt dann auch $\widetilde{F_{23}^\lambda}=F_{23}^\lambda$.

\medbreak
Daher greift auch hier das Argument, dass Idempotenttransport nur parallele Kanten erzeugt. Es bleiben daher, wie behauptet, nur die Inklusionskanten zwischen verschiedenen Komponenten in Abbildung \ref{fig:gyoja:better_compatibility_graph_A4} übrig.

\bigbreak
Schritt 3: Verifizieren von Z4.

%
Wir wollen also surjektive Homomorphismen $\psi_\lambda: \IZ^{d_\lambda\times d_\lambda} \to F^\lambda \OmegaGy F^\lambda$ konstruieren. Dazu verwenden wir die Präsentation von $\IZ^{d_\lambda\times d_\lambda}$ aus Lemma \ref{gyoja:matrixrings_pathrelations}.

Wir setzen zur Abkürzung $X_{IJ}^\lambda := F^\lambda X_{IJ} F^\lambda$. Nach Konstruktion gilt $X_{IJ}^\lambda X_{JI}^\lambda = F_I^\lambda$ für alle transversalen Kanten $I\leftrightarrows J$, wenn $F_I^\lambda$ und $F_J^\lambda$ nicht verschwinden, sowie $X_{IJ}^\lambda = X_{JI}^\lambda = 0$ andernfalls.

\medbreak
Wenn wir die Knoten der jeweiligen starken Zusammenhangskomponente in \ref{fig:gyoja:better_compatibility_graph_A4} mit der zugehörigen Indexmenge bezeichnen (was konfliktfrei möglich ist, da keine Indexmenge mit höherer Vielfachheit als $1$ vorkommt), definiert
\[\psi_\lambda:\IZ^{d_\lambda\times d_\lambda}\to F^\lambda \OmegaGy F^\lambda, e_{II} \mapsto F_{I}^\lambda, e_{IJ} \mapsto X_{IJ}^\lambda\]
einen Morphismus $\psi_\lambda: \IZ^{d_\lambda\times d_\lambda}\to F^\lambda \OmegaGy F^\lambda$ für diejenigen Komponenten, die ohne Inklusionskanten eine Gerade bilden (also in alle außer die mit $\lambda=(3,1^2)$ bezeichnete Komponente).
%

\medbreak
Es ist die Surjektivität von $\psi_\lambda$ zu beweisen, d.\,h. dass alle $X_{IJ}^\lambda$ im Bild von $\psi_\lambda$ sind. Für die transversalen Kanten ist das durch die Konstruktion klar. Für $\lambda=(5),(4,1),(2,1^3)$ und $(1^5)$ ist daher nichts weiter zu tun.

Für $\lambda=(3,2)$ müssen wir die Inklusionskanten $X_{13,3}$ und $X_{24,2}$ betrachten. Dazu benutzen wir die Relation $(\gamma)$:
\begin{align*}
	X_{24,2}^{(3,2)} &= F_{24}^{(3,2)} X_{24,2} F_2^{(3,2)} \\
	&= (X_{24,3} X_{3,24}) X_{24,2} (X_{2,13} X_{13,2}) \\
	&= X_{24,3} X_{3,24} (X_{24,2} X_{2,13}) X_{13,2} \\
	&\overset{\mathclap{(\gamma^{42})}}{=} X_{24,3} X_{3,24} (X_{24,14} X_{14,13} + X_{24,134} X_{134,13} -X_{24,23} \underbrace{X_{23,13}) X_{13,2}}_{=0} \\
	&= X_{24,3} X_{3,24} X_{24,14} X_{14,13} X_{13,2} + X_{24,3} \underbrace{X_{3,24} X_{24,134}}_{=0} X_{134,13} X_{13,2} & \\
	&= X_{24,3}^{(3,2)} X_{3,24}^{(3,2)} X_{24,14}^{(3,2)} X_{14,13}^{(3,2)} X_{13,2}^{(3,2)} \quad\text{da es Element von } F^{(3,2)}\OmegaGy F^{(3,2)} \text{ ist.}\\
	&\in\im(\psi_{(3,2)})
\end{align*}
Durch Anwenden des Graphautomorphimus folgt auch $X_{13,3}^{(3,2)}\in\im(\psi_{(3,2)})$, durch Anwenden von $\delta$ folgt $X_{124,24}^{(2^2,1)},X_{134,13}^{(2^2,1)}\in\im(\psi_{(2^2,1)})$. Also liegen alle $X_{IJ}^\lambda$ im Bild von $\psi_\lambda$ für $\lambda=(3,2), (2^2,1)$ und wir haben die Surjektivität auch in diesen beiden Fällen bewiesen.

\medbreak
Es bleibt der Fall $\lambda=(3,1^2)$ übrig, d.\,h. das äußere Quadrat der Spiegelungsdarstellung. Wir sortieren die sechs zweielementigen Teilmengen in der Reihenfolge $\Set{1,2}$, $\Set{1,3}$, $\Set{1,4}$, $\Set{2,3}$, $\Set{2,4}$, $\Set{3,4}$. Wir behaupten, dass der folgende Homomorphismus ${\psi_{(3,1^2)}: \IZ^{6\times 6}\to F^{(3,1^2)}\OmegaGy F^{(3,1^2)}}$ wohldefiniert ist:
\[\begin{pmatrix}
	e_{11} & e_{12} & & & & \\
	e_{21} & e_{22} & e_{23} & & & \\
	& e_{32} & e_{33} & e_{34} & & \\
	& & e_{43} & e_{44} & e_{45} & \\
	& & & e_{54} & e_{55} & e_{56} \\
	& & & & e_{65} & e_{66}
\end{pmatrix} \mapsto \left(\begin{smallmatrix}
	F_{12}^\lambda & X_{12,13}^\lambda  & & & & \\
	X_{13,12}^\lambda & F_{13}^\lambda & X_{13,14}^\lambda & & & \\
	& X_{14,13}^\lambda & F_{14}^\lambda & X_{14,13}^\lambda X_{13,23}^\lambda & & \\
	& & X_{23,13}^\lambda X_{13,14}^\lambda &  F_{23}^\lambda & X_{23,24}^\lambda & \\
	& & & X_{24,23}^\lambda & F_{24}^\lambda & X_{24,34}^\lambda \\
	& & & & X_{34,24}^\lambda & F_{34}^\lambda
\end{smallmatrix}\right)\]
Fast alle Relationen aus \ref{gyoja:matrixrings_pathrelations} sind dabei durch die Konstruktion der Idempotente bereits erfüllt. Wir müssen nur noch
\[X_{14,13}^\lambda X_{13,23}^\lambda \cdot X_{23,13}^\lambda X_{13,14}^\lambda = F_{14}^\lambda \quad\text{und}\]
\[X_{23,13}^\lambda X_{13,14}^\lambda \cdot X_{14,13}^\lambda X_{13,23}^\lambda = F_{23}^\lambda\]
nachprüfen. Das folgt aus \ref{gyoja:idempotenttransport}, denn in $X_{13,23}^\lambda \cdot X_{23,13}^\lambda = F^\lambda(X_{13,23} F_{23}^\lambda X_{23,13}^\lambda)F^\lambda$ tritt der Transport des Idempotents $F_{23}^\lambda$ entlang der Kante $\Set{23}\to\Set{13}$ auf, welcher genau der Rücktransport zu demjenigen Idempotenttransport ist, den wir zur Definition von $F_{23}^\lambda$ benutzt haben. Also liefert er genau das ursprüngliche Idempotent $F_{13}^\lambda$ zurück. Daraus folgt $X_{14,13}^\lambda X_{13,23}^\lambda \cdot X_{23,13}^\lambda X_{13,14}^\lambda = F^\lambda( X_{14,13} F_{13}^\lambda X_{13,14})F^\lambda = F^\lambda F_{14}^\lambda F^\lambda=F_{14}^\lambda$ aufgrund der Definition von $F_{14}^\lambda$. Durch Anwenden von $\delta$ ergibt sich auch die Gültigkeit der zweiten Gleichung.

Wir überzeugen uns erneut von der Surjektivität von $\psi_\lambda$. Nach Definition sind bereits $X_{13,12}^\lambda$, $X_{14,13}^\lambda$, $X_{24,23}^\lambda$ und $X_{34,24}^\lambda$ im Bild. Es gilt außerdem
\begin{align*}
	X_{23,13}^\lambda &= X_{23,13}^\lambda F_{13}^\lambda \\
	&= X_{23,13}^\lambda (X_{13,14}^\lambda X_{14,13}^\lambda) \\
	&= (X_{23,13}^\lambda X_{13,14}^\lambda) X_{14,13}^\lambda \\
	&\in \im(\psi_\lambda)
\end{align*}
sowie
\begin{align*}
	X_{24,14}^\lambda &= X_{24,14}^\lambda F_{14}^\lambda \\
	&= X_{24,14}^\lambda (X_{14,13}^\lambda X_{13,14}^\lambda) \\
	&= (X_{24,14}^\lambda X_{14,13}^\lambda) X_{13,14}^\lambda \\
	&= (X_{24,23}^\lambda X_{23,13}^\lambda+X_{24,2}^\lambda X_{2,13}^\lambda - X_{24,134}^\lambda X_{134,13}^\lambda) X_{13,14}^\lambda \quad\text{aufgrund von }(\gamma^{24}) \\
	&= (X_{24,23}^\lambda X_{23,13}^\lambda) X_{13,14}^\lambda + \big(X_{24,2}^\lambda \smash{\underbrace{F_2^\lambda}_{=0}} X_{2,13}^\lambda - X_{24,134}^\lambda \smash{\underbrace{F_{134}^\lambda}_{=0}} X_{134,13}^\lambda\big) X_{13,14}^\lambda \\
	&= (X_{24,23}^\lambda X_{23,13}^\lambda) X_{13,14}^\lambda \\
	&\in\im(\psi_\lambda).
\end{align*}
Durch Anwenden des Graphautomorphismus erhalten wir, dass auch $X_{12,13}^\lambda$, $X_{13,14}^\lambda$, $X_{13,23}^\lambda$, $X_{14,24}^\lambda$, $X_{23,24}^\lambda$ und $X_{24,34}^\lambda$ im Bild liegen, womit $\psi_\lambda$ also surjektiv ist.
\end{proof}

\subsection{Die Zerlegungsvermutung für \texorpdfstring{$B_3$}{Typ B und Rang Drei}}

\begin{theorem}
Die Zerlegungsvermutung gilt für Coxeter"=Gruppen vom Typ $B_3$.
\end{theorem}
\begin{proof}
Wir wollen erneut die Relationen vom Typ $(\alpha)$ ausnutzen, um $\OmegaGy$ als Quotienten einer verbesserten Pfadalgebra darzustellen. Als zugrundeliegenden Köcher wollen wir dafür den in Abbildung \ref{fig:gyoja:better_compatibility_graph_B3} abgebildeten bekommen (wobei erneut Inklusionskanten weggelassen wurden).

\begin{figure}[htp]
	\index{terms}{Kompatibilitätsgraph!verbesserter}
	\centering
	\begin{tikzpicture}

\node[Vertex] (E_empty) at (0,0) {$\emptyset$};

\node[Vertex] (E_0a) at (-1.5,1) {$0$};

\node[Vertex] (E_1a) at ( 0.0,1) {$1$};
\node[Vertex] (E_2a) at (+1.5,1) {$2$};

\node[Vertex] (E_0b) at (-1.5,2) {$0$};
\node[Vertex] (E_1b) at ( 0.0,2) {$1$};
\node[Vertex] (E_2b) at (+1.5,2) {$2$};

\node[Vertex] (E_0c) at (-1.5,3) {$0$};
\node[Vertex] (E_1c) at (-0.5,3) {$1$};
\node[Vertex] (E_02c)at (-0.5,5) {$02$};

\node[Vertex] (E_1d) at ( 0.5,3) {$1$};
\node[Vertex] (E_02d)at ( 0.5,5) {$02$};
\node[Vertex] (E_12d)at ( 1.5,5) {$12$};

\node[Vertex] (E_01b) at (-1.5,6) {$01$};
\node[Vertex] (E_02b) at ( 0.0,6) {$02$};
\node[Vertex] (E_12b) at ( 1.5,6) {$12$};

\node[Vertex] (E_01a) at (-1.5,7) {$01$};
\node[Vertex] (E_02a) at ( 0.0,7) {$02$};

\node[Vertex] (E_12a) at ( 1.5,7) {$12$};

\node[Vertex] (E_123) at (0,8) {$012$};

\path[EdgeT]
	(E_1a) edge (E_2a)
	
	(E_0b) edge (E_1b)
	(E_1b) edge (E_2b)
	
	(E_0c) edge (E_1c)
	(E_1c) edge (E_02c)
	
	(E_1d) edge (E_02d)
	(E_02d) edge (E_12d)
	
	(E_01b) edge (E_02b)
	(E_02b) edge (E_12b)
	
	(E_01a) edge (E_02a);

\node (lambda_3_empty)   at (6,0) {$(3),\emptyset$};
\node (lambda_empty_3)   at (5,1) {$\emptyset,(3)$};
\node (lambda_21_empty)  at (7,1) {$(2,1),\emptyset$};
\node (lambda_2_1)       at (6,2) {$(2),(1)$};
\node (lambda_1_2)       at (5,4) {$(1),(2)$};
\node (lambda_11_1)      at (7,4) {$(1^2),(1)$};
\node (lambda_1_11)      at (6,6) {$(1),(1^2)$};
\node (lambda_empty_21)  at (5,7) {$\emptyset,(2,1)$};
\node (lambda_111_empty) at (7,7) {$(1^3),\emptyset$};
\node (lambda_empty_111) at (6,8) {$\emptyset,(1^3)$};

\end{tikzpicture}
	\caption{Verbesserte Zerlegung von $\OmegaGy(B_3)$ als Pfadalgebra}
	\label{fig:gyoja:better_compatibility_graph_B3}
\end{figure}

Fangen wir zunächst an, uns die Idempotente zu beschaffen, d.\,h. Z1 und Z2 nachzuweisen. Unsere Konstruktion wird zugleich auch Z6 erfüllen (und Z7 trivialerweise, da es keine nichttrivialen Graphautomorphismen von $B_3$ gibt). Wir halten fest, dass jeder gute Ring $\IZ[\tfrac{1}{2}]$ enthält, wir dürfen also durch $2$ dividieren.

\bigbreak
Schritt 0:

Erneut setzen wir $F^{(3),\emptyset}:=E_\emptyset$ und $F^{\emptyset,(1^3)}:=E_{012}$.

\medbreak
Schritt 1:

Aus $(\alpha^{21})$ folgt
\[E_2 = X_{2,1} X_{1,2}\]
und daher ist $F_1' := X_{1,2} X_{2,1}$ ein Idempotent $\leq E_1$. Aus $(\alpha^{12})$ folgt weiterhin
\[E_1 = X_{1,2} X_{2,1} + X_{1,02} X_{02,1}\]
sodass also $F_1'':= X_{1,02} X_{02,1}$ ein zu $F_1'$ orthogonales Idempotent ist. Diese beiden Idempotente wollen wir weiter zerlegen.

\medbreak
Schritt 2:

Wir erinnern daran, dass das in den Relationen vom Typ $(\alpha)$ vorkommende Polynom $\tau_{m-1}$ für $m=4$ die Form $\tau_{4-1}(T) = T^3-2T$ hat. Aus $(\alpha^{01})$ und $(\alpha^{10})$ folgt daher:
\begin{align}
	0 &= X_{0,1} X_{1,0} X_{0,1} + X_{0,1} X_{1,02} X_{02,1} - 2X_{0,1} \label{eq:gyoja_B3:1}\tag{1}\\
	0 &= X_{1,0} X_{0,1} X_{1,0} + X_{1,02} X_{02,1} X_{1,0} - 2X_{1,0} \label{eq:gyoja_B3:2}\tag{2}
\end{align}
Indem wir $f:=X_{1,0} X_{0,1}$ setzen, die obere Gleichung von links mit $X_{1,0}$ und die untere von rechts mit $X_{0,1}$ multiplizieren, erhalten wir:
\begin{align}
	0 &= f^2 + f F_1'' - 2f \label{eq:gyoja_B3:3}\tag{3}\\
	0 &= f^2 + F_1'' f - 2f \label{eq:gyoja_B3:4}\tag{4}
\end{align}
Daher gilt
\[f F_1'' = F_1'' f =: f'' \]
und somit auch
\[f F_1' = F_1' f =: f' \]
da $E_1 = F_1' + F_1''$. Wir multiplizieren \eqref{eq:gyoja_B3:3} mit $F_1''$ und \eqref{eq:gyoja_B3:4} mit $F_1'$ und erhalten:
\begin{align}
	0 &= f''^2 - f'' \label{eq:gyoja_B3:5}\tag{5}\\
	0 &= f'^2 - 2f'  \label{eq:gyoja_B3:6}\tag{6}
\end{align}

Das liefert uns die folgende Zerlegung in paarweise orthogonale Idempotente:
\begin{align}
	E_1 = F_1' + F_1'' = \big(\underbrace{\tfrac{1}{2}f'}_{=:F_1^{(2),(1)}}\big) + \big(\underbrace{F_1' - \tfrac{1}{2}f'}_{=:F_1^{(2,1),\emptyset}}\big) + \big(\underbrace{f''}_{=:F_1^{(1),(2)}}\big) + \big(\underbrace{F_1'' - f''}_{=:F_1^{(1^2),(1)}}\big)
	\label{eq:gyoja_B3:idempotentsF1}\tag{7}
\end{align}
Mit diesen Bezeichnungen gilt $X_{1,0}X_{0,1} = 2F_1^{(2),(1)} + F_1^{(1),(2)}$. Wir definieren die restlichen Idempotente durch Idempotenttransport nach Lemma \ref{gyoja:idempotenttransport} bzw. Anwenden der Dualität, d.\,h. wir setzen
\begin{align*}
	F_0^{(1),(2)} &:= X_{0,1} F_1^{(1),(2)} X_{1,0},\\
	F_0^{(2),(1)} &:= \tfrac{1}{2} X_{0,1} F_1^{(2),(1)} X_{1,0}, \\
	F_0^{\emptyset,(3)} &:= E_0 - F_0^{(1),(2)} - F_0^{(2),(1)}, \\
	F_2^{(2),(1)} &:= X_{2,1} F_1^{(2),(1)} X_{1,2}, \\
	F_2^{(2,1),\emptyset} &:= X_{2,1} F_1^{(2,1),\emptyset} X_{1,2},\\ 
	F_{01}^{\lambda,\mu} &:= \delta(F_2^{\mu^\prime,\lambda^\prime}), \\
	F_{02}^{\lambda,\mu} &:= \delta(F_1^{\mu^\prime,\lambda^\prime}), \\
	F_{12}^{\lambda,\mu} &:= \delta(F_0^{\mu^\prime,\lambda^\prime})
\end{align*}
sowie alle anderen $F_I^{\lambda,\mu}:=0$.

\bigbreak
Schritt 3: Verifizieren von Z3.

Dazu müssen wir uns davon überzeugen, dass $F_I^{\lambda,\mu} X_{IJ} F_J^{\lambda',\mu'}=0$ gilt, wann immer $(\lambda,\mu)\not\preceq(\lambda',\mu')$ ist (siehe \ref{J_alg:def:order_Lambda} für die Definition der partiellen Ordnung $\preceq$ auf $\Irr(B_3)=\Set{(\lambda,\mu) | \lambda \vdash a, \mu \vdash b, a+b=3}$). Wir zeigen sogar, dass die einzigen nicht in Abbildung \ref{fig:gyoja:better_compatibility_graph_B3} dargestellten Kanten die Inklusionskanten sind.

Es gilt:
\begin{align*}
	X_{0,1} F_1^{(2,1),\emptyset} &= X_{0,1} (F_1' - \tfrac{1}{2} f') \\
	&= X_{0,1} (E_1 - \tfrac{1}{2} X_{1,0}X_{0,1}) F_1' \\
	&= \tfrac{1}{2} (2 X_{0,1} - X_{0,1} X_{1,0} X_{0,1}) F_1' \\
	&\overset{\text{\eqref{eq:gyoja_B3:1}}}{=} \tfrac{1}{2} (X_{0,1} F_1'') F_1' \\
	&= 0 \\
	X_{0,1} F_1^{(1^2),(1)} &= X_{0,1} (F_1'' - f'') \\
	&= X_{0,1} (E_1 - X_{1,0}X_{0,1}) F_1'' \\
	&= (X_{0,1} - X_{0,1} X_{1,0} X_{0,1}) F_1'' \\
	&\overset{\text{\eqref{eq:gyoja_B3:1}}}{=} (X_{0,1} F_1'' - X_{0,1}) F_1'' \\
	&= 0  
\end{align*}
Das heißt, dass es keine transversale Kanten von $F_1^{(2,1),\emptyset}$ oder $F_1^{(1^2),(1)}$ zu Ecken gibt, die mit $\Set{0}$ gelabelt sind. Da die Idempotente, die mit $\Set{0}$ gelabelt sind, durch Idempotenttransport definiert wurden, folgt aus \ref{gyoja:idempotenttransport} auch $F_0^{\emptyset,(3)}X_{0,1} = 0$, d.\,h. dass auch keine Kanten von mit $\Set{1}$ gelabelten Ecken zu $F_0^{\emptyset,(3)}$ existieren.

Ganz analog gilt auch:
\begin{align*}
	F_1^{(2,1),\emptyset} X_{1,0} &= (F_1' - \tfrac{1}{2} f') X_{1,0}\\
	&= F_1' (E_1 - \tfrac{1}{2} X_{1,0}X_{0,1}) X_{1,0} \\
	&= \tfrac{1}{2} F_1' (2 X_{1,0} - X_{1,0} X_{0,1} X_{1,0}) \\
	&\overset{\text{\eqref{eq:gyoja_B3:2}}}{=} \tfrac{1}{2} F_1' (F_1'' X_{1,0}) \\
	&= 0 \\
	F_1^{(1^2),(1)} X_{1,0} &= (F_1'' - f'') X_{1,0} \\
	&= F_1'' (E_1 - X_{1,0}X_{0,1}) X_{1,0} \\
	&= F_1'' (X_{1,0} - X_{1,0} X_{0,1} X_{1,0}) \\
	&\overset{\text{\eqref{eq:gyoja_B3:2}}}{=} F_1'' (F_1'' X_{1,0} - X_{1,0})\\
	&= 0
\end{align*}
Das heißt, dass auch in der umgekehrten Richtung keine Kanten von Ecken, die mit $\Set{0}$ gelabelt sind, zu $F_1^{(2,1),\emptyset}$ oder $F_1^{(1^2),(1)}$ existieren. Wieder folgt aus \ref{gyoja:idempotenttransport} auch $X_{1,0}F_0^{\emptyset,(3)} = 0$, d.\,h. dass auch keine Kanten von $F_0^{\emptyset,(3)}$ zu mit $\Set{1}$ gelabelten Ecken existieren.

Da die Idempotente $F_0^{(1),(2)}$ und $F_0^{(2),(1)}$ durch Idempotenttransport definiert wurden, existieren auch keine Kanten $F_0^{(1),(2)} \leftrightarrows F_1^{(2),(1)}$ oder $F_0^{(2),(1)} \leftrightarrows F_1^{(1),(2)}$. Ebenso existieren keine Kanten $F_2^{(2,1),\emptyset} \leftrightarrows F_1^{(2),(1)}$ oder $F_2^{(2),(1)} \leftrightarrows F_1^{(2,1),\emptyset}$, da auch die Idempotente $F_2^{(2,1),\emptyset}$ und $F_2^{(2),(1)}$ durch Idempotenttransport definiert wurden.

Da $\delta$ ein Antiautomorphismus ist, ergibt sich die symmetrische Situation auch für die Ecken mit Labeln $\Set{01},\Set{02},\Set{12}$.

Zu prüfen bleibt, ob es Kanten ${F_1^{(1),(2)} \leftrightarrows F_{02}^{(1^2),(1)}}$ oder ${F_1^{(1^2),(1)} \leftrightarrows F_{02}^{(1),(2)}}$ gibt. Dazu beweisen wir, dass auch die beiden Idempotente $F_{02}^{(1^2),(1)}$ und $F_{02}^{(1),(2)}$ durch Idempotenttransport gegeben sind. Dies folgt erneut aus einem Blick auf die $(\alpha)$-Relationen.
\begin{align*}
	0 &= X_{02,1} X_{1,0} X_{0,1} + X_{02,1} X_{1,02} X_{02,1} + X_{02,12} X_{12,02} X_{02,1} - 2X_{02,1}  \label{eq:gyoja_B3:8}\tag{8}
\end{align*}
Indem wir von rechts mit $X_{1,02}$ multiplizieren und $X_{1,0}X_{0,1} = F_1^{(1),(2)} + 2 F_1^{(2),(1)}$ sowie $X_{02,12} X_{12,02} =F_{02}^{(1^2),(1)} + 2F_{02}^{(1),(1^2)}$ und $F_{02}'':=X_{02,1}X_{1,02}=\delta(F_1'')$ benutzen, erhalten wir:
\begin{align*}
	0 &= X_{02,1} ( F_1^{(1),(2)} + 2F_1^{(2),(1)} ) X_{1,02} + F_{02}'' F_{02}'' + (F_{02}^{(1^2),(1)} + 2F_{02}^{(1),(1^2)})F_{02}'' - 2F_{02}'' \\
	&= X_{02,1} ( F_1^{(1),(2)} + 2F_1^{(2),(1)} ) X_{1,02} + F_{02}^{(1^2),(1)} F_{02}''+ 2\smash{\underbrace{F_{02}^{(1),(1^2)}F_{02}''}_{=0}} - F_{02}'' \\
	&= X_{02,1} ( F_1^{(1),(2)} + 2F_1^{(2),(1)} ) X_{1,02} + F_{02}^{(1^2),(1)} - F_{02}'' \\
	&= X_{02,1} ( F_1^{(1),(2)} + 2F_1^{(2),(1)} ) X_{1,02} + (-F_{02}^{(1),(2)})
\end{align*}
Wir erhalten also
\[	F_{02}^{(1),(2)} = X_{02,1} F_1^{(1),(2)} X_{02,1} + 2\cdot X_{02,1} F_1^{(2),(1)} X_{1,02}.\]
Aus der $(\alpha^{2,1})$-Relation
\[ 0 = \sum_{K} X_{02,K}^2 X_{K,2}^1 = X_{02,1} X_{1,2}\]
folgt $X_{02,1} F_1' = 0$ und, da $F_1^{(2),(1)}\leq F_1'$ ist, somit auch $X_{02,1} F_1^{(2),(1)} = 0$, womit wir schlussendlich
\[ F_{02}^{(1),(2)} =  X_{02,1} F_1^{(1),(2)} X_{1,02} \]
erhalten, d.\,h. $F_{02}^{(1),(2)}$ ist ein entlang $\Set{02}\leftrightarrows\Set{1}$ transportiertes Idempotent. Wegen $F_1'' = F_1^{(1),(2)} + F_1^{(1^2),(1)}$ gilt ebenfalls
\[ F_{02}^{(1^2),(1)} =  X_{02,1} F_1^{(1^2),(1)} X_{1,02}\]
womit sich dann aus Lemma \ref{gyoja:idempotenttransport}, wie gewünscht, ergibt, dass es keine außer den in Abbildung \ref{fig:gyoja:better_compatibility_graph_B3} eingezeichneten Kanten zwischen den mit $\Set{1}$ und $\Set{02}$ gelabelten Idempotenten gibt. Das zeigt, dass Z3 gilt.

%

\bigbreak
Schritt 4: Verifizieren von Z4.

Für eindimensionale Darstellungen ist dabei nicht viel zu tun. Wir definieren in diesen Fällen $\psi_{\lambda,\mu}: \IZ[\tfrac{1}{2}]^{1\times 1}\to F^{\lambda,\mu} \IZ[\tfrac{1}{2}]\OmegaGy F^{\lambda,\mu}$ auf die einzig mögliche Weise durch $\psi_{\lambda,\mu}(e_{11}):=F^{\lambda,\mu}$. Die Surjektivität ist hier automatisch gegeben, da innerhalb der vier Komponenten in Abbildung \ref{fig:gyoja:better_compatibility_graph_B3}, die zu den eindimensionalen Charakteren gehören, keine Kanten vorkommen und daher $F^{\lambda,\mu} \IZ[\tfrac{1}{2}]\OmegaGy F^{\lambda,\mu} = \IZ[\tfrac{1}{2}] F^{\lambda,\mu}$ gilt.

\medbreak
\begin{table}[htp]
\centering
\begin{tabular}{lccl}
Charakter $\chi_{\lambda,\mu}$ & Abbildung $\psi_{\lambda,\mu}$ \\
\hline
$\chi_{(2,1),\emptyset} $ & 
	$\begin{pmatrix}
	e_{11} & e_{12} \\ e_{21} & e_{22}	
	\end{pmatrix}$&$\mapsto$&$\begin{pmatrix}
		E_1 & X_{1,2} \\ X_{2,1} & E_2
	\end{pmatrix}$ \\
$\chi_{\emptyset,(2,1)}$ &
	$\begin{pmatrix}
	e_{11} & e_{12} \\ e_{21} & e_{22}	
	\end{pmatrix}$&$\mapsto$&$\begin{pmatrix}
		E_{02} & -X_{02,01} \\ -X_{01,02} & E_{01}
	\end{pmatrix}$ \\
$\chi_{(2),(1)}$ &
	$\begin{pmatrix}
	e_{11} & e_{12} & \\ e_{21} & e_{22} & e_{23} \\ & e_{32} & e_{33}
	\end{pmatrix}$&$\mapsto$&$\begin{pmatrix}
		E_0 & X_{0,1} & \\ \tfrac{1}{2}X_{1,0} & E_1 & X_{1,2} \\ & X_{2,1} & E_2
	\end{pmatrix}$ \\
$\chi_{(1),(1^2)}$ &
	$\begin{pmatrix}
	e_{11} & e_{12} & \\ e_{21} & e_{22} & e_{23} \\ & e_{32} & e_{33}
	\end{pmatrix}$&$\mapsto$&$\begin{pmatrix}
		E_{12} & -X_{12,02} & \\ -\tfrac{1}{2}X_{02,12} & E_{02} & -X_{02,01} \\ & -X_{01,02} & E_{01}
	\end{pmatrix}$ \\
$\chi_{(1),(2)}$ &
	$\begin{pmatrix}
	e_{11} & e_{12} & \\ e_{21} & e_{22} & e_{23} \\ & e_{32} & e_{33}
	\end{pmatrix}$&$\mapsto$&$\begin{pmatrix}
		E_0 & X_{0,1} & \\ -X_{1,0} & E_1 & X_{1,02} \\ & X_{02,1} & E_{02}
	\end{pmatrix}$ \\
$\chi_{(1^2),(1)}$ &
	$\begin{pmatrix}
	e_{11} & e_{12} & \\ e_{21} & e_{22} & e_{23} \\ & e_{32} & e_{33}
	\end{pmatrix}$&$\mapsto$&$\begin{pmatrix}
		E_{12} & X_{12,02} & \\ X_{02,12} & E_{02} & -X_{02,1} \\ & -X_{1,02} & E_1
	\end{pmatrix}$ \\
\end{tabular}
\caption{Morphismen $\psi_{\lambda,\mu} : \IZ[\tfrac{1}{2}]^{d_{\lambda,\mu}\times d_{\lambda,\mu}} \twoheadrightarrow F^{\lambda,\mu} \IZ[\tfrac{1}{2}]\OmegaGy F^{\lambda,\mu}$}
\label{tab:inverses_B3}
\end{table}

%

In Tabelle \ref{tab:inverses_B3} sind die Morphismen $\psi_{\lambda,\mu}: \IZ[\tfrac{1}{2}]^{d_{\lambda,\mu}\times d_{\lambda,\mu}} \to F^{\lambda,\mu} \IZ[\tfrac{1}{2}]\OmegaGy F^{\lambda,\mu}$ für die zwei- und dreidimensionalen Charaktere angegeben. Die Tabelle ist dabei so zu lesen, dass die Angabe $e_{ij}\mapsto a$ in der Zeile für den Charakter $\chi_{\lambda,\mu}$ für die Abbildungsvorschrift $\psi_{\lambda,\mu}(e_{ij}):=F^{\lambda,\mu} a F^{\lambda,\mu}$ steht.

Dabei benutzen wir erneut, dass sich $\IZ[\tfrac{1}{2}]^{d\times d}$ als $\IZ[\tfrac{1}{2}]$"~Algebra durch die Präsentation in Lemma \ref{gyoja:matrixrings_pathrelations} beschreiben lässt. Diese Relationen sind tatsächlich erfüllt, wie aus der Konstruktion der Idempotente $F^{\lambda,\mu}$ folgt.

\medbreak
Die Konstruktion der Morphismen $\psi_{\lambda,\mu}$ sichert zu, dass jedes Idempotent $F_I^{\lambda,\mu}$ für alle $I$ und jedes Element $F^{\lambda,\mu} X_{IJ} F^{\lambda,\mu}$ für transversale Kanten $I\leftrightarrows J$ im Bild von $\psi_{\lambda,\mu}$ liegt. Für $(\lambda,\mu) \in\Set{ (\emptyset,(2,1)),((2,1),\emptyset),((2),(1)),((1),(1^2))}$ ist das bereits ausreichend, da alle in der zu $F^{\lambda,\mu}$ gehörigen starken Komponente vorkommenden Kanten transversale Kanten sind.

\medbreak
Für $\chi_{(1),(2)}$ könnte jedoch eine Inklusionskante $\Set{0}\to\Set{0,2}$ und für $\chi_{(1^2),(1)}$ könnte eine Inklusionskante $\Set{1}\to\Set{1,2}$ vorkommen. Wir zeigen nun noch, dass dies nicht der Fall ist, und vollenden damit den Beweis. Wir benutzen dabei die $(\gamma)$-Relationen:
\begin{align*}
	X_{02,0}^{(1),(2)} &= F^{(1),(2)} X_{02,0} F_0^{(1),(2)} \\
	&= F^{(1),(2)} X_{02,0} (X_{0,1} F_1^{(1),(2)} X_{1,0}) \\
	&= F^{(1),(2)} (X_{02,0} X_{0,1}) F_1^{(1),(2)} X_{1,0} \\
	&\overset{\mathclap{(\gamma^{20})}}{=} F^{(1),(2)} (X_{02,2} X_{2,1} + X_{02,12}X_{12,1} - X_{02,01} X_{01,1}) F_1^{(1),(2)} X_{1,0} \\
	&= (X_{02,2}^{(1),(2)} X_{2,1}^{(1),(2)} + X_{02,12}^{(1),(2)} X_{12,1}^{(1),(2)} - X_{02,01}^{(1),(2)} X_{01,1}^{(1),(2)}) F_1^{(1),(2)} X_{1,0}
\end{align*}
Hier verschwinden alle Summanden in der Klammer, da in der $(1),(2)$ bezeichneten Komponente in Abbildung \ref{fig:gyoja:better_compatibility_graph_B3} keine Knoten mit Label $\Set{2}$, $\Set{12}$ oder $\Set{01}$ liegen. Durch Anwenden von $\delta$ folgt ebenso $F_{12}^{(1^2),(1)} X_{12,1} = 0$.
\end{proof}

\index{terms}{Einparameterfall|)}
\index{terms}{Vermutung!$W$-Graph-Zerlegungs-|)}

\appendix
\chapter{Algorithmen}
\section{Coxeter-Gruppen und Hecke-Algebren}

\subsection{Bruhat-Ordnung}
\index{terms}{Bruhat!-Ordnung|(}

\begin{convention}
Wir vereinbaren, dass für alle Funktionen dieses Abschnitts und der folgenden Abschnitte die Coxeter"=Gruppe $(W,S)$ implizit bekannt ist. In praktischen Implementierungen heißt das möglicherweise, dass entweder zusätzliche Parameter übergeben werden müssen oder die Datenstukturen anderweitig mit diesem Wissen versehen werden müssen (z.\,B. durch geeignete Klassen in einem objektorierentierten Paradigma).

Da wir nicht in die Details einer solchen Implementierung eintauchen wollen, legen wir außerdem fest, dass eine gewisse Grundfunktionalität für Coxeter"=Gruppen und deren Elemente bereits zur Verfügung steht:
\begin{itemize}
	\item Der Stern \lstinline!*! soll so implementiert sein (z.\,B. durch Operator-Überladung oder GAPs Method-Selection-Mechanismus), dass er für Elemente der Coxeter"=Gruppe das Produkt $W\times W\to W$ berechnet.
	\item Eine Funktion \lstinline!Length! mit einem Parameter, die die Längenfunktion $l:W\to\IZ$ zur Verfügung stellt.
	\item Eine Funktion \lstinline!LeftDescentSet! mit einem Parameter, die die Abstiegsmenge
	\[D_L(w):=\Set{s\in S | sw<w}=\Set{s\in S | l(sw)<l(w)}\]
	berechnet.
	\item Eine Funktion \lstinline!CanonicalLeftAscentSet! wird benötigt, um $W$ effizient zu durchlaufen. Diese Funktion soll für eine anfangs fest gewählte Nummerierung $s_1, \ldots, s_n$ von $S$ jedem $w\in W$ die Menge
	\[A_L(w) := \Set{s_i\in S | s_i w>w \ \text{und für alle}\ j<i: s_j(s_i w)> s_i w}\]
	zuordnen. Dadurch wird eine Baumstruktur auf $W$ definiert: Die Wurzel ist $1$ und der Vorgänger eines $w\neq 1$ ist $s_i w$, wobei $i$ der kleinste Index mit $s_i w < w$ ist. Die Funktion \lstinline!CanonicalLeftAscentSet! berechnet also alle Nachfolger von $w$ in diesem Baum. (Siehe \cite{stembridge2001computational} für Details und Anwendungen dieses Konzepts.)
\end{itemize}

Zusätzlich dazu werden wir eine gewisse Standardfunktionalität voraussetzen, die es in den meisten Programmierumgebungen gibt oder die trivial zu implementieren ist. Insbesondere werden wir mit Mengen, Listen und Dictionaries als Datenstrukturen arbeiten. Wir werden dafür stets selbsterklärende Namen wählen. Im ersten Beispiel ist dies etwa die Funktion \lstinline!SomeElementOf!, die einfach irgendein Element der Menge auswählen soll.
\end{convention}

\begin{remark}
Es gibt verschiedene Möglichkeiten, diese Operationen für Elemente von Coxeter"=Gruppen tatsächlich zu implementieren. Am wichtigsten ist dabei eine effiziente Lösung des Wortproblems und eine effiziente Implementierung der Längenfunktion, da sich alle anderen notwendigen Operationen darauf zurückführen lassen. Denkbar wären etwa die folgenden Optionen:
\begin{itemize}
	\item Man kann direkt mit reduzierten Worten arbeiten. Dies ist sicherlich der allgemeinste Ansatz, der für beliebige Coxeter"=Gruppen funktioniert. Dabei ist eine effiziente Lösung des Wortproblems das größte Hindernis, d.\,h. der Test, wann zwei reduzierte Worte gleich sind. Der Satz von Tits zeigt, dass das Wortproblem für Coxeter"=Gruppen grundsätzlich algorithmisch lösbar ist. Das Finden einer effizienten Lösung ist jedoch nichttrivial und bedarf detaillierterer Überlegungen.
	\item Kristallographische Gruppen kann man als Matrizengruppen über den ganzen Zahlen realisieren, z.B. indem man eine geeignete Version der Spiegelungsdarstellung betrachtet. Damit kann man alle notwendigen Operationen effizient implementieren (siehe etwa \cite{humphreys1992coxeter} zu den hier nicht besprochenen Eigenschaften der Spiegelungsdarstellung). Die Spiegelungsdarstellung für beliebige Coxeter"=Gruppen zu verwenden, ist theoretisch auch möglich, erfordert aber für unendliche Coxeter"=Gruppen die Fähigkeit, in gewissen Zahlkörpern $\IQ\subseteq K\subseteq\IQ^\text{ab}\cap\IR$ exakt zu rechnen (das kann GAP beispielsweise) und exakt zu testen, ob ein Element $x\in K$ positiv oder negativ ist (das kann GAP momentan nicht).
	\item Endliche Coxeter"=Gruppen kann man zusätzlich dazu auch als Permutationsgruppen auf ihrem Wurzelsystem auffassen. Dies ist z.\,B. in CHEVIE implementiert.
\end{itemize}
Zu Implementierungsdetails siehe z.\,B. CHEVIE (\cite{GapChevie}), GAP generell (\cite{GAP3} und \cite{GAP4}) oder Meinolf Gecks PyCox-Implementierung (\cite{geck2012pycox}).
\end{remark}

\begin{remark}
\index{terms}{Liftungseigenschaft}
Beginnen wir mit einem einfachen Algorithmus, der mit Hilfe der Liftungseigenschaft der Bruhat"=Ordnung testet, ob für zwei Elemente $y,w\in W$ bezüglich der Bruhat"=Ordnung $y\leq w$ gilt oder nicht.
\end{remark}

\begin{algorithm}[Bruhat-Ordnung]\label{algo:bruhat_simple}
Funktion \textbf{Bruhat}:
\begin{itemize}
	\item Input-Parameter:
	\begin{itemize}
		\item y: \lstinline!ElementOfW!
		\item w: \lstinline!ElementOfW!
	\end{itemize}
	\item Rückgabe-Wert: \lstinline!Boolean!
	
	\lstinline!true! falls $y\leq w$ in der Bruhat"=Ordnung gilt, \lstinline!false! andernfalls.
\end{itemize}
\end{algorithm}
\begin{lstlisting}
function Bruhat(ElementOfW y, ElementOfW w)
	returns Boolean
{
	ElementOfW s;
	Integer Leny := Length(y);
	Integer Lenw := Length(w);
	
	while(0 < Leny and Leny < Lenw){
		s := SomeElementOf(LeftDescentSet(w));
		w := s*w;
		Lenw := Lenw-1;

		if(s in LeftDescentSet(y)){
			y := s*y;
			Leny := Leny-1;
		}
	}
	if(0 = Leny){
		return true;
	}
	if(Leny >= Lenw){
		if(y = w){
			return true;
		}else{
			return false;
		}
	}
}
\end{lstlisting}

\begin{remark}
Folgende Abwandlung des Algorithmus gibt zusätzlich aus, wie $y$ als Teilwort einer reduzierten Darstellung von $w$ realisiert werden kann.
\end{remark}

\begin{algorithm}[Bruhat-Ordnung 2]\label{algo:bruhat_subwords}
Funktion \textbf{BruhatSubWord}:
\begin{itemize}
	\item Input-Parameter:
	\begin{itemize}
		\item y: \lstinline!ElementOfW!
		\item w: \lstinline!ElementOfW!
	\end{itemize}
	\item Output-Parameter:
	\begin{itemize}
		\item RedWordy: \lstinline!ListOfBooleans!
		
		Eine Liste $(b_1,\ldots,b_n)\in\Set{0,1}^n$ derart, dass die Indizes \mbox{$1\leq i_1<\ldots<i_k\leq n$} mit $b_i=1$ ein Teilwort $s_{i_1} \cdots s_{i_k}$ von \lstinline!RedWordw! definieren, das ein reduzierter Ausdruck von $y$ ist.
		\item RedWordw: \lstinline!ListOfElementsOfW!
		
		Eine Liste $(s_1,\ldots,s_n)$ mit $s_i\in S$ derart, dass $s_1 \cdots s_n$ ein reduzierter Ausdruck von $w$ ist.
	\end{itemize}
	\item Rückgabe-Wert: \lstinline!Boolean!
	
	\lstinline!true! falls $y\leq w$ in der Bruhat"=Ordnung gilt, \lstinline!false! andernfalls.
\end{itemize}
\end{algorithm}
\begin{lstlisting}
function BruhatSubword(
	Input: ElementOfW y, ElementOfW w,
	Output: ListOfBooleans RedWordy, ListOfElementOfWs RedWordw
	)
	returns Boolean
{
	ElementOfW s;
	Integer Leny := Length(y);
	Integer Lenw := Length(w);
	
	while(0 < Leny and Leny < Lenw){
		s := SomeElementOf(LeftDescentSet(w));
		w := s*w;
		Lenw := Lenw-1;
		RedWordw.Append(s);
		
		if(s in LeftDescentSet(y)){
			y := s*y;
			Leny := Leny-1;
			RedWordy.Append(true);
		}else{
			RedWordy.Append(false);
		}
	}
	if(0 = Leny){
		while(0 < Lenw){
			s := SomeElementOf(LeftDescentSet(w));
			w := s*w;
			Lenw := Lenw-1;
			RedWordy.Append(false);
			RedWordw.Append(s);
		}
		return true;
	}
	if(Leny >= Lenw){
		if(y = w){
			while(0 < Lenw){
				s := SomeElementOf(LeftDescentSet(w));
				w := s*w;
				Lenw := Lenw-1;
				RedWordy.Append(true);
				RedWordw.Append(s);
			}
			return true;
		}else{
			return false;
		}
	}
}
\end{lstlisting}

\begin{remark}
Wenn man bereits ein reduziertes Wort von $w$ kennt, kann man natürlich die Zeile \lstinline!s := SomeElementOf(LeftDescentSet(w))! dadurch ersetzen, dass man jeweils den nächsten Buchstaben in diesem Wort auswählt. Auf diese Weise wird $y$ als Teilwort dieses gegebenen Wortes dargestellt.
\end{remark}

\begin{remark}
Die Schleifenbedingung der ersten \lstinline!while!-Schleife ist optimiert, indem beide Randfälle $l(y)=0$ und ${l(y)=l(w)}$ als Abbruchbedingung dienen. Da dann jedoch abgebrochen werden kann, bevor \lstinline!RedWordw! und \lstinline!RedWordy! vollständig erstellt wurden, sind in beiden Fällen noch "`Restschleifen"' implementiert, deren Zweck nicht mehr der Test auf $y\leq w$ ist, sondern die Vervollständigung von \lstinline!RedWordy! bzw. \lstinline!RedWordw!.
\end{remark}

\begin{remark}
Der Ablauf des Algorithmus kann aus dem Wort \lstinline!RedWordw! und der Information in \lstinline!RedWordy! rekonstruiert werden. Das kann man nutzen, um weitere Dinge zu berechnen. Der folgende Algorithmus kann beispielsweise aus zwei solchen Worten, wie sie \lstinline!BruhatSubword! ausgibt, das Bruhat"=Intervall $[y,w]$ rekonstruieren.

Man beachte, dass der eben besprochene Algorithmus zwei Folgen $y_k, w_k$ definiert. Wenn $w=s_1 \ldots s_m$ das reduzierte Wort ist, das \lstinline!BruhatSubword! findet, dann sind diese Folgen durch die Rückwärtsrekursion
\begin{alignat*}{3}
w_m &:= w,\quad &w_{k-1} &:= s_{m-k+1} w_k \\
y_m &:= y, &y_{k-1} &:= \begin{cases} s_{m-k+1} y_k & \text{falls}\,s_{m-k+1} y_k < y_k \\ y_k & \text{falls}\,s_{m-k+1} y_k > y_k\end{cases}
\end{alignat*}
gegeben. Es folgt in beiden Fällen $y_{k-1} < s_{m-k+1} y_{k-1}$ für alle $k=1...m$, d.\,h. die Voraussetzungen von \ref{bruhat:intervals} sind in jedem Schritt erfüllt und wir können $[y_k,w_k]$ rekursiv aus $[y_{k-1},w_{k-1}]$ berechnen.
\end{remark}

\begin{algorithm}[Bruhat-Intervalle]\label{algo:bruhat_interval}
\index{terms}{Bruhat!-Intervall}
Funktion \textbf{BruhatInterval}:
\begin{itemize}
	\item Input-Parameter:
	\begin{itemize}
		\item y: \lstinline!ElementOfW!
		\item w: \lstinline!ElementOfW!
	\end{itemize}
	\item Rückgabe-Wert: \lstinline!SetOfElementsOfW!
	
	Das Bruhat"=Intervall $[y,w]:=\Set{z\in W | y\leq z\leq w}$.
\end{itemize}
\end{algorithm}

\begin{lstlisting}[mathescape=true]
function BruhatInterval(ElementOfW y, ElementOfW w)
	returns SetOfElementOfW
{
	ListOfBooleans RedWordy;
	ListOfElementsOfW RedWordw;
	if(not BruhatSubwords(y,w,out RedWordy,out RedWordw)){
		return $\emptyset$;
	}
	
	SetOfElementsOfW X := {1};
	ElementOfW s;
	ElementOfW u := 1;

	for(i := Length(RedWordw)-1; i >= 0; i := i-1){
		s := RedWordw[i];
		if(RedWordy[i]){
			// $[u,w] \rightsquigarrow [su,sw]$
			X := $\Set{x\in X | su\leq x}\cup\Set{sx | x\in X, l(sx)>l(x)}$;
			u := s*u;
		}else{
			// $[u,w] \rightsquigarrow [u,sw]$
			X := $X\cup\Set{sx | x\in X, l(sx)>l(x)}$;
		}
	}
	
	return X;
}
\end{lstlisting}

\begin{remark}
Die Indizierung ist hier $0$-basiert. Das muss man natürlich anpassen, wenn man mit einer Programmierumgebung arbeitet, die bei $1$ zu zählen anfängt anstatt bei $0$, wie es eigentlich sein müsste.
\end{remark}
\begin{remark}
Man kann, indem man von Links- zu Rechtsmultiplikation übergeht, die Schleife auch so strukturieren, dass \lstinline!RedWordw! von vorne nach hinten durchlaufen wird.

Die hier angegebene Variante hat jedoch den Vorteil, dass man die letzten beiden Algorithmen auch kombinieren könnte, wenn man nur an Bruhat"=Intervallen, aber nicht an reduzierten Wörtern interessiert ist. Dann entfiele die Notwendigkeit für eine der beiden "`Restschleifen"'. Man kann sie streichen und beginnt die Intervallberechnung dafür direkt mit \lstinline!X:={y}!. Das hat den Vorteil, dass man dann nur noch jeweils ${l(w)-l(y)}$ Iterationen für die Ermittlung der Teilwortinformation und des Intervalls benötigt, statt jeweils $l(w)$ Iterationen, wie es jetzt der Fall ist.
\end{remark}

\index{terms}{Bruhat!-Ordnung|)}

\subsection{Kazhdan-Lusztig-Algorithmen}
\index{terms}{Kazhdan-Lusztig!-Polynom|(}\index{terms}{Kazhdan-Lusztig!-$\mu$|(}

\begin{convention}
Wir vereinbaren nun zusätzlich, dass auf der Coxeter"=Gruppe $(W,S)$ eine Gewichtsfunktion $L:W\to\Gamma$ gegeben ist und die folgenden Funktionen implizit wissen, dass sie Berechnungen für diese Coxeter"=Gruppe"=mit"=Gewicht durchführen, so wie die obigen Funktionen wussten, dass sie sich auf $(W,S)$ beziehen.

Zusätzlich wollen wir vereinbaren, dass die folgenden Funktionalitäten bereits vorhanden sind:
\begin{itemize}
	\item Die Operatoren \lstinline!+! und \lstinline!-! sollen für Elemente von $\Gamma$ die Gruppenverknüpfungen realisieren. Entsprechend sollen \lstinline!+!, \lstinline!-! und \lstinline!*! auch für Elemente von $\IZ[\Gamma]$ implementiert sein. Des Weiteren sei \lstinline!vPow! eine Funktion, die $\Gamma\cup\Set{\infty}\to\IZ[\Gamma],\gamma\mapsto v^\gamma$ implementiert, und \lstinline!BarInvolution! eine Funktion, die $F[\Gamma]\to F[\Gamma], f\mapsto\overline{f}$ implementiert.
	\item Wir brauchen außerdem Funktionen \lstinline!PositivePart! und \lstinline!NegativePart!, die aus einem ${\alpha\in\IZ[\Gamma]}$ den positiven bzw. negativen Anteil extrahieren, d.\,h. diejenigen Elemente $\alpha_+\in\IZ[\Gamma_{>0}], \alpha_-\in\IZ[\Gamma_{<0}]$, für die $\alpha=\alpha_+ + \alpha_0 + \alpha_-$ mit $\alpha_0\in\IZ$ gilt.
	\item Eine Funktion \lstinline!Weight! mit einem Parameter, die die Gewichtsfunktion $L:W\to\Gamma$ berechnet.
\end{itemize}
\end{convention}

\begin{remark}
Wir beginnen mit einer Verbesserung des Algorithmus aus \citep[2.1.9]{geckjacon} zur Berechung von $P_{y,w}^\ast$ und $\mu_{y,w}^s$. Da die Berechnung immer noch stark rekursiv ist, bietet sich eine Implementierung durch dynamische Programmierung an, d.\,h. wir nehmen an, dass zwei Arrays/Dictionaries/Hash-Tables/... \lstinline!TableP! und \lstinline!TableMu! verwaltet werden, mit denen die bereits berechneten Werte von $P_{y,w}^\ast$ beziehungsweise die von Null verschiedenen Werten von $\mu_{y,w}^s$ abrufbar sind.

\index{terms}{Kritisches Paar}
Aus Speicherplatzgründen sollte in \lstinline!TableP! nur für kritische Paare $(y,w)$ etwas gespeichert werden. Zur Erinnerung: \udot{Kritische Paare} sind diejenigen $(y,w)\in W\times W$, für die $y\leq w$ und
\[\forall t\in S: (tw<w \implies L(t)>0 \wedge ty<y) \wedge (wt<w \implies L(t)>0 \wedge yt<y)\]
gilt. Aus den Eigenschaften in \ref{KL:def:KL_poly} folgt, dass für nichtkritische Paare $(y,w)$ stets $P_{y,w}^\ast = 0$ ist oder ein kritisches Paar $(y',w')$ und ein $\gamma\in\Gamma_{\geq 0}$ existieren derart, dass $P_{y,w}^\ast = v^{-\gamma} P_{y',w'}^\ast$ gilt.

Der Zweck folgender Hilfsfunktion ist es, genau diese Berechnung von $(y',w')$ und $\gamma$ aus $(y,w)$ zu bewerkstelligen.
\end{remark}

\begin{algorithm}[Kritische Paare]\label{algo:KL_crit_pair}
Funktion \textbf{KLCriticalPair}:
\begin{itemize}
	\item Input-Parameter:
	\begin{itemize}
		\item y: \lstinline!ElementOfW!
		\item w: \lstinline!ElementOfW!
	\end{itemize}
	\item Output-Parameter:
	\begin{itemize}
		\item u: \lstinline!ElementOfW!
		\item v: \lstinline!ElementOfW!
		
		$(u,v)$ ist ein kritisches Paar, falls $\gamma<\infty$ zurückgegeben wird. Ansonsten ist $P_{y,w}^\ast=0$ und es werden diejenigen Elemente ausgegeben, an denen dieser Fakt erkannt wurde.
	\end{itemize}
	\item Rückgabe-Wert: \lstinline!ElementOfGamma!

	Ein $\gamma\in\Gamma_{\geq 0}\cup\Set{\infty}$ derart, dass entweder $\gamma=\infty$ und $P_{y,w}^\ast=0$ oder $\gamma\in\Gamma_{\geq 0}$, $(u,v)$ ein kritisches Paar ist und $P_{y,w}^\ast = v^{-\gamma} P_{u,v}^\ast$ gilt.
\end{itemize}
\end{algorithm}
\begin{lstlisting}[mathescape=true]
function KLCriticalPair(
	Input: ElementOfW y, ElementOfW w,
	Output: ElementOfW u, ElementOfW v)
	returns ElementOfGamma
{
	ElementOfGamma gamma := 0;
	
	while(true){
		if(not Bruhat(y,w)){
			u := y;
			v := w;
			return $\infty$;
		}
		
		if($\exists t\in S : tw<w \wedge L(t)=0$){
			w := t*w;
			y := t*y;
			continue;
		}
		if($\exists t\in S : wt<w \wedge L(t)=0$){
			w := w*t;
			y := y*t;
			continue;
		}
		
		if($\exists t\in S : tw<w \wedge L(t)>0 \wedge ty>y$){
			y := t*y;
			gamma := gamma+Weight(t);
			continue;
		}
		if($\exists t\in S : wt<w \wedge L(t)>0 \wedge yt>y$){
			y := y*t;
			gamma := gamma+Weight(t);
			continue;
		}
		
		break;
	}

	u := y;
	v := w;
	return gamma;
}
\end{lstlisting}

\begin{remark}
Haben wir dies, so können wir die Rekursion für die Polynome $P_{y,w}^\ast$ und $\mu_{y,w}^s$ wie folgt implementieren. Diese Algorithmen sind Verfeinerungen von \citep[2.1.9]{geckjacon}.
\end{remark}

\begin{algorithm}[Kazhdan-Lusztig-Polynome $P_{y,w}^\ast$, $\mu_{y,w}^s$ und $P_{y,w}$]\label{algo:KL_poly_mu}
Funktion \textbf{KLPoly}:
\begin{itemize}
	\item Input-Parameter:
	\begin{itemize}
		\item y: \lstinline!ElementOfW!
		\item w: \lstinline!ElementOfW!
	\end{itemize}
	\item Rückgabe-Wert: \lstinline!LaurentPolynomial!

	Das Polynom $P_{y,w}^\ast$.
\end{itemize}

Funktion \textbf{KLMu}:
\begin{itemize}
	\item Input-Parameter:
	\begin{itemize}
		\item y: \lstinline!ElementOfW!
		\item w: \lstinline!ElementOfW!
		\item s: \lstinline!ElementOfS!
	\end{itemize}
	\item Rückgabe-Wert: \lstinline!LaurentPolynomial!

	Das Polynom $\mu_{y,w}^s$, falls $s\in S$, $L(s)>0$ und $sy<y<z<sz$ ist, andernfalls $0$.
\end{itemize}

Funktion \textbf{KazhdanLusztigPolynomial}:
\begin{itemize}
	\item Input-Parameter:
	\begin{itemize}
		\item y: \lstinline!ElementOfW!
		\item w: \lstinline!ElementOfW!
	\end{itemize}
	\item Rückgabe-Wert: \lstinline!LaurentPolynomial!

	Das Polynom $P_{y,w}:=v^{L(w)-L(y)}P_{y,w}^\ast$.
\end{itemize}

\end{algorithm}
\begin{lstlisting}[mathescape=true]
function KLPoly(ElementOfW y, ElementOfW w)
	returns LaurentPolynomial
{
	LaurentPolynomial KLP;
	ElementOfW t;
	
	ElementOfW u;
	ElementOfW v;
	ElementOfW tu;
	ElementOfW tv;

	BruhatInterval BInterval;

	
	ElementOfGamma gamma := KLCriticalPair(y,w,out u,out v);	

	if(gamma == $\infty$){
		KLP := 0;
	}elif(u == v){
		KLP := vPow(-gamma);
	}else{
		if([u,v] in TableP){
			KLP := TableP[u,v];
		}else{
			t := SomeElementOf(LeftDescentSet(v));
		
			tu := t*u;
			tv := t*v;
		
			KLP := vPow(Weight(t))*KLPoly(u,tv) + KLPoly(tu,tv);

			if(Length(v) > 2){
				BInterval := $\Set{z\in [u,tv] | tz<z}$;
			
				foreach(z in BInterval){
					KLP := KLP - KLPoly(u,z)*KLMu(z,tv,tt);
				}
			}

			TableP[u,v]:=KLP;
		}

		KLP := vPow(-gamma)*KLP;
	}

	return KLP;
}

function KLMu(ElementOfW y,ElementOfW w,ElementOfW s)
	returns LaurentPolynomial
{
	SetOfElementsOfW J;
	
	ElementOfGamma gamma;
	LaurentPolynomial KLP;
	LaurentPolynomial mu;
	
	LaurentPolynomial alpha;
	
	BruhatInterval BInterval;


	// Schritt 1: Check 
	J := SetDifference(LeftDescentSet(y), LeftDescentSet(w));
	if($s\notin J$ or 0==Weight(s) or $y \not\leq w$){
		return 0;
	}
	

	if([y,w,s] in TableMu){
		return TableMu[y,w,s];
	}

	KLP := KLPoly(y,w);
	
	// Einparameterfall
	if($\forall s_i\in S: L(s_1) = L(s_i)$){
		mu := vPow(Weight(s)) * KLP;
		mu := mu - NegativePart(mu);
	}
	// Multiparameterfall
	else{
		alpha := vPow(Weight(s))*KLP;
		alpha := alpha - NegativePart(alpha);
	
		BInterval := $\Set{z\in [y,w] | z\neq y, sz<z}$;

		foreach(z in BInterval){
			alpha := alpha - KLPoly(y,z)*KLMu(z,w,ss);
			alpha := alpha - NegativePart(alpha);
		}

		mu := alpha+BarInvolution(PositivePart(alpha));
	fi;

	TableMu[y,w,s] := mu;
	return mu;
}

function KazhdanLusztigPolynomial(ElementOfW y, ElementOfW w)
	returns LaurentPolynomial
{
	return vPow(Weight(w)-Weight(y))*KLPoly(y,w);
}
\end{lstlisting}

\begin{remark}
\index{terms}{Einparameterfall}\index{terms}{Coxeter!-Gruppe!reduzible}\index{terms}{Induktion}
Eine einfache Optimierung ergibt sich, wenn man den Einparameterfall nicht für ganz $S$ testet, sondern nur testet, ob $L$ auf derjenigen Teilmenge $J\subseteq S$ konstant ist, welche minimal mit der Eigenschaft $w\in W_J$ ist, d.h. $J=\Set{s_{i_1},\ldots,s_{i_l}}$ für ein reduziertes Wort $w=s_{i_1}\cdots s_{i_l}$.

Eine weitere Optimierung ergibt sich, wenn man die parabolische Untergruppe $W_J$ weiter in ihre irreduziblen Komponenten zerlegt und ausnutzt, dass für eine Zerlegung ${W_J=W_1\times W_2}$ die Gleichungen $P_{y_1y_2, w_1w_2}^\ast = P_{y_1, w_1}^\ast \cdot P_{y_2, w_2}^\ast$ sowie
\[\mu_{y_1 y_2, w_1 w_2}^s =\begin{cases} \mu_{y_1,w_1}^s & \text{falls }s\in S_1\text{ und }y_2=w_2 \\ \mu_{y_2,w_2}^s & \text{falls } s\in S_2\text{ und }y_1=w_1 \\ 0 & \text{sonst}\end{cases}\]
für alle $y_i,w_i\in W_i$ und alle $s\in S$ gelten.

Noch weiter optimieren kann man, wenn man beliebige parabolische Untergruppen von $W_J$ betrachtet und die Induktionseigenschaften von Kazhdan-Lusztig"=Polynomen und "~Zellen benutzt. Je nachdem, an welcher Information man genau interessiert ist, kann dies eine enorme Zeit- und Speicherplatzeinsparung bedeuten. Für theoretische Details siehe \cite{howlett2003inducingI}, \cite{howlett2004inducingII}, \cite{geck2003induction} sowie \cite{geck2012pycox} für Implementationsdetails.
\end{remark}
\begin{remark}
\index{terms}{Kazhdan-Lusztig!-Zellen}\index{terms}{Zellmodul}\index{terms}{Zusammenhangskomponente!starke}\index{terms}{Tiefensuche}
Hat man erst einmal die $\mu_{y,w}^s$"~Werte kann man auch die Links-, Rechts- und zweiseitigen Kazhdan-Lusztig"=Zellen berechnen, indem man in dem in \ref{KL:def:cells} definierten Graphen nach den starken Zusammenhangskomponenten sucht. Dafür gibt es Linearzeitalgorithmen mittels Tiefensuche, beispielsweise Tarjans Algorithmus (siehe \cite[Ch.\,4]{tarjan1972depth}).

Nicht nur die Zellen selbst, auch die Matrixdarstellungen der zugehörigen Zellmoduln können in dieser Situation einfach bestimmt werden. Indem man direkt die dazugehörigen $W$"~Graphen abspeichert, spart man Speicherplatz.
\end{remark}

\index{terms}{Kazhdan-Lusztig!-Polynom|)}\index{terms}{Kazhdan-Lusztig!-$\mu$|)}
\section{Balancieren von Matrixdarstellungen}

\index{terms}{Darstellung!balancierte|(}

\begin{remark}
Wir haben in \ref{balanced_reps:invariant_blf2} bewiesen, dass zu jedem Isomorphietyp $\lambda$ auch eine balancierte Darstellung $\rho:H\to K^{d\times d}$ vom Isomorphietyp $\lambda$ existiert. Der Beweis war sogar konstruktiv genug, um ihn als Algorithmus zu implementieren. An entscheidender Stelle wurde jedoch eine Orthogonalbasis für die von $\Omega\in K^{d\times d}$ auf $K^d$ induzierte Bilinearform benutzt. In dem Fall, für den wir uns am meisten interessieren, nämlich dass $H$ eine Hecke"=Algebra, $F\subseteq\IR$ und $K:=\QuotFld(F[\Gamma])$ ist, führt solch ein Vorgehen jedoch zu einer Gradexplosion in den Matrixeinträgen. Besonders bei großen Dimensionen erfordert dies unnötig viel Speicherplatz.

In diesem Abschnitt wollen wir ein alternatives Vorgehen beschreiben, mit dem keine Grad\-ex\-plo\-sion stattfindet. Dies basiert im Wesentlichen auf Satz \ref{balanced_reps:invariant_blf1}. Wir werden zeigen, dass zu jeder irreduziblen Matrixdarstellung: $\rho: H\to K^{d\times d}$ mit invarianter Bilinearform $\Omega=\Omega^\text{Tr}\in GL_d(K)$ ein Basiswechsel $Q\in GL_d(K)$ existiert, der ${Q^T \Omega Q\in GL_d(\mathcal{O})}$ und $\Omega \equiv \diag(\ldots) \mod\mathfrak{m}$ und gleichzeitig $\deg((Q^T \Omega Q)_{ij})\leq\deg(\Omega_{ij})$ für alle Indizes $i,j$ erfüllt. Nach Satz \ref{balanced_reps:invariant_blf1} bedeutet dies, dass $Q^{-1}\rho Q$ eine balancierte Darstellung ist.
\end{remark}

\subsection{Theorie}

\begin{remark}
Die Idee dafür ist, eine modifizierte Cholesky-Zerlegung von $\Omega$ durchzuführen.\index{terms}{Cholesky-Zerlegung}
\end{remark}

\begin{theorem}[Balanciertheit ohne Gradexplosion]\label{balanced_reps:algo_degree}
\index{terms}{$\ast$-symmetrisch}\index{terms}{Symmetrische Algebra}\index{terms}{Bewertung}\index{terms}{Bewertung!-sring}
\index{symbols}{nuinfty@$\nu_\infty$}
Sei $F$ ein formal reeller Körper, $K:=\QuotFld(F[\Gamma])$, $\nu: K\to\Gamma\cup\Set{\infty}$ die kanonische Bewertung auf $K$ und $\mathcal{O}\subseteq K$ der Bewertungsring von $\nu$.

Weiter sei $H$ eine endlichdimensionale, zerfallend halbeinfache $K$"~Algebra, $\ast$ ein involutiver Antiautomorphismus, $\tau$ eine Spurform und $B$ eine $\ast$"~symmetrische Basis von $H$.
Außerdem sei $\rho: H\to K^{d\times d}$ eine irreduzible Matrixdarstellung und $\Omega\in GL_d(K)$ eine beliebige invariante Bilinearform für $\rho$ mit $\nu(\Omega)=0$ (man beachte, dass dies durch Multiplikation mit einem geeigneten $v^\gamma$ immer erreicht werden kann).

\medbreak
Dann existiert eine Folge $Q^{(1)}, Q^{(2)}, \ldots, Q^{(d)}\in GL_d(F[\Gamma])$ so, dass für die durch $\Omega^{(0)}:=\Omega$ und $\Omega^{(i)} := {Q^{(i)}}^\text{Tr} \cdot \Omega^{(i-1)} \cdot Q^{(i)}$ ($i=1,\ldots,d$) definierte Folge gilt:
\begin{enumerate}
	\item $\Omega^{(i)} = f\cdot \sum_{x\in B} \big(A_x^{(i)}\big)^{\text{Tr}} \big(A_x^{(i)}\big)$ für geeignete $A_x^{(i)} \in\mathcal{O}^{d\times d}$ und eine von $i$ unabhängige Konstante $f\in\mathcal{O}^\times$.
	\item Die Reduktion von $\Omega^{(i)}$ ist bereits teilweise diagonal:
	\[\Omega^{(i)} \equiv
	\left(\begin{array}{cccc}
	D^{(i)} & 0    & \rdelim){2}{1em} & \rdelim\}{1}{3em}[$i$]   \\
	0 & \ast &                  & \rdelim\}{1}{3em}[$d-i$]
\end{array}\right.\mod\mathfrak{m}\]
	Dabei ist $D^{(i)}\in GL_i(F)$ diagonal.
	\item Die Grade sind monoton fallend:
	\[\nu_\infty\big(\Omega^{(i)}\big) \leq \nu_\infty\big(\Omega^{(i-1)}\big)\]
	Dabei sei $\nu_\infty: K\to\Gamma\cup\Set{-\infty}$ die durch $\nu_\infty\big(\sum_\gamma a_\gamma v^\gamma\big) := \sup\Set{\gamma | a_\gamma\neq 0}$ definierte und via $\nu_\infty\big(\frac{f}{g}\big):=\nu_\infty(f)-\nu_\infty(g)$ auf $K$ fortgesetzte Gradfunktion.
\end{enumerate}
Insbesondere erfüllt $\Omega^{(d)}$ die Voraussetzungen von \ref{balanced_reps:invariant_blf1} und mit $Q:=Q^{(1)} Q^{(2)} \cdots Q^{(d)}$ ist $Q^{-1} \rho Q$ eine balancierte Darstellung.
\end{theorem}
\begin{proof}
Dazu sei an Lemma \ref{formally_real:symmetric_matrices} erinnert. Wir werden dieses Lemma nun immerzu benutzen.

\medbreak
Fangen wir zunächst an, a. für $i=0$ nachzuprüfen. b. ist trivialerweise wahr für $i=0$.

Wir wissen bereits, dass
\[\Omega':=\sum_{x\in B} \rho(x)^\text{Tr} \rho(x)\]
eine invariante Bilinearform für $\rho$ ist (siehe \ref{balanced_reps:invariant_blf2}). Für $\alpha:=\min\Set{\nu(\rho(x)) | x\in B}$ gilt $\nu(\Omega)=2\alpha$ aufgrund von \ref{formally_real:symmetric_matrices}. Wenn wir also $\Omega'$ durch $\Omega'':=v^{-2\alpha}\Omega'$ ersetzen, erhalten wir eine invariante Bilinearform mit $\nu(\Omega'')=0$.

Da $\rho$ irreduzibel ist, unterscheiden sich $\Omega$ und $\Omega''$ also nur um Multiplikation mit einer Konstanten $f\in K$. Da $\nu(\Omega)=\nu(\Omega'')=0$ ist, muss auch $\nu(f)=0$ sein, d.\,h. $f\in\mathcal{O}^\times$. Daraus folgt, dass $\Omega$ die Form
\[\Omega = f\sum_{x\in B} A_x^\text{Tr} A_x\]
mit $f\in\mathcal{O}^\times$ und $A_x\in\mathcal{O}^{d\times d}$ hat. Wichtig ist nun, dass die Existenz so einer Summendarstellung invariant unter Basiswechseln ist: Wenn $\Omega$ diese Form hat, so ist auch $Q^\text{Tr} \Omega Q$ von dieser Gestalt für alle $Q\in K^{d\times d}$.

\medbreak
Wir gehen induktiv vor: Haben wir $\Omega^{(i-1)}$ bereits vorliegen, dann konstruieren wir zwei Matrizen $R^{(i)}, E^{(i)}\in GL_d(K)$ und setzen $Q^{(i)} := E^{(i)} R^{(i)}$.

Wir setzen dafür $\gamma_i := \frac{1}{2}\nu\big(\Omega_{ii}^{(i-1)}\big)$, $E^{(i)} := \diag(1,\ldots,1,v^{-\gamma_i},1,\ldots,1)$ und führen einen Basiswechsel mit $E^{(i)}$ durch, d.\,h. wir ersetzen $\Omega^{(i-1)}$ durch
\[\Omega^{(i-\frac{1}{2})} := {E^{(i)}}^\text{Tr} \Omega^{(i-1)} E^{(i)} = f\sum_{x\in B} \big(A_x^{(i-1)} E^{(i)}\big)^\text{Tr} \big(A_x^{(i-1)} E^{(i)}\big).\]
Aufgrund von Lemma \ref{formally_real:symmetric_matrices} ist $\gamma_i = \min\Set{\nu(A_x^{(i-1)} e_i) | x\in B}$, d.\,h. wir erhalten ${v^{-\gamma_i} A_x^{(i-1)} e_i \in\mathcal{O}^d}$ für alle $x\in B$. Daher ist $A_x^{(i-\frac{1}{2})} := A_x^{(i-1)} E^{(i)} \in\mathcal{O}^{d\times d}$ und somit auch $\Omega^{(i-\frac{1}{2})}\in\mathcal{O}^{d\times d}$. Außerdem sind die Grade in $\Omega^{(i-\frac{1}{2})}$ nicht angewachsen, sondern in der $i$"~ten Zeile und Spalte um $\gamma_i$ verringert worden. Es gilt also, wie gewünscht, ebenfalls $\nu_\infty\big(\Omega^{(i-\frac{1}{2})}\big) \leq \nu_\infty\big(\Omega^{(i-1)}\big)$.

\medbreak
Jetzt ist nach Konstruktion $\Omega_{ii}^{(i-\frac{1}{2})}=v^{-2\gamma_i} \Omega_{ii}^{(i-1)}\in\mathcal{O}^\times$. Wir führen nun einen weiteren Basiswechsel mit $R^{(i)}\in F^{d\times d}$ durch, wobei
\[R_{xy}^{(i)} := \begin{cases} 1 & x=y \\
0 & x\neq i \,\text{oder}\, x>y \\
-\frac{\Omega_{iy}^{(i-\frac{1}{2})}}{\Omega_{ii}^{(i-\frac{1}{2})}}\mod\mathfrak{m} & x=i \,\text{und}\, x<y
\end{cases}.\]
Also ist
\[R^{(i)} = \begin{pmatrix}
	1 &        &   &   & & & \\
	  & \ddots &   &   & & & \\
	  &        & 1 &   & & & \\
	  &        &   & 1 & -\frac{\Omega_{i,i+1}^{(i-\frac{1}{2})}}{\Omega_{ii}^{(i-\frac{1}{2})}} & \cdots & -\frac{\Omega_{i,d}^{(i-\frac{1}{2})}}{\Omega_{ii}^{(i-\frac{1}{2})}} \\
	  &        &   &   & 1 &        &   \\
	  &        &   &   &   & \ddots &   \\
	  &        &   &   &   &        & 1 \\
\end{pmatrix}.\]
Jetzt führen wir den Basiswechsel mit $R^{(i)}$ durch, d.\,h.
\[\Omega^{(i)} := {R^{(i)}}^\text{Tr} \Omega^{(i-\frac{1}{2})} R^{(i)} \equiv
  \begin{array}{cccccc}
  \ldelim({3}{1em}
	& D^{(i-1)} & 0 & 0 & \rdelim){3}{1em} & \rdelim\}{2}{3em}[$=:D^{(i)}$]\\
	& 0 & \Omega_{ii}^{(i-\frac{1}{2})} & 0 & & \\
	& 0 & 0 & \ast & & 
	\end{array}
	\mod\mathfrak{m}.\]
Weiter ist auch
\[\Omega^{(i)} = f\sum_{x\in B} \big(A_x^{(i-\frac{1}{2})} R^{(i)}\big)^\text{Tr} \big(A_x^{(i-\frac{1}{2})} R^{(i)}\big).\]
Wegen $R^{(i)}\in GL_d(F)$ ist $A_x^{(i)} := A_x^{(i-\frac{1}{2})} R^{(i)}\in \mathcal{O}^{d\times d}$ und die Grade haben sich in diesem Schritt gar nicht geändert, d.\,h. $\nu_\infty(\Omega^{(i)}) = \nu_\infty(\Omega^{(i-\frac{1}{2})}) \leq \nu_\infty(\Omega^{(i-1)})$.

Damit sind alle Behauptungen des Induktionsschritts gezeigt.
\end{proof}

\begin{remark}
War $\rho$ bereits balanciert, dann werden bei diesem Vorgehen alle $Q^{(i)}\in GL_d(F)$ gewählt, es findet also wirklich nur eine Cholesky-Zerlegung von $\Omega\mod\mathfrak{m}$ statt und mehr nicht.
\end{remark}

\subsection{Implementierung}

\begin{convention}
Wir vereinbaren, dass für alle Funktionen dieses Abschnitts die Coxeter"=Gruppe"=mit"=Gewicht $(W,S,L)$ und deren Hecke"=Algebra $H$ implizit bekannt sind.

Zusätzlich zu den vorherigen Funktionen wollen wir vereinbaren, dass die folgenden Funktionalitäten bereits vorhanden sind:
\begin{itemize}
	\item Eine Funktion \lstinline!Valuation! mit einem Parameter, welche die Bewertungsfunktion \mbox{$\nu:F(\Gamma)\to\Gamma\cup\Set{\infty}$} berechnet, sowie eine Funktion \lstinline!LowestTerm! mit einem Parameter, die jedem Element $x=v^\gamma \cdot f$ von $K^\times$ mit $f\in\mathcal{O}^\times$ das Element $f\mod\mathfrak{m}\in F^\times$ zuordnet.
	\item Ein Typ/eine Klasse/... \lstinline!Representation!, der/die endlichdimensionale Matrixdarstellungen $\rho:H\to K^{d\times d}$ speichert. Es soll durch \lstinline!Rep.Dimension! die Dimension $d$ und durch \lstinline!Rep.Generators! die Liste der Matrizen $\rho(T_s)$ zugänglich sein.
	\item Eine Funktion \lstinline!InvariantBilinearForm!, die eine solche Matrixdarstellung nimmt und eine invariante Bilinearform $B\in GL_d(K)$ berechnet, z.\,B. indem das lineare Gleichungssystem $B\rho(T_s) = \rho(T_s)^\text{Tr} B$ gelöst wird.
\end{itemize}
\end{convention}

\begin{algorithm}[Balancieren von Matrixdarstellungen]\label{algo:balance}
Funktion \textbf{BalanceRepresentation}:
\begin{itemize}
	\item Input-Parameter:
	\begin{itemize}
		\item rho: \lstinline!Representation!
		
		Eine Matrixdarstellung $\rho: H\to K^{d\times d}$.
	\end{itemize}
	\item Output-Parameter:
	\begin{itemize}
		\item Q: \lstinline!Matrix!
		
		Eine Matrix $Q\in GL_d(K)$ derart, dass $\rho'(h):=Q^{-1}\rho Q$ eine balancierte Matrixdarstellung ist.
		\item Omega: \lstinline!Matrix!
		
		Eine symmetrische Matrix $\Omega\in\mathcal{O}^{d\times d}$ mit $\Omega\rho'(h)=\rho'(h)^\text{Tr}\Omega$.
		\item D: \lstinline!DiagonalMatrix!
		
		Die Diagonalmatrix $\Omega\mod\mathfrak{m} \in GL_d(F)$.
	\end{itemize}
	\item Rückgabe-Wert: Keiner.
\end{itemize}
\end{algorithm}
\begin{lstlisting}[mathescape=true]
function BalanceRepresentation(
	Input: Representation rho,
	Output: Matrix Q,
	        Matrix Omega,
	        DiagonalMatrix D)
	returns nothing
{
	Omega := InvariantBilinearForm(rho);
	Omega := vPow(-Valuation(Omega))*Omega;


	Q := IdentityMatrix(dim x dim);

	for(i := 0; i < Rep.Dimension; i := i+1){
		gamma := 1/2*Valuation($\Omega_{ii}$);
		$D_{ii}$ :=	LowestTerm($\Omega_{ii}$);

		// Berechnung von $\Omega^{(i-\frac{1}{2})}$
		for(j := 0; j < Rep.Dimension; j := j+1){
			$\Omega_{ij}$ := vPow(-gamma)*$\Omega_{ij}$;
			$\Omega_{ji}$ := vPow(-gamma)*$\Omega_{ji}$;
		}

		R := IdentityMatrix(dim x dim);
		for(j := i+1; j < Rep.Dimension; j := j+1){
			gamma := Valuation($\Omega_{ij}$);

			if(0 == gamma){
				$R_{ij}$ := - LowestTerm($\Omega_{ij}$) / $D_{ii}$;
			}
		}
		Omega := TransposedMatrix(R)*Omega*R;
		
		for(j := i; j < Rep.Dimension; j := j+1){
			$R_{ij}$ := vPow(-gamma)*$R_{ij}$;
		}
		Q := Q*R;
	}
}	
\end{lstlisting}

\begin{remark}
Erneut ist dies eine Implementierung für $0$-basierte Indizierung.
\end{remark}
\begin{remark}
Dieser Algorithmus führt im Wesentlichen exakt das aus, was auch im Beweis von Satz \ref{balanced_reps:algo_degree} geschieht. Man kann noch optimieren, indem man beispielsweise die Matrixmultiplikationen \lstinline!Omega := TransposedMatrix(R)*Omega*R! manuell implementiert, da \lstinline!R! ja sehr dünn besetzt ist.

Außerdem kann man auf \lstinline!Q := Q*R! ganz verzichten, wenn man statt $Q$ die Matrix $Q^{-1}$ berechnet. Diese kann man nämlich direkt angeben und zeilenweise aufbauen. Mit den Bezeichnungen aus dem Beweis von \ref{balanced_reps:algo_degree} gilt:
\[(Q^{-1})_{ij} = \begin{cases} v^{\gamma_i}\cdot\left( \frac{\Omega_{ij}^{(i-\frac{1}{2})}}{\Omega_{ii}^{(i-\frac{1}{2})}}\mod\mathfrak{m}\right) & \text{falls}\,i\leq j \\ 0 & \text{sonst} \end{cases}\]
Dies entspricht dann auch mehr dem Vorgehen bei einer richtigen Cholesky-Zerlegung.\index{terms}{Cholesky-Zerlegung}
\end{remark}

\begin{remark}
\index{terms}{Tiefensuche}
Um Balanciertheit nur zu testen, reicht obiger Algorithmus. Will man aber zusätzlich Führende"=Koeffizienten"=Matrizen berechnen, dann benötigt man die Werte $a_\lambda$. Der folgende Algorithmus berechnet sie mit Hilfe der Formel (siehe Korollar \ref{balanced_reps:characters})
\[a_\lambda = -\min\Set{\nu(\chi_\lambda(T_w)) | w\in W}.\]

Um Speicherplatz zu sparen, wird dabei eine rekursive Tiefensuche durch $W$ durchgeführt. Die Tiefensuche ist so strukturiert, dass man mit einer einzigen $d\times d$ Matrix $M=v^{a_\lambda} \rho(T_w)$ mit Einträgen aus $K$ auskommt, welche man durch Multiplikation mit $\rho(T_s)^{\pm 1}$ für ein $s\in S$ in jedem Schritt anpasst. Die Rekursionstiefe ist durch die Länge des längsten Elements $l(w_0)$ beschränkt.
\end{remark}

\begin{algorithm}[Berechnung von $a_\lambda$]\label{algo:aValue}
Funktion \textbf{aValue}:
\begin{itemize}
	\item Input-Parameter:
	\begin{itemize}
		\item rho: \lstinline!Representation!
		
		Eine Matrixdarstellung $\rho: H\to K^{d\times d}$.
	\end{itemize}
	\item Rückgabe-Wert: \lstinline!ElementOfGamma!
	
	Es wird $\min\Set{\nu(\chi(T_w)) | w\in W}$ ausgegeben, wobei $\chi$ der Charakter von $\rho$ ist. Dies ist genau $\min\Set{a_\lambda | \lambda\in\Lambda, \rho\text{ hat einen irreduziblen Konstituenten }\rho_\lambda}$.
\end{itemize}
\end{algorithm}
\begin{lstlisting}[mathescape=true]
function aValue(Representation rho)
	returns ElementOfGamma
{
	Matrix M := IdentityMatrix(dim x dim);

	ElementOfGamma alpha = 0;

	function dfs(ElementOfW w)
		returns nothing
	{
		foreach(s in CanonicalLeftAscentSet(w)){
			M := rho.Generators[s]*M;
			dfs(s*w);
			M := rho.Generators[s]^(-1)*M;
		}
		
		alpha := Minimum(alpha,Valuation(Trace(M)));
	}
	dfs(1);
	
	return -alpha;
}
\end{lstlisting}

\begin{remark}
\index{terms}{Tiefensuche}
Der folgende Algorithmus bestimmt die Führenden"=Koeffizienten"=Matrizen einer Matrix"=Darstellung $\rho: H\to K^{d\times d}$. Es wird dabei die gleiche Tiefensuche benutzt. $\rho$ ist genau dann balanciert, wenn alle dabei berechneten Matrizen $M$ in $\mathcal{O}^{d\times d}$ sind und die führenden Koeffizienten sind gegebenenfalls durch $M\mod\mathfrak{m}^{d\times d}$ bestimmt.

Da viele dieser Matrizen verschwinden, bietet es sich wieder an, mit einer dünn besetzen Speicherstruktur zu arbeiten wie einem Dictionary, einer Hash"=Table oder etwas Ähnlichem.
\end{remark}

\begin{algorithm}[Extrahieren von Führenden-Koeffizienten-Matrizen]\label{algo:lead_coeff}
\index{terms}{Führende Koeffizienten}
Funktion \textbf{ComputeLeadingCoefficients}:
\begin{itemize}
	\item Input-Parameter:
	\begin{itemize}
		\item rho: \lstinline!Representation!
		
		Eine Matrixdarstellung $\rho: H\to K^{d\times d}$. Ist $\rho$ nicht balanciert, wird ein Fehler ausgegeben. Ist $\rho$ balanciert, so wird ein Dictionary \lstinline!rho.LeadCoeff! angelegt, das die von Null verschiedenen Führenden-Koeffizienten-Matrizen $c(w) = v^{a_\lambda} \rho(T_w)\mod\mathfrak{m}\in F^{d\times d}$ verwaltet.
	\end{itemize}
	\item Rückgabe-Wert: Keiner.
\end{itemize}
\end{algorithm}
\begin{lstlisting}[mathescape=true]
function ComputeLeadingCoefficients(Representation rho)
	returns nothing
{
	rho.LeadCoeff := new Dictionary;

	Matrix M := aValue(rho) * IdentityMatrix(dim x dim);
	ElementOfGamma gamma;


	function dfs(ElementOfW w)
		returns nothing
	{
		foreach(s in CanonicalLeftAscentSet(w)){
			M := rho.Generators[s]*M;
			dfs(s*w);
			M := rho.Generators[s]^(-1)*M;
		}

		gamma := Valuation(M);
		
		if(0 == gamma){
			Rep.LeadCoeff[w] := LowestTerm(M);
		}elif(gamma < 0){
			Error("Representation not balanced!");
		}
	}

	dfs(1);
}
\end{lstlisting}

\begin{remark}
\index{terms}{Matrix!dünn besetzte}
Die Führenden"=Koeffizienten"=Matrizen selbst sind ebenfalls dünn besetzt, hier ist also weiteres Einsparpotential vorhanden.
\end{remark}
\begin{remark}
Man kann beide Algorithmen natürlich auch kombinieren, indem man im Dictionary die Matrizen $v^\alpha \rho(T_w)$ für die bisher besuchten $w\in W$ und den bisherigen Minimalwert $\alpha$ abspeichert und jedes Mal, wenn eine echt kleinere Bewertung $\alpha_\text{neu}$ gefunden wird, das Dictionary zurückgesetzt wird.

Das erfordert allerdings eine vorsichtigere Behandlung des Fehlerfalls, dass $\rho$ nicht balanciert ist, da erst am Ende der Berechnung feststeht, ob der Fehler wirklich eingetreten ist oder nicht.
\end{remark}
\begin{remark}
\index{terms}{Darstellung!balancierte!strikt}
Wenn man nicht nur eine balancierte Darstellung übergeben hat, sondern eine, die stark balanciert im Sinne von \ref{balanced_reps:strictly_balanced} ist, kann man den Speicheraufwand aufgrund der Relation $d_\mathfrak{s} c(x)_\mathfrak{st} = d_\mathfrak{t} c(x^{-1})_\mathfrak{ts}$ für alle $\mathfrak{s},\mathfrak{t}$ und alle $x\in W$ noch einmal um die Hälfte drücken.
\end{remark}

\index{terms}{Darstellung!balancierte|)}
\section{Zellmoduln}
\index{terms}{Zellmodul|(}

\subsection{Theorie}

\begin{remark}
Wir wissen aus Lemma \ref{CellAlg:ex:CellModules_Hecke} im Prinzip, wie die Zellmoduln von Hecke"=Algebren aussehen. Ich möchte trotzdem einige Vereinfachungen besprechen.
\end{remark}

\begin{lemma}
\index{terms}{Lusztig-Isomorphismus}\index{terms}{Kazhdan-Lusztig!-Zellen}
Es gelten $(\spadesuit)$, $(\clubsuit)$ und $(\vardiamond)$. Sei $(W,S,L)$ eine Coxeter"=Gruppe"=mit"=Gewicht, $H$ ihre Hecke"=Algebra, $F:=\IQ_W$, $K:=F(\Gamma)$ und $\rho:KH\to K^{d_\lambda\times d_\lambda}$ eine irreduzible, balancierte Matrixdarstellung von $H$ vom Isomorphietyp $\lambda$.

Sei $\psi:=\overline{\rho}\circ\phi$ die zugeordnete Zelldarstellung von $H$. Es gilt:
\[\psi(h) = \sum_{\substack{d\in\mathcal{D} \\ \overline{\rho}(t_d)\neq 0}} \sum_{z\in\mathfrak{C}_d} n_d \sigma_{\mathfrak{C}_d}(h)_{zd} \cdot \overline{\rho}(t_z)\]
Dabei meint $\mathfrak{C}_d$ die eindeutige Linkszelle, die $d$ enthält.
\end{lemma}
\begin{proof}
Es gilt zunächst einmal für alle $w\in W$:
\[\psi(C_w) = \overline{\rho}(\phi(C_w)) = \sum_{\substack{z\in W, d\in\mathcal{D} \\ z \sim_\mathcal{LR} d}} n_d h_{w,d,z} \overline{\rho}(t_z)\]
Ist nun in einem der Summanden $h_{w,d,z}\neq 0$, dann ist $z \preceq_\mathcal{L} d$. Da jedoch auch $z \sim_\mathcal{LR} d$ gilt, folgt $z \sim_\mathcal{L} d$ (aufgrund von \citep[2.5.9]{geckjacon} und $(\spadesuit)$). 

Aus $(\vardiamond)$ und \citep[1.6.19]{geckjacon} folgt $n_d t_d \cdot t_z = t_z$ für alle $z\sim_\mathcal{R} d$, d.\,h. $\overline{\rho}(t_d)\neq 0$ genau dann, wenn $\overline{\rho}(t_z)\neq 0$ für alle $z \sim_\mathcal{R} d$ gilt.

Damit reduziert sich unsere Summe zu
\[\psi(C_w) = \sum_{\substack{d\in\mathcal{D} \\ \overline{\rho}(t_d)\neq 0}} \sum_{z\in\mathfrak{C}_d} n_d h_{w,d,z} \overline{\rho}(t_z).\]

Nun ist aber per definitionem
\[C_w C_d = \sum_z h_{w,d,z} C_z\]
d.\,h. $h_{w,d,z}$ ist der Eintrag an der Position $(z,d)$ der Darstellungsmatrix von $C_w$ auf dem regulären Linksmodul $H$. Wenn $z$ und $d$ aus der Linkszelle $\mathfrak{C}_d$ sind, stimmt dies mit dem Eintrag an Position $(z,d)$ des Linkszellmoduls von $\mathfrak{C}_d$ überein, d.\,h. mit $\sigma(C_w)_{zd}$. Wir erhalten:
\[\psi(C_w) = \sum_{\substack{d\in\mathcal{D} \\ \overline{\rho}(t_d)\neq 0}} \sum_{z\in\mathfrak{C}_d} n_d \sigma_{\mathfrak{C}_d}(C_w)_{zd} \overline{\rho}(t_z)\]
Da die $C_w$ eine Basis von $H$ sind, folgt die Behauptung.
\end{proof}

\begin{remark}
\index{terms}{Vermutung!von Geck-Müller}
Wollen wir also $\psi$ und $\rho$ vergleichen (etwa um die Vermutung von Geck und Müller über Zellularität von $W$"~Graph-Darstellungen zu überprüfen), dann müssen wir $\rho(T_s)$ mit $\psi(T_s)$ vergleichen. Dazu brauchen wir alle $\sigma_{\mathfrak{C}_d}(T_s)$, d.\,h. wir müssen die $\mu_{y,w}^s$ für alle $y,w\in\mathfrak{C}_d$ berechnen. Das geht mit Algorithmus \ref{algo:KL_poly_mu}.
\end{remark}
\begin{remark}
Wenn wir nur an einer Darstellung interessiert sind, ist es nicht effizient, ganz $\mathcal{D}$ zu berechnen und daraus diejenigen $d$ mit $\overline{\rho}(t_d)\neq 0$ auszusondern, denn für die Berechnung von $\mathcal{D}$ müssen wir entweder alle Isomorphietypen von Darstellungen durchlaufen und eine Summe für jedes $w\in W$ berechnen (siehe \ref{J_alg:def:J_algebra}) oder die Kazhdan-Lusztig-Polynome $P_{1,z}^\ast$ für alle $z\in W$ berechnen (siehe \ref{KL:def:Lusztig_a}).

Besser ist es, eine Vorauswahl zu treffen und erst dann die Kazhdan-Lusztig-Polynome heranzuziehen. Dafür dient das folgende Lemma.
\end{remark}

\begin{lemma}
\index{terms}{Kazhdan-Lusztig!-Zellen}
Es gelten die Lusztig"=Vermutungen. Sei $(W,S,L)$ eine Coxeter"=Gruppe"=mit"=Gewicht, $H=H(W,S,L)$ ihre Hecke"=Algebra, $F:=\IQ_W$, $K:=F(\Gamma)$ und $\rho:KH\to K^{d_\lambda\times d_\lambda}$ eine irreduzible, balancierte Matrixdarstellung von $H$ vom Isomorphietyp $\lambda$.

Sei $d\in\mathcal{D}$ mit $\overline{\rho}(t_d)\neq 0$. Dann gilt:
\begin{enumerate}
	\item $L(d) - a_\lambda \in 2\Gamma$.
	\item $d^2=1$.
	\item $n_d t_d$ ist ein Idempotent mit projektiven Moduln ${J t_d = \operatorname{span}\Set{t_w | w \sim_\mathcal{L} d}}$ bzw. ${t_d J = \operatorname{span}\Set{t_w | w \sim_\mathcal{R} d}}$.
	\item Ist $w\in W$ ein Element derart, dass $w^2=1$ gilt sowie $\overline{\rho}(t_w)$ ungleich Null und $\eta_w:=\pm\overline{\rho}(t_w)$ für eine Wahl des Vorzeichens ein Idempotent ist, dann gilt mit diesen Bezeichnungen
	\begin{enumerate}
		\item $\eta_w \eta_d = \begin{cases} 0 & \text{falls}\,w\not\sim_\mathcal{L} d \\ \eta_w & \text{falls}\,w\sim_\mathcal{L}d \end{cases}$ und
		\item $\eta_d \eta_w = \eta_w \eta_d$.
	\end{enumerate}
\end{enumerate}
\end{lemma}
\begin{proof}
a. Ist $d\in\mathcal{D}$ und $\overline{\rho}(t_d)\neq 0$, so gilt laut Definition $\textbf{a}(d)=\Delta(d)$ und $P_{1,d}^\ast\neq 0$. Aus \citep[2.3.14]{geckjacon} folgt $\textbf{a}(d)=a_\lambda$, aus der Definition $\Delta(d)=\nu(\overline{P_{1,d}^\ast})$. Nun sagt aber Lemma \ref{KL:lemma:symmetry_KL_poly}, dass alle nichtverschwindenden Koeffizienten in $P_{1,d}^\ast$ zu Potenzen gehören müssen, die dieselbe Parität wie $L(d)+L(1)=L(d)$ haben. Der Koeffizient vor $v^{\Delta(d)}$ ist ungleich Null, also gilt $a_\lambda = \Delta(d) \equiv L(d) \mod 2\Gamma$.

\medbreak
b. ist \textbf{P6}.

\medbreak
c. folgt aus \citep[1.6.19]{geckjacon}.

\medbreak
d. Wenn $w$ ebenfalls eine Involution ist, gilt
\begin{align*}
	t_w \cdot n_d t_d = t_w &\iff w \sim_\mathcal{L} d \\
	&\iff w=w^{-1} \sim_\mathcal{R} d^{-1}=d \\
	&\iff t_w = n_d t_d \cdot t_w \\
\intertext{sowie analog}
	t_w \cdot n_d t_d = 0 &\iff w \not\sim_\mathcal{L} d \\
	&\iff w \not\sim_\mathcal{R} d \\
	&\iff 0 = n_d t_d \cdot t_w. \qedhere
\end{align*}
\end{proof}

\subsection{Implementierungen}

\begin{convention}
Wir treffen erneut ein paar Vereinbarungen an unsere Programmierumgebung:
\begin{itemize}
	\item Das Hecke"=Algebra"=Objekt soll via \lstinline!H.DistinguishedInvolutions! eine Teilmenge von $W$ verwalten, welche die bereits ermittelten $d\in\mathcal{D}$ speichert.
\end{itemize}
\end{convention}

\begin{algorithm}[Duflo-Involutionen]\label{algo:dist_inv}
Funktion \textbf{GetDistinguishedInvolutions}:
\begin{itemize}
	\item Input-Parameter:
	\begin{itemize}
		\item rho: \lstinline!Representation!
		
		Eine Matrixdarstellung $\rho: H\to K^{d\times d}$.
	\end{itemize}
	\item Rückgabe-Wert: \lstinline!SetOfElementsOfW!
	
	$\Set{d\in\mathcal{D} | \overline{\rho}(t_d)\neq 0}$.
\end{itemize}
\end{algorithm}
\begin{lstlisting}[mathescape=true]
function GetDistinguishedInvolutions(Representation rho)
	returns SetOfElementsOfW
{
	ComputeLeadingCoefficients(rho);

	D := $\Set{z | \overline{\rho}(t_z)\neq 0}$;
	
	if(H.DistinguishedInvolutions$\cap D \neq \emptyset$){
		return H.DistinguishedInvolutions$\cap D$;
	}

	ListOfMatrices IdemMat;
	ListOfSetsOfElementsOfW FibMat;

	Matrix E;
	Matrix F;
	
	Integer n;
	
	foreach(z in D){
		if(z*z != 1 or 1/2*(Weight(z)-rho.aValue) $\notin\Gamma$){
			continue;
		}

		E := GetLeadingCoefficient(rho,z);
		F := E*E;
		
		if(E == F){
			// Nichts tun
		}elif(E == -F){
			E := -E;
		}else{
			continue;
		}

		if(E in IdemMat){
			i := Position of E in IdemMat;
			AddToSet(FibMat[i],z);
		}else{
			AddToList(IdemMat,E);
			AddToList(FibMat,$\Set{z}$);
		}
	}


	for(i := 0; i < Length(IdemMat); i := i+1){
		// In diesem Fall wurde $\eta_z$ bereits ausgeschlossen.
		if(IsEmpty(IdemMat[i])){
			continue;
		}

		E := IdemMat[i];
			
		for(j := i+1; j<Length(IdemMat); j := j+1){
			// In diesem Fall wurde $\eta_w$ bereits ausgeschlossen.
			if(IsEmpty(IdemMat[j])){
				continue;
			}

			F := IdemMat[j];

			A := E*F;
			if(A != F*E){
				Empty(IdemMat[i]);
				Empty(FibMat[i]);
				
				Empty(IdemMat[j]);
				Empty(FibMat[j]);
				
				break;
			}

			if(A != 0){
				if(A != F){
					Empty(IdemMat[i]);
					Empty(FibMat[i]);
				
					break;
				}

				if(A != E){
					Empty(IdemMat[j]);
					Empty(FibMat[j]);
				}
			
			}
		}
	}


	D := [];
	
	for(i := 0; i<Length(IdemMat); i := i+1){
		foreach(z in FibMat[ii]){
			if( rho.aValue = Valuation(BarInvolution(KLPoly(1,z))) ){
				AddToSet(D,z);

				// Lusztigs Definition von $n_z$
				n := LowestTerm(BarInvolution(KLPoly(1,z)));
				// Gecks Definition von $\tilde{n}_z$
				H.nValue[z] := (-1)^Length(z)*n;
			}
		}
	}

	H.DistinguishedInvolutions := H.DistinguishedInvolutions$\cup D$;

	return D;
}
\end{lstlisting}

\begin{algorithm}[Zellmoduln nach Geck]\label{algo:cell_rep}
Funktion \textbf{GetCellRepresentation}:
\begin{itemize}
	\item Input-Parameter:
	\begin{itemize}
		\item rho: \lstinline!Representation!
		
		Eine balancierte Matrixdarstellung $\rho: H\to K^{d\times d}$. Falls $\rho$ nicht balanciert ist, wird ein Fehler ausgegeben.
	\end{itemize}
	\item Rückgabe-Wert: \lstinline!Representation!
	
	Die Matrixdarstellung $\psi:=\overline{\rho}\circ\phi$ des Zellmoduls, der von $\rho$ definiert wird.
\end{itemize}
\end{algorithm}
\begin{lstlisting}[mathescape=true]
function GetCellRepresentation(Representation rho)
	returns Representation
{
	Representation psi := Copy(rho);
	Representation sigma;

	SetOfElementsOfW DcapF := GetDistinguishedInvolutions(psi);
	SetOfElementsOfW C;

	Matrix A;

	// Neue Matrizen anlegen.
	foreach(s in CoxeterGenerators(W)){
		psi.Generators[s] := 0;
	}

	foreach(d in DcapF){
		C := LeftCell(d);
		sigma := LeftCellRepresentation(C);

		foreach(z in $\Set{z\in C | \overline{\rho}(t_z)\neq 0}$){
			// GetLeadingCoefficient erzeugt einen Fehler,
			// wenn $\rho$ nicht balanciert ist.
			A := H.nValue[d]*GetLeadingCoefficient(rho,z);

			foreach(s in CoxeterGenerators(W)){
				psi.Generators[s] := psi.Generators[s]
				                     + sigma.Generators[s][z,d] * A;
			}
		}
	}

	return psi;
}
\end{lstlisting}

\index{terms}{Zellmodul|)}


\indexprologue{Die Seitenzahlen verweisen auf Orte, an denen das jeweilige Symbol (erneut) definiert wurde.}
\printindex[symbols]
\printindex[terms]

\bibliography{diss}

\backmatter
\pagenumbering{gobble} 
\addchap{Ehrenwörtliche Erklärung}

Hiermit erkläre ich,
\begin{itemize}
	\item dass mir die Promotionsordnung der Fakultät bekannt ist,
	\item dass ich die Dissertation selbst angefertigt habe, keine Textabschnitte oder Ergebnisse eines Dritten oder eigenen Prüfungsarbeiten ohne Kennzeichnung übernommen und alle von mir benutzten Hilfsmittel, persönlichen Mitteilungen und Quellen in meiner Arbeit angegeben habe,
	\item dass ich die Hilfe eines Promotionsberaters nicht in Anspruch genommen habe und dass Dritte weder unmittelbar noch mittelbar geldwerte Leistungen von mir für Arbeiten erhalten haben, die im Zusammenhang mit dem Inhalt der vorgelegten Dissertation stehen und
	\item dass ich die Dissertation noch nicht als Prüfungsarbeit für eine staatliche oder andere wissenschaftliche Prüfung eingereicht habe.
\end{itemize}

Bei der Auswahl und Auswertung des Materials sowie bei der Herstellung des Manuskripts haben mich folgende Personen unterstützt: Prof. Dr. Burkhard Külshammer, PD Dr. Jürgen Müller.

Ich habe weder die gleiche, noch eine in wesentlichen Teilen ähnliche, noch eine andere Abhandlung bereits bei einer anderen Hochschule als Dissertation eingereicht.

\vspace{2cm}
Jena, den 

\end{document}